\documentclass[11pt,reqno,a4paper]{amsart}
\usepackage[T1]{fontenc}
\usepackage{lmodern}
\usepackage[margin=22mm]{geometry}
\usepackage{amsmath,amssymb,amsthm,mathtools}
\usepackage{microtype,graphicx,xcolor,booktabs,array}
\usepackage{tikz,float,needspace}
\usetikzlibrary{arrows.meta,decorations.markings,decorations.pathreplacing,calc}
\usepackage[hidelinks]{hyperref}
\usepackage{xurl}
\allowdisplaybreaks[2]
\numberwithin{equation}{section}
\newtheorem{theorem}{Theorem}[section]
\newtheorem{proposition}[theorem]{Proposition}
\newtheorem{lemma}[theorem]{Lemma}
\newtheorem{corollary}[theorem]{Corollary}
\newtheorem{conjecture}[theorem]{Conjecture}
\theoremstyle{definition}
\newtheorem{definition}[theorem]{Definition}
\newtheorem{notation}[theorem]{Notation}

\newtheorem{remark}[theorem]{Remark}

\newtheorem*{ac}{Acknowledgements}
\newcommand{\cC}{\mathcal C}

\newcommand{\Hom}{\operatorname{Hom}}
\newcommand{\id}{\operatorname{id}}

\newcommand{\tpl}{\stackrel{\mathrm{TPL}}{=}}
\newcommand{\tpr}{\stackrel{\mathrm{TPR}}{=}}

\newcommand{\TPL}{\operatorname{TPL}}
\newcommand{\TPR}{\operatorname{TPR}}
\newif\ifabeliangreen
\abeliangreentrue
\definecolor{AbelianGreen}{RGB}{0,115,55}

\usepackage{longtable}
\usetikzlibrary{arrows.meta,decorations.markings,calc}
\tikzset{
 sd/base/.style={line width=.48pt,line cap=butt,line join=round,
   every node/.style={font=\normalfont\fontsize{8}{9.2}\selectfont,inner sep=0pt,outer sep=0pt}},
 sd/box/.style={draw,fill=white,inner xsep=2pt,inner ysep=1.6pt,outer sep=0pt,minimum height=4mm},
 sd/string/.style={line width=.48pt,line cap=butt},
 sd/label/.style={inner sep=0pt,outer sep=0pt},
 sd/edge-label/.style={inner sep=0pt,outer sep=0pt},
 sd/arrow/.style={postaction={decorate},decoration={markings,
    mark=at position .54 with {\arrow{Stealth[length=2.5pt,width=2.3pt]}}}},
 sd/reverse/.style={postaction={decorate},decoration={markings,
    mark=at position .54 with {\arrowreversed{Stealth[length=2.5pt,width=2.3pt]}}}}
}
\newcommand{\SD}[1]{\ensuremath{\vcenter{\hbox{\csname NGD@#1\endcsname}}}}

\tikzset{ng diagram/.style={line width=.45pt,line cap=round,line join=round,>={Stealth[length=3pt,width=2.6pt]},every node/.append style={font=\normalfont\fontsize{8}{9.5}\selectfont,inner sep=1.5pt}}}
\newcommand{\NGDiagram}[1]{\ensuremath{\vcenter{\hbox{\csname ngdiagram@#1\endcsname}}}}

\newcommand{\NGYBase}{\draw (.75,.75)--(0,0)--(-.75,.75);\draw (0,0)--(0,-.75);}

\pgfkeys{/ng/labels/.is family,/ng/labels,
1/.initial=,
2/.initial=,
3/.initial=,
4/.initial=,
5/.initial=,
6/.initial=,
7/.initial=,
8/.initial=,
9/.initial=,
10/.initial=,
11/.initial=,
12/.initial=,
13/.initial=,
14/.initial=,
15/.initial=}
\newcommand{\NGLabel}[1]{\pgfkeysvalueof{/ng/labels/#1}}
\newcommand{\NGShape}[2]{\begingroup\pgfqkeys{/ng/labels}{#2}%
  \csname ngshape@#1\endcsname\endgroup}

\expandafter\def\csname ngdiagram@D001\endcsname{%
  \NGShape{P001}{1={i_1},2={i_2},3={i_3},4={i_4},5={i_5},6={i_6},7={i_7},8={i_8},9={i_9},10={\alpha_1},11={\alpha_2},12={\alpha_3},13={\alpha_4},14={\alpha_5},15={\alpha_6}}%
}

\expandafter\def\csname ngdiagram@D002\endcsname{%
  \NGShape{O001}{1={\beta},2={\alpha},3={a},4={b},5={c}}%
}

\expandafter\def\csname ngdiagram@D003\endcsname{%
  \NGShape{O002}{1={a},2={b},3={c}}%
}

\expandafter\def\csname ngdiagram@D004\endcsname{%
  \NGShape{O003}{1={a},2={g}}%
}

\expandafter\def\csname ngdiagram@D005\endcsname{%
  \NGShape{O004}{1={g},2={g},3={a},4={ag},5={a}}%
}

\expandafter\def\csname ngdiagram@D006\endcsname{%
  \NGShape{O005}{1={g},2={h},3={k},4={},5={g{+}h},6={k{-}h},7={g'},8={h'},9={k'}}%
}

\expandafter\def\csname ngdiagram@D007\endcsname{%
  \NGShape{O006}{1={g{+}k},2={g},3={k},4={g'},5={k'}}%
}

\expandafter\def\csname ngdiagram@D008\endcsname{%
  \NGShape{O007}{1={g},2={h}}%
}

\expandafter\def\csname ngdiagram@D009\endcsname{%
  \NGShape{O008}{1={g{+}h}}%
}

\expandafter\def\csname ngdiagram@D010\endcsname{%
  \NGShape{T001}{1={g{+}h},2={g{+}k},3={g},4={h},5={k},6={k{-}h}}%
}

\expandafter\def\csname ngdiagram@D011\endcsname{%
  \NGShape{T002}{1={g{+}h},2={g},3={h}}%
}

\expandafter\def\csname ngdiagram@D012\endcsname{%
  \NGShape{O009}{1={g},2={h},3={g{+}h}}%
}

\expandafter\def\csname ngdiagram@D013\endcsname{%
  \NGShape{O010}{1={g{+}h},2={g{+}h}}%
}

\expandafter\def\csname ngdiagram@D014\endcsname{%
  \NGShape{O011}{1={g}}%
}

\expandafter\def\csname ngdiagram@D015\endcsname{%
  \NGShape{O012}{1={\varrho},2={\varrho},3={\varrho},4={g},5={h}}%
}

\expandafter\def\csname ngdiagram@D016\endcsname{%
  \NGShape{O013}{1={\varrho},2={\varrho},3={\varrho},4={g},5={h}}%
}

\expandafter\def\csname ngdiagram@D017\endcsname{%
  \NGShape{T003}{1={g},2={h}}%
}

\expandafter\def\csname ngdiagram@D018\endcsname{%
  \NGShape{P002}{1={g{+}h},2={h},3={g},4={k}}%
}

\expandafter\def\csname ngdiagram@D019\endcsname{%
  \NGShape{T004}{1={g{+}h},2={g},3={h}}%
}

\expandafter\def\csname ngdiagram@D020\endcsname{%
  \NGShape{T005}{1={g{+}h},2={k}}%
}

\expandafter\def\csname ngdiagram@D021\endcsname{%
  \NGShape{T006}{1={k},2={g}}%
}

\expandafter\def\csname ngdiagram@D022\endcsname{%
  \NGShape{T007}{1={h},2={k}}%
}

\expandafter\def\csname ngdiagram@D023\endcsname{%
  \NGShape{T008}{1={g{+}h},2={g},3={h}}%
}

\expandafter\def\csname ngdiagram@D024\endcsname{%
  \NGShape{T003}{1={h},2={g}}%
}

\expandafter\def\csname ngdiagram@D025\endcsname{%
  \NGShape{T009}{1={g},2={h}}%
}

\expandafter\def\csname ngdiagram@D026\endcsname{%
  \NGShape{O014}{1={\alpha}}%
}

\expandafter\def\csname ngdiagram@D027\endcsname{%
  \NGShape{O015}{1={g},2={\alpha}}%
}

\expandafter\def\csname ngdiagram@D028\endcsname{%
  \NGShape{O016}{1={g{+}h},2={g},3={h}}%
}

\expandafter\def\csname ngdiagram@D029\endcsname{%
  \NGShape{O017}{1={g{+}h}}%
}

\expandafter\def\csname ngdiagram@D030\endcsname{%
  \NGShape{O018}{1={g},2={h},3={\alpha}}%
}

\expandafter\def\csname ngdiagram@D031\endcsname{%
  \NGShape{O019}{1={\alpha},2={g{+}h},3={g},4={g},5={h},6={h}}%
}

\expandafter\def\csname ngdiagram@D032\endcsname{%
  \NGShape{O020}{1={g{+}h},2={\alpha}}%
}

\expandafter\def\csname ngdiagram@D033\endcsname{%
  \NGShape{O021}{1={\alpha},2={g}}%
}

\expandafter\def\csname ngdiagram@D034\endcsname{%
  \NGShape{O022}{1={\alpha},2={g}}%
}

\expandafter\def\csname ngdiagram@D035\endcsname{%
  \NGShape{O023}{1={h},2={g},3={\alpha}}%
}

\expandafter\def\csname ngdiagram@D036\endcsname{%
  \NGShape{O024}{1={h},2={g},3={\alpha}}%
}

\expandafter\def\csname ngdiagram@D037\endcsname{%
  \NGShape{O025}{1={g},2={h},3={\alpha}}%
}

\expandafter\def\csname ngdiagram@D038\endcsname{%
  \NGShape{O026}{1={g},2={h},3={\alpha}}%
}

\expandafter\def\csname ngdiagram@D039\endcsname{%
  \NGShape{O027}{1={h},2={g},3={\alpha}}%
}

\expandafter\def\csname ngdiagram@D040\endcsname{%
  \NGShape{O028}{1={h},2={g},3={\alpha}}%
}

\expandafter\def\csname ngdiagram@D041\endcsname{%
  \NGShape{O029}{1={h}}%
}

\expandafter\def\csname ngdiagram@D042\endcsname{%
  \NGShape{O030}{1={h},2={k}}%
}

\expandafter\def\csname ngdiagram@D043\endcsname{%
  \NGShape{O031}{1={l},2={h},3={k}}%
}

\expandafter\def\csname ngdiagram@D044\endcsname{%
  \NGShape{O032}{1={l},2={l},3={h},4={k}}%
}

\expandafter\def\csname ngdiagram@D045\endcsname{%
  \NGShape{O033}{1={h},2={k}}%
}

\expandafter\def\csname ngdiagram@D046\endcsname{%
  \NGShape{O030}{1={h},2={{-}h}}%
}

\expandafter\def\csname ngdiagram@D047\endcsname{%
  \NGShape{O034}{1={g},2={g},3={h},4={{-}h}}%
}

\expandafter\def\csname ngdiagram@D048\endcsname{%
  \NGShape{O035}{1={h{-}g},2={{-}h{+}g}}%
}

\expandafter\def\csname ngdiagram@D049\endcsname{%
  \NGShape{O036}{1={h},2={{-}h},3={1},4={1}}%
}

\expandafter\def\csname ngdiagram@D050\endcsname{%
  \NGShape{O033}{1={0},2={0}}%
}

\expandafter\def\csname ngdiagram@D051\endcsname{%
  \NGShape{O034}{1={h},2={h},3={h},4={{-}h}}%
}

\expandafter\def\csname ngdiagram@D052\endcsname{%
  \NGShape{O030}{1={0},2={0}}%
}

\expandafter\def\csname ngdiagram@D053\endcsname{%
  \NGShape{O037}{1={g}}%
}

\expandafter\def\csname ngdiagram@D054\endcsname{%
  \NGShape{O038}{1={h}}%
}

\expandafter\def\csname ngdiagram@D055\endcsname{%
  \NGShape{O038}{1={g}}%
}

\expandafter\def\csname ngdiagram@D056\endcsname{%
  \NGShape{O037}{1={0}}%
}

\expandafter\def\csname ngdiagram@D057\endcsname{%
  \NGShape{O039}{1={g},2={0}}%
}

\expandafter\def\csname ngdiagram@D058\endcsname{%
  \NGShape{O040}{1={g},2={h}}%
}

\expandafter\def\csname ngdiagram@D059\endcsname{%
  \NGShape{O038}{1={h{-}g}}%
}

\expandafter\def\csname ngdiagram@D060\endcsname{%
  \NGShape{O038}{1={g}}%
}

\expandafter\def\csname ngdiagram@D061\endcsname{%
  \NGShape{O041}{1={h}}%
}

\expandafter\def\csname ngdiagram@D062\endcsname{%
  \NGShape{O037}{1={k}}%
}

\expandafter\def\csname ngdiagram@D063\endcsname{%
  \NGShape{O042}{1={g},2={h},3={k}}%
}

\expandafter\def\csname ngdiagram@D064\endcsname{%
  \NGShape{O043}{1={g},2={h},3={k},4={l}}%
}

\expandafter\def\csname ngdiagram@D065\endcsname{%
  \NGShape{O044}{1={g},2={h},3={k},4={l},5={l},6={l}}%
}

\expandafter\def\csname ngdiagram@D066\endcsname{%
  \NGShape{O045}{1={0},2={g}}%
}

\expandafter\def\csname ngdiagram@D067\endcsname{%
  \NGShape{O046}{1={h},2={h'}}%
}

\expandafter\def\csname ngdiagram@D068\endcsname{%
  \NGShape{O047}{1={h}}%
}

\expandafter\def\csname ngdiagram@D069\endcsname{%
  \NGShape{O046}{1={k},2={h'}}%
}

\expandafter\def\csname ngdiagram@D070\endcsname{%
  \NGShape{O047}{1={l}}%
}

\expandafter\def\csname ngdiagram@D071\endcsname{%
  \NGShape{O046}{1={h},2={k'}}%
}

\expandafter\def\csname ngdiagram@D072\endcsname{%
  \NGShape{T010}{1={0},2={g},3={k'},4={h}}%
}

\expandafter\def\csname ngdiagram@D073\endcsname{%
  \NGShape{O048}{1={l}}%
}

\expandafter\def\csname ngdiagram@D074\endcsname{%
  \NGShape{O045}{1={0},2={g}}%
}

\expandafter\def\csname ngdiagram@D075\endcsname{%
  \NGShape{O046}{1={0},2={0}}%
}

\expandafter\def\csname ngdiagram@D076\endcsname{%
  \NGShape{O049}{1={0},2={g},3={h}}%
}

\expandafter\def\csname ngdiagram@D077\endcsname{%
  \NGShape{O050}{1={h},2={h'},3={h}}%
}

\expandafter\def\csname ngdiagram@D078\endcsname{%
  \NGShape{O051}{1={h},2={h'},3={h},4={h}}%
}

\expandafter\def\csname ngdiagram@D079\endcsname{%
  \NGShape{T011}{1={g},2={h},3={k'}}%
}

\expandafter\def\csname ngdiagram@D080\endcsname{%
  \NGShape{T012}{1={g},2={r},3={h'},4={p}}%
}

\expandafter\def\csname ngdiagram@D081\endcsname{%
  \NGShape{T013}{1={h},2={k'}}%
}

\expandafter\def\csname ngdiagram@D082\endcsname{%
  \NGShape{T014}{1={0},2={k},3={h'},4={h}}%
}

\expandafter\def\csname ngdiagram@D083\endcsname{%
  \NGShape{T015}{1={g},2={r},3={h'},4={p},5={g},6={g},7={g}}%
}

\expandafter\def\csname ngdiagram@D084\endcsname{%
  \NGShape{O052}{}%
}

\expandafter\def\csname ngdiagram@D085\endcsname{%
  \NGShape{O053}{}%
}

\expandafter\def\csname ngdiagram@D086\endcsname{%
  \NGShape{P003}{1={g},2={h},3={g{+}h},4={k},5={k{-}h},6={g{+}k},7={l},8={h{-}l},9={g{+}l}}%
}

\expandafter\def\csname ngdiagram@D087\endcsname{%
  \NGShape{T016}{1={g{+}h},2={g},3={g{+}k},4={k{-}h},5={k},6={h}}%
}

\expandafter\def\csname ngdiagram@D088\endcsname{%
  \NGShape{T017}{1={g{+}h},2={g{+}l},3={g},4={h},5={l},6={h{-}l}}%
}

\expandafter\def\csname ngdiagram@D089\endcsname{%
  \NGShape{T018}{1={g},2={l},3={k},4={g{+}k},5={k{-}l},6={g{+}l}}%
}

\expandafter\def\csname ngdiagram@D090\endcsname{%
  \NGShape{T019}{1={h},2={h{-}l},3={k{-}h},4={k},5={k{-}l},6={l}}%
}

\expandafter\def\csname ngdiagram@D091\endcsname{%
  \NGShape{T020}{1={g{+}h},2={g{+}l},3={g{+}k},4={k{-}h},5={h{-}l},6={k{-}l}}%
}

\expandafter\def\csname ngdiagram@D092\endcsname{%
  \NGShape{P004}{1={g},2={h},3={g{+}h},4={k},5={h{-}k},6={g{+}k}}%
}

\expandafter\def\csname ngdiagram@D093\endcsname{%
  \NGShape{T021}{1={g{+}h},2={g},3={h}}%
}

\expandafter\def\csname ngdiagram@D094\endcsname{%
  \NGShape{T022}{1={g{+}h},2={g{+}k},3={g},4={h},5={k},6={h{-}k}}%
}

\expandafter\def\csname ngdiagram@D095\endcsname{%
  \NGShape{T023}{1={g},2={k},3={g{+}k}}%
}

\expandafter\def\csname ngdiagram@D096\endcsname{%
  \NGShape{T024}{1={h},2={h{-}k},3={k}}%
}

\expandafter\def\csname ngdiagram@D097\endcsname{%
  \NGShape{T025}{1={g{+}h},2={h{-}k},3={g{+}k}}%
}

\expandafter\def\csname ngdiagram@D098\endcsname{%
  \NGShape{P005}{1={g{+}h},2={h},3={g},4={k},5={g{+}h{+}k}}%
}

\expandafter\def\csname ngdiagram@D099\endcsname{%
  \NGShape{T008}{1={g{+}h},2={g},3={h}}%
}

\expandafter\def\csname ngdiagram@D100\endcsname{%
  \NGShape{T026}{1={g{+}h},2={g{+}h{+}k},3={k}}%
}

\expandafter\def\csname ngdiagram@D101\endcsname{%
  \NGShape{T027}{1={g},2={g{+}h{+}k},3={h{+}k}}%
}

\expandafter\def\csname ngdiagram@D102\endcsname{%
  \NGShape{T028}{1={k},2={h},3={h{+}k}}%
}

\expandafter\def\csname ngdiagram@D103\endcsname{%
  \NGShape{T029}{1={g{+}h},2={g{+}h{+}k},3={g},4={h},5={k},6={k{+}h}}%
}

\expandafter\def\csname ngdiagram@D104\endcsname{%
  \NGShape{P006}{1={g{+}h},2={h},3={g},4={k}}%
}

\expandafter\def\csname ngdiagram@D105\endcsname{%
  \NGShape{P007}{1={g{+}h},2={h},3={g},4={k},5={l'}}%
}

\expandafter\def\csname ngdiagram@D106\endcsname{%
  \NGShape{P008}{1={h},2={g},3={k}}%
}

\expandafter\def\csname ngdiagram@D107\endcsname{%
  \NGShape{P009}{1={g},2={h},3={g{+}h}}%
}

\expandafter\def\csname ngdiagram@D108\endcsname{%
  \NGShape{P010}{1={g{+}h},2={h},3={g},4={k},5={l'}}%
}

\expandafter\def\csname ngdiagram@D109\endcsname{%
  \NGShape{T030}{1={g{+}h},2={l'},3={k}}%
}

\expandafter\def\csname ngdiagram@D110\endcsname{%
  \NGShape{T011}{1={g},2={s},3={l'}}%
}

\expandafter\def\csname ngdiagram@D111\endcsname{%
  \NGShape{T011}{1={h},2={k},3={s'}}%
}

\expandafter\def\csname ngdiagram@D112\endcsname{%
  \NGShape{P011}{1={h},2={g},3={k}}%
}

\expandafter\def\csname ngdiagram@D113\endcsname{%
  \NGShape{T005}{1={h},2={g}}%
}

\expandafter\def\csname ngdiagram@D114\endcsname{%
  \NGShape{T003}{1={h},2={k}}%
}

\expandafter\def\csname ngdiagram@D115\endcsname{%
  \NGShape{T031}{1={g{-}k},2={g},3={k}}%
}

\expandafter\def\csname ngdiagram@D116\endcsname{%
  \NGShape{T032}{1={h},2={k{-}g}}%
}

\expandafter\def\csname ngdiagram@D117\endcsname{%
  \NGShape{T033}{1={k},2={g},3={g{-}k}}%
}

\expandafter\def\csname ngdiagram@D118\endcsname{%
  \NGShape{P012}{1={g},2={h},3={k},4={l'}}%
}

\expandafter\def\csname ngdiagram@D119\endcsname{%
  \NGShape{P013}{1={g},2={h},3={k'},4={l}}%
}

\expandafter\def\csname ngdiagram@D120\endcsname{%
  \NGShape{P014}{1={g},2={h},3={k},4={l'}}%
}

\expandafter\def\csname ngdiagram@D121\endcsname{%
  \NGShape{P015}{1={g},2={h},3={k},4={l'}}%
}

\expandafter\def\csname ngdiagram@D122\endcsname{%
  \NGShape{T011}{1={g},2={k},3={s'}}%
}

\expandafter\def\csname ngdiagram@D123\endcsname{%
  \NGShape{T011}{1={h},2={s},3={l'}}%
}

\expandafter\def\csname ngdiagram@D124\endcsname{%
  \NGShape{T011}{1={g{+}h},2={k},3={l'}}%
}

\expandafter\def\csname ngdiagram@D125\endcsname{%
  \NGShape{T034}{1={h},2={g},3={g{+}h}}%
}

\expandafter\def\csname ngdiagram@D126\endcsname{%
  \NGShape{P016}{1={g},2={h},3={k'},4={l}}%
}

\expandafter\def\csname ngdiagram@D127\endcsname{%
  \NGShape{T035}{1={g},2={s},3={k'}}%
}

\expandafter\def\csname ngdiagram@D128\endcsname{%
  \NGShape{T011}{1={h},2={l},3={s'}}%
}

\expandafter\def\csname ngdiagram@D129\endcsname{%
  \NGShape{T011}{1={h},2={s},3={k'}}%
}

\expandafter\def\csname ngdiagram@D130\endcsname{%
  \NGShape{T035}{1={g},2={l},3={s'}}%
}

\expandafter\def\csname ngdiagram@D131\endcsname{%
  \NGShape{P017}{1={g},2={h},3={k},4={l'}}%
}

\expandafter\def\csname ngdiagram@D132\endcsname{%
  \NGShape{T035}{1={g},2={k},3={l'}}%
}

\expandafter\def\csname ngdiagram@D133\endcsname{%
  \NGShape{T036}{1={h},2={k},3={s'}}%
}

\expandafter\def\csname ngdiagram@D134\endcsname{%
  \NGShape{T035}{1={g},2={s},3={t'}}%
}

\expandafter\def\csname ngdiagram@D135\endcsname{%
  \NGShape{T037}{1={h},2={t},3={l'}}%
}

\expandafter\def\csname ngdiagram@D136\endcsname{%
  \NGShape{P018}{1={g},2={h_1},3={h_2},4={h_3},5={k'}}%
}

\expandafter\def\csname ngdiagram@D137\endcsname{%
  \NGShape{P019}{1={g},2={h_1},3={0},4={h_3{+}h_2},5={(k{+}h_2)'}}%
}

\expandafter\def\csname ngdiagram@D138\endcsname{%
  \NGShape{T038}{1={h_1},2={s'},3={h_2},4={h_3}}%
}

\expandafter\def\csname ngdiagram@D139\endcsname{%
  \NGShape{T039}{1={g},2={k'},3={s}}%
}

\expandafter\def\csname ngdiagram@D140\endcsname{%
  \NGShape{T040}{1={h_1},2={s'},3={h_2},4={h_3},5={h_2},6={h_2},7={h_2}}%
}

\expandafter\def\csname ngdiagram@D141\endcsname{%
  \NGShape{T041}{1={g},2={k'},3={s},4={h_2},5={h_2},6={h_2}}%
}

\expandafter\def\csname ngdiagram@D142\endcsname{%
  \NGShape{T042}{1={h_1},2={(s{+}h_2)'},3={0},4={v}}%
}

\expandafter\def\csname ngdiagram@D143\endcsname{%
  \NGShape{T043}{1={g},2={u'},3={s{+}h_2}}%
}

\expandafter\def\csname ngdiagram@D144\endcsname{%
  \NGShape{T044}{1={g},2={h_2},3={s'}}%
}

\expandafter\def\csname ngdiagram@D145\endcsname{%
  \NGShape{T045}{1={h_3},2={k'},3={t},4={s}}%
}

\expandafter\def\csname ngdiagram@D146\endcsname{%
  \NGShape{T046}{1={g},2={h_1},3={t'}}%
}

\expandafter\def\csname ngdiagram@D147\endcsname{%
  \NGShape{T047}{1={},2={g},3={h_2},4={s'},5={h_2},6={h_2},7={h_2},8={}}%
}

\expandafter\def\csname ngdiagram@D148\endcsname{%
  \NGShape{T048}{1={h_3},2={k'},3={t},4={s},5={h_2},6={h_2},7={h_2}}%
}

\expandafter\def\csname ngdiagram@D149\endcsname{%
  \NGShape{T049}{1={g},2={0},3={w'}}%
}

\expandafter\def\csname ngdiagram@D150\endcsname{%
  \NGShape{T050}{1={v},2={u'},3={t},4={w}}%
}

\expandafter\def\csname ngdiagram@D151\endcsname{%
  \NGShape{P018}{1={g},2={0},3={0},4={0},5={k'}}%
}

\expandafter\def\csname ngdiagram@D152\endcsname{%
  \NGShape{P018}{1={g},2={0},3={0},4={0},5={k}}%
}

\expandafter\def\csname ngdiagram@D153\endcsname{%
  \NGShape{P020}{1={g},2={h},3={g{+}h}}%
}

\expandafter\def\csname ngdiagram@D154\endcsname{%
  \NGShape{P021}{1={g},2={k_1},3={k_2},4={k_3},5={k_4}}%
}

\expandafter\def\csname ngdiagram@D155\endcsname{%
  \NGShape{T051}{1={g{+}h},2={g},3={h}}%
}

\expandafter\def\csname ngdiagram@D156\endcsname{%
  \NGShape{T052}{1={g{+}h},2={g},3={h}}%
}

\expandafter\def\csname ngdiagram@D157\endcsname{%
  \NGShape{T035}{1={g},2={k},3={s'}}%
}

\expandafter\def\csname ngdiagram@D158\endcsname{%
  \NGShape{T035}{1={h},2={s},3={t'}}%
}

\expandafter\def\csname ngdiagram@D159\endcsname{%
  \NGShape{T053}{1={g{+}h},2={t},3={k'}}%
}

\expandafter\def\csname ngdiagram@D160\endcsname{%
  \NGShape{T003}{1={g},2={s}}%
}

\expandafter\def\csname ngdiagram@D161\endcsname{%
  \NGShape{T003}{1={h},2={s}}%
}

\expandafter\def\csname ngdiagram@D162\endcsname{%
  \NGShape{T054}{1={g{+}h},2={s}}%
}

\expandafter\def\csname ngdiagram@D163\endcsname{%
  \NGShape{O045}{1={0},2={{-}g}}%
}

\expandafter\def\csname ngdiagram@D164\endcsname{%
  \NGShape{P022}{1={g_1},2={g_2},3={g_3},4={h_1},5={h_2},6={h_3}}%
}

\expandafter\def\csname ngdiagram@D165\endcsname{%
  \NGShape{T055}{1={0},2={u},3={0},4={0}}%
}

\expandafter\def\csname ngdiagram@D166\endcsname{%
  \NGShape{T056}{1={g},2={u'},3={{-}h},4={v}}%
}

\expandafter\def\csname ngdiagram@D167\endcsname{%
  \NGShape{O054}{1={a},2={b},3={c}}%
}

\expandafter\def\csname ngdiagram@D168\endcsname{%
  \NGShape{O055}{1={g}}%
}

\expandafter\def\csname ngdiagram@D169\endcsname{%
  \NGShape{O056}{1={g^*}}%
}

\expandafter\def\csname ngshape@P001\endcsname{%
\begin{tikzpicture}[scale=1.125,ng diagram]
\draw (0,2)--(3,2);
\draw (1,1)--(2,1);
\draw (0,0)--(3,0);
\draw (0,0) -- (0,2) -- (1,1) -- (0,0);
\draw (3,0)--(3,2)--(2,1)--(3,0);
\draw[->] (0,2) -- (1.5,2) node [above] {$\NGLabel{1}$};
\draw[->] (1,1) -- (1.5,1) node [above] {$\NGLabel{2}$};
\draw[->] (0,0) -- (1.5,0) node [above] {$\NGLabel{3}$};
\draw[->] (0,2) --++ (.5,-.5) node [right] {$\NGLabel{4}$};
\draw[->] (1,1) --++ (-.5,-.5) node [right] {$\NGLabel{5}$};
\draw[->] (0,0) -- (0,1) node [left] {$\NGLabel{6}$};
\draw[->] (3,2) --++ (-.5,-.5) node [left] {$\NGLabel{7}$};
\draw[->] (2,1) --++ (.5,-.5) node [left] {$\NGLabel{8}$};
\draw[->] (3,0) -- (3,1) node [right] {$\NGLabel{9}$};
\node at (0+.15,2-.3) {$\NGLabel{10}$};
\node at (1-.2,1) {$\NGLabel{11}$};
\node at (0+.15,0+.3) {$\NGLabel{12}$};
\node at (3-.15,2-.3) {$\NGLabel{13}$};
\node at (2+.2,1) {$\NGLabel{14}$};
\node at (3-.15,0+.3) {$\NGLabel{15}$};
\end{tikzpicture}%
}

\expandafter\def\csname ngshape@O001\endcsname{%
\begin{tikzpicture}[scale=1.4,ng diagram]
\draw (0,0) circle [radius=.5];
\draw [->](0,.5) arc[radius=.5, start angle=90, end angle=180];
\draw [->](0,.5) arc[radius=.5, start angle=90, end angle=0];
\draw [->](0,.5)--(0,0);
\draw (0,-.5)--(0,.5);
\node at (-.12,.33)[anchor=east] {$\NGLabel{1}$};
\node at (-.12,-.33)[anchor=east] {$\NGLabel{2}$};
\node at (-.5,0)[left] {$\NGLabel{3}$};
\node at (0,0)[right] {$\NGLabel{4}$};
\node at (.5,0)[right] {$\NGLabel{5}$};
\end{tikzpicture}%
}

\expandafter\def\csname ngshape@O002\endcsname{%
\begin{tikzpicture}[scale=1.4,ng diagram]
\draw (-.5,0) arc[radius=.5, start angle=180, end angle=0];
\draw (0,.5)--(0,0);
\draw [->](0,.5) arc[radius=.5, start angle=90, end angle=180];
\draw [->](0,.5) arc[radius=.5, start angle=90, end angle=0];
\draw [->](0,.5)--(0,0);
\draw (0,0)--(0,-.2);
\draw (.5,0)--(.5,-.2);
\draw (-.5,0)--(-.5,-.2);
\node at (-.5,-.12)[left] {$\NGLabel{1}$};
\node at (0,-.12)[right] {$\NGLabel{2}$};
\node at (.5,-.12)[right] {$\NGLabel{3}$};
\end{tikzpicture}%
}

\expandafter\def\csname ngshape@O003\endcsname{%
\begin{tikzpicture}[scale=1,ng diagram]
\draw (0,0)--(0,.9*1.5);
\draw [->](0,.9)--(0,.45*1.5) node [left] {$\NGLabel{1}$};
\draw (.5,0)--(.5,.9*1.5);
\draw [->](.5,.9)--(.5,.45*1.5) node [left] {$\NGLabel{2}$};
\end{tikzpicture}%
}

\expandafter\def\csname ngshape@O004\endcsname{%
\begin{tikzpicture}[scale=1,ng diagram]
\draw (0,0)--(0,.9*1.5);
\draw (.5,.9*1.5)--(0,.6*1.5)--(0,.3*1.5)--(.5,0);
\draw [->](.5,.9*1.5)--(.25,.75*1.5) node [above] {$\NGLabel{1}$};
\draw [->](0,.3*1.5)--(.25,.15*1.5) node[below] {$\NGLabel{2}$};
\draw [->] (0,.9*1.5)--(0,.75*1.5) node[left] {$\NGLabel{3}$};
\draw [->](0,.6*1.5)--(0,.45*1.5)node [left] {$\NGLabel{4}$};
\draw [->](0,.3*1.5)--(0,.15*1.5)node [left] {$\NGLabel{5}$};
\end{tikzpicture}%
}

\expandafter\def\csname ngshape@O005\endcsname{%
\begin{tikzpicture}[scale=1,ng diagram]
\draw (0,0)--(0,.75);
\draw (.5,-.5)--(0,0)--(-.5,-.5);
\draw (.5,1.25)--(0,.75)--(-.5,1.25);
\draw (.5,1.25) arc [radius=.5/1.414, start angle=135, end angle=45];
\draw (1,1.25)--(1.5,.75)--(2,1.25);
\draw (1.5,.75)--(1.5,0);
\draw (.5,-.5) arc [radius=.5/1.414, start angle=225, end angle=315];
\draw (1,-.5)--(1.5,0)--(2,-.5);
\draw [->](0,.75)--(0,.375);
\draw [->](1.5,.75)--(1.5,.375);
\draw [->] (-.5,1.25)--(-.25,1);
\draw [->] (2,1.25)--(1.75,1);
\draw [-<](.5,1.25) arc[radius=.5/1.414,start angle=135, end angle=90];
\draw [-<] (-.5,-.5)--(-.25,-.25);
\draw [-<] (2,-.5)--(1.75,-.25);
\draw [->](.5,-.5) arc[radius=.5/1.414,start angle=225, end angle=270];
\node at (-.25,1) [left] {$\NGLabel{1}$};
\node at (.75,1.1) {$\NGLabel{2}$};
\node at (1.75,1)[right] {$\NGLabel{3}$};
\node at (1.5,.375)[left] {$\NGLabel{4}$};
\node at (0,.375)[left]{$\NGLabel{5}$};
\node at (1.5,.375)[left] {$\NGLabel{6}$};
\node at (-.25,-.25) [left] {$\NGLabel{7}$};
\node at (.75,.75-1.1) {$\NGLabel{8}$};
\node at (1.75,-.25)[right] {$\NGLabel{9}$};
\end{tikzpicture}%
}

\expandafter\def\csname ngshape@O006\endcsname{%
\begin{tikzpicture}[scale=1,ng diagram]
\draw (0,0)--(0,.75);
\draw (.5,-.5)--(0,0)--(-.5,-.5);
\draw (.5,1.25)--(0,.75)--(-.5,1.25);
\draw [->](0,.75)--(0,.375);
\node at (0,.375)[left]{$\NGLabel{1}$};
\node at (-.25,1) [left] {$\NGLabel{2}$};
\node at (.25,1)[right] {$\NGLabel{3}$};
\node at (-.25,-.25) [left] {$\NGLabel{4}$};
\node at (.25,-.25)[right] {$\NGLabel{5}$};
\end{tikzpicture}%
}

\expandafter\def\csname ngshape@O007\endcsname{%
\begin{tikzpicture}[scale=1,ng diagram]
\draw (0,0)--(0,.75);
\draw (.5,-.5)--(0,0)--(-.5,-.5);
\draw (.5,1.25)--(0,.75)--(-.5,1.25);
\draw (.5,1.25) arc [radius=.5/1.414, start angle=135, end angle=45];
\draw (1,1.25)--(1.5,.75)--(2,1.25);
\draw (1.5,.75)--(1.5,0);
\draw (.5,-.5) arc [radius=.5/1.414, start angle=225, end angle=315];
\draw (1,-.5)--(1.5,0)--(2,-.5);
\draw [->](0,.75)--(0,.375);
\draw [->](1.5,.75)--(1.5,.375);
\node at (0,.375)[left] {$\NGLabel{1}$};
\node at (1.5,.375)[left] {$\NGLabel{2}$};
\end{tikzpicture}%
}

\expandafter\def\csname ngshape@O008\endcsname{%
\begin{tikzpicture}[scale=1,ng diagram]
\draw (0,0)--(0,.75);
\draw (.5,-.5)--(0,0)--(-.5,-.5);
\draw (.5,1.25)--(0,.75)--(-.5,1.25);
\draw [->](0,.75)--(0,.375);
\node at (0,.375)[left]{$\NGLabel{1}$};
\end{tikzpicture}%
}

\expandafter\def\csname ngshape@T001\endcsname{%
\begin{tikzpicture}[scale=1.125,ng diagram]
\draw[->] (0,0) 	node [below] {}	-- (1,0) 		node [below] 	{$\NGLabel{1}$}; \draw (1,0) 		-- (2,0);
\draw[->] (2,0) 	node [below] {} 	-- (1.5,.8660254) 	node [right] 	{$\NGLabel{2}$}; \draw (1.5,.8660254) 	-- (1,1.7320508);
\draw[->] (1,1.7320508)	node [above] {} 	-- (.5,.8660254) 	node [left] 	{$\NGLabel{3}$}; \draw (.5,.8660254) 	-- (0,0);
\draw[-<] (0,0) 	-- (.5,.2886751) 	node [above] 	{$\NGLabel{4}$}; \draw (.5,.2886751) 	-- (1,0.5773503);
\draw[-<] (1,0.5773503) node [below] {} -- (1,1.1547005) 	node [right]	{$\NGLabel{5}$}; \draw (1,1.1547005) 	-- (1,1.7320508);
\draw[->] (1,0.5773503) -- (1.5,.2886751) 	node [above] 	{$\NGLabel{6}$}; \draw (1.5,.2886751) 	-- (2,0);
\end{tikzpicture}%
}

\expandafter\def\csname ngshape@T002\endcsname{%
\begin{tikzpicture}[scale=1.125,ng diagram]
\draw[] (0,0) 	node [below] {}	-- (1,0) 		node [below] 	{}; \draw (1,0) 		-- (2,0);
\draw[-<] (2,0) 	node [below] {} 	-- (1.5,.8660254) 	node [right] 	{$\NGLabel{1}$}; \draw (1.5,.8660254) 	-- (1,1.7320508);
\draw[-<] (1,1.7320508)	node [above] {} 	-- (.5,.8660254) 	node [left] 	{$\NGLabel{2}$}; \draw (.5,.8660254) 	-- (0,0);
\draw[] (0,0) 	-- (.5,.2886751) 	node [above] 	{}; \draw (.5,.2886751) 	-- (1,0.5773503);
\draw[->] (1,0.5773503) node [below] {} -- (1,1.1547005) 	node [right]	{$\NGLabel{3}$}; \draw (1,1.1547005) 	-- (1,1.7320508);
\draw[] (1,0.5773503) -- (1.5,.2886751) 	node [above] 	{}; \draw (1.5,.2886751) 	-- (2,0);
\end{tikzpicture}%
}

\expandafter\def\csname ngshape@O009\endcsname{%
\begin{tikzpicture}[scale=1,ng diagram]
\draw (0,0) rectangle (2,1);
\draw (1,0)--(1,1);
\draw [->](0,1)--(0,.5) node [left] {$\NGLabel{1}$};
\draw [->](1,1)--(1,.5) node [left] {$\NGLabel{2}$};
\draw [-<](2,1)--(2,.5) node [left] {$\NGLabel{3}$};
\draw (0,1) arc[radius=1, start angle=180, end angle=0];
\draw (0,0) arc[radius=1, start angle=180, end angle=360];
\end{tikzpicture}%
}

\expandafter\def\csname ngshape@O010\endcsname{%
\begin{tikzpicture}[scale=1,ng diagram]
\draw (0,0) rectangle (2,1);
\draw [->](0,1)--(0,.5) node [left] {$\NGLabel{1}$};
\draw [-<](2,1)--(2,.5) node [left] {$\NGLabel{2}$};
\draw (0,1) arc[radius=1, start angle=180, end angle=0];
\draw (0,0) arc[radius=1, start angle=180, end angle=360];
\end{tikzpicture}%
}

\expandafter\def\csname ngshape@O011\endcsname{%
\begin{tikzpicture}[scale=1,ng diagram]
\draw (-.5,0) arc[radius=.5, start angle=180, end angle=0];
\draw (0,.5)--(0,0);
\draw [->](0,.5) arc[radius=.5, start angle=90, end angle=180];
\draw [](0,.5) arc[radius=.5, start angle=90, end angle=0];
\draw [](0,.5)--(0,0);
\draw (0,0)--(0,-.2);
\draw (.5,0)--(.5,-.2);
\draw (-.5,0)--(-.5,-.2);
\node at (-.4,0)[left] {$\NGLabel{1}$};
\end{tikzpicture}%
}

\expandafter\def\csname ngshape@O012\endcsname{%
\begin{tikzpicture}[scale=1,ng diagram]
\draw (0,0)--(0,2.1);
\draw (0,1.4)--(-.75,1.4);
\draw [-<](-.75,1.4)--(-.4,1.4);
\draw (0,.7)--(.75,.7);
\draw[-<](0,.7)--(.4,.7);
\node at (0,1.9)[right] {$\NGLabel{1}$};
\node at (0,1)[left] {$\NGLabel{2}$};
\node at (0,.4)[left] {$\NGLabel{3}$};
\node at (-.4,1.4)[above]{$\NGLabel{4}$};
\node at (.4,.7)[above]{$\NGLabel{5}$};
\end{tikzpicture}%
}

\expandafter\def\csname ngshape@O013\endcsname{%
\begin{tikzpicture}[scale=1,ng diagram]
\draw (0,0)--(0,2.1);
\draw (0,.7)--(-.75,.7);
\draw [-<](-.75,1.4/2)--(-.4,1.4/2);
\draw (0,1.4)--(.75,1.4);
\draw [-<](0,1.4)--(.4,1.4);
\node at (0,1.9)[left] {$\NGLabel{1}$};
\node at (0,1)[right] {$\NGLabel{2}$};
\node at (0,.4)[right] {$\NGLabel{3}$};
\node at (-.4,1.4/2)[above]{$\NGLabel{4}$};
\node at (.4,1.4)[above]{$\NGLabel{5}$};
\end{tikzpicture}%
}

\expandafter\def\csname ngshape@T003\endcsname{%
\begin{tikzpicture}[scale=1.125,ng diagram]
\draw[->] (0,0) 	node [below] {}	-- (1,0) 		node [below] 	{$\NGLabel{1}$}; \draw (1,0) 		-- (2,0);
\draw(2,0) 	node [below] {} 	-- (1.5,.8660254) 	node [right] 	{}; \draw (1.5,.8660254) 	-- (1,1.7320508);
\draw (1,1.7320508)	node [above] {} 	-- (.5,.8660254) 	node [left] 	{}; \draw (.5,.8660254) 	-- (0,0);
\draw (0,0) 	-- (.5,.2886751) 	node [above] 	{};
\draw (.5,.2886751) 	-- (1,0.5773503);
\draw[-<] (1,0.5773503) node [below] {} -- (1,1.1547005) 	node [right]	{$\NGLabel{2}$}; \draw (1,1.1547005) 	-- (1,1.7320508);
\draw(1,0.5773503) -- (1.5,.2886751) 	node [above] 	{}; \draw (1.5,.2886751) 	-- (2,0);
\end{tikzpicture}%
}

\expandafter\def\csname ngshape@P002\endcsname{%
\begin{tikzpicture}[scale=1.125,ng diagram]
\draw (0,2)--(3,2);
\draw (1,1)--(2,1);
\draw (0,0)--(3,0);
\draw (0,0) -- (0,2) -- (1,1) -- (0,0);
\draw (3,0)--(3,2)--(2,1)--(3,0);
\draw[] (0,2) -- (1.5,2) node [above] {};
\draw[] (1,1) -- (1.5,1) node [above] {};
\draw[->] (0,0) -- (1.5,0) node [above] {$\NGLabel{1}$};
\draw[] (0,2) --++ (.5,-.5) node [right] {};
\draw[->] (1,1) --++ (-.5,-.5) node [right] {$\NGLabel{2}$};
\draw[-<] (0,0) -- (0,1) node [left] {$\NGLabel{3}$};
\draw[-<] (3,2) --++ (-.5,-.5) node [left] {$\NGLabel{4}$};
\draw[] (2,1) --++ (.5,-.5) node [left] {};
\draw[] (3,0) -- (3,1) node [right] {};
\end{tikzpicture}%
}

\expandafter\def\csname ngshape@T004\endcsname{%
\begin{tikzpicture}[scale=1.125,ng diagram]
\draw[->] (0,0) 	node [below] {}	-- (1,0) 		node [below] 	{$\NGLabel{1}$}; \draw (1,0) 		-- (2,0);
\draw (2,0) 	node [below] {} 	-- (1.5,.8660254) 	node [right] 	{}; \draw (1.5,.8660254) 	-- (1,1.7320508);
\draw[->] (1,1.7320508)	node [above] {} 	-- (.5,.8660254) 	node [left] 	{$\NGLabel{2}$}; \draw (.5,.8660254) 	-- (0,0);
\draw[-<] (0,0) 	-- (.5,.2886751) 	node [above] 	{$\NGLabel{3}$};
\draw (.5,.2886751) 	-- (1,0.5773503);
\draw(1,0.5773503) node [below] {} -- (1,1.1547005) 	node [right]	{}; \draw (1,1.1547005) 	-- (1,1.7320508);
\draw (1,0.5773503) -- (1.5,.2886751) 	node [above] 	{}; \draw (1.5,.2886751) 	-- (2,0);
\end{tikzpicture}%
}

\expandafter\def\csname ngshape@T005\endcsname{%
\begin{tikzpicture}[scale=1.125,ng diagram]
\draw[->] (0,0) 	node [below] {}	-- (1,0) 		node [below] 	{$\NGLabel{1}$}; \draw (1,0) 		-- (2,0);
\draw(2,0) 	node [below] {} 	-- (1.5,.8660254) 	node [right] 	{}; \draw (1.5,.8660254) 	-- (1,1.7320508);
\draw (1,1.7320508)	node [above] {} 	-- (.5,.8660254) 	node [left] 	{}; \draw (.5,.8660254) 	-- (0,0);
\draw (0,0) 	-- (.5,.2886751) 	node [above] 	{};
\draw (.5,.2886751) 	-- (1,0.5773503);
\draw[->] (1,0.5773503) node [below] {} -- (1,1.1547005) 	node [right]	{$\NGLabel{2}$}; \draw (1,1.1547005) 	-- (1,1.7320508);
\draw(1,0.5773503) -- (1.5,.2886751) 	node [above] 	{}; \draw (1.5,.2886751) 	-- (2,0);
\end{tikzpicture}%
}

\expandafter\def\csname ngshape@T006\endcsname{%
\begin{tikzpicture}[scale=1.125,ng diagram]
\draw (0,0) 	node [below] {}	-- (1,0) node [below] 	{}; \draw (1,0) 		-- (2,0);
\draw[-<] (2,0) 	node [below] {} -- (1.5,.8660254) 	node [right] 	{$\NGLabel{1}$}; \draw (1.5,.8660254) 	-- (1,1.7320508);
\draw (1,1.7320508)	node [above] {} -- (.5,.8660254) 	node [left] 	{}; \draw (.5,.8660254) 	-- (0,0);
\draw[->] (0,0) 	-- (.5,.2886751) 	node [above] 	{$\NGLabel{2}$}; \draw (.5,.2886751) 	-- (1,0.5773503);
\draw (1,0.5773503) node [below] {} -- (1,1.1547005) 	node [right]	{}; \draw (1,1.1547005) 	-- (1,1.7320508);
\draw (1,0.5773503) -- (1.5,.2886751) 	node [above] 	{}; \draw (1.5,.2886751) 	-- (2,0);
\end{tikzpicture}%
}

\expandafter\def\csname ngshape@T007\endcsname{%
\begin{tikzpicture}[scale=1.125,ng diagram]
\draw (0,0) 	node [below] {}	-- (1,0) 		node [below] 	{}; \draw (1,0) 		-- (2,0);
\draw (2,0) 	node [below] {} 	-- (1.5,.8660254) 	node [right] 	{}; \draw (1.5,.8660254) 	-- (1,1.7320508);
\draw[-<] (1,1.7320508)	node [above] {} 	-- (.5,.8660254) 	node [left] 	{$\NGLabel{1}$}; \draw (.5,.8660254) 	-- (0,0);
\draw(0,0) 	-- (.5,.2886751) 	node [above] 	{}; \draw (.5,.2886751) 	-- (1,0.5773503);
\draw(1,0.5773503) node [below] {} -- (1,1.1547005) 	node [right]	{}; \draw (1,1.1547005) 	-- (1,1.7320508);
\draw[-<] (1,0.5773503) -- (1.5,.2886751) 	node [above] 	{$\NGLabel{2}$}; \draw (1.5,.2886751) 	-- (2,0);
\end{tikzpicture}%
}

\expandafter\def\csname ngshape@T008\endcsname{%
\begin{tikzpicture}[scale=1.125,ng diagram]
\draw[->] (0,0) 	node [below] {}	-- (1,0) 		node [below] 	{$\NGLabel{1}$}; \draw (1,0) 		-- (2,0);
\draw[] (2,0) 	node [below] {} 	-- (1.5,.8660254) 	node [right] 	{}; \draw (1.5,.8660254) 	-- (1,1.7320508);
\draw[->] (1,1.7320508)	node [above] {} 	-- (.5,.8660254) 	node [left] 	{$\NGLabel{2}$}; \draw (.5,.8660254) 	-- (0,0);
\draw[-<] (0,0) 	-- (.5,.2886751) 	node [above] 	{$\NGLabel{3}$}; \draw (.5,.2886751) 	-- (1,0.5773503);
\draw[] (1,0.5773503) node [below] {} -- (1,1.1547005) 	node [right]	{}; \draw (1,1.1547005) 	-- (1,1.7320508);
\draw[] (1,0.5773503) -- (1.5,.2886751) 	node [above] 	{}; \draw (1.5,.2886751) 	-- (2,0);
\end{tikzpicture}%
}

\expandafter\def\csname ngshape@T009\endcsname{%
\begin{tikzpicture}[scale=1.125,ng diagram]
\draw[-<] (0,0) 	node [below] {}	-- (1,0) 		node [below] 	{$\NGLabel{1}$}; \draw (1,0) 		-- (2,0);
\draw(2,0) 	node [below] {} 	-- (1.5,.8660254) 	node [right] 	{}; \draw (1.5,.8660254) 	-- (1,1.7320508);
\draw (1,1.7320508)	node [above] {} 	-- (.5,.8660254) 	node [left] 	{}; \draw (.5,.8660254) 	-- (0,0);
\draw (0,0) 	-- (.5,.2886751) 	node [above] 	{};
\draw (.5,.2886751) 	-- (1,0.5773503);
\draw[->] (1,0.5773503) node [below] {} -- (1,1.1547005) 	node [right]	{$\NGLabel{2}$}; \draw (1,1.1547005) 	-- (1,1.7320508);
\draw(1,0.5773503) -- (1.5,.2886751) 	node [above] 	{}; \draw (1.5,.2886751) 	-- (2,0);
\end{tikzpicture}%
}

\expandafter\def\csname ngshape@O014\endcsname{%
\begin{tikzpicture}[scale=1,xscale=1.3,ng diagram]
\NGYBase
\node at (0,0)[left] {$\NGLabel{1}$};
\end{tikzpicture}%
}

\expandafter\def\csname ngshape@O015\endcsname{%
\begin{tikzpicture}[scale=1,xscale=1.3,ng diagram]
\NGYBase
\draw (.4,.4)--(-.4,.4);
\draw [-<](-.4,.4)--(0,.4);
\node at (0,.4) [above] {$\NGLabel{1}$};
\node at (0,0)[left] {$\NGLabel{2}$};
\end{tikzpicture}%
}

\expandafter\def\csname ngshape@O016\endcsname{%
\begin{tikzpicture}[scale=0.75,xscale=1.3,ng diagram]
\draw (0,0)--(1.5,0)--(.75,1)--(0,0);
\node at (.75,1.25)[right] {$\NGLabel{1}$};
\node at (.375,.5)[left] {$\NGLabel{2}$};
\node at (1.175,.5)[right] {$\NGLabel{3}$};
\draw (.75,1)--(.75,1.5);
\draw (0,0)--(-.5,-.5);
\draw (1.5,0)--(2,-.5);
\draw [->] (.75,1.5)--(.75,1.25);
\draw [->](.75,1)--(.375,.5);
\draw [->](.75,1)--(1.125,.5);
\end{tikzpicture}%
}

\expandafter\def\csname ngshape@O017\endcsname{%
\begin{tikzpicture}[scale=0.75,xscale=1.3,ng diagram]
\draw (.75,1.5)--(.75,0)--(-.5,-.5)--(.75,0)--(2,-.5);
\node at (.75,.75)[right] {$\NGLabel{1}$};
\draw [->](.75,1.5)--(.75,.75);
\end{tikzpicture}%
}

\expandafter\def\csname ngshape@O018\endcsname{%
\begin{tikzpicture}[scale=1.5,xscale=1.3,ng diagram]
\NGYBase
\draw (.5,.5)--(-.5,.5);
\draw (.25,.25)--(-.25,.25);
\draw [-<](-.5,.5)--(0,.5);
\draw [-<](-.25,.25)--(0,.25);
\node at (0,.48) [above] {$\NGLabel{1}$};
\node at (0,.21) [above] {$\NGLabel{2}$};
\node at (0,0)[left] {$\NGLabel{3}$};
\end{tikzpicture}%
}

\expandafter\def\csname ngshape@O019\endcsname{%
\begin{tikzpicture}[scale=1.5,xscale=1.3,ng diagram]
\NGYBase
\draw (-.6,.6)--(-.3,.6)--(-.15,.45)--(.15,.45)--(.3,.6)--(.6,.6);
\draw (-.3,.3)--(-.15,.45)--(.15,.45)--(.3,.3);
\draw [->](.6,.6)--(.45,.6);
\draw [-<](-.6,.6)--(-.45,.6);
\draw [->](.3,.3)--(.225,.375);
\draw [-<](-.3,.3)--(-.225,.375);
\draw [->](.15,.45)--(0,.45);
\node at (0,0)[left] {$\NGLabel{1}$};
\node at (0,.42) [above] {$\NGLabel{2}$};
\node at (.45,.6)[above] {$\NGLabel{3}$};
\node at (-.45,.6)[above] {$\NGLabel{4}$};
\node at (.15,.28) {$\NGLabel{5}$};
\node at (-.15,.28) {$\NGLabel{6}$};
\end{tikzpicture}%
}

\expandafter\def\csname ngshape@O020\endcsname{%
\begin{tikzpicture}[scale=1.5,xscale=1.3,ng diagram]
\NGYBase
\draw (.45,.45)--(-.45,.45);
\draw[->] (.45,.45)--(0,.45);
\node at (0,.45) [above] {$\NGLabel{1}$};
\node at (0,0)[left] {$\NGLabel{2}$};
\end{tikzpicture}%
}

\expandafter\def\csname ngshape@O021\endcsname{%
\begin{tikzpicture}[scale=1,xscale=1.8,ng diagram]
\NGYBase
\draw (0,-.45)--(.45,.45);
\draw [->](0,-.45)--(.225,0);
\node at (0,0)[left] {$\NGLabel{1}$};
\node at (.225,0)[right] {$\NGLabel{2}$};
\end{tikzpicture}%
}

\expandafter\def\csname ngshape@O022\endcsname{%
\begin{tikzpicture}[scale=1,xscale=1.8,ng diagram]
\NGYBase
\draw (0,-.45)--(-.45,.45);
\draw [-<](0,-.45)--(-.225,0);
\node at (-.1,0) {$\NGLabel{1}$};
\node at (-.225,0)[left] {$\NGLabel{2}$};
\end{tikzpicture}%
}

\expandafter\def\csname ngshape@O023\endcsname{%
\begin{tikzpicture}[scale=1.5,xscale=1.4,ng diagram]
\NGYBase
\draw (.25,.25)--(-.25,.25);
\draw [-<](-.25,.25)--(0,.25);
\draw (0,-.45)--(-.6,.6);
\draw [-<](0,-.45)--(-.45*.6/1.05,0);
\node at (0,.25) [above] {$\NGLabel{1}$};
\node at (-.2,0) [left] {$\NGLabel{2}$};
\node at (-.1,0) {$\NGLabel{3}$};
\end{tikzpicture}%
}

\expandafter\def\csname ngshape@O024\endcsname{%
\begin{tikzpicture}[scale=1.5,xscale=1.4,ng diagram]
\NGYBase
\draw (.5,.5)--(-.5,.5);
\draw [-<](-.5,.5)--(0,.5);
\draw (0,-.45)--(-.35,.35);
\draw [-<](0,-.45)--(-.45*.35/.8,0);
\node at (0,.48) [above] {$\NGLabel{1}$};
\node at (-.2,0) [left] {$\NGLabel{2}$};
\node at (-.1,0) {$\NGLabel{3}$};
\end{tikzpicture}%
}

\expandafter\def\csname ngshape@O025\endcsname{%
\begin{tikzpicture}[scale=1.5,xscale=1.4,ng diagram]
\NGYBase
\draw (.5,.5)--(-.5,.5);
\draw [-<](-.5,.5)--(0,.5);
\draw (0,-.45)--(.35,.35);
\draw [->](0,-.45)--(.45*.35/.8,0);
\node at (0,.48) [above] {$\NGLabel{1}$};
\node at (.2,0) [right] {$\NGLabel{2}$};
\node at (-.1,0) {$\NGLabel{3}$};
\end{tikzpicture}%
}

\expandafter\def\csname ngshape@O026\endcsname{%
\begin{tikzpicture}[scale=1.5,xscale=1.4,ng diagram]
\NGYBase
\draw (.25,.25)--(-.25,.25);
\draw [-<](-.25,.25)--(0,.25);
\draw (0,-.45)--(.6,.6);
\draw [->](0,-.45)--(.45*.6/1.05,0);
\node at (0,.25) [above] {$\NGLabel{1}$};
\node at (.2,0) [right] {$\NGLabel{2}$};
\node at (-.1,0) {$\NGLabel{3}$};
\end{tikzpicture}%
}

\expandafter\def\csname ngshape@O027\endcsname{%
\begin{tikzpicture}[scale=1.5,xscale=1.4,ng diagram]
\NGYBase
\draw (0,-.3)--(-.45,.45);
\draw [-<](0,-.3)--(-.3*.45/.75,0);
\draw (0,-.6)--(.35,.35);
\draw [->](0,-.6)--(.6*.35/.95,0);
\node at (-0.3*.45/.75,0) [left] {$\NGLabel{1}$};
\node at (.2,0) [right] {$\NGLabel{2}$};
\node at (-.1,0) {$\NGLabel{3}$};
\end{tikzpicture}%
}

\expandafter\def\csname ngshape@O028\endcsname{%
\begin{tikzpicture}[scale=1.5,xscale=1.4,ng diagram]
\NGYBase
\draw (0,-.3)--(.45,.45);
\draw [->](0,-.3)--(.3*.45/.75,0);
\draw (0,-.6)--(-.35,.35);
\draw [-<](0,-.6)--(-.6*.35/.95,0);
\node at (-.6*.35/.95,0) [left] {$\NGLabel{1}$};
\node at (.2,0) [right] {$\NGLabel{2}$};
\node at (-.1,0) {$\NGLabel{3}$};
\end{tikzpicture}%
}

\expandafter\def\csname ngshape@O029\endcsname{%
\begin{tikzpicture}[scale=0.75,ng diagram]
\NGYBase
\node at (0,0)[left] {$\NGLabel{1}$};
\end{tikzpicture}%
}

\expandafter\def\csname ngshape@O030\endcsname{%
\begin{tikzpicture}[scale=1,ng diagram]
\draw (0,0) circle [radius=1];
\draw (0,1)--(0,-1);
\node at (-.06,.76)[anchor=east] {$\NGLabel{1}$};
\node at (-.06,-.76)[anchor=east] {$\NGLabel{2}$};
\end{tikzpicture}%
}

\expandafter\def\csname ngshape@O031\endcsname{%
\begin{tikzpicture}[scale=1,ng diagram]
\draw (0,0) circle [radius=1];
\draw (0,1)--(0,-1);
\draw[red] (-.2,0) arc[radius=.3, start angle=0, end angle=360];
\draw [red,-<](-.2,0) arc[radius=.3, start angle=0, end angle=90];
\node at (-.5,.4) {$\NGLabel{1}$};
\node at (-.06,.76)[anchor=east] {$\NGLabel{2}$};
\node at (-.06,-.76)[anchor=east] {$\NGLabel{3}$};
\end{tikzpicture}%
}

\expandafter\def\csname ngshape@O032\endcsname{%
\begin{tikzpicture}[scale=1,ng diagram]
\draw (0,0) circle [radius=1];
\draw (0,1)--(0,-1);
\draw [red] (0,.6)--(-.8,.6);
\draw [red,-<] (0,.6)--(-.4,.6);
\node at (-.4,.45){$\NGLabel{1}$};
\draw [red] (0,-.6)--(-.8,-.6);
\draw [red,->] (0,-.6)--(-.4,-.6);
\node at (-.4,-.45){$\NGLabel{2}$};
\node at (-.06,.76)[anchor=east] {$\NGLabel{3}$};
\node at (-.06,-.76)[anchor=east] {$\NGLabel{4}$};
\end{tikzpicture}%
}

\expandafter\def\csname ngshape@O033\endcsname{%
\begin{tikzpicture}[scale=1.3,ng diagram]
\draw (0,0) circle [radius=.5];
\draw (0,.5) arc[radius=.5, start angle=90, end angle=180];
\draw (0,.5) arc[radius=.5, start angle=90, end angle=0];
\draw (0,.5)--(0,0);
\draw (0,-.5)--(0,.5);
\node at (-.05,.32)[anchor=east] {$\NGLabel{1}$};
\node at (-.05,-.32)[anchor=east] {$\NGLabel{2}$};
\end{tikzpicture}%
}

\expandafter\def\csname ngshape@O034\endcsname{%
\begin{tikzpicture}[scale=1,ng diagram]
\draw (0,0) circle [radius=1];
\draw (0,1)--(0,-1);
\draw [red] (0,.6)--(.8,.6);
\draw [red,->] (0,.6)--(.4,.6);
\node at (.4,.45){$\NGLabel{1}$};
\draw [red] (0,-.6)--(.8,-.6);
\draw [red,-<] (0,-.6)--(.4,-.6);
\node at (.4,-.45){$\NGLabel{2}$};
\node at (-.06,.76)[anchor=east] {$\NGLabel{3}$};
\node at (-.06,-.76)[anchor=east] {$\NGLabel{4}$};
\end{tikzpicture}%
}

\expandafter\def\csname ngshape@O035\endcsname{%
\begin{tikzpicture}[scale=1,xscale=1.25,ng diagram]
\draw (0,0) circle [radius=1];
\draw (0,1)--(0,-1);
\node at (-.06,.60)[anchor=east] {$\NGLabel{1}$};
\node at (-.06,-.60)[anchor=east] {$\NGLabel{2}$};
\end{tikzpicture}%
}

\expandafter\def\csname ngshape@O036\endcsname{%
\begin{tikzpicture}[scale=1,ng diagram]
\draw (0,0) circle [radius=1];
\draw (0,1)--(0,-1);
\node at (-.06,.76)[anchor=east] {$\NGLabel{1}$};
\node at (-.06,-.76)[anchor=east] {$\NGLabel{2}$};
\draw [red] (-.6,.8)--(-.6,1)arc[radius=.2,start angle=180, end angle=90]--(.4,1.2) arc[radius=.2, start angle=90, end angle=0]--(.6,.8);
\draw [red,-<] (-.6,.8)--(-.6,1)arc[radius=.2,start angle=180, end angle=90]--(0,1.2);
\node at (0,1.2) [above] {$\NGLabel{3}$};
\draw [red] (-.6,-.8)--(-.6,-1)arc[radius=.2,start angle=180, end angle=270]--(.4,-1.2) arc[radius=.2, start angle=270, end angle=360]--(.6,-.8);
\draw [red,->] (-.6,-.8)--(-.6,-1)arc[radius=.2,start angle=180, end angle=270]--(0,-1.2);
\node at (0,-1.2) [below] {$\NGLabel{4}$};
\end{tikzpicture}%
}

\expandafter\def\csname ngshape@O037\endcsname{%
\begin{tikzpicture}[scale=1,ng diagram]
\NGYBase
\node at (0,0)[left] {$\NGLabel{1}$};
\end{tikzpicture}%
}

\expandafter\def\csname ngshape@O038\endcsname{%
\begin{tikzpicture}[scale=1,ng diagram]
\NGYBase
\node at (0,0)[right] {$\NGLabel{1}$};
\end{tikzpicture}%
}

\expandafter\def\csname ngshape@O039\endcsname{%
\begin{tikzpicture}[scale=1,ng diagram]
\NGYBase
\draw [red](.4,.4)--(-.4,.4);
\draw [red,->](.4,.4)--(0,.4) node [above] {$\NGLabel{1}$};
\node at (0,0)[left] {$\NGLabel{2}$};
\end{tikzpicture}%
}

\expandafter\def\csname ngshape@O040\endcsname{%
\begin{tikzpicture}[scale=1,ng diagram]
\NGYBase
\draw [red](.4,.4)--(-.4,.4);
\draw [red,->](.4,.4)--(0,.4) node [above] {$\NGLabel{1}$};
\node at (0,0)[right] {$\NGLabel{2}$};
\end{tikzpicture}%
}

\expandafter\def\csname ngshape@O041\endcsname{%
\begin{tikzpicture}[scale=1,ng diagram]
\NGYBase
\node at (0,0)[above] {$\NGLabel{1}$};
\end{tikzpicture}%
}

\expandafter\def\csname ngshape@O042\endcsname{%
\begin{tikzpicture}[scale=1.1,ng diagram]
\draw (0,0)--(1.5,0)--(.75,1)--(0,0);
\node at (.75,1)[below] {$\NGLabel{1}$};
\node at (1.4,.14)[left] {$\NGLabel{2}$};
\node at (.1,.14)[right] {$\NGLabel{3}$};
\draw (.75,1)--(.75,1.5);
\draw (0,0)--(-.5,-.5);
\draw (1.5,0)--(2,-.5);
\end{tikzpicture}%
}

\expandafter\def\csname ngshape@O043\endcsname{%
\begin{tikzpicture}[scale=1.1,ng diagram]
\draw (0,0)--(1.5,0)--(.75,1)--(0,0);
\node at (.75,1)[below] {$\NGLabel{1}$};
\node at (1.4,.14)[left] {$\NGLabel{2}$};
\node at (.1,.14)[right] {$\NGLabel{3}$};
\draw (.75,1)--(.75,1.5);
\draw (0,0)--(-.5,-.5);
\draw (1.5,0)--(2,-.5);
\draw [red](.75,.4) circle [radius=.2];
\node at (.75,.1)[] {$\NGLabel{4}$};
\draw [->,red] (.55,.4) arc[radius=.2, start angle=180, end angle=0];
\end{tikzpicture}%
}

\expandafter\def\csname ngshape@O044\endcsname{%
\begin{tikzpicture}[scale=1.1,ng diagram]
\draw (0,0)--(1.5,0)--(.75,1)--(0,0);
\node at (.75,1)[below] {$\NGLabel{1}$};
\node at (1.4,.14)[left] {$\NGLabel{2}$};
\node at (.1,.14)[right] {$\NGLabel{3}$};
\draw (.75,1)--(.75,1.5);
\draw (0,0)--(-.5,-.5);
\draw (1.5,0)--(2,-.5);
\draw [red](.6,0)--(.75*.4,.4);
\draw [->,red](.6,0)--(.3+.75*.2,.2);
\draw [red](1.5-.6,0)--(1.5-.75*.4,.4);
\draw [-<,red](1.5-.6,0)--(1.5-.3-.75*.2,.2);
\draw [red](.75*.6,.6)--(1.5-.75*.6,.6);
\draw [->,red](.75*.6,.6)--(.75,.6);
\node at (.75,.45){$\NGLabel{4}$};
\node at (.45+.75*.2,.2){$\NGLabel{5}$};
\node at (1.5-.45-.75*.2,.2){$\NGLabel{6}$};
\end{tikzpicture}%
}

\expandafter\def\csname ngshape@O045\endcsname{%
\begin{tikzpicture}[scale=1,ng diagram]
\draw (0,0)--(.75,0);
\draw (-.5,.5)--(0,0)--(-.5,-.5);
\draw (1.25,.5)--(.75,0)--(1.25,-.5);
\node at (0,0) [left] {$\NGLabel{1}$};
\node at (.75,0) [right] {$\NGLabel{2}$};
\end{tikzpicture}%
}

\expandafter\def\csname ngshape@O046\endcsname{%
\begin{tikzpicture}[scale=1,ng diagram]
\draw (0,0)--(0,.75);
\draw (.5,-.5)--(0,0)--(-.5,-.5);
\draw (.5,1.25)--(0,.75)--(-.5,1.25);
\node at (0,0)[left] {$\NGLabel{1}$};
\node at (0,.75)[left] {$\NGLabel{2}$};
\end{tikzpicture}%
}

\expandafter\def\csname ngshape@O047\endcsname{%
\begin{tikzpicture}[scale=1,ng diagram]
\draw (0,0)--(0,.75);
\draw [->](0,.75)--(0,.375);
\draw (.5,-.5)--(0,0)--(-.5,-.5);
\draw (.5,1.25)--(0,.75)--(-.5,1.25);
\node at (0,.375)[left] {$\NGLabel{1}$};
\end{tikzpicture}%
}

\expandafter\def\csname ngshape@T010\endcsname{%
\begin{tikzpicture}[scale=1.125,ng diagram]
\draw[] (0,0) 	node [below] {}	-- (1,0) 		node [below] 	{}; \draw (1,0) 		-- (2,0);
\node at (.32,.30) {$\NGLabel{1}$};
\draw[] (2,0) 	node [below] {} 	-- (1.5,.8660254) 	node [right] 	{}; \draw (1.5,.8660254) 	-- (1,1.7320508);
\node at (1.60,.36) {$\NGLabel{2}$};
\draw[] (1,1.7320508)-- (.5,.8660254) 	node [left] 	{};
\node at (.87,1.25) {$\NGLabel{3}$};
\draw (.5,.8660254) 	-- (0,0);
\draw[] (0,0) 	-- (.5,.2886751) 	node [above] 	{}; \draw (.5,.2886751) 	-- (1,0.5773503);
\draw[] (1,0.5773503) node [below] {} -- (1,1.1547005) 	node [right]	{}; \draw (1,1.1547005) 	-- (1,1.7320508);
\node at (.84,.76) {$\NGLabel{4}$};
\draw[] (1,0.5773503) -- (1.5,.2886751) 	node [above] 	{}; \draw (1.5,.2886751) 	-- (2,0);
\end{tikzpicture}%
}

\expandafter\def\csname ngshape@O048\endcsname{%
\begin{tikzpicture}[scale=0.75,ng diagram]
\draw (0,0)--(0,.75);
\draw [->](0,.75)--(0,.375);
\draw (.5,-.5)--(0,0)--(-.5,-.5);
\draw (.5,1.25)--(0,.75)--(-.5,1.25);
\node at (0,.375)[left] {$\NGLabel{1}$};
\end{tikzpicture}%
}

\expandafter\def\csname ngshape@O049\endcsname{%
\begin{tikzpicture}[scale=1,xscale=1.3,ng diagram]
\draw (0,0)--(.75,0);
\draw (-.5,.5)--(0,0)--(-.5,-.5);
\draw (1.25,.5)--(.75,0)--(1.25,-.5);
\node at (0,0) [left] {$\NGLabel{1}$};
\node at (.75,0) [right] {$\NGLabel{2}$};
\draw (1.15,.4)--(1.15,-.4);
\draw [->](1.15,.4)--(1.15,0);
\node at (1.15,0)[right] {$\NGLabel{3}$};
\end{tikzpicture}%
}

\expandafter\def\csname ngshape@O050\endcsname{%
\begin{tikzpicture}[scale=1,xscale=1.3,ng diagram]
\draw (0,0)--(0,.75);
\draw (.5,-.5)--(0,0)--(-.5,-.5);
\draw (.5,1.25)--(0,.75)--(-.5,1.25);
\node at (0,0)[left] {$\NGLabel{1}$};
\node at (0,.75)[left] {$\NGLabel{2}$};
\draw (.4,1.15)--(.4,-.4);
\draw [->](.4,1.15)--(.4,.375);
\node at (.4,.375)[right] {$\NGLabel{3}$};
\end{tikzpicture}%
}

\expandafter\def\csname ngshape@O051\endcsname{%
\begin{tikzpicture}[scale=1,xscale=1.3,ng diagram]
\draw (0,0)--(0,.75);
\draw (.5,-.5)--(0,0)--(-.5,-.5);
\draw (.5,1.25)--(0,.75)--(-.5,1.25);
\node at (0,0)[left] {$\NGLabel{1}$};
\node at (0,.75)[left] {$\NGLabel{2}$};
\draw (.4,1.15)--(0,.425);
\draw [-<](0,.425)--(.2,.7875);
\draw (0,.325)--(.4,-.4);
\draw [->](0,.325)--(.2,-0.0375);
\node at (.2,.7875) [right] {$\NGLabel{3}$};
\node at (.2,-0.0375)[right] {$\NGLabel{4}$};
\end{tikzpicture}%
}

\expandafter\def\csname ngshape@T011\endcsname{%
\begin{tikzpicture}[scale=1.125,ng diagram]
\draw[->] (0,0) 	node [below] {}	-- (1,0) 		node [below] 	{$\NGLabel{1}$}; \draw (1,0) 		-- (2,0);
\draw(2,0) 	node [below] {} 	-- (1.5,.8660254) 	node [right] 	{}; \draw (1.5,.8660254) 	-- (1,1.7320508);
\draw (1,1.7320508)	node [above] {} 	-- (.5,.8660254) 	node [left] 	{}; \draw (.5,.8660254) 	-- (0,0);
\draw (0,0) 	-- (.5,.2886751) 	node [above] 	{};
\draw (.5,.2886751) 	-- (1,0.5773503);
\draw[] (1,0.5773503) node [below] {} -- (1,1.1547005) 	node [right]	{}; \draw (1,1.1547005) 	-- (1,1.7320508);
\draw(1,0.5773503) -- (1.5,.2886751) 	node [above] 	{}; \draw (1.5,.2886751) 	-- (2,0);
\node at (1,1.7320508*.25+.1) [below] {$\NGLabel{2}$};
\node at (1,1*1.7320508) [above] {$\NGLabel{3}$};
\end{tikzpicture}%
}

\expandafter\def\csname ngshape@T012\endcsname{%
\begin{tikzpicture}[scale=1.125,ng diagram]
\draw[] (0,0) 	node [below] {}	-- (1,0) 		node [below] 	{}; \draw (1,0) 		-- (2,0);
\node at (.32,.30) {$\NGLabel{1}$};
\draw[] (2,0) 	node [below] {} 	-- (1.5,.8660254) 	node [right] 	{}; \draw (1.5,.8660254) 	-- (1,1.7320508);
\node at (1.55,.40) {$\NGLabel{2}$};
\draw[] (1,1.7320508)-- (.5,.8660254) 	node [left] 	{};
\node at (.87,1.25) {$\NGLabel{3}$};
\draw (.5,.8660254) 	-- (0,0);
\draw[] (0,0) 	-- (.5,.2886751) 	node [above] 	{}; \draw (.5,.2886751) 	-- (1,0.5773503);
\draw[] (1,0.5773503) node [below] {} -- (1,1.1547005) 	node [right]	{}; \draw (1,1.1547005) 	-- (1,1.7320508);
\node at (1,.627) [left]{$\NGLabel{4}$};
\draw[] (1,0.5773503) -- (1.5,.2886751) 	node [above] 	{}; \draw (1.5,.2886751) 	-- (2,0);
\end{tikzpicture}%
}

\expandafter\def\csname ngshape@T013\endcsname{%
\begin{tikzpicture}[scale=1.125,ng diagram]
\draw(2,0) 	node [below] {} 	-- (1.5,.8660254) 	node [right] 	{}; \draw (1.5,.8660254) 	-- (1,1.7320508);
\draw (1,1.7320508)	node [above] {} 	-- (.5,.8660254) 	node [left] 	{}; \draw (.5,.8660254) 	-- (0,0);
\draw (0,0) 	-- (.5,.2886751) 	node [above] 	{};
\draw (.5,.2886751) 	-- (1,0.5773503);
\draw[] (1,0.5773503) node [below] {} -- (1,1.1547005) 	node [right]	{}; \draw (1,1.1547005) 	-- (1,1.7320508);
\draw(1,0.5773503) -- (1.5,.2886751) 	node [above] 	{}; \draw (1.5,.2886751) 	-- (2,0);
\node at (1,1.7320508*.25+.1) [below] {$\NGLabel{1}$};
\node at (1,1*1.7320508) [above] {$\NGLabel{2}$};
\end{tikzpicture}%
}

\expandafter\def\csname ngshape@T014\endcsname{%
\begin{tikzpicture}[scale=1.125,ng diagram]
\draw[] (0,0) 	node [below] {}	-- (1,0) 		node [below] 	{}; \draw (1,0) 		-- (2,0);
\node at (.32,.30) {$\NGLabel{1}$};
\draw[] (2,0) 	node [below] {} 	-- (1.5,.8660254) 	node [right] 	{}; \draw (1.5,.8660254) 	-- (1,1.7320508);
\node at (1.55,.40) {$\NGLabel{2}$};
\draw[] (1,1.7320508)-- (.5,.8660254) 	node [left] 	{};
\node at (.87,1.25) {$\NGLabel{3}$};
\draw (.5,.8660254) 	-- (0,0);
\draw[] (0,0) 	-- (.5,.2886751) 	node [above] 	{}; \draw (.5,.2886751) 	-- (1,0.5773503);
\draw[] (1,0.5773503) node [below] {} -- (1,1.1547005) 	node [right]	{}; \draw (1,1.1547005) 	-- (1,1.7320508);
\node at (.78,.73) {$\NGLabel{4}$};
\draw[] (1,0.5773503) -- (1.5,.2886751) 	node [above] 	{}; \draw (1.5,.2886751) 	-- (2,0);
\end{tikzpicture}%
}

\expandafter\def\csname ngshape@T015\endcsname{%
\begin{tikzpicture}[scale=1.125,ng diagram]
\draw[] (0,0) 	node [below] {}	-- (1,0) 		node [below] 	{}; \draw (1,0) 		-- (2,0);
\node at (.32,.30) {$\NGLabel{1}$};
\draw[] (2,0) 	node [below] {} 	-- (1.5,.8660254) 	node [right] 	{}; \draw (1.5,.8660254) 	-- (1,1.7320508);
\node at (1.55,.40) {$\NGLabel{2}$};
\draw[] (1,1.7320508)-- (.5,.8660254) 	node [left] 	{};
\node at (.87,1.25) {$\NGLabel{3}$};
\draw (.5,.8660254) 	-- (0,0);
\draw[] (0,0) 	-- (.5,.2886751) 	node [above] 	{}; \draw (.5,.2886751) 	-- (1,0.5773503);
\draw[] (1,0.5773503) node [below] {} -- (1,1.1547005) 	node [right]	{}; \draw (1,1.1547005) 	-- (1,1.7320508);
\node at (1,.627)[left] {$\NGLabel{4}$};
\draw[] (1,0.5773503) -- (1.5,.2886751) 	node [above] 	{}; \draw (1.5,.2886751) 	-- (2,0);
\draw [-<,red](.5,0)--(.5,.25/1.7320508);
\draw [-<,red](1.5,.5/1.7320508)--(1.5,.25/1.7320508);
\draw [-<,red] (.75,.75/1.7320508)--(1,.75/1.7320508);
\draw [red](.5,0)--(.5,.5/1.7320508);
\draw [red](1.5,0)--(1.5,.5/1.7320508);
\draw [red](.75,.75/1.7320508)--(1.25,.75/1.7320508);
\node at  (.5,.25/1.7320508) [right] {$\NGLabel{5}$};
\node at  (1.5,.25/1.7320508) [left] {$\NGLabel{6}$};
\node at (1,.75/1.7320508)[below] {$\NGLabel{7}$};
\end{tikzpicture}%
}

\expandafter\def\csname ngshape@O052\endcsname{%
\begin{tikzpicture}[scale=1,ng diagram]
\draw (0,0)--(0,1.5);
\draw (.5,0)--(.5,1.5);
\draw (1,0)--(1,1.5);
\end{tikzpicture}%
}

\expandafter\def\csname ngshape@O053\endcsname{%
\begin{tikzpicture}[scale=1,ng diagram]
\draw (0,0) arc[radius=.5, start angle=180, end angle=0];
\draw (0,1.5) arc[radius=.5, start angle=180, end angle=360];
\end{tikzpicture}%
}

\expandafter\def\csname ngshape@P003\endcsname{%
\begin{tikzpicture}[scale=0.75,ng diagram]
\draw (0,2)--(3,2);
\draw (1,1)--(2,1);
\draw (0,0)--(3,0);
\draw (0,0) -- (0,2) -- (1,1) -- (0,0);
\draw (3,0)--(3,2)--(2,1)--(3,0);
\draw[->] (0,2) -- (1.5,2) node [above] {$\NGLabel{1}$};
\draw[->] (1,1) -- (1.5,1) node [above] {$\NGLabel{2}$};
\draw[-<] (0,0) -- (1.5,0) node [above] {$\NGLabel{3}$};
\draw[->] (0,2) --++ (.5,-.5) node [right] {$\NGLabel{4}$};
\draw[->] (1,1) --++ (-.5,-.5) node [right] {$\NGLabel{5}$};
\draw[->] (0,0) -- (0,1) node [left] {$\NGLabel{6}$};
\draw[-<] (3,2) --++ (-.5,-.5) node [left] {$\NGLabel{7}$};
\draw[->] (2,1) --++ (.5,-.5) node [left] {$\NGLabel{8}$};
\draw[-<] (3,0) -- (3,1) node [right] {$\NGLabel{9}$};
\end{tikzpicture}%
}

\expandafter\def\csname ngshape@T016\endcsname{%
\begin{tikzpicture}[scale=1.125,ng diagram]
\draw[-<] (0,0) 	node [below] {}	-- (1,0) 		node [below] 	{$\NGLabel{1}$}; \draw (1,0) 		-- (2,0);
\draw[-<] (2,0) 	node [below] {} 	-- (1.5,.8660254) 	node [right] 	{$\NGLabel{2}$}; \draw (1.5,.8660254) 	-- (1,1.7320508);
\draw[-<] (1,1.7320508)	node [above] {} 	-- (.5,.8660254) 	node [left] 	{$\NGLabel{3}$}; \draw (.5,.8660254) 	-- (0,0);
\draw[-<] (0,0) 	-- (.5,.2886751) 	node [above] 	{$\NGLabel{4}$}; \draw (.5,.2886751) 	-- (1,0.5773503);
\draw[-<] (1,0.5773503) node [below] {} -- (1,1.1547005) 	node [right]	{$\NGLabel{5}$}; \draw (1,1.1547005) 	-- (1,1.7320508);
\draw[->] (1,0.5773503) -- (1.5,.2886751) 	node [above] 	{$\NGLabel{6}$}; \draw (1.5,.2886751) 	-- (2,0);
\end{tikzpicture}%
}

\expandafter\def\csname ngshape@T017\endcsname{%
\begin{tikzpicture}[scale=1.125,ng diagram]
\draw[-<] (0,0) 	node [below] {}	-- (1,0) 		node [below] 	{$\NGLabel{1}$}; \draw (1,0) 		-- (2,0);
\draw[-<] (2,0) 	node [below] {} 	-- (1.5,.8660254) 	node [right] 	{$\NGLabel{2}$}; \draw (1.5,.8660254) 	-- (1,1.7320508);
\draw[-<] (1,1.7320508)	node [above] {} 	-- (.5,.8660254) 	node [left] 	{$\NGLabel{3}$}; \draw (.5,.8660254) 	-- (0,0);
\draw[->] (0,0) 	-- (.5,.2886751) 	node [above] 	{$\NGLabel{4}$}; \draw (.5,.2886751) 	-- (1,0.5773503);
\draw[->] (1,0.5773503) node [below] {} -- (1,1.1547005) 	node [right]	{$\NGLabel{5}$}; \draw (1,1.1547005) 	-- (1,1.7320508);
\draw[->] (1,0.5773503) -- (1.5,.2886751) 	node [above] 	{$\NGLabel{6}$}; \draw (1.5,.2886751) 	-- (2,0);
\end{tikzpicture}%
}

\expandafter\def\csname ngshape@T018\endcsname{%
\begin{tikzpicture}[scale=1.125,ng diagram]
\draw[->] (0,0) 	node [below] {}	-- (1,0) 		node [below] 	{$\NGLabel{1}$}; \draw (1,0) 		-- (2,0);
\draw[-<] (2,0) 	node [below] {} 	-- (1.5,.8660254) 	node [right] 	{$\NGLabel{2}$}; \draw (1.5,.8660254) 	-- (1,1.7320508);
\draw[-<] (1,1.7320508)	node [above] {} 	-- (.5,.8660254) 	node [left] 	{$\NGLabel{3}$}; \draw (.5,.8660254) 	-- (0,0);
\draw[-<] (0,0) 	-- (.5,.2886751) 	node [above] 	{$\NGLabel{4}$}; \draw (.5,.2886751) 	-- (1,0.5773503);
\draw[-<] (1,0.5773503) node [below] {} -- (1,1.1547005) 	node [right]	{}; \draw (1,1.1547005) 	-- (1,1.7320508);
\draw[-<] (1,0.5773503) -- (1.5,.2886751) 	node [above] 	{$\NGLabel{6}$}; \draw (1.5,.2886751) 	-- (2,0);
\node at (1,.69)[right=1.5pt] {$\NGLabel{5}$};
\end{tikzpicture}%
}

\expandafter\def\csname ngshape@T019\endcsname{%
\begin{tikzpicture}[scale=1.125,ng diagram]
\draw[->] (0,0) 	node [below] {}	-- (1,0) 		node [below] 	{$\NGLabel{1}$}; \draw (1,0) 		-- (2,0);
\draw[->] (2,0) 	node [below] {} 	-- (1.5,.8660254) 	node [right] 	{$\NGLabel{2}$}; \draw (1.5,.8660254) 	-- (1,1.7320508);
\draw[-<] (1,1.7320508)	node [above] {} 	-- (.5,.8660254) 	node [left] 	{$\NGLabel{3}$}; \draw (.5,.8660254) 	-- (0,0);
\draw[-<] (0,0) 	-- (.5,.2886751) 	node [above] 	{$\NGLabel{4}$}; \draw (.5,.2886751) 	-- (1,0.5773503);
\draw[-<] (1,0.5773503) node [below] {} -- (1,1.1547005) 	node [right]	{}; \draw (1,1.1547005) 	-- (1,1.7320508);
\draw[-<] (1,0.5773503) -- (1.5,.2886751) 	node [above] 	{$\NGLabel{6}$}; \draw (1.5,.2886751) 	-- (2,0);
\node at (1,.69)[right=1.5pt] {$\NGLabel{5}$};
\end{tikzpicture}%
}

\expandafter\def\csname ngshape@T020\endcsname{%
\begin{tikzpicture}[scale=1.125,ng diagram]
\draw[-<] (0,0) 	node [below] {}	-- (1,0) 		node [below] 	{$\NGLabel{1}$}; \draw (1,0) 		-- (2,0);
\draw[-<] (2,0) 	node [below] {} 	-- (1.5,.8660254) 	node [right] 	{$\NGLabel{2}$}; \draw (1.5,.8660254) 	-- (1,1.7320508);
\draw[-<] (1,1.7320508)	node [above] {} 	-- (.5,.8660254) 	node [left] 	{$\NGLabel{3}$}; \draw (.5,.8660254) 	-- (0,0);
\draw[-<] (0,0) 	-- (.5,.2886751) 	node[above] {}; \draw (.5,.2886751) 	-- (1,0.5773503);
\draw[-<] (1,0.5773503) node [below] {} -- (1,1.1547005) 	node[right] {}; \draw (1,1.1547005) 	-- (1,1.7320508);
\draw[->] (1,0.5773503) -- (1.5,.2886751) 	node[above] {}; \draw (1.5,.2886751) 	-- (2,0);
\node at (.65,.16) {$\NGLabel{4}$};
\node at (1.35,.16) {$\NGLabel{5}$};
\node at (1.24,.72) {$\NGLabel{6}$};
\end{tikzpicture}%
}

\expandafter\def\csname ngshape@P004\endcsname{%
\begin{tikzpicture}[scale=0.75,ng diagram]
\draw (0,2)--(3,2);
\draw (1,1)--(2,1);
\draw (0,0)--(3,0);
\draw (0,0) -- (0,2) -- (1,1) -- (0,0);
\draw (3,0)--(3,2)--(2,1)--(3,0);
\draw[->] (0,2) -- (1.5,2) node [above] {$\NGLabel{1}$};
\draw[->] (1,1) -- (1.5,1) node [above] {$\NGLabel{2}$};
\draw[-<] (0,0) -- (1.5,0) node [above] {$\NGLabel{3}$};
\draw[] (0,2) --++ (.5,-.5) node [right] {};
\draw[] (1,1) --++ (-.5,-.5) node [right] {};
\draw[] (0,0) -- (0,1) node [left] {};
\draw[-<] (3,2) --++ (-.5,-.5) node [left] {$\NGLabel{4}$};
\draw[->] (2,1) --++ (.5,-.5) node [left] {$\NGLabel{5}$};
\draw[-<] (3,0) -- (3,1) node [right] {$\NGLabel{6}$};
\end{tikzpicture}%
}

\expandafter\def\csname ngshape@T021\endcsname{%
\begin{tikzpicture}[scale=1.125,ng diagram]
\draw[-<] (0,0) 	node [below] {}	-- (1,0) 		node [below] 	{$\NGLabel{1}$}; \draw (1,0) 		-- (2,0);
\draw[-<] (2,0) 	node [below] {} 	-- (1.5,.8660254) 	node [right] 	{$\NGLabel{2}$}; \draw (1.5,.8660254) 	-- (1,1.7320508);
\draw[] (1,1.7320508)	node [above] {} 	-- (.5,.8660254) 	node [left] 	{}; \draw (.5,.8660254) 	-- (0,0);
\draw[] (0,0) 	-- (.5,.2886751) 	node [above] 	{}; \draw (.5,.2886751) 	-- (1,0.5773503);
\draw[] (1,0.5773503) node [below] {} -- (1,1.1547005) 	node [right]	{}; \draw (1,1.1547005) 	-- (1,1.7320508);
\draw[->] (1,0.5773503) -- (1.5,.2886751) 	node [above] 	{$\NGLabel{3}$}; \draw (1.5,.2886751) 	-- (2,0);
\end{tikzpicture}%
}

\expandafter\def\csname ngshape@T022\endcsname{%
\begin{tikzpicture}[scale=1.125,ng diagram]
\draw[-<] (0,0) 	node [below] {}	-- (1,0) 		node [below] 	{$\NGLabel{1}$}; \draw (1,0) 		-- (2,0);
\draw[-<] (2,0) 	node [below] {} 	-- (1.5,.8660254) 	node [right] 	{$\NGLabel{2}$}; \draw (1.5,.8660254) 	-- (1,1.7320508);
\draw[-<] (1,1.7320508)	node [above] {} 	-- (.5,.8660254) 	node [left] 	{$\NGLabel{3}$}; \draw (.5,.8660254) 	-- (0,0);
\draw[->] (0,0) 	-- (.5,.2886751) 	node [above] 	{$\NGLabel{4}$}; \draw (.5,.2886751) 	-- (1,0.5773503);
\draw[->] (1,0.5773503) node [below] {} -- (1,1.1547005) 	node [right]	{$\NGLabel{5}$}; \draw (1,1.1547005) 	-- (1,1.7320508);
\draw[->] (1,0.5773503) -- (1.5,.2886751) 	node[above] {}; \draw (1.5,.2886751) 	-- (2,0);
\node at (1.35,.16) {$\NGLabel{6}$};
\end{tikzpicture}%
}

\expandafter\def\csname ngshape@T023\endcsname{%
\begin{tikzpicture}[scale=1.125,ng diagram]
\draw[->] (0,0) 	node [below] {}	-- (1,0) 		node [below] 	{$\NGLabel{1}$}; \draw (1,0) 		-- (2,0);
\draw[-<] (2,0) 	node [below] {} 	-- (1.5,.8660254) 	node [right] 	{$\NGLabel{2}$}; \draw (1.5,.8660254) 	-- (1,1.7320508);
\draw[] (1,1.7320508)	node [above] {} 	-- (.5,.8660254) 	node [left] 	{}; \draw (.5,.8660254) 	-- (0,0);
\draw[] (0,0) 	-- (.5,.2886751) 	node [above] 	{}; \draw (.5,.2886751) 	-- (1,0.5773503);
\draw[] (1,0.5773503) node [below] {} -- (1,1.1547005) 	node [right]	{}; \draw (1,1.1547005) 	-- (1,1.7320508);
\draw[-<] (1,0.5773503) -- (1.5,.2886751) 	node[above] {}; \draw (1.5,.2886751) 	-- (2,0);
\node at (1.35,.16) {$\NGLabel{3}$};
\end{tikzpicture}%
}

\expandafter\def\csname ngshape@T024\endcsname{%
\begin{tikzpicture}[scale=1.125,ng diagram]
\draw[->] (0,0) 	node [below] {}	-- (1,0) 		node [below] 	{$\NGLabel{1}$}; \draw (1,0) 		-- (2,0);
\draw[->] (2,0) 	node [below] {} 	-- (1.5,.8660254) 	node [right] 	{$\NGLabel{2}$}; \draw (1.5,.8660254) 	-- (1,1.7320508);
\draw[] (1,1.7320508)	node [above] {} 	-- (.5,.8660254) 	node [left] 	{}; \draw (.5,.8660254) 	-- (0,0);
\draw[] (0,0) 	-- (.5,.2886751) 	node [above] 	{}; \draw (.5,.2886751) 	-- (1,0.5773503);
\draw[] (1,0.5773503) node [below] {} -- (1,1.1547005) 	node [right]	{}; \draw (1,1.1547005) 	-- (1,1.7320508);
\draw[-<] (1,0.5773503) -- (1.5,.2886751) 	node [above] 	{$\NGLabel{3}$}; \draw (1.5,.2886751) 	-- (2,0);
\end{tikzpicture}%
}

\expandafter\def\csname ngshape@T025\endcsname{%
\begin{tikzpicture}[scale=1.125,ng diagram]
\draw[-<] (0,0) 	node [below] {}	-- (1,0) 		node [below] 	{$\NGLabel{1}$}; \draw (1,0) 		-- (2,0);
\draw[-<] (2,0) 	node [below] {} 	-- (1.5,.8660254) 	node [right] 	{$\NGLabel{2}$}; \draw (1.5,.8660254) 	-- (1,1.7320508);
\draw[] (1,1.7320508)	node [above] {} 	-- (.5,.8660254) 	node [left] 	{}; \draw (.5,.8660254) 	-- (0,0);
\draw[] (0,0) 	-- (.5,.2886751) 	node [above] 	{}; \draw (.5,.2886751) 	-- (1,0.5773503);
\draw[] (1,0.5773503) node [below] {} -- (1,1.1547005) 	node [right]	{}; \draw (1,1.1547005) 	-- (1,1.7320508);
\draw[->] (1,0.5773503) -- (1.5,.2886751) 	node[above] {}; \draw (1.5,.2886751) 	-- (2,0);
\node at (1.35,.16) {$\NGLabel{3}$};
\end{tikzpicture}%
}

\expandafter\def\csname ngshape@P005\endcsname{%
\begin{tikzpicture}[scale=0.75,ng diagram]
\draw (0,2)--(3,2);
\draw (1,1)--(2,1);
\draw (0,0)--(3,0);
\draw (0,0) -- (0,2) -- (1,1) -- (0,0);
\draw (3,0)--(3,2)--(2,1)--(3,0);
\draw[] (0,2) -- (1.5,2) node [above] {};
\draw[] (1,1) -- (1.5,1) node [above] {};
\draw[->] (0,0) -- (1.5,0) node [above] {$\NGLabel{1}$};
\draw[] (0,2) --++ (.5,-.5) node [right] {};
\draw[->] (1,1) --++ (-.5,-.5) node [right] {$\NGLabel{2}$};
\draw[-<] (0,0) -- (0,1) node [left] {$\NGLabel{3}$};
\draw[] (3,2) --++ (-.5,-.5) node [left] {};
\draw[->] (2,1) --++ (.5,-.5) node [left] {$\NGLabel{4}$};
\draw[->] (3,0) -- (3,1) node [right] {$\NGLabel{5}$};
\end{tikzpicture}%
}

\expandafter\def\csname ngshape@T026\endcsname{%
\begin{tikzpicture}[scale=1.125,ng diagram]
\draw[->] (0,0) 	node [below] {}	-- (1,0) 		node [below] 	{$\NGLabel{1}$}; \draw (1,0) 		-- (2,0);
\draw[->] (2,0) 	node [below] {} 	-- (1.5,.8660254) 	node [right] 	{$\NGLabel{2}$}; \draw (1.5,.8660254) 	-- (1,1.7320508);
\draw[] (1,1.7320508)	node [above] {} 	-- (.5,.8660254) 	node [left] 	{}; \draw (.5,.8660254) 	-- (0,0);
\draw[] (0,0) 	-- (.5,.2886751) 	node [above] 	{}; \draw (.5,.2886751) 	-- (1,0.5773503);
\draw[] (1,0.5773503) node [below] {} -- (1,1.1547005) 	node [right]	{}; \draw (1,1.1547005) 	-- (1,1.7320508);
\draw[->] (1,0.5773503) -- (1.5,.2886751) 	node [above] 	{$\NGLabel{3}$}; \draw (1.5,.2886751) 	-- (2,0);
\end{tikzpicture}%
}

\expandafter\def\csname ngshape@T027\endcsname{%
\begin{tikzpicture}[scale=1.125,ng diagram]
\draw[] (0,0) 	node [below] {}	-- (1,0) 		node [below] 	{}; \draw (1,0) 		-- (2,0);
\draw[] (2,0) 	node [below] {} 	-- (1.5,.8660254) 	node [right] 	{}; \draw (1.5,.8660254) 	-- (1,1.7320508);
\draw[] (1,1.7320508)	node [above] {} 	-- (.5,.8660254) 	node [left] 	{}; \draw (.5,.8660254) 	-- (0,0);
\draw[->] (0,0) 	-- (.5,.2886751) 	node [above] 	{$\NGLabel{1}$}; \draw (.5,.2886751) 	-- (1,0.5773503);
\draw[-<] (1,0.5773503) node [below] {} -- (1,1.1547005) 	node[right] {}; \draw (1,1.1547005) 	-- (1,1.7320508);
\draw[->] (1,0.5773503) -- (1.5,.2886751) 	node[above] {}; \draw (1.5,.2886751) 	-- (2,0);
\node at (1.08,.16) {$\NGLabel{2}$};
\node at (1,.69)[right=1.5pt] {$\NGLabel{3}$};
\end{tikzpicture}%
}

\expandafter\def\csname ngshape@T028\endcsname{%
\begin{tikzpicture}[scale=1.125,ng diagram]
\draw[] (0,0) 	node [below] {}	-- (1,0) 		node [below] 	{}; \draw (1,0) 		-- (2,0);
\draw[->] (2,0) 	node [below] {} 	-- (1.5,.8660254) 	node [right] 	{$\NGLabel{1}$}; \draw (1.5,.8660254) 	-- (1,1.7320508);
\draw[-<] (1,1.7320508)	node [above] {} 	-- (.5,.8660254) 	node [left] 	{$\NGLabel{2}$}; \draw (.5,.8660254) 	-- (0,0);
\draw[] (0,0) 	-- (.5,.2886751) 	node [above] 	{}; \draw (.5,.2886751) 	-- (1,0.5773503);
\draw[-<] (1,0.5773503) node [below] {} -- (1,1.1547005) 	node[right] {}; \draw (1,1.1547005) 	-- (1,1.7320508);
\draw[] (1,0.5773503) -- (1.5,.2886751) 	node [above] 	{}; \draw (1.5,.2886751) 	-- (2,0);
\node at (1,.69)[right=1.5pt] {$\NGLabel{3}$};
\end{tikzpicture}%
}

\expandafter\def\csname ngshape@T029\endcsname{%
\begin{tikzpicture}[scale=1.125,ng diagram]
\draw[->] (0,0) 	node [below] {}	-- (1,0) 		node [below] 	{$\NGLabel{1}$}; \draw (1,0) 		-- (2,0);
\draw[->] (2,0) 	node [below] {} 	-- (1.5,.8660254) 	node [right] 	{$\NGLabel{2}$}; \draw (1.5,.8660254) 	-- (1,1.7320508);
\draw[->] (1,1.7320508)	node [above] {} 	-- (.5,.8660254) 	node [left] 	{$\NGLabel{3}$}; \draw (.5,.8660254) 	-- (0,0);
\draw[-<] (0,0) 	-- (.5,.2886751) 	node [above] 	{$\NGLabel{4}$}; \draw (.5,.2886751) 	-- (1,0.5773503);
\draw[-<] (1,0.5773503) node [below] {} -- (1,1.1547005) 	node[right] {}; \draw (1,1.1547005) 	-- (1,1.7320508);
\draw[->] (1,0.5773503) -- (1.5,.2886751) 	node [above] 	{$\NGLabel{5}$}; \draw (1.5,.2886751) 	-- (2,0);
\node at (1,.69)[right=1.5pt] {$\NGLabel{6}$};
\end{tikzpicture}%
}

\expandafter\def\csname ngshape@P006\endcsname{%
\begin{tikzpicture}[scale=0.75,ng diagram]
\draw (0,2)--(3,2);
\draw (1,1)--(2,1);
\draw (0,0)--(3,0);
\draw (0,0) -- (0,2) -- (1,1) -- (0,0);
\draw (3,0)--(3,2)--(2,1)--(3,0);
\draw[] (0,2) -- (1.5,2) node [above] {};
\draw[] (1,1) -- (1.5,1) node [above] {};
\draw[->] (0,0) -- (1.5,0) node [above] {$\NGLabel{1}$};
\draw[] (0,2) --++ (.5,-.5) node [right] {};
\draw[->] (1,1) --++ (-.5,-.5) node [right] {$\NGLabel{2}$};
\draw[-<] (0,0) -- (0,1) node [left] {$\NGLabel{3}$};
\draw[-<] (3,2) --++ (-.5,-.5) node [left] {$\NGLabel{4}$};
\draw[] (2,1) --++ (.5,-.5) node [left] {};
\draw[] (3,0) -- (3,1) node [right] {};
\end{tikzpicture}%
}

\expandafter\def\csname ngshape@P007\endcsname{%
\begin{tikzpicture}[scale=0.84375,ng diagram]
\draw (0,2)--(3,2);
\draw (1,1)--(2,1);
\draw (0,0)--(3,0);
\draw (0,0) -- (0,2) -- (1,1) -- (0,0);
\draw (3,0)--(3,2)--(2,1)--(3,0);
\draw[] (0,2) -- (1.5,2) node [above] {};
\draw[] (1,1) -- (1.5,1) node [above] {};
\draw[->] (0,0) -- (1.5,0) node [above] {$\NGLabel{1}$};
\draw[] (0,2) --++ (.5,-.5) node [right] {};
\draw[->] (1,1) --++ (-.5,-.5) node [right] {$\NGLabel{2}$};
\draw[-<] (0,0) -- (0,1) node [left] {$\NGLabel{3}$};
\draw[] (3,2) --++ (-.5,-.5) node [left] {};
\draw[] (2,1) --++ (.5,-.5) node [left] {};
\draw[] (3,0) -- (3,1) node [right] {};
\node at (1.9,.9) {$\NGLabel{4}$};
\node at (3.1,2.1) {$\NGLabel{5}$};
\end{tikzpicture}%
}

\expandafter\def\csname ngshape@P008\endcsname{%
\begin{tikzpicture}[scale=0.84375,ng diagram]
\draw (0,2)--(3,2);
\draw (1,1)--(2,1);
\draw (0,0)--(3,0);
\draw (0,0) -- (0,2) -- (1,1) -- (0,0);
\draw (3,0)--(3,2)--(2,1)--(3,0);
\draw[] (0,2) -- (1.5,2) node [above] {};
\draw[->] (1,1) -- (1.5,1) node [above] {$\NGLabel{1}$};
\draw[] (0,0) -- (1.5,0) node [above] {};
\draw[] (0,2) --++ (.5,-.5) node [right] {};
\draw[] (1,1) --++ (-.5,-.5) node [right] {};
\draw[-<] (0,0) -- (0,1) node [left] {$\NGLabel{2}$};
\draw[] (3,2) --++ (-.5,-.5) node [left] {};
\draw[] (2,1) --++ (.5,-.5) node [left] {};
\draw[->] (3,0) -- (3,1) node [right] {$\NGLabel{3}$};
\end{tikzpicture}%
}

\expandafter\def\csname ngshape@P009\endcsname{%
\begin{tikzpicture}[scale=0.84375,ng diagram]
\draw (0,2)--(3,2);
\draw (1,1)--(2,1);
\draw (0,0)--(3,0);
\draw (0,0) -- (0,2) -- (1,1) -- (0,0);
\draw (3,0)--(3,2)--(2,1)--(3,0);
\draw[->] (0,2) -- (1.5,2) node [above] {$\NGLabel{1}$};
\draw[->] (1,1) -- (1.5,1) node [above] {$\NGLabel{2}$};
\draw[-<] (0,0) -- (1.5,0) node [above] {$\NGLabel{3}$};
\draw[] (0,2) --++ (.5,-.5) node [right] {};
\draw[] (1,1) --++ (-.5,-.5) node [right] {};
\draw[] (0,0) -- (0,1) node [left] {};
\draw[] (3,2) --++ (-.5,-.5) node [left] {};
\draw[] (2,1) --++ (.5,-.5) node [left] {};
\draw[] (3,0) -- (3,1) node [right] {};
\end{tikzpicture}%
}

\expandafter\def\csname ngshape@P010\endcsname{%
\begin{tikzpicture}[scale=0.75,ng diagram]
\draw (0,2)--(3,2);
\draw (1,1)--(2,1);
\draw (0,0)--(3,0);
\draw (0,0) -- (0,2) -- (1,1) -- (0,0);
\draw (3,0)--(3,2)--(2,1)--(3,0);
\draw[] (0,2) -- (1.5,2) node [above] {};
\draw[] (1,1) -- (1.5,1) node [above] {};
\draw[->] (0,0) -- (1.5,0) node [above] {$\NGLabel{1}$};
\draw[] (0,2) --++ (.5,-.5) node [right] {};
\draw[->] (1,1) --++ (-.5,-.5) node [right] {$\NGLabel{2}$};
\draw[-<] (0,0) -- (0,1) node [left] {$\NGLabel{3}$};
\draw[] (3,2) --++ (-.5,-.5) node [left] {};
\draw[] (2,1) --++ (.5,-.5) node [left] {};
\draw[] (3,0) -- (3,1) node [right] {};
\node at (1.9,.9) {$\NGLabel{4}$};
\node at (3.1,2.1) {$\NGLabel{5}$};
\end{tikzpicture}%
}

\expandafter\def\csname ngshape@T030\endcsname{%
\begin{tikzpicture}[scale=1.125,ng diagram]
\draw[->] (0,0) 	node [below] {}	-- (1,0) 		node [below] 	{$\NGLabel{1}$}; \draw (1,0) 		-- (2,0);
\draw[] (2,0) 	node [below] {} 	-- (1.5,.8660254) 	node [right] 	{}; \draw (1.5,.8660254) 	-- (1,1.7320508);
\draw[] (1,1.7320508)	node [above] {$\NGLabel{2}$} 	-- (.5,.8660254) 	node [left] 	{}; \draw (.5,.8660254) 	-- (0,0);
\draw[] (0,0) 	-- (.5,.2886751) 	node [above] 	{}; \draw (.5,.2886751) 	-- (1,0.5773503);
\draw[] (1,0.5773503) node [below] {$\NGLabel{3}$} -- (1,1.1547005) 	node [right]	{}; \draw (1,1.1547005) 	-- (1,1.7320508);
\draw[] (1,0.5773503) -- (1.5,.2886751) 	node [above] 	{}; \draw (1.5,.2886751) 	-- (2,0);
\end{tikzpicture}%
}

\expandafter\def\csname ngshape@P011\endcsname{%
\begin{tikzpicture}[scale=0.75,ng diagram]
\draw (0,2)--(3,2);
\draw (1,1)--(2,1);
\draw (0,0)--(3,0);
\draw (0,0) -- (0,2) -- (1,1) -- (0,0);
\draw (3,0)--(3,2)--(2,1)--(3,0);
\draw[] (0,2) -- (1.5,2) node [above] {};
\draw[->] (1,1) -- (1.5,1) node [above] {$\NGLabel{1}$};
\draw[] (0,0) -- (1.5,0) node [above] {};
\draw[] (0,2) --++ (.5,-.5) node [right] {};
\draw[] (1,1) --++ (-.5,-.5) node [right] {};
\draw[-<] (0,0) -- (0,1) node [left] {$\NGLabel{2}$};
\draw[] (3,2) --++ (-.5,-.5) node [left] {};
\draw[] (2,1) --++ (.5,-.5) node [left] {};
\draw[->] (3,0) -- (3,1) node [right] {$\NGLabel{3}$};
\end{tikzpicture}%
}

\expandafter\def\csname ngshape@T031\endcsname{%
\begin{tikzpicture}[scale=1.125,ng diagram]
\draw[] (0,0) 	node [below] {}	-- (1,0) 		node [below] 	{}; \draw (1,0) 		-- (2,0);
\draw[-<] (2,0) 	node [below] {} 	-- (1.5,.8660254) 	node [right] 	{$\NGLabel{1}$}; \draw (1.5,.8660254) 	-- (1,1.7320508);
\draw[-<] (1,1.7320508)	node [above] {} 	-- (.5,.8660254) 	node [left] 	{$\NGLabel{2}$}; \draw (.5,.8660254) 	-- (0,0);
\draw[] (0,0) 	-- (.5,.2886751) 	node [above] 	{}; \draw (.5,.2886751) 	-- (1,0.5773503);
\draw[-<] (1,0.5773503) node [below] {} -- (1,1.1547005) 	node [right]	{$\NGLabel{3}$}; \draw (1,1.1547005) 	-- (1,1.7320508);
\draw[] (1,0.5773503) -- (1.5,.2886751) 	node [above] 	{}; \draw (1.5,.2886751) 	-- (2,0);
\end{tikzpicture}%
}

\expandafter\def\csname ngshape@T032\endcsname{%
\begin{tikzpicture}[scale=1.125,ng diagram]
\draw[->] (0,0) 	node [below] {}	-- (1,0) 		node [below] 	{$\NGLabel{1}$}; \draw (1,0) 		-- (2,0);
\draw(2,0) 	node [below] {} 	-- (1.5,.8660254) 	node [right] 	{}; \draw (1.5,.8660254) 	-- (1,1.7320508);
\draw (1,1.7320508)	node [above] {} 	-- (.5,.8660254) 	node [left] 	{}; \draw (.5,.8660254) 	-- (0,0);
\draw (0,0) 	-- (.5,.2886751) 	node [above] 	{};
\draw (.5,.2886751) 	-- (1,0.5773503);
\draw[-<] (1,0.5773503) node [below] {} -- (1,1.1547005) 	node[right] {}; \draw (1,1.1547005) 	-- (1,1.7320508);
\draw(1,0.5773503) -- (1.5,.2886751) 	node [above] 	{}; \draw (1.5,.2886751) 	-- (2,0);
\node at (1,.69)[right=1.5pt] {$\NGLabel{2}$};
\end{tikzpicture}%
}

\expandafter\def\csname ngshape@T033\endcsname{%
\begin{tikzpicture}[scale=1.125,ng diagram]
\draw[] (0,0) 	node [below] {}	-- (1,0) 		node [below] 	{}; \draw (1,0) 		-- (2,0);
\draw[->] (2,0) 	node [below] {} 	-- (1.5,.8660254) 	node [right] 	{$\NGLabel{1}$}; \draw (1.5,.8660254) 	-- (1,1.7320508);
\draw[->] (1,1.7320508)	node [above] {} 	-- (.5,.8660254) 	node [left] 	{$\NGLabel{2}$}; \draw (.5,.8660254) 	-- (0,0);
\draw[] (0,0) 	-- (.5,.2886751) 	node [above] 	{}; \draw (.5,.2886751) 	-- (1,0.5773503);
\draw[->] (1,0.5773503) node [below] {} -- (1,1.1547005) 	node[right] {}; \draw (1,1.1547005) 	-- (1,1.7320508);
\draw[] (1,0.5773503) -- (1.5,.2886751) 	node [above] 	{}; \draw (1.5,.2886751) 	-- (2,0);
\node at (1,.69)[right=1.5pt] {$\NGLabel{3}$};
\end{tikzpicture}%
}

\expandafter\def\csname ngshape@P012\endcsname{%
\begin{tikzpicture}[scale=0.84375,ng diagram]
\draw (0,2)--(3,2);
\draw (1,1)--(2,1);
\draw (0,0)--(3,0);
\draw (0,0) -- (0,2) -- (1,1) -- (0,0);
\draw (3,0)--(3,2)--(2,1)--(3,0);
\draw[] (0,2) -- (1.5,2) node [above] {};
\draw[] (1,1) -- (1.5,1) node [above] {};
\draw[] (0,0) -- (1.5,0) node [above] {};
\draw[] (0,2) --++ (.5,-.5) node [right] {};
\draw[-<] (1,1) --++ (-.5,-.5) node [right] {$\NGLabel{1}$};
\draw[] (0,0) -- (0,1) node [left] {};
\draw[] (3,2) --++ (-.5,-.5) node [left] {};
\draw[-<] (2,1) --++ (.5,-.5) node [left] {$\NGLabel{2}$};
\draw[] (3,0) -- (3,1) node [right] {};
\node at (.24,1.58) {$\NGLabel{3}$};
\node at (2.76,1.58){$\NGLabel{4}$};
\end{tikzpicture}%
}

\expandafter\def\csname ngshape@P013\endcsname{%
\begin{tikzpicture}[scale=0.84375,ng diagram]
\draw (0,2)--(3,2);
\draw (1,1)--(2,1);
\draw (0,0)--(3,0);
\draw (0,0) -- (0,2) -- (1,1) -- (0,0);
\draw (3,0)--(3,2)--(2,1)--(3,0);
\draw[] (0,2) -- (1.5,2) node [above] {};
\draw[] (1,1) -- (1.5,1) node [above] {};
\draw[] (0,0) -- (1.5,0) node [above] {};
\draw[] (0,2) --++ (.5,-.5) node [right] {};
\draw[->] (1,1) --++ (-.5,-.5) node [right] {$\NGLabel{1}$};
\draw[] (0,0) -- (0,1) node [left] {};
\draw[] (3,2) --++ (-.5,-.5) node [left] {};
\draw[] (2,1) --++ (.5,-.5) node [left] {};
\draw[->] (3,0) -- (3,1) node [right] {$\NGLabel{2}$};
\node at (-.1,2.1) {$\NGLabel{3}$};
\node at (2,1)[right] {$\NGLabel{4}$};
\end{tikzpicture}%
}

\expandafter\def\csname ngshape@P014\endcsname{%
\begin{tikzpicture}[scale=0.84375,ng diagram]
\draw (0,2)--(3,2);
\draw (1,1)--(2,1);
\draw (0,0)--(3,0);
\draw (0,0) -- (0,2) -- (1,1) -- (0,0);
\draw (3,0)--(3,2)--(2,1)--(3,0);
\draw[] (0,2) -- (1.5,2) node [above] {};
\draw[->] (1,1) -- (1.5,1) node [above] {$\NGLabel{1}$};
\draw[] (0,0) -- (1.5,0) node [above] {};
\draw[] (0,2) --++ (.5,-.5) node [right] {};
\draw[] (1,1) --++ (-.5,-.5) node [right] {};
\draw[] (0,0) -- (0,1) node [left] {};
\draw[] (3,2) --++ (-.5,-.5) node [left] {};
\draw[] (2,1) --++ (.5,-.5) node [left] {};
\draw[->] (3,0) -- (3,1) node [right] {$\NGLabel{2}$};
\node at (-.1,2.1) {$\NGLabel{3}$};
\node at (-.1,-.1) {$\NGLabel{4}$};
\end{tikzpicture}%
}

\expandafter\def\csname ngshape@P015\endcsname{%
\begin{tikzpicture}[scale=0.75,ng diagram]
\draw (0,2)--(3,2);
\draw (1,1)--(2,1);
\draw (0,0)--(3,0);
\draw (0,0) -- (0,2) -- (1,1) -- (0,0);
\draw (3,0)--(3,2)--(2,1)--(3,0);
\draw[] (0,2) -- (1.5,2) node [above] {};
\draw[] (1,1) -- (1.5,1) node [above] {};
\draw[] (0,0) -- (1.5,0) node [above] {};
\draw[] (0,2) --++ (.5,-.5) node [right] {};
\draw[-<] (1,1) --++ (-.5,-.5) node [right] {$\NGLabel{1}$};
\draw[] (0,0) -- (0,1) node [left] {};
\draw[] (3,2) --++ (-.5,-.5) node [left] {};
\draw[-<] (2,1) --++ (.5,-.5) node [left] {$\NGLabel{2}$};
\draw[] (3,0) -- (3,1) node [right] {};
\node at (.24,1.58) {$\NGLabel{3}$};
\node at (2.76,1.58){$\NGLabel{4}$};
\end{tikzpicture}%
}

\expandafter\def\csname ngshape@T034\endcsname{%
\begin{tikzpicture}[scale=1.125,ng diagram]
\draw[] (0,0) 	node [below] {}	-- (1,0) 		node [below] 	{}; \draw (1,0) 		-- (2,0);
\draw[-<] (2,0) 	node [below] {} 	-- (1.5,.8660254) 	node [right] 	{$\NGLabel{1}$}; \draw (1.5,.8660254) 	-- (1,1.7320508);
\draw[->] (1,1.7320508)	node [above] {} 	-- (.5,.8660254) 	node [left] 	{$\NGLabel{2}$}; \draw (.5,.8660254) 	-- (0,0);
\draw[] (0,0) 	-- (.5,.2886751) 	node [above] 	{}; \draw (.5,.2886751) 	-- (1,0.5773503);
\draw[->] (1,0.5773503) node [below] {} -- (1,1.1547005) 	node[right] {}; \draw (1,1.1547005) 	-- (1,1.7320508);
\draw[] (1,0.5773503) -- (1.5,.2886751) 	node [above] 	{}; \draw (1.5,.2886751) 	-- (2,0);
\node at (1,.69)[right=1.5pt] {$\NGLabel{3}$};
\end{tikzpicture}%
}

\expandafter\def\csname ngshape@P016\endcsname{%
\begin{tikzpicture}[scale=0.75,ng diagram]
\draw (0,2)--(3,2);
\draw (1,1)--(2,1);
\draw (0,0)--(3,0);
\draw (0,0) -- (0,2) -- (1,1) -- (0,0);
\draw (3,0)--(3,2)--(2,1)--(3,0);
\draw[] (0,2) -- (1.5,2) node [above] {};
\draw[] (1,1) -- (1.5,1) node [above] {};
\draw[] (0,0) -- (1.5,0) node [above] {};
\draw[] (0,2) --++ (.5,-.5) node [right] {};
\draw[->] (1,1) --++ (-.5,-.5) node [right] {$\NGLabel{1}$};
\draw[] (0,0) -- (0,1) node [left] {};
\draw[] (3,2) --++ (-.5,-.5) node [left] {};
\draw[] (2,1) --++ (.5,-.5) node [left] {};
\draw[->] (3,0) -- (3,1) node [right] {$\NGLabel{2}$};
\node at (-.1,2.1) {$\NGLabel{3}$};
\node at (2,1)[right] {$\NGLabel{4}$};
\end{tikzpicture}%
}

\expandafter\def\csname ngshape@T035\endcsname{%
\begin{tikzpicture}[scale=1.125,ng diagram]
\draw[->] (0,0) 	node [below] {}	-- (1,0) 		node [below] 	{$\NGLabel{1}$}; \draw (1,0) 		-- (2,0);
\draw(2,0) 	node [below] {} 	-- (1.5,.8660254) 	node [right] 	{}; \draw (1.5,.8660254) 	-- (1,1.7320508);
\draw (1,1.7320508)	node [above] {} 	-- (.5,.8660254) 	node [left] 	{}; \draw (.5,.8660254) 	-- (0,0);
\draw (0,0) 	-- (.5,.2886751) 	node [above] 	{};
\draw (.5,.2886751) 	-- (1,0.5773503);
\draw[] (1,0.5773503) node [below] {} -- (1,1.1547005) 	node [right]	{}; \draw (1,1.1547005) 	-- (1,1.7320508);
\draw(1,0.5773503) -- (1.5,.2886751) 	node [above] 	{}; \draw (1.5,.2886751) 	-- (2,0);
\node at (1.23,.78)  {$\NGLabel{2}$};
\node at (1.18,1.20)  {$\NGLabel{3}$};
\end{tikzpicture}%
}

\expandafter\def\csname ngshape@P017\endcsname{%
\begin{tikzpicture}[scale=0.75,ng diagram]
\draw (0,2)--(3,2);
\draw (1,1)--(2,1);
\draw (0,0)--(3,0);
\draw (0,0) -- (0,2) -- (1,1) -- (0,0);
\draw (3,0)--(3,2)--(2,1)--(3,0);
\draw[] (0,2) -- (1.5,2) node [above] {};
\draw[->] (1,1) -- (1.5,1) node [above] {$\NGLabel{1}$};
\draw[] (0,0) -- (1.5,0) node [above] {};
\draw[] (0,2) --++ (.5,-.5) node [right] {};
\draw[] (1,1) --++ (-.5,-.5) node [right] {};
\draw[] (0,0) -- (0,1) node [left] {};
\draw[] (3,2) --++ (-.5,-.5) node [left] {};
\draw[] (2,1) --++ (.5,-.5) node [left] {};
\draw[->] (3,0) -- (3,1) node [right] {$\NGLabel{2}$};
\node at (-.1,2.1) {$\NGLabel{3}$};
\node at (-.1,-.1) {$\NGLabel{4}$};
\end{tikzpicture}%
}

\expandafter\def\csname ngshape@T036\endcsname{%
\begin{tikzpicture}[scale=1.125,ng diagram]
\draw[-<] (0,0) 	node [below] {}	-- (1,0) 		node [below] 	{$\NGLabel{1}$}; \draw (1,0) 		-- (2,0);
\draw(2,0) 	node [below] {} 	-- (1.5,.8660254) 	node [right] 	{}; \draw (1.5,.8660254) 	-- (1,1.7320508);
\draw (1,1.7320508)	node [above] {} 	-- (.5,.8660254) 	node [left] 	{}; \draw (.5,.8660254) 	-- (0,0);
\draw (0,0) 	-- (.5,.2886751) 	node [above] 	{};
\draw (.5,.2886751) 	-- (1,0.5773503);
\draw[] (1,0.5773503) node [below] {} -- (1,1.1547005) 	node [right]	{}; \draw (1,1.1547005) 	-- (1,1.7320508);
\draw(1,0.5773503) -- (1.5,.2886751) 	node [above] 	{}; \draw (1.5,.2886751) 	-- (2,0);
\node at (1,1.7320508*.25)  {$\NGLabel{2}$};
\node at (1,1*1.7320508+.1)  {$\NGLabel{3}$};
\end{tikzpicture}%
}

\expandafter\def\csname ngshape@T037\endcsname{%
\begin{tikzpicture}[scale=1.125,ng diagram]
\draw[->] (0,0) 	node [below] {}	-- (1,0) 		node [below] 	{$\NGLabel{1}$}; \draw (1,0) 		-- (2,0);
\draw(2,0) 	node [below] {} 	-- (1.5,.8660254) 	node [right] 	{}; \draw (1.5,.8660254) 	-- (1,1.7320508);
\draw (1,1.7320508)	node [above] {} 	-- (.5,.8660254) 	node [left] 	{}; \draw (.5,.8660254) 	-- (0,0);
\draw (0,0) 	-- (.5,.2886751) 	node [above] 	{};
\draw (.5,.2886751) 	-- (1,0.5773503);
\draw[] (1,0.5773503) node [below] {} -- (1,1.1547005) 	node [right]	{}; \draw (1,1.1547005) 	-- (1,1.7320508);
\draw(1,0.5773503) -- (1.5,.2886751) 	node [above] 	{}; \draw (1.5,.2886751) 	-- (2,0);
\node at (1,1.7320508*.25)  {$\NGLabel{2}$};
\node at (1,1*1.7320508+.1)  {$\NGLabel{3}$};
\end{tikzpicture}%
}

\expandafter\def\csname ngshape@P018\endcsname{%
\begin{tikzpicture}[scale=0.75,ng diagram]
\draw (0,2)--(3,2);
\draw (1,1)--(2,1);
\draw (0,0)--(3,0);
\draw (0,0) -- (0,2) -- (1,1) -- (0,0);
\draw (3,0)--(3,2)--(2,1)--(3,0);
\draw[] (0,2) -- (1.5,2) node [above] {};
\draw[] (1,1) -- (1.5,1) node [above] {};
\draw[] (0,0) -- (1.5,0) node [above] {};
\draw[] (0,2) --++ (.5,-.5) node [right] {};
\draw[] (1,1) --++ (-.5,-.5) node [right] {};
\draw[] (0,0) -- (0,1) node [left] {};
\draw[] (3,2) --++ (-.5,-.5) node [left] {};
\draw[] (2,1) --++ (.5,-.5) node [left] {};
\draw[->] (3,0) -- (3,1) node [right] {$\NGLabel{1}$};
\node at (.24,.42) {$\NGLabel{2}$};
\node at (.24,1.58) {$\NGLabel{3}$};
\node at (1,1)[left] {$\NGLabel{4}$};
\node at (2,1)[right] {$\NGLabel{5}$};
\end{tikzpicture}%
}

\expandafter\def\csname ngshape@P019\endcsname{%
\begin{tikzpicture}[scale=0.75,ng diagram,xscale=2.0]
\draw (0,2)--(3,2);
\draw (1,1)--(2,1);
\draw (0,0)--(3,0);
\draw (0,0) -- (0,2) -- (1,1) -- (0,0);
\draw (3,0)--(3,2)--(2,1)--(3,0);
\draw[] (0,2) -- (1.5,2) node [above] {};
\draw[] (1,1) -- (1.5,1) node [above] {};
\draw[] (0,0) -- (1.5,0) node [above] {};
\draw[] (0,2) --++ (.5,-.5) node [right] {};
\draw[] (1,1) --++ (-.5,-.5) node [right] {};
\draw[] (0,0) -- (0,1) node [left] {};
\draw[] (3,2) --++ (-.5,-.5) node [left] {};
\draw[] (2,1) --++ (.5,-.5) node [left] {};
\draw[->] (3,0) -- (3,1) node [right] {$\NGLabel{1}$};
\node at (.24,.42) {$\NGLabel{2}$};
\node at (.24,1.58) {$\NGLabel{3}$};
\node at (.48,1) {$\NGLabel{4}$};
\node at (2.52,1) {$\NGLabel{5}$};
\end{tikzpicture}%
}

\expandafter\def\csname ngshape@T038\endcsname{%
\begin{tikzpicture}[scale=1.125,ng diagram]
\draw[] (0,0) 	node [below] {}	-- (1,0) 		node [below] 	{}; \draw (1,0) 		-- (2,0);
\node at (.36,.36) {$\NGLabel{1}$};
\draw[] (2,0) 	node [below] {$\NGLabel{2}$} 	-- (1.5,.8660254) 	node [right] 	{}; \draw (1.5,.8660254) 	-- (1,1.7320508);
\draw[] (1,1.7320508)-- (.5,.8660254) 	node [left] 	{};
\node at (.82,1.20) {$\NGLabel{3}$};
\draw (.5,.8660254) 	-- (0,0);
\draw[] (0,0) 	-- (.5,.2886751) 	node [above] 	{}; \draw (.5,.2886751) 	-- (1,0.5773503);
\draw[] (1,0.5773503) node [below] {} -- (1,1.1547005) 	node [right]	{}; \draw (1,1.1547005) 	-- (1,1.7320508);
\node at (.77,.76) {$\NGLabel{4}$};
\draw[] (1,0.5773503) -- (1.5,.2886751) 	node [above] 	{}; \draw (1.5,.2886751) 	-- (2,0);
\end{tikzpicture}%
}

\expandafter\def\csname ngshape@T039\endcsname{%
\begin{tikzpicture}[scale=1.125,ng diagram]
\draw[] (0,0) 	node [below] {}	-- (1,0) 		node [below] 	{}; \draw (1,0) 		-- (2,0);
\draw[->](2,0) 	node [below] {} 	-- (1.5,.8660254) 	node [right] 	{$\NGLabel{1}$}; \draw (1.5,.8660254) 	-- (1,1.7320508);
\draw (1,1.7320508)	node [above] {} 	-- (.5,.8660254) 	node [left] 	{}; \draw (.5,.8660254) 	-- (0,0);
\draw (0,0) 	-- (.5,.2886751) 	node [above] 	{};
\draw (.5,.2886751) 	-- (1,0.5773503);
\draw (1,0.5773503) node [below] {} -- (1,1.1547005) 	node [right]	{}; \draw (1,1.1547005) 	-- (1,1.7320508);
\draw(1,0.5773503) -- (1.5,.2886751) 	node [above] 	{}; \draw (1.5,.2886751) 	-- (2,0);
\node at (1,1.7320508*.4) [right] {$\NGLabel{2}$};
\node at (0,0)[below]  {$\NGLabel{3}$};
\end{tikzpicture}%
}

\expandafter\def\csname ngshape@T040\endcsname{%
\begin{tikzpicture}[scale=1.125,ng diagram]
\draw[] (0,0) 	node [below] {}	-- (1,0) 		node [below] 	{}; \draw (1,0) 		-- (2,0);
\node at (.36,.36) {$\NGLabel{1}$};
\draw[] (2,0) 	node [below] {$\NGLabel{2}$} 	-- (1.5,.8660254) 	node [right] 	{}; \draw (1.5,.8660254) 	-- (1,1.7320508);
\draw[] (1,1.7320508)-- (.5,.8660254) 	node [left] 	{};
\node at (.82,1.20) {$\NGLabel{3}$};
\draw (.5,.8660254) 	-- (0,0);
\draw[] (0,0) 	-- (.5,.2886751) 	node [above] 	{}; \draw (.5,.2886751) 	-- (1,0.5773503);
\draw[] (1,0.5773503) node [below] {} -- (1,1.1547005) 	node [right]	{}; \draw (1,1.1547005) 	-- (1,1.7320508);
\node at (.77,.76) {$\NGLabel{4}$};
\draw[] (1,0.5773503) -- (1.5,.2886751) 	node [above] 	{}; \draw (1.5,.2886751) 	-- (2,0);
\draw [red] (1,1.232)--(.5/1.7320508+1,1.232);
\draw [red,->] (1,1.232)--(.25/1.7320508+1,1.232);
\node at (.25/1.7320508+1,1.232)[below] {$\NGLabel{5}$};
\draw [red](1,1/1.7320508+.3)--(1+.15*1.7320508,1/1.7320508-.15);
\draw [red,-<](1,1/1.7320508+.3)--(1+.075*1.7320508,1/1.7320508+.075);
\node at (1+.075*1.7320508,1/1.7320508+.075)[right] {$\NGLabel{6}$};
\draw [red] (2-.3*1.7320508,.3)--(2-.3/1.7320508,.3);
\draw [red,-<] (2-.3*1.7320508,.3)--(2-.2*1.7320508,.3);
\node at (2-.2*1.7320508,.3) [above] {$\NGLabel{7}$};
\end{tikzpicture}%
}

\expandafter\def\csname ngshape@T041\endcsname{%
\begin{tikzpicture}[scale=1.125,ng diagram]
\draw[] (0,0) 	node [below] {}	-- (1,0) 		node [below] 	{}; \draw (1,0) 		-- (2,0);
\draw[->](2,0) 	node [below] {} 	-- (1.5,.8660254) 	node [right] 	{$\NGLabel{1}$}; \draw (1.5,.8660254) 	-- (1,1.7320508);
\draw (1,1.7320508)	node [above] {} 	-- (.5,.8660254) 	node [left] 	{}; \draw (.5,.8660254) 	-- (0,0);
\draw (0,0) 	-- (.5,.2886751) 	node [above] 	{};
\draw (.5,.2886751) 	-- (1,0.5773503);
\draw (1,0.5773503) node [below] {} -- (1,1.1547005) 	node [right]	{}; \draw (1,1.1547005) 	-- (1,1.7320508);
\draw(1,0.5773503) -- (1.5,.2886751) 	node [above] 	{}; \draw (1.5,.2886751) 	-- (2,0);
\node at (1,1.7320508*.4) [right] {$\NGLabel{2}$};
\node at (0,0)[below]  {$\NGLabel{3}$};
\begin{scope}[shift={(2,0)},xscale=-1]
\draw [red] (1,1.232)--(.5/1.7320508+1,1.232);
\draw [red,-<] (1,1.232)--(.25/1.7320508+1,1.232);
\node at (.25/1.7320508+1,1.232)[below] {$\NGLabel{4}$};
\draw [red](1,1/1.7320508+.3)--(1+.15*1.7320508,1/1.7320508-.15);
\draw [red,->](1,1/1.7320508+.3)--(1+.075*1.7320508,1/1.7320508+.075);
\node at (1+.075*1.7320508,1/1.7320508+.075)[left] {$\NGLabel{5}$};
\draw [red] (2-.3*1.7320508,.3)--(2-.3/1.7320508,.3);
\draw [red,->] (2-.3*1.7320508,.3)--(2-.2*1.7320508,.3);
\node at (2-.2*1.7320508,.3) [above] {$\NGLabel{6}$};
\end{scope}
\end{tikzpicture}%
}

\expandafter\def\csname ngshape@T042\endcsname{%
\begin{tikzpicture}[scale=1.125,ng diagram]
\draw[] (0,0) 	node [below] {}	-- (1,0) 		node [below] 	{}; \draw (1,0) 		-- (2,0);
\node at (.36,.36) {$\NGLabel{1}$};
\draw[] (2,0) 	node [below] {$\NGLabel{2}$} 	-- (1.5,.8660254) 	node [right] 	{}; \draw (1.5,.8660254) 	-- (1,1.7320508);
\draw[] (1,1.7320508)-- (.5,.8660254) 	node [left] 	{};
\node at (.82,1.20) {$\NGLabel{3}$};
\draw (.5,.8660254) 	-- (0,0);
\draw[] (0,0) 	-- (.5,.2886751) 	node [above] 	{}; \draw (.5,.2886751) 	-- (1,0.5773503);
\draw[] (1,0.5773503) node [below] {} -- (1,1.1547005) 	node [right]	{}; \draw (1,1.1547005) 	-- (1,1.7320508);
\node at (.84,.76) {$\NGLabel{4}$};
\draw[] (1,0.5773503) -- (1.5,.2886751) 	node [above] 	{}; \draw (1.5,.2886751) 	-- (2,0);
\end{tikzpicture}%
}

\expandafter\def\csname ngshape@T043\endcsname{%
\begin{tikzpicture}[scale=1.125,ng diagram]
\draw[] (0,0) 	node [below] {}	-- (1,0) 		node [below] 	{}; \draw (1,0) 		-- (2,0);
\draw[->](2,0) 	node [below] {} 	-- (1.5,.8660254) 	node [right] {}; \draw (1.5,.8660254) 	-- (1,1.7320508);
\draw (1,1.7320508)	node [above] {} 	-- (.5,.8660254) 	node [left] 	{}; \draw (.5,.8660254) 	-- (0,0);
\draw (0,0) 	-- (.5,.2886751) 	node [above] 	{};
\draw (.5,.2886751) 	-- (1,0.5773503);
\draw (1,0.5773503) node [below] {} -- (1,1.1547005) 	node [right]	{}; \draw (1,1.1547005) 	-- (1,1.7320508);
\draw(1,0.5773503) -- (1.5,.2886751) 	node [above] 	{}; \draw (1.5,.2886751) 	-- (2,0);
\node at (1.25,1.299)[right=2pt] {$\NGLabel{1}$};
\node at (1.17,.72) {$\NGLabel{2}$};
\node at (0,0)[below]  {$\NGLabel{3}$};
\begin{scope}[shift={(2,0)},xscale=-1]
\end{scope}
\end{tikzpicture}%
}

\expandafter\def\csname ngshape@T044\endcsname{%
\begin{tikzpicture}[scale=1.125,ng diagram]
\draw[] (0,0) 	node [below] {}	-- (1,0) 		node [below] 	{}; \draw (1,0) 		-- (2,0);
\draw[] (2,0) 	node [below] {} 	-- (1.5,.8660254) 	node [right] 	{}; \draw (1.5,.8660254) 	-- (1,1.7320508);
\draw[] (1,1.7320508)	node [above] {} 	-- (.5,.8660254) 	node [left] 	{}; \draw (.5,.8660254) 	-- (0,0);
\draw[] (0,0) 	-- (.5,.2886751) 	node [above] 	{}; \draw (.5,.2886751) 	-- (1,0.5773503);
\draw[] (1,0.5773503) node [below] {} -- (1,1.1547005) 	node [right]	{}; \draw (1,1.1547005) 	-- (1,1.7320508);
\draw[->] (1,0.5773503) -- (1.5,.2886751) 	node [above] 	{$\NGLabel{1}$}; \draw (1.5,.2886751) 	-- (2,0);
\node at (.36,.36){$\NGLabel{2}$};
\node at (.82,1.20) {$\NGLabel{3}$};
\end{tikzpicture}%
}

\expandafter\def\csname ngshape@T045\endcsname{%
\begin{tikzpicture}[scale=1.125,ng diagram]
\draw[] (0,0) 	node [below] {}	-- (1,0) 		node [below] 	{}; \draw (1,0) 		-- (2,0);
\node at (.36,.36) {$\NGLabel{1}$};
\draw[] (2,0) 	node [below] {} 	-- (1.5,.8660254) 	node [right] 	{}; \draw (1.5,.8660254) 	-- (1,1.7320508);
\node at (1.56,.40) {$\NGLabel{2}$};
\draw[] (1,1.7320508)-- (.5,.8660254) 	node [left] 	{};
\node at (.82,1.20) {$\NGLabel{3}$};
\draw (.5,.8660254) 	-- (0,0);
\draw[] (0,0) 	-- (.5,.2886751) 	node [above] 	{}; \draw (.5,.2886751) 	-- (1,0.5773503);
\draw[] (1,0.5773503) node [below] {} -- (1,1.1547005) 	node [right]	{}; \draw (1,1.1547005) 	-- (1,1.7320508);
\node at (.77,.76) {$\NGLabel{4}$};
\draw[] (1,0.5773503) -- (1.5,.2886751) 	node [above] 	{}; \draw (1.5,.2886751) 	-- (2,0);
\end{tikzpicture}%
}

\expandafter\def\csname ngshape@T046\endcsname{%
\begin{tikzpicture}[scale=1.125,ng diagram]
\draw[] (0,0) 	node [below] {}	-- (1,0) 		node [below] 	{}; \draw (1,0) 		-- (2,0);
\draw[->] (2,0) 	node [below] {} 	-- (1.5,.8660254) 	node [right] 	{$\NGLabel{1}$}; \draw (1.5,.8660254) 	-- (1,1.7320508);
\draw[] (1,1.7320508)	node [above] {} 	-- (.5,.8660254) 	node [left] 	{}; \draw (.5,.8660254) 	-- (0,0);
\draw[] (0,0) 	-- (.5,.2886751) 	node [above] 	{}; \draw (.5,.2886751) 	-- (1,0.5773503);
\draw[] (1,0.5773503) node [below] {} -- (1,1.1547005) 	node [right]	{}; \draw (1,1.1547005) 	-- (1,1.7320508);
\draw[] (1,0.5773503) -- (1.5,.2886751) 	node [above] 	{}; \draw (1.5,.2886751) 	-- (2,0);
\node at (.36,.36){$\NGLabel{2}$};
\node at (.77,.76) {$\NGLabel{3}$};
\end{tikzpicture}%
}

\expandafter\def\csname ngshape@T047\endcsname{%
\begin{tikzpicture}[scale=1.125,ng diagram]
\draw[] (0,0) 	node [below] {}	-- (1,0) 		node [below] 	{}; \draw (1,0) 		-- (2,0);
\draw[] (2,0) 	node [below] {} 	-- (1.5,.8660254) 	node [right] 	{}; \draw (1.5,.8660254) 	-- (1,1.7320508);
\draw[] (1,1.7320508)	node [above] {} 	-- (.5,.8660254) 	node [left] 	{$\NGLabel{1}$}; \draw (.5,.8660254) 	-- (0,0);
\draw[] (0,0) 	-- (.5,.2886751) 	node[above] {}; \draw (.5,.2886751) 	-- (1,0.5773503);
\draw[] (1,0.5773503) node [below] {} -- (1,1.1547005) 	node [right]	{}; \draw (1,1.1547005) 	-- (1,1.7320508);
\draw[->] (1,0.5773503) -- (1.5,.2886751) 	node [above] 	{$\NGLabel{2}$}; \draw (1.5,.2886751) 	-- (2,0);
\node at (.36,.36){$\NGLabel{3}$};
\node at (.82,1.20) {$\NGLabel{4}$};
\draw [red,domain=-60:240] plot ({1+.3*cos(\x)}, {1.7320508+.3*sin(\x)});
\draw [red,domain=-60:90,->] plot ({1+.3*cos(\x)}, {1.7320508+.3*sin(\x)});
\node at (1,1.7320508+.3)[above] {$\NGLabel{5}$};
\draw [red,domain=60:360] plot ({.3*cos(\x)}, {.3*sin(\x)});
\draw [red,domain=60:270,->] plot ({0+.3*cos(\x)}, {0+.3*sin(\x)});
\node at (0,-.3) [below] {$\NGLabel{6}$};
\draw [red,domain=180:480] plot ({2+.3*cos(\x)}, {.3*sin(\x)});
\draw [red,domain=180:270,->] plot ({2+.3*cos(\x)}, {0+.3*sin(\x)});
\node at (2,-.3) [below] {$\NGLabel{7}$};
\node at (.72,.16) {$\NGLabel{8}$};
\end{tikzpicture}%
}

\expandafter\def\csname ngshape@T048\endcsname{%
\begin{tikzpicture}[scale=1.125,ng diagram]
\draw[] (0,0) 	node [below] {}	-- (1,0) 		node [below] 	{}; \draw (1,0) 		-- (2,0);
\node at (.36,.36) {$\NGLabel{1}$};
\draw[] (2,0) 	node [below] {} 	-- (1.5,.8660254) 	node [right] 	{}; \draw (1.5,.8660254) 	-- (1,1.7320508);
\node at (1.56,.40) {$\NGLabel{2}$};
\draw[] (1,1.7320508)-- (.5,.8660254) 	node [left] 	{};
\node at (.82,1.20) {$\NGLabel{3}$};
\draw (.5,.8660254) 	-- (0,0);
\draw[] (0,0) 	-- (.5,.2886751) 	node [above] 	{}; \draw (.5,.2886751) 	-- (1,0.5773503);
\draw[] (1,0.5773503) node [below] {} -- (1,1.1547005) 	node [right]	{}; \draw (1,1.1547005) 	-- (1,1.7320508);
\node at (.77,.76) {$\NGLabel{4}$};
\draw[] (1,0.5773503) -- (1.5,.2886751) 	node [above] 	{}; \draw (1.5,.2886751) 	-- (2,0);
\draw[red] (.5,0)--(.5,.5/1.7320508);
\draw[red,->] (.5,0)--(.5,.25/1.7320508);
\node at (.5,.25/1.7320508)[right] {$\NGLabel{5}$};
\draw[red] (1.5,0)--(1.5,.5/1.7320508);
\draw[red,-<] (1.5,0)--(1.5,.25/1.7320508);
\node at (1.5,.25/1.7320508)[left] {$\NGLabel{6}$};
\draw [red](.75,.75/1.7320508)--(1.25,.75/1.7320508);
\draw [->,red](.75,.75/1.7320508)--(1,.75/1.7320508);
\node at (1,.75/1.7320508)[below] {$\NGLabel{7}$};
\end{tikzpicture}%
}

\expandafter\def\csname ngshape@T049\endcsname{%
\begin{tikzpicture}[scale=1.125,ng diagram]
\draw[] (0,0) 	node [below] {}	-- (1,0) 		node [below] 	{}; \draw (1,0) 		-- (2,0);
\draw[] (2,0) 	node [below] {} 	-- (1.5,.8660254) 	node [right] 	{}; \draw (1.5,.8660254) 	-- (1,1.7320508);
\draw[] (1,1.7320508)	node [above] {} 	-- (.5,.8660254) 	node [left] 	{}; \draw (.5,.8660254) 	-- (0,0);
\draw[] (0,0) 	-- (.5,.2886751) 	node [above] 	{}; \draw (.5,.2886751) 	-- (1,0.5773503);
\draw[] (1,0.5773503) node [below] {} -- (1,1.1547005) 	node [right]	{}; \draw (1,1.1547005) 	-- (1,1.7320508);
\draw[->] (1,0.5773503) -- (1.5,.2886751) 	node [above] 	{$\NGLabel{1}$}; \draw (1.5,.2886751) 	-- (2,0);
\node at (.36,.36){$\NGLabel{2}$};
\node at (.85,1.24) {$\NGLabel{3}$};
\end{tikzpicture}%
}

\expandafter\def\csname ngshape@T050\endcsname{%
\begin{tikzpicture}[scale=1.125,ng diagram]
\draw[] (0,0) 	node [below] {}	-- (1,0) 		node [below] 	{}; \draw (1,0) 		-- (2,0);
\node at (.36,.36) {$\NGLabel{1}$};
\draw[] (2,0) 	node [below] {} 	-- (1.5,.8660254) 	node [right] 	{}; \draw (1.5,.8660254) 	-- (1,1.7320508);
\node at (1.61,.38) {$\NGLabel{2}$};
\draw[] (1,1.7320508)-- (.5,.8660254) 	node [left] 	{};
\node at (.82,1.20) {$\NGLabel{3}$};
\draw (.5,.8660254) 	-- (0,0);
\draw[] (0,0) 	-- (.5,.2886751) 	node [above] 	{}; \draw (.5,.2886751) 	-- (1,0.5773503);
\draw[] (1,0.5773503) node [below] {} -- (1,1.1547005) 	node [right]	{}; \draw (1,1.1547005) 	-- (1,1.7320508);
\node at (.83,.77) {$\NGLabel{4}$};
\draw[] (1,0.5773503) -- (1.5,.2886751) 	node [above] 	{}; \draw (1.5,.2886751) 	-- (2,0);
\end{tikzpicture}%
}

\expandafter\def\csname ngshape@P020\endcsname{%
\begin{tikzpicture}[scale=0.75,ng diagram]
\draw (0,2)--(3,2);
\draw (1,1)--(2,1);
\draw (0,0)--(3,0);
\draw (0,0) -- (0,2) -- (1,1) -- (0,0);
\draw (3,0)--(3,2)--(2,1)--(3,0);
\draw[->] (0,2) -- (1.5,2) node [above] {$\NGLabel{1}$};
\draw[->] (1,1) -- (1.5,1) node [above] {$\NGLabel{2}$};
\draw[-<] (0,0) -- (1.5,0) node [above] {$\NGLabel{3}$};
\draw[] (0,2) --++ (.5,-.5) node [right] {};
\draw[] (1,1) --++ (-.5,-.5) node [right] {};
\draw[] (0,0) -- (0,1) node [left] {};
\draw[] (3,2) --++ (-.5,-.5) node [left] {};
\draw[] (2,1) --++ (.5,-.5) node [left] {};
\draw[] (3,0) -- (3,1) node [right] {};
\end{tikzpicture}%
}

\expandafter\def\csname ngshape@P021\endcsname{%
\begin{tikzpicture}[scale=0.75,ng diagram]
\draw (0,2)--(3,2);
\draw (1,1)--(2,1);
\draw (0,0)--(3,0);
\draw (0,0) -- (0,2) -- (1,1) -- (0,0);
\draw (3,0)--(3,2)--(2,1)--(3,0);
\draw[] (0,2) -- (1.5,2) node [above] {};
\draw[] (0,2) --++ (.5,-.5) node [right] {};
\draw[] (1,1) --++ (-.5,-.5) node [right] {};
\draw[->] (0,0) -- (1.5,0) node [above] {$\NGLabel{1}$};
\draw[] (0,0) -- (0,1) node [left] {};
\draw[] (3,2) --++ (-.5,-.5) node [left] {};
\draw[] (2,1) --++ (.5,-.5) node [left] {};
\draw[] (3,0) -- (3,1) node [right] {};
\node at (1,1)[left] {$\NGLabel{2}$};
\node at (.15,1.8)[below]{$\NGLabel{3}$};
\node at (3-.15,1.8)[below]{$\NGLabel{4}$};
\node at (2,1)[right]{$\NGLabel{5}$};
\end{tikzpicture}%
}

\expandafter\def\csname ngshape@T051\endcsname{%
\begin{tikzpicture}[scale=1.125,ng diagram]
\draw[-<] (0,0) 	node [below] {}	-- (1,0) 		node [below] 	{$\NGLabel{1}$}; \draw (1,0) 		-- (2,0);
\draw[-<] (2,0) 	node [below] {} 	-- (1.5,.8660254) 	node [right] 	{$\NGLabel{2}$}; \draw (1.5,.8660254) 	-- (1,1.7320508);
\draw[] (1,1.7320508)	node [above] {} 	-- (.5,.8660254) 	node [left] 	{}; \draw (.5,.8660254) 	-- (0,0);
\draw[] (0,0) 	-- (.5,.2886751) 	node [above] 	{}; \draw (.5,.2886751) 	-- (1,0.5773503);
\draw[] (1,0.5773503) node [below] {} -- (1,1.1547005) 	node [right]	{}; \draw (1,1.1547005) 	-- (1,1.7320508);
\draw[-<] (1,0.5773503) -- (1.5,.2886751) 	node [above] 	{$\NGLabel{3}$}; \draw (1.5,.2886751) 	-- (2,0);
\end{tikzpicture}%
}

\expandafter\def\csname ngshape@T052\endcsname{%
\begin{tikzpicture}[scale=1.125,ng diagram]
\draw[-<] (0,0) 	node [below] {}	-- (1,0) 		node [below] 	{$\NGLabel{1}$}; \draw (1,0) 		-- (2,0);
\draw[] (2,0) 	node [below] {} 	-- (1.5,.8660254) 	node [right] 	{}; \draw (1.5,.8660254) 	-- (1,1.7320508);
\draw[-<] (1,1.7320508)	node [above] {} 	-- (.5,.8660254) 	node [left] 	{$\NGLabel{2}$}; \draw (.5,.8660254) 	-- (0,0);
\draw[->] (0,0) 	-- (.5,.2886751) 	node [above] 	{$\NGLabel{3}$}; \draw (.5,.2886751) 	-- (1,0.5773503);
\draw[] (1,0.5773503) node [below] {} -- (1,1.1547005) 	node [right]	{}; \draw (1,1.1547005) 	-- (1,1.7320508);
\draw[] (1,0.5773503) -- (1.5,.2886751) 	node [above] 	{}; \draw (1.5,.2886751) 	-- (2,0);
\end{tikzpicture}%
}

\expandafter\def\csname ngshape@T053\endcsname{%
\begin{tikzpicture}[scale=1.125,ng diagram]
\draw[-<] (0,0) 	node [below] {}	-- (1,0) 		node [below] 	{$\NGLabel{1}$}; \draw (1,0) 		-- (2,0);
\draw(2,0) 	node [below] {} 	-- (1.5,.8660254) 	node [right] 	{}; \draw (1.5,.8660254) 	-- (1,1.7320508);
\draw (1,1.7320508)	node [above] {} 	-- (.5,.8660254) 	node [left] 	{}; \draw (.5,.8660254) 	-- (0,0);
\draw (0,0) 	-- (.5,.2886751) 	node [above] 	{};
\draw (.5,.2886751) 	-- (1,0.5773503);
\draw[] (1,0.5773503) node [below] {} -- (1,1.1547005) 	node [right]	{}; \draw (1,1.1547005) 	-- (1,1.7320508);
\draw(1,0.5773503) -- (1.5,.2886751) 	node [above] 	{}; \draw (1.5,.2886751) 	-- (2,0);
\node at (1.23,.78)  {$\NGLabel{2}$};
\node at (1.18,1.20)  {$\NGLabel{3}$};
\end{tikzpicture}%
}

\expandafter\def\csname ngshape@T054\endcsname{%
\begin{tikzpicture}[scale=1.125,ng diagram]
\draw[-<] (0,0) 	node [below] {}	-- (1,0) 		node [below] 	{$\NGLabel{1}$}; \draw (1,0) 		-- (2,0);
\draw(2,0) 	node [below] {} 	-- (1.5,.8660254) 	node [right] 	{}; \draw (1.5,.8660254) 	-- (1,1.7320508);
\draw (1,1.7320508)	node [above] {} 	-- (.5,.8660254) 	node [left] 	{}; \draw (.5,.8660254) 	-- (0,0);
\draw (0,0) 	-- (.5,.2886751) 	node [above] 	{};
\draw (.5,.2886751) 	-- (1,0.5773503);
\draw[-<] (1,0.5773503) node [below] {} -- (1,1.1547005) 	node [right]	{$\NGLabel{2}$}; \draw (1,1.1547005) 	-- (1,1.7320508);
\draw(1,0.5773503) -- (1.5,.2886751) 	node [above] 	{}; \draw (1.5,.2886751) 	-- (2,0);
\end{tikzpicture}%
}

\expandafter\def\csname ngshape@P022\endcsname{%
\begin{tikzpicture}[scale=0.75,ng diagram]
\draw (0,2)--(3,2);
\draw (1,1)--(2,1);
\draw (0,0)--(3,0);
\draw (0,0) -- (0,2) -- (1,1) -- (0,0);
\draw (3,0)--(3,2)--(2,1)--(3,0);
\draw[] (0,2) -- (1.5,2) node [above] {};
\draw[] (1,1) -- (1.5,1) node [above] {};
\draw[] (0,0) -- (1.5,0) node [above] {};
\draw[] (0,2) --++ (.5,-.5) node [right] {};
\draw[] (1,1) --++ (-.5,-.5) node [right] {};
\draw[] (0,0) -- (0,1) node [left] {};
\draw[] (3,2) --++ (-.5,-.5) node [left] {};
\draw[] (2,1) --++ (.5,-.5) node [left] {};
\draw[] (3,0) -- (3,1) node [right] {};
\node at (0+.15,2-.3) {$\NGLabel{1}$};
\node at (1-.2,1) {$\NGLabel{2}$};
\node at (0+.15,0+.3) {$\NGLabel{3}$};
\node at (3-.15,2-.3) {$\NGLabel{4}$};
\node at (2+.2,1) {$\NGLabel{5}$};
\node at (3-.15,0+.3) {$\NGLabel{6}$};
\end{tikzpicture}%
}

\expandafter\def\csname ngshape@T055\endcsname{%
\begin{tikzpicture}[scale=1.125,ng diagram]
\draw[] (0,0) 	node [below] {}	-- (1,0) 		node [below] 	{}; \draw (1,0) 		-- (2,0);
\node at (.33,.33) {$\NGLabel{1}$};
\draw[] (2,0) 	node [below] {} 	-- (1.5,.8660254) 	node [right] 	{}; \draw (1.5,.8660254) 	-- (1,1.7320508);
\node at (1.50,.48) {$\NGLabel{2}$};
\draw[] (1,1.7320508)-- (.5,.8660254) 	node [left] 	{};
\node at (.82,1.20) {$\NGLabel{3}$};
\draw (.5,.8660254) 	-- (0,0);
\draw[] (0,0) 	-- (.5,.2886751) 	node [above] 	{}; \draw (.5,.2886751) 	-- (1,0.5773503);
\draw[] (1,0.5773503) node [below] {} -- (1,1.1547005) 	node [right]	{}; \draw (1,1.1547005) 	-- (1,1.7320508);
\node at (.77,.76) {$\NGLabel{4}$};
\draw[] (1,0.5773503) -- (1.5,.2886751) 	node [above] 	{}; \draw (1.5,.2886751) 	-- (2,0);
\end{tikzpicture}%
}

\expandafter\def\csname ngshape@T056\endcsname{%
\begin{tikzpicture}[scale=1.125,ng diagram]
\draw[] (0,0) 	node [below] {}	-- (1,0) 		node [below] 	{}; \draw (1,0) 		-- (2,0);
\node at (.36,.36) {$\NGLabel{1}$};
\draw[] (2,0) 	node [below] {} 	-- (1.5,.8660254) 	node [right] 	{}; \draw (1.5,.8660254) 	-- (1,1.7320508);
\node at (1.61,.38) {$\NGLabel{2}$};
\draw[] (1,1.7320508)-- (.5,.8660254) 	node [left] 	{};
\node at (.82,1.20) {$\NGLabel{3}$};
\draw (.5,.8660254) 	-- (0,0);
\draw[] (0,0) 	-- (.5,.2886751) 	node [above] 	{}; \draw (.5,.2886751) 	-- (1,0.5773503);
\draw[] (1,0.5773503) node [below] {} -- (1,1.1547005) 	node [right]	{}; \draw (1,1.1547005) 	-- (1,1.7320508);
\node at (.83,.77) {$\NGLabel{4}$};
\draw[] (1,0.5773503) -- (1.5,.2886751) 	node [above] 	{}; \draw (1.5,.2886751) 	-- (2,0);
\end{tikzpicture}%
}

\expandafter\def\csname ngshape@O054\endcsname{%
\begin{tikzpicture}[scale=1,ng diagram]
\begin{scope}
\draw (-.5,.3) arc[radius=.5, start angle=180, end angle=360];
\draw (0,.5)--(0,0);
\draw [-<](0,-.2) arc[radius=.5, start angle=270, end angle=180];
\draw [-<](0,-.2) arc[radius=.5, start angle=270, end angle=360];
\draw [-<](0,-.2)--(0,.3);
\draw (0,.3)--(0,.5);
\draw (.5,.3)--(.5,.5);
\draw (-.5,.3)--(-.5,.5);
\node at (-.5,.3)[left] {$\NGLabel{1}$};
\node at (0,.3)[left] {$\NGLabel{2}$};
\node at (.5,.3)[left] {$\NGLabel{3}$};
\end{scope}
\end{tikzpicture}%
}

\expandafter\def\csname ngshape@O055\endcsname{%
\begin{tikzpicture}[scale=1,ng diagram]
\draw (-.5,0) arc[radius=.5, start angle=180, end angle=0];
\draw (0,.5)--(0,0);
\draw [](0,.5) arc[radius=.5, start angle=90, end angle=180];
\draw [](0,.5) arc[radius=.5, start angle=90, end angle=0];
\draw [](0,.5)--(0,0);
\draw (0,0)--(0,-.2);
\draw (.5,0)--(.5,-.2);
\draw (-.5,0)--(-.5,-.2);
\node at (-.19,.29) {$\NGLabel{1}$};
\end{tikzpicture}%
}

\expandafter\def\csname ngshape@O056\endcsname{%
\begin{tikzpicture}[scale=1,ng diagram]
\begin{scope}
\draw (-.5,.3) arc[radius=.5, start angle=180, end angle=360];
\draw (0,.5)--(0,0);
\draw [-<](0,-.2) arc[radius=.5, start angle=270, end angle=180];
\draw [-<](0,-.2) arc[radius=.5, start angle=270, end angle=360];
\draw [-<](0,-.2)--(0,.3);
\draw (0,.3)--(0,.5);
\draw (.5,.3)--(.5,.5);
\draw (-.5,.3)--(-.5,.5);
\node at (-.19,.00) {$\NGLabel{1}$};
\end{scope}
\end{tikzpicture}%
}

\usetikzlibrary{arrows.meta,decorations.markings}
\tikzset{ng/core/.style={line width=.45pt,line cap=round,line join=round,
 every node/.style={font=\normalfont\fontsize{8}{9.5}\selectfont,inner sep=2pt}},
 ng/arrow/.style={postaction={decorate},decoration={markings,
 mark=at position .5 with {\arrow{Stealth[length=3pt,width=2.6pt]}}}},
 ng/tetrahedron/.style={ng/core,scale=1.125},
 ng/vertex/.style={inner sep=1.8pt},
 ng/box/.style={draw,fill=white,minimum width=9mm,minimum height=5mm,inner sep=2pt}}

\pgfkeys{/ng/prism/.is family,/ng/prism,
 e1/.initial=X_1,e2/.initial=X_2,e3/.initial=X_3,
 e4/.initial=X_4,e5/.initial=X_5,e6/.initial=X_6,
 e7/.initial=X_7,e8/.initial=X_8,e9/.initial=X_9,
 v1/.initial=\alpha_1,v2/.initial=\alpha_2,v3/.initial=\alpha_3,
 v4/.initial=\alpha_4,v5/.initial=\alpha_5,v6/.initial=\alpha_6,scale/.initial=1}

\pgfkeys{/ng/tetra/.is family,/ng/tetra,
 e1/.initial=X_1,e2/.initial=X_2,e3/.initial=X_3,
 e4/.initial=X_4,e5/.initial=X_5,e6/.initial=X_6,
 center/.initial=\alpha,left/.initial=\beta,top/.initial=\gamma,
 right/.initial=\delta}

\expandafter\def\csname NGD@D010\endcsname{%
\begin{tikzpicture}[sd/base,x=1mm,y=1mm]
\draw[sd/arrow] (19.0000,10.9697) -- (38.0000,0.0000);
\draw[sd/reverse] (19.0000,10.9697) -- (19.0000,32.9090);
\draw[sd/arrow] (19.0000,10.9697) -- (0.0000,0.0000);
\draw[sd/reverse] (38.0000,0.0000) -- (0.0000,0.0000);
\draw[sd/arrow] (19.0000,32.9090) -- (0.0000,0.0000);
\draw[sd/arrow] (38.0000,0.0000) -- (19.0000,32.9090);
\fill (19.0000,9.0697) circle[radius=.30mm];
\fill (-1.6454,-0.9500) circle[radius=.30mm];
\fill (19.0000,34.8090) circle[radius=.30mm];
\fill (39.6454,-0.9500) circle[radius=.30mm];
\node[sd/label] at (19.7600,-2.3496) {$g+h$};
\node[sd/label] at (7.9760,18.2120) {$g$};
\node[sd/label] at (32.9253,18.1319) {$g+k$};
\node[sd/label] at (17.0848,21.5005) {$k$};
\node[sd/label] at (8.6921,7.7618) {$h$};
\node[sd/label] at (26.2084,2.3933) {$k-h$};
\end{tikzpicture}
}

\expandafter\def\csname NGD@D011\endcsname{%
\begin{tikzpicture}[sd/base,x=1mm,y=1mm]
\draw[sd/string] (19.0000,10.9697) -- (38.0000,0.0000);
\draw[sd/arrow] (19.0000,10.9697) -- (19.0000,32.9090);
\draw[sd/string] (19.0000,10.9697) -- (0.0000,0.0000);
\draw[sd/string] (38.0000,0.0000) -- (0.0000,0.0000);
\draw[sd/reverse] (19.0000,32.9090) -- (0.0000,0.0000);
\draw[sd/reverse] (38.0000,0.0000) -- (19.0000,32.9090);
\fill (19.0000,9.0697) circle[radius=.30mm];
\fill (-1.6454,-0.9500) circle[radius=.30mm];
\fill (19.0000,34.8090) circle[radius=.30mm];
\fill (39.6454,-0.9500) circle[radius=.30mm];
\node[sd/label] at (7.9760,18.2120) {$g$};
\node[sd/label] at (32.9595,18.1516) {$g+h$};
\node[sd/label] at (17.0392,21.5005) {$h$};
\end{tikzpicture}
}

\expandafter\def\csname NGD@D017\endcsname{%
\begin{tikzpicture}[sd/base,x=1mm,y=1mm]
\draw[sd/string] (19.0000,10.9697) -- (38.0000,0.0000);
\draw[sd/reverse] (19.0000,10.9697) -- (19.0000,32.9090);
\draw[sd/string] (19.0000,10.9697) -- (0.0000,0.0000);
\draw[sd/reverse] (38.0000,0.0000) -- (0.0000,0.0000);
\draw[sd/string] (19.0000,32.9090) -- (0.0000,0.0000);
\draw[sd/string] (38.0000,0.0000) -- (19.0000,32.9090);
\fill (19.0000,9.0697) circle[radius=.30mm];
\fill (-1.6454,-0.9500) circle[radius=.30mm];
\fill (19.0000,34.8090) circle[radius=.30mm];
\fill (39.6454,-0.9500) circle[radius=.30mm];
\node[sd/label] at (19.7600,-1.9786) {$g$};
\node[sd/label] at (17.0392,21.5005) {$h$};
\end{tikzpicture}
}

\expandafter\def\csname NGD@D024\endcsname{%
\begin{tikzpicture}[sd/base,x=1mm,y=1mm]
\draw[sd/string] (19.0000,10.9697) -- (38.0000,0.0000);
\draw[sd/reverse] (19.0000,10.9697) -- (19.0000,32.9090);
\draw[sd/string] (19.0000,10.9697) -- (0.0000,0.0000);
\draw[sd/reverse] (38.0000,0.0000) -- (0.0000,0.0000);
\draw[sd/string] (19.0000,32.9090) -- (0.0000,0.0000);
\draw[sd/string] (38.0000,0.0000) -- (19.0000,32.9090);
\fill (19.0000,9.0697) circle[radius=.30mm];
\fill (-1.6454,-0.9500) circle[radius=.30mm];
\fill (19.0000,34.8090) circle[radius=.30mm];
\fill (39.6454,-0.9500) circle[radius=.30mm];
\node[sd/label] at (19.7600,-2.0763) {$h$};
\node[sd/label] at (17.1387,21.5005) {$g$};
\end{tikzpicture}
}

\expandafter\def\csname NGD@D025\endcsname{%
\begin{tikzpicture}[sd/base,x=1mm,y=1mm]
\draw[sd/string] (19.0000,10.9697) -- (38.0000,0.0000);
\draw[sd/arrow] (19.0000,10.9697) -- (19.0000,32.9090);
\draw[sd/string] (19.0000,10.9697) -- (0.0000,0.0000);
\draw[sd/arrow] (38.0000,0.0000) -- (0.0000,0.0000);
\draw[sd/string] (19.0000,32.9090) -- (0.0000,0.0000);
\draw[sd/string] (38.0000,0.0000) -- (19.0000,32.9090);
\fill (19.0000,9.0697) circle[radius=.30mm];
\fill (-1.6454,-0.9500) circle[radius=.30mm];
\fill (19.0000,34.8090) circle[radius=.30mm];
\fill (39.6454,-0.9500) circle[radius=.30mm];
\node[sd/label] at (19.7600,-1.9786) {$g$};
\node[sd/label] at (17.0392,21.5005) {$h$};
\end{tikzpicture}
}

\expandafter\def\csname NGD@D072\endcsname{%
\begin{tikzpicture}[sd/base,x=1mm,y=1mm]
\draw[sd/string] (16.5000,9.5263) -- (33.0000,0.0000);
\draw[sd/string] (16.5000,9.5263) -- (16.5000,28.5788);
\draw[sd/string] (16.5000,9.5263) -- (0.0000,0.0000);
\draw[sd/string] (33.0000,0.0000) -- (0.0000,0.0000);
\draw[sd/string] (16.5000,28.5788) -- (0.0000,0.0000);
\draw[sd/string] (33.0000,0.0000) -- (16.5000,28.5788);
\node[sd/label] at (14.2916,10.8013) {${h}$};
\node[sd/label] at (5.5154,5.5154) {${0}$};
\node[sd/label] at (13.9636,19.1128) {${k}'$};
\node[sd/label] at (27.4846,5.5154) {${g}$};
\end{tikzpicture}
}

\expandafter\def\csname NGD@D079\endcsname{%
\begin{tikzpicture}[sd/base,x=1mm,y=1mm]
\draw[sd/string] (19.0000,10.9697) -- (38.0000,0.0000);
\draw[sd/string] (19.0000,10.9697) -- (19.0000,32.9090);
\draw[sd/string] (19.0000,10.9697) -- (0.0000,0.0000);
\draw[sd/reverse] (38.0000,0.0000) -- (0.0000,0.0000);
\draw[sd/string] (19.0000,32.9090) -- (0.0000,0.0000);
\draw[sd/string] (38.0000,0.0000) -- (19.0000,32.9090);
\node[sd/label] at (19.0000,8.4197) {${h}$};
\fill (-1.6454,-0.9500) circle[radius=.30mm];
\node[sd/label] at (19.0000,35.2090) {${k}'$};
\fill (39.6454,-0.9500) circle[radius=.30mm];
\node[sd/label] at (19.7600,-1.9786) {$g$};
\end{tikzpicture}
}

\expandafter\def\csname NGD@D080\endcsname{%
\begin{tikzpicture}[sd/base,x=1mm,y=1mm]
\draw[sd/string] (16.5000,9.5263) -- (33.0000,0.0000);
\draw[sd/string] (16.5000,9.5263) -- (16.5000,28.5788);
\draw[sd/string] (16.5000,9.5263) -- (0.0000,0.0000);
\draw[sd/string] (33.0000,0.0000) -- (0.0000,0.0000);
\draw[sd/string] (16.5000,28.5788) -- (0.0000,0.0000);
\draw[sd/string] (33.0000,0.0000) -- (16.5000,28.5788);
\node[sd/label] at (14.5081,10.6763) {${p}$};
\node[sd/label] at (5.5154,5.5154) {${g}$};
\node[sd/label] at (13.8989,18.8713) {${h}'$};
\node[sd/label] at (28.0149,4.9851) {${r}$};
\end{tikzpicture}
}

\expandafter\def\csname NGD@D082\endcsname{%
\begin{tikzpicture}[sd/base,x=1mm,y=1mm]
\draw[sd/string] (16.5000,9.5263) -- (33.0000,0.0000);
\draw[sd/string] (16.5000,9.5263) -- (16.5000,28.5788);
\draw[sd/string] (16.5000,9.5263) -- (0.0000,0.0000);
\draw[sd/string] (33.0000,0.0000) -- (0.0000,0.0000);
\draw[sd/string] (16.5000,28.5788) -- (0.0000,0.0000);
\draw[sd/string] (33.0000,0.0000) -- (16.5000,28.5788);
\node[sd/label] at (14.2916,10.8013) {${h}$};
\node[sd/label] at (5.5154,5.5154) {${0}$};
\node[sd/label] at (13.8989,18.8713) {${h}'$};
\node[sd/label] at (27.1310,5.8690) {${k}$};
\end{tikzpicture}
}

\expandafter\def\csname NGD@D086\endcsname{%
\begin{tikzpicture}[sd/base,x=1mm,y=1mm]
\draw[sd/arrow] (0.0000,25.2000) -- (42.0000,25.2000);
\draw[sd/arrow] (12.1800,12.6000) -- (29.8200,12.6000);
\draw[sd/reverse] (0.0000,0.0000) -- (42.0000,0.0000);
\draw[sd/arrow] (0.0000,25.2000) -- (12.1800,12.6000);
\draw[sd/arrow] (12.1800,12.6000) -- (0.0000,0.0000);
\draw[sd/arrow] (0.0000,0.0000) -- (0.0000,25.2000);
\draw[sd/reverse] (42.0000,25.2000) -- (29.8200,12.6000);
\draw[sd/arrow] (29.8200,12.6000) -- (42.0000,0.0000);
\draw[sd/reverse] (42.0000,0.0000) -- (42.0000,25.2000);
\fill (1.0870,22.5114) circle[radius=.30mm];
\fill (10.2800,12.6000) circle[radius=.30mm];
\fill (1.0870,2.6886) circle[radius=.30mm];
\fill (40.9130,22.5114) circle[radius=.30mm];
\fill (31.7200,12.6000) circle[radius=.30mm];
\fill (40.9130,2.6886) circle[radius=.30mm];
\node[sd/label] at (-4.5021,12.0960) {$g+k$};
\node[sd/label] at (46.1494,12.0960) {$g+l$};
\node[sd/label] at (20.1600,27.1786) {$g$};
\node[sd/label] at (20.1600,-2.3496) {$g+h$};
\node[sd/label] at (4.1462,17.5085) {$k$};
\node[sd/label] at (9.4957,3.4953) {$k-h$};
\node[sd/label] at (34.6358,20.6192) {$l$};
\node[sd/label] at (32.6866,3.6716) {$h-l$};
\node[sd/label] at (20.6472,10.5237) {$h$};
\end{tikzpicture}
}

\expandafter\def\csname NGD@D092\endcsname{%
\begin{tikzpicture}[sd/base,x=1mm,y=1mm]
\draw[sd/arrow] (0.0000,25.2000) -- (42.0000,25.2000);
\draw[sd/arrow] (12.1800,12.6000) -- (29.8200,12.6000);
\draw[sd/reverse] (0.0000,0.0000) -- (42.0000,0.0000);
\draw[sd/string] (0.0000,25.2000) -- (12.1800,12.6000);
\draw[sd/string] (12.1800,12.6000) -- (0.0000,0.0000);
\draw[sd/string] (0.0000,0.0000) -- (0.0000,25.2000);
\draw[sd/reverse] (42.0000,25.2000) -- (29.8200,12.6000);
\draw[sd/arrow] (29.8200,12.6000) -- (42.0000,0.0000);
\draw[sd/reverse] (42.0000,0.0000) -- (42.0000,25.2000);
\fill (1.0870,22.5114) circle[radius=.30mm];
\fill (10.2800,12.6000) circle[radius=.30mm];
\fill (1.0870,2.6886) circle[radius=.30mm];
\fill (40.9130,22.5114) circle[radius=.30mm];
\fill (31.7200,12.6000) circle[radius=.30mm];
\fill (40.9130,2.6886) circle[radius=.30mm];
\node[sd/label] at (46.5021,12.0960) {$g+k$};
\node[sd/label] at (20.1600,27.1786) {$g$};
\node[sd/label] at (20.1600,-2.3496) {$g+h$};
\node[sd/label] at (34.4534,20.7955) {$k$};
\node[sd/label] at (32.5043,3.4953) {$h-k$};
\node[sd/label] at (20.6472,10.5237) {$h$};
\end{tikzpicture}
}

\expandafter\def\csname NGD@D098\endcsname{%
\begin{tikzpicture}[sd/base,x=1mm,y=1mm]
\draw[sd/string] (0.0000,25.2000) -- (42.0000,25.2000);
\draw[sd/string] (12.1800,12.6000) -- (29.8200,12.6000);
\draw[sd/arrow] (0.0000,0.0000) -- (42.0000,0.0000);
\draw[sd/string] (0.0000,25.2000) -- (12.1800,12.6000);
\draw[sd/arrow] (12.1800,12.6000) -- (0.0000,0.0000);
\draw[sd/reverse] (0.0000,0.0000) -- (0.0000,25.2000);
\draw[sd/string] (42.0000,25.2000) -- (29.8200,12.6000);
\draw[sd/arrow] (29.8200,12.6000) -- (42.0000,0.0000);
\draw[sd/arrow] (42.0000,0.0000) -- (42.0000,25.2000);
\fill (1.0870,22.5114) circle[radius=.30mm];
\fill (10.2800,12.6000) circle[radius=.30mm];
\fill (1.0870,2.6886) circle[radius=.30mm];
\fill (40.9130,22.5114) circle[radius=.30mm];
\fill (31.7200,12.6000) circle[radius=.30mm];
\fill (40.9130,2.6886) circle[radius=.30mm];
\node[sd/label] at (-1.8613,12.0960) {$g$};
\node[sd/label] at (49.1885,12.0960) {$g+h+k$};
\node[sd/label] at (20.1600,-2.3496) {$g+h$};
\node[sd/label] at (4.6099,8.2183) {$h$};
\node[sd/label] at (33.9662,4.9085) {$k$};
\end{tikzpicture}
}

\expandafter\def\csname NGD@D104\endcsname{%
\begin{tikzpicture}[sd/base,x=1mm,y=1mm]
\draw[sd/string] (0.0000,25.2000) -- (42.0000,25.2000);
\draw[sd/string] (12.1800,12.6000) -- (29.8200,12.6000);
\draw[sd/arrow] (0.0000,0.0000) -- (42.0000,0.0000);
\draw[sd/string] (0.0000,25.2000) -- (12.1800,12.6000);
\draw[sd/arrow] (12.1800,12.6000) -- (0.0000,0.0000);
\draw[sd/reverse] (0.0000,0.0000) -- (0.0000,25.2000);
\draw[sd/reverse] (42.0000,25.2000) -- (29.8200,12.6000);
\draw[sd/string] (29.8200,12.6000) -- (42.0000,0.0000);
\draw[sd/string] (42.0000,0.0000) -- (42.0000,25.2000);
\fill (1.0870,22.5114) circle[radius=.30mm];
\fill (10.2800,12.6000) circle[radius=.30mm];
\fill (1.0870,2.6886) circle[radius=.30mm];
\fill (40.9130,22.5114) circle[radius=.30mm];
\fill (31.7200,12.6000) circle[radius=.30mm];
\fill (40.9130,2.6886) circle[radius=.30mm];
\node[sd/label] at (-1.8613,12.0960) {$g$};
\node[sd/label] at (20.1600,-2.3496) {$g+h$};
\node[sd/label] at (4.6099,8.2183) {$h$};
\node[sd/label] at (34.4534,20.7955) {$k$};
\end{tikzpicture}
}

\expandafter\def\csname NGD@D108\endcsname{%
\begin{tikzpicture}[sd/base,x=1mm,y=1mm]
\draw[sd/string] (0.0000,25.2000) -- (42.0000,25.2000);
\draw[sd/string] (12.1800,12.6000) -- (29.8200,12.6000);
\draw[sd/arrow] (0.0000,0.0000) -- (42.0000,0.0000);
\draw[sd/string] (0.0000,25.2000) -- (12.1800,12.6000);
\draw[sd/arrow] (12.1800,12.6000) -- (0.0000,0.0000);
\draw[sd/reverse] (0.0000,0.0000) -- (0.0000,25.2000);
\draw[sd/string] (42.0000,25.2000) -- (29.8200,12.6000);
\draw[sd/string] (29.8200,12.6000) -- (42.0000,0.0000);
\draw[sd/string] (42.0000,0.0000) -- (42.0000,25.2000);
\fill (1.0870,22.5114) circle[radius=.30mm];
\fill (10.2800,12.6000) circle[radius=.30mm];
\fill (1.0870,2.6886) circle[radius=.30mm];
\node[sd/label] at (39.6385,19.3593) {${l}'$};
\node[sd/label] at (32.8700,12.6000) {${k}$};
\fill (40.9130,2.6886) circle[radius=.30mm];
\node[sd/label] at (-1.8613,12.0960) {$g$};
\node[sd/label] at (20.1600,-2.3496) {$g+h$};
\node[sd/label] at (4.6099,8.2183) {$h$};
\end{tikzpicture}
}

\expandafter\def\csname NGD@D112\endcsname{%
\begin{tikzpicture}[sd/base,x=1mm,y=1mm]
\draw[sd/string] (0.0000,25.2000) -- (42.0000,25.2000);
\draw[sd/arrow] (12.1800,12.6000) -- (29.8200,12.6000);
\draw[sd/string] (0.0000,0.0000) -- (42.0000,0.0000);
\draw[sd/string] (0.0000,25.2000) -- (12.1800,12.6000);
\draw[sd/string] (12.1800,12.6000) -- (0.0000,0.0000);
\draw[sd/reverse] (0.0000,0.0000) -- (0.0000,25.2000);
\draw[sd/string] (42.0000,25.2000) -- (29.8200,12.6000);
\draw[sd/string] (29.8200,12.6000) -- (42.0000,0.0000);
\draw[sd/arrow] (42.0000,0.0000) -- (42.0000,25.2000);
\fill (1.0870,22.5114) circle[radius=.30mm];
\fill (10.2800,12.6000) circle[radius=.30mm];
\fill (1.0870,2.6886) circle[radius=.30mm];
\fill (40.9130,22.5114) circle[radius=.30mm];
\fill (31.7200,12.6000) circle[radius=.30mm];
\fill (40.9130,2.6886) circle[radius=.30mm];
\node[sd/label] at (-1.8613,12.0960) {$g$};
\node[sd/label] at (43.9152,12.0960) {$k$};
\node[sd/label] at (20.6472,10.5237) {$h$};
\end{tikzpicture}
}

\expandafter\def\csname NGD@D121\endcsname{%
\begin{tikzpicture}[sd/base,x=1mm,y=1mm]
\draw[sd/string] (0.0000,25.2000) -- (42.0000,25.2000);
\draw[sd/string] (12.1800,12.6000) -- (29.8200,12.6000);
\draw[sd/string] (0.0000,0.0000) -- (42.0000,0.0000);
\draw[sd/string] (0.0000,25.2000) -- (12.1800,12.6000);
\draw[sd/reverse] (12.1800,12.6000) -- (0.0000,0.0000);
\draw[sd/string] (0.0000,0.0000) -- (0.0000,25.2000);
\draw[sd/string] (42.0000,25.2000) -- (29.8200,12.6000);
\draw[sd/reverse] (29.8200,12.6000) -- (42.0000,0.0000);
\draw[sd/string] (42.0000,0.0000) -- (42.0000,25.2000);
\node[sd/label] at (2.1741,19.8229) {${k}$};
\fill (10.2800,12.6000) circle[radius=.30mm];
\fill (1.0870,2.6886) circle[radius=.30mm];
\node[sd/label] at (39.6385,19.3593) {${l}'$};
\fill (31.7200,12.6000) circle[radius=.30mm];
\fill (40.9130,2.6886) circle[radius=.30mm];
\node[sd/label] at (4.7101,8.1214) {$g$};
\node[sd/label] at (33.9427,4.8857) {$h$};
\end{tikzpicture}
}

\expandafter\def\csname NGD@D126\endcsname{%
\begin{tikzpicture}[sd/base,x=1mm,y=1mm]
\draw[sd/string] (0.0000,25.2000) -- (42.0000,25.2000);
\draw[sd/string] (12.1800,12.6000) -- (29.8200,12.6000);
\draw[sd/string] (0.0000,0.0000) -- (42.0000,0.0000);
\draw[sd/string] (0.0000,25.2000) -- (12.1800,12.6000);
\draw[sd/arrow] (12.1800,12.6000) -- (0.0000,0.0000);
\draw[sd/string] (0.0000,0.0000) -- (0.0000,25.2000);
\draw[sd/string] (42.0000,25.2000) -- (29.8200,12.6000);
\draw[sd/string] (29.8200,12.6000) -- (42.0000,0.0000);
\draw[sd/arrow] (42.0000,0.0000) -- (42.0000,25.2000);
\node[sd/label] at (2.6426,18.6640) {${k}'$};
\fill (10.2800,12.6000) circle[radius=.30mm];
\fill (1.0870,2.6886) circle[radius=.30mm];
\fill (40.9130,22.5114) circle[radius=.30mm];
\node[sd/label] at (32.6200,12.6000) {${l}$};
\fill (40.9130,2.6886) circle[radius=.30mm];
\node[sd/label] at (43.9608,12.0960) {$h$};
\node[sd/label] at (4.7101,8.1214) {$g$};
\end{tikzpicture}
}

\expandafter\def\csname NGD@D131\endcsname{%
\begin{tikzpicture}[sd/base,x=1mm,y=1mm]
\draw[sd/string] (0.0000,25.2000) -- (42.0000,25.2000);
\draw[sd/arrow] (12.1800,12.6000) -- (29.8200,12.6000);
\draw[sd/string] (0.0000,0.0000) -- (42.0000,0.0000);
\draw[sd/string] (0.0000,25.2000) -- (12.1800,12.6000);
\draw[sd/string] (12.1800,12.6000) -- (0.0000,0.0000);
\draw[sd/string] (0.0000,0.0000) -- (0.0000,25.2000);
\draw[sd/string] (42.0000,25.2000) -- (29.8200,12.6000);
\draw[sd/string] (29.8200,12.6000) -- (42.0000,0.0000);
\draw[sd/arrow] (42.0000,0.0000) -- (42.0000,25.2000);
\node[sd/label] at (2.1741,19.8229) {${k}$};
\fill (10.2800,12.6000) circle[radius=.30mm];
\node[sd/label] at (2.3615,5.8407) {${l}'$};
\fill (40.9130,22.5114) circle[radius=.30mm];
\fill (31.7200,12.6000) circle[radius=.30mm];
\fill (40.9130,2.6886) circle[radius=.30mm];
\node[sd/label] at (43.9608,12.0960) {$h$};
\node[sd/label] at (20.6472,10.6214) {$g$};
\end{tikzpicture}
}

\expandafter\def\csname NGD@D136\endcsname{%
\begin{tikzpicture}[sd/base,x=1mm,y=1mm]
\draw[sd/string] (0.0000,25.2000) -- (42.0000,25.2000);
\draw[sd/string] (12.1800,12.6000) -- (29.8200,12.6000);
\draw[sd/string] (0.0000,0.0000) -- (42.0000,0.0000);
\draw[sd/string] (0.0000,25.2000) -- (12.1800,12.6000);
\draw[sd/string] (12.1800,12.6000) -- (0.0000,0.0000);
\draw[sd/string] (0.0000,0.0000) -- (0.0000,25.2000);
\draw[sd/string] (42.0000,25.2000) -- (29.8200,12.6000);
\draw[sd/string] (29.8200,12.6000) -- (42.0000,0.0000);
\draw[sd/arrow] (42.0000,0.0000) -- (42.0000,25.2000);
\node[sd/label] at (2.9238,17.9687) {${h_2}$};
\node[sd/label] at (8.1300,12.6000) {${h_3}$};
\node[sd/label] at (2.9238,7.2313) {${h_1}$};
\fill (40.9130,22.5114) circle[radius=.30mm];
\node[sd/label] at (33.6200,12.6000) {${k}'$};
\fill (40.9130,2.6886) circle[radius=.30mm];
\node[sd/label] at (43.8613,12.0960) {$g$};
\end{tikzpicture}
}

\expandafter\def\csname NGD@D138\endcsname{%
\begin{tikzpicture}[sd/base,x=1mm,y=1mm]
\draw[sd/string] (16.5000,9.5263) -- (33.0000,0.0000);
\draw[sd/string] (16.5000,9.5263) -- (16.5000,28.5788);
\draw[sd/string] (16.5000,9.5263) -- (0.0000,0.0000);
\draw[sd/string] (33.0000,0.0000) -- (0.0000,0.0000);
\draw[sd/string] (16.5000,28.5788) -- (0.0000,0.0000);
\draw[sd/string] (33.0000,0.0000) -- (16.5000,28.5788);
\node[sd/label] at (13.6421,11.1763) {${h_3}$};
\node[sd/label] at (7.8135,7.8135) {${h_1}$};
\node[sd/label] at (13.6400,17.9054) {${h_2}$};
\node[sd/label] at (34.9919,-1.1500) {${s}'$};
\end{tikzpicture}
}

\expandafter\def\csname NGD@D139\endcsname{%
\begin{tikzpicture}[sd/base,x=1mm,y=1mm]
\draw[sd/string] (19.0000,10.9697) -- (38.0000,0.0000);
\draw[sd/string] (19.0000,10.9697) -- (19.0000,32.9090);
\draw[sd/string] (19.0000,10.9697) -- (0.0000,0.0000);
\draw[sd/string] (38.0000,0.0000) -- (0.0000,0.0000);
\draw[sd/string] (19.0000,32.9090) -- (0.0000,0.0000);
\draw[sd/arrow] (38.0000,0.0000) -- (19.0000,32.9090);
\node[sd/label] at (21.6414,12.4947) {${k}'$};
\node[sd/label] at (-1.9919,-1.1500) {${s}$};
\fill (19.0000,34.8090) circle[radius=.30mm];
\fill (39.6454,-0.9500) circle[radius=.30mm];
\node[sd/label] at (30.7840,16.8956) {$g$};
\end{tikzpicture}
}

\expandafter\def\csname NGD@D144\endcsname{%
\begin{tikzpicture}[sd/base,x=1mm,y=1mm]
\draw[sd/arrow] (19.0000,10.9697) -- (38.0000,0.0000);
\draw[sd/string] (19.0000,10.9697) -- (19.0000,32.9090);
\draw[sd/string] (19.0000,10.9697) -- (0.0000,0.0000);
\draw[sd/string] (38.0000,0.0000) -- (0.0000,0.0000);
\draw[sd/string] (19.0000,32.9090) -- (0.0000,0.0000);
\draw[sd/string] (38.0000,0.0000) -- (19.0000,32.9090);
\fill (19.0000,9.0697) circle[radius=.30mm];
\node[sd/label] at (7.8135,7.8135) {${h_2}$};
\node[sd/label] at (16.5930,23.9259) {${s}'$};
\fill (39.6454,-0.9500) circle[radius=.30mm];
\node[sd/label] at (26.9992,3.7630) {$g$};
\end{tikzpicture}
}

\expandafter\def\csname NGD@D145\endcsname{%
\begin{tikzpicture}[sd/base,x=1mm,y=1mm]
\draw[sd/string] (16.5000,9.5263) -- (33.0000,0.0000);
\draw[sd/string] (16.5000,9.5263) -- (16.5000,28.5788);
\draw[sd/string] (16.5000,9.5263) -- (0.0000,0.0000);
\draw[sd/string] (33.0000,0.0000) -- (0.0000,0.0000);
\draw[sd/string] (16.5000,28.5788) -- (0.0000,0.0000);
\draw[sd/string] (33.0000,0.0000) -- (16.5000,28.5788);
\node[sd/label] at (14.5081,10.6763) {${s}$};
\node[sd/label] at (7.8135,7.8135) {${h_3}$};
\node[sd/label] at (14.7400,22.0105) {${t}$};
\node[sd/label] at (26.0704,6.9296) {${k}'$};
\end{tikzpicture}
}

\expandafter\def\csname NGD@D146\endcsname{%
\begin{tikzpicture}[sd/base,x=1mm,y=1mm]
\draw[sd/string] (19.0000,10.9697) -- (38.0000,0.0000);
\draw[sd/string] (19.0000,10.9697) -- (19.0000,32.9090);
\draw[sd/string] (19.0000,10.9697) -- (0.0000,0.0000);
\draw[sd/string] (38.0000,0.0000) -- (0.0000,0.0000);
\draw[sd/string] (19.0000,32.9090) -- (0.0000,0.0000);
\draw[sd/arrow] (38.0000,0.0000) -- (19.0000,32.9090);
\node[sd/label] at (16.5751,12.3697) {${t}'$};
\node[sd/label] at (7.8135,7.8135) {${h_1}$};
\fill (19.0000,34.8090) circle[radius=.30mm];
\fill (39.6454,-0.9500) circle[radius=.30mm];
\node[sd/label] at (30.7840,16.8956) {$g$};
\end{tikzpicture}
}

\expandafter\def\csname NGD@D153\endcsname{%
\begin{tikzpicture}[sd/base,x=1mm,y=1mm]
\draw[sd/arrow] (0.0000,25.2000) -- (42.0000,25.2000);
\draw[sd/arrow] (12.1800,12.6000) -- (29.8200,12.6000);
\draw[sd/reverse] (0.0000,0.0000) -- (42.0000,0.0000);
\draw[sd/string] (0.0000,25.2000) -- (12.1800,12.6000);
\draw[sd/string] (12.1800,12.6000) -- (0.0000,0.0000);
\draw[sd/string] (0.0000,0.0000) -- (0.0000,25.2000);
\draw[sd/string] (42.0000,25.2000) -- (29.8200,12.6000);
\draw[sd/string] (29.8200,12.6000) -- (42.0000,0.0000);
\draw[sd/string] (42.0000,0.0000) -- (42.0000,25.2000);
\fill (1.0870,22.5114) circle[radius=.30mm];
\fill (10.2800,12.6000) circle[radius=.30mm];
\fill (1.0870,2.6886) circle[radius=.30mm];
\fill (40.9130,22.5114) circle[radius=.30mm];
\fill (31.7200,12.6000) circle[radius=.30mm];
\fill (40.9130,2.6886) circle[radius=.30mm];
\node[sd/label] at (20.1600,27.1786) {$g$};
\node[sd/label] at (20.1600,-2.3496) {$g+h$};
\node[sd/label] at (20.6472,10.5237) {$h$};
\end{tikzpicture}
}

\expandafter\def\csname NGD@D154\endcsname{%
\begin{tikzpicture}[sd/base,x=1mm,y=1mm]
\draw[sd/string] (0.0000,25.2000) -- (42.0000,25.2000);
\draw[sd/string] (12.1800,12.6000) -- (29.8200,12.6000);
\draw[sd/arrow] (0.0000,0.0000) -- (42.0000,0.0000);
\draw[sd/string] (0.0000,25.2000) -- (12.1800,12.6000);
\draw[sd/string] (12.1800,12.6000) -- (0.0000,0.0000);
\draw[sd/string] (0.0000,0.0000) -- (0.0000,25.2000);
\draw[sd/string] (42.0000,25.2000) -- (29.8200,12.6000);
\draw[sd/string] (29.8200,12.6000) -- (42.0000,0.0000);
\draw[sd/string] (42.0000,0.0000) -- (42.0000,25.2000);
\node[sd/label] at (2.8300,18.2005) {${k_2}$};
\node[sd/label] at (8.1300,12.6000) {${k_1}$};
\fill (1.0870,2.6886) circle[radius=.30mm];
\node[sd/label] at (39.1700,18.2005) {${k_3}$};
\node[sd/label] at (33.8700,12.6000) {${k_4}$};
\fill (40.9130,2.6886) circle[radius=.30mm];
\node[sd/label] at (20.1600,-1.9786) {$g$};
\end{tikzpicture}
}

\expandafter\def\csname NGD@D164\endcsname{%
\begin{tikzpicture}[sd/base,x=1mm,y=1mm]
\draw[sd/string] (0.0000,25.2000) -- (42.0000,25.2000);
\draw[sd/string] (12.1800,12.6000) -- (29.8200,12.6000);
\draw[sd/string] (0.0000,0.0000) -- (42.0000,0.0000);
\draw[sd/string] (0.0000,25.2000) -- (12.1800,12.6000);
\draw[sd/string] (12.1800,12.6000) -- (0.0000,0.0000);
\draw[sd/string] (0.0000,0.0000) -- (0.0000,25.2000);
\draw[sd/string] (42.0000,25.2000) -- (29.8200,12.6000);
\draw[sd/string] (29.8200,12.6000) -- (42.0000,0.0000);
\draw[sd/string] (42.0000,0.0000) -- (42.0000,25.2000);
\node[sd/label] at (2.5489,18.8958) {${g_1}$};
\node[sd/label] at (8.6300,12.6000) {${g_2}$};
\node[sd/label] at (2.5489,6.3042) {${g_3}$};
\node[sd/label] at (39.0762,17.9687) {${h_1}$};
\node[sd/label] at (33.8700,12.6000) {${h_2}$};
\node[sd/label] at (39.0762,7.2313) {${h_3}$};
\end{tikzpicture}
}

\expandafter\def\csname NGD@E-00-conjugated\endcsname{%
\begin{tikzpicture}[sd/base,x=1mm,y=1mm]
\node[sd/box,minimum width=8.4289mm,minimum height=4.9786mm] (b0) at (0.0000,-6.6893) {$D_h^{-1}$};
\node[sd/box,minimum width=6.7800mm,minimum height=4.0000mm] (b1) at (-0.8245,-15.5786) {$p_0$};
\node[sd/box,minimum width=6.7800mm,minimum height=4.0000mm] (b2) at (-0.8245,-23.9786) {$q_0$};
\node[sd/box,minimum width=6.7800mm,minimum height=4.0000mm] (b3) at (-0.8245,-32.3786) {$D_h$};
\draw[sd/string] (-2.5952,0.0000) -- ([xshift=-2.5952mm]b0.north);
\node[sd/label,anchor=south] at (-2.5952,1.0000) {$\varrho$};
\draw[sd/string] (1.4048,0.0000) -- ([xshift=1.4048mm]b0.north);
\node[sd/label,anchor=south] at (1.4048,1.0000) {$\varrho$};
\draw[sd/string] ([xshift=-3.1145mm]b0.south) -- ([xshift=-2.2900mm]b1.north);
\draw[sd/string] ([xshift=0.8855mm]b0.south) -- ([xshift=1.7100mm]b1.north);
\draw[sd/string] ([xshift=0.0000mm]b1.south) -- ([xshift=0.0000mm]b2.north);
\draw[sd/string] ([xshift=-2.2900mm]b2.south) -- ([xshift=-2.2900mm]b3.north);
\draw[sd/string] ([xshift=1.7100mm]b2.south) -- ([xshift=1.7100mm]b3.north);
\draw[sd/string] ([xshift=-2.2900mm]b3.south) -- (-3.1145,-39.3786);
\node[sd/label,anchor=north] at (-3.1145,-40.3786) {$\varrho$};
\draw[sd/string] ([xshift=1.7100mm]b3.south) -- (0.8855,-39.3786);
\node[sd/label,anchor=north] at (0.8855,-40.3786) {$\varrho$};
\end{tikzpicture}
}

\expandafter\def\csname NGD@E-hh\endcsname{%
\begin{tikzpicture}[sd/base,x=1mm,y=1mm]
\node[sd/box,minimum width=6.5000mm,minimum height=4.0000mm] (b0) at (0.0000,-6.2000) {$p_h$};
\node[sd/box,minimum width=6.5000mm,minimum height=4.0000mm] (b1) at (0.0000,-14.6000) {$q_h$};
\draw[sd/string] (-2.1500,0.0000) -- ([xshift=-2.1500mm]b0.north);
\node[sd/label,anchor=south] at (-2.1500,1.0000) {$\varrho$};
\draw[sd/string] (1.8500,0.0000) -- ([xshift=1.8500mm]b0.north);
\node[sd/label,anchor=south] at (1.8500,1.0000) {$\varrho$};
\draw[sd/string] ([xshift=0.0000mm]b0.south) -- ([xshift=0.0000mm]b1.north);
\draw[sd/string] ([xshift=-2.1500mm]b1.south) -- (-2.1500,-21.6000);
\node[sd/label,anchor=north] at (-2.1500,-22.6000) {$\varrho$};
\draw[sd/string] ([xshift=1.8500mm]b1.south) -- (1.8500,-21.6000);
\node[sd/label,anchor=north] at (1.8500,-22.6000) {$\varrho$};
\end{tikzpicture}
}

\expandafter\def\csname NGD@E-hk\endcsname{%
\begin{tikzpicture}[sd/base,x=1mm,y=1mm]
\node[sd/box,minimum width=6.5000mm,minimum height=4.0000mm] (b0) at (0.0000,-6.2000) {$p_h$};
\node[sd/box,minimum width=6.5000mm,minimum height=4.0000mm] (b1) at (0.0000,-14.6000) {$q_k$};
\draw[sd/string] (-2.1500,0.0000) -- ([xshift=-2.1500mm]b0.north);
\node[sd/label,anchor=south] at (-2.1500,1.0000) {$\varrho$};
\draw[sd/string] (1.8500,0.0000) -- ([xshift=1.8500mm]b0.north);
\node[sd/label,anchor=south] at (1.8500,1.0000) {$\varrho$};
\draw[sd/string] ([xshift=0.0000mm]b0.south) -- ([xshift=0.0000mm]b1.north);
\draw[sd/string] ([xshift=-2.1500mm]b1.south) -- (-2.1500,-21.6000);
\node[sd/label,anchor=north] at (-2.1500,-22.6000) {$\varrho$};
\draw[sd/string] ([xshift=1.8500mm]b1.south) -- (1.8500,-21.6000);
\node[sd/label,anchor=north] at (1.8500,-22.6000) {$\varrho$};
\end{tikzpicture}
}

\expandafter\def\csname NGD@E-product\endcsname{%
\begin{tikzpicture}[sd/base,x=1mm,y=1mm]
\node[sd/box,minimum width=6.5000mm,minimum height=4.0000mm] (b0) at (0.0000,-6.2000) {$p_h$};
\node[sd/box,minimum width=6.5000mm,minimum height=4.0000mm] (b1) at (0.0000,-14.6000) {$q_k$};
\node[sd/box,minimum width=6.5000mm,minimum height=4.0000mm] (b2) at (0.0000,-23.0000) {$p_p$};
\node[sd/box,minimum width=6.5000mm,minimum height=4.0000mm] (b3) at (0.0000,-31.4000) {$q_q$};
\draw[sd/string] (-2.1500,0.0000) -- ([xshift=-2.1500mm]b0.north);
\node[sd/label,anchor=south] at (-2.1500,1.0000) {$\varrho$};
\draw[sd/string] (1.8500,0.0000) -- ([xshift=1.8500mm]b0.north);
\node[sd/label,anchor=south] at (1.8500,1.0000) {$\varrho$};
\draw[sd/string] ([xshift=0.0000mm]b0.south) -- ([xshift=0.0000mm]b1.north);
\draw[sd/string] ([xshift=-2.1500mm]b1.south) -- ([xshift=-2.1500mm]b2.north);
\draw[sd/string] ([xshift=1.8500mm]b1.south) -- ([xshift=1.8500mm]b2.north);
\draw[sd/string] ([xshift=0.0000mm]b2.south) -- ([xshift=0.0000mm]b3.north);
\draw[sd/string] ([xshift=-2.1500mm]b3.south) -- (-2.1500,-38.4000);
\node[sd/label,anchor=north] at (-2.1500,-39.4000) {$\varrho$};
\draw[sd/string] ([xshift=1.8500mm]b3.south) -- (1.8500,-38.4000);
\node[sd/label,anchor=north] at (1.8500,-39.4000) {$\varrho$};
\end{tikzpicture}
}

\expandafter\def\csname NGD@E-reduced\endcsname{%
\begin{tikzpicture}[sd/base,x=1mm,y=1mm]
\node[sd/box,minimum width=6.5000mm,minimum height=4.0000mm] (b0) at (0.0000,-6.2000) {$p_h$};
\node[sd/box,minimum width=6.5000mm,minimum height=4.0000mm] (b1) at (0.0000,-14.6000) {$q_q$};
\draw[sd/string] (-2.1500,0.0000) -- ([xshift=-2.1500mm]b0.north);
\node[sd/label,anchor=south] at (-2.1500,1.0000) {$\varrho$};
\draw[sd/string] (1.8500,0.0000) -- ([xshift=1.8500mm]b0.north);
\node[sd/label,anchor=south] at (1.8500,1.0000) {$\varrho$};
\draw[sd/string] ([xshift=0.0000mm]b0.south) -- ([xshift=0.0000mm]b1.north);
\draw[sd/string] ([xshift=-2.1500mm]b1.south) -- (-2.1500,-21.6000);
\node[sd/label,anchor=north] at (-2.1500,-22.6000) {$\varrho$};
\draw[sd/string] ([xshift=1.8500mm]b1.south) -- (1.8500,-21.6000);
\node[sd/label,anchor=north] at (1.8500,-22.6000) {$\varrho$};
\end{tikzpicture}
}

\expandafter\def\csname NGD@F-u\endcsname{%
\begin{tikzpicture}[sd/base,x=1mm,y=1mm]
\node[sd/box,minimum width=8.1617mm,minimum height=4.1117mm] (b0) at (0.0000,-6.2558) {$\mathsf F(u)$};
\draw[sd/string] (-2.6397,0.0000) -- ([xshift=-2.6397mm]b0.north);
\node[sd/label,anchor=south] at (-2.6397,1.0000) {$\varrho$};
\draw[sd/string] (1.3603,0.0000) -- ([xshift=1.3603mm]b0.north);
\node[sd/label,anchor=south] at (1.3603,1.0000) {$\varrho$};
\draw[sd/string] ([xshift=0.0000mm]b0.south) -- (0.0000,-13.3117);
\node[sd/label,anchor=north] at (0.0000,-14.3117) {$\varrho$};
\end{tikzpicture}
}

\expandafter\def\csname NGD@H-0g\endcsname{%
\begin{tikzpicture}[sd/base,x=1mm,y=1mm]
\node[sd/box,minimum width=9.5774mm,minimum height=4.1117mm] (b0) at (-3.1882,-6.2558) {$\mathsf S(u_0)$};
\node[sd/box,minimum width=12.6890mm,minimum height=4.4128mm] (b1) at (1.6325,-14.9181) {$\mathsf F(\varrho^2u_g)$};
\draw[sd/string] (-3.1882,0.0000) -- ([xshift=0.0000mm]b0.north);
\node[sd/label,anchor=south] at (-3.1882,1.0000) {$\varrho$};
\draw[sd/string] (3.7654,0.0000) -- ([xshift=2.1329mm]b1.north);
\node[sd/label,anchor=south] at (3.7654,1.0000) {$\varrho$};
\draw[sd/string] ([xshift=-3.6887mm]b0.south) -- (-6.8769,-22.1244);
\node[sd/label,anchor=north] at (-6.8769,-23.1244) {$\varrho$};
\draw[sd/string] ([xshift=2.1511mm]b0.south) -- ([xshift=-2.6696mm]b1.north);
\draw[sd/string] ([xshift=0.0000mm]b1.south) -- (1.6325,-22.1244);
\node[sd/label,anchor=north] at (1.6325,-23.1244) {$\varrho$};
\end{tikzpicture}
}

\expandafter\def\csname NGD@H-conjugated\endcsname{%
\begin{tikzpicture}[sd/base,x=1mm,y=1mm]
\node[sd/box,minimum width=6.7800mm,minimum height=4.0000mm] (b0) at (-0.8245,-6.2000) {$D_h$};
\node[sd/box,minimum width=6.7800mm,minimum height=4.0000mm] (b1) at (-0.8245,-14.6000) {$H_g$};
\node[sd/box,minimum width=8.4289mm,minimum height=4.9786mm] (b2) at (0.0000,-23.4893) {$D_h^{-1}$};
\draw[sd/string] (-3.1145,0.0000) -- ([xshift=-2.2900mm]b0.north);
\node[sd/label,anchor=south] at (-3.1145,1.0000) {$\varrho$};
\draw[sd/string] (0.8855,0.0000) -- ([xshift=1.7100mm]b0.north);
\node[sd/label,anchor=south] at (0.8855,1.0000) {$\varrho$};
\draw[sd/string] ([xshift=-2.2900mm]b0.south) -- ([xshift=-2.2900mm]b1.north);
\draw[sd/string] ([xshift=1.7100mm]b0.south) -- ([xshift=1.7100mm]b1.north);
\draw[sd/string] ([xshift=-2.2900mm]b1.south) -- ([xshift=-3.1145mm]b2.north);
\draw[sd/string] ([xshift=1.7100mm]b1.south) -- ([xshift=0.8855mm]b2.north);
\draw[sd/string] ([xshift=-2.5952mm]b2.south) -- (-2.5952,-30.9786);
\node[sd/label,anchor=north] at (-2.5952,-31.9786) {$\varrho$};
\draw[sd/string] ([xshift=1.4048mm]b2.south) -- (1.4048,-30.9786);
\node[sd/label,anchor=north] at (1.4048,-31.9786) {$\varrho$};
\end{tikzpicture}
}

\expandafter\def\csname NGD@H-pq\endcsname{%
\begin{tikzpicture}[sd/base,x=1mm,y=1mm]
\node[sd/box,minimum width=9.6435mm,minimum height=4.1703mm] (b0) at (-3.1161,-6.2851) {$\mathsf S(u_p)$};
\node[sd/box,minimum width=12.6109mm,minimum height=4.4128mm] (b1) at (1.6325,-14.9766) {$\mathsf F(\varrho^2u_q)$};
\draw[sd/string] (-3.1161,0.0000) -- ([xshift=0.0000mm]b0.north);
\node[sd/label,anchor=south] at (-3.1161,1.0000) {$\varrho$};
\draw[sd/string] (3.8706,0.0000) -- ([xshift=2.2381mm]b1.north);
\node[sd/label,anchor=south] at (3.8706,1.0000) {$\varrho$};
\draw[sd/string] ([xshift=-3.7218mm]b0.south) -- (-6.8379,-22.1830);
\node[sd/label,anchor=north] at (-6.8379,-23.1830) {$\varrho$};
\draw[sd/string] ([xshift=2.1270mm]b0.south) -- ([xshift=-2.6216mm]b1.north);
\draw[sd/string] ([xshift=0.0000mm]b1.south) -- (1.6325,-22.1830);
\node[sd/label,anchor=north] at (1.6325,-23.1830) {$\varrho$};
\end{tikzpicture}
}

\expandafter\def\csname NGD@P-l\endcsname{%
\begin{tikzpicture}[sd/base,x=1mm,y=1mm]
\node[sd/box,minimum width=6.5000mm,minimum height=4.0000mm] (b0) at (0.0000,-6.2000) {$w_l$};
\node[sd/box,minimum width=6.5000mm,minimum height=4.0000mm] (b1) at (0.0000,-14.6000) {$\check w_l$};
\draw[sd/string] (-2.1500,0.0000) -- ([xshift=-2.1500mm]b0.north);
\node[sd/label,anchor=south] at (-2.1500,1.0000) {$\varrho$};
\draw[sd/string] (1.8500,0.0000) -- ([xshift=1.8500mm]b0.north);
\node[sd/label,anchor=south] at (1.8500,1.0000) {$\varrho$};
\draw[sd/string] ([xshift=0.0000mm]b0.south) -- ([xshift=0.0000mm]b1.north);
\node[sd/label,anchor=west] at (1.0000,-10.4000) {$l$};
\draw[sd/string] ([xshift=-2.1500mm]b1.south) -- (-2.1500,-21.6000);
\node[sd/label,anchor=north] at (-2.1500,-22.6000) {$\varrho$};
\draw[sd/string] ([xshift=1.8500mm]b1.south) -- (1.8500,-21.6000);
\node[sd/label,anchor=north] at (1.8500,-22.6000) {$\varrho$};
\end{tikzpicture}
}

\expandafter\def\csname NGD@S-u\endcsname{%
\begin{tikzpicture}[sd/base,x=1mm,y=1mm]
\node[sd/box,minimum width=8.1130mm,minimum height=4.1117mm] (b0) at (0.0000,-6.2558) {$\mathsf S(u)$};
\draw[sd/string] (0.0000,0.0000) -- ([xshift=0.0000mm]b0.north);
\node[sd/label,anchor=south] at (0.0000,1.0000) {$\varrho$};
\draw[sd/string] ([xshift=-2.6478mm]b0.south) -- (-2.6478,-13.3117);
\node[sd/label,anchor=north] at (-2.6478,-14.3117) {$\varrho$};
\draw[sd/string] ([xshift=1.3522mm]b0.south) -- (1.3522,-13.3117);
\node[sd/label,anchor=north] at (1.3522,-14.3117) {$\varrho$};
\end{tikzpicture}
}

\expandafter\def\csname NGD@TPE-left-first\endcsname{%
\begin{tikzpicture}[sd/base,x=1mm,y=1mm]
\draw[sd/reverse] (19.0000,10.9697) -- (38.0000,0.0000);
\draw[sd/reverse] (19.0000,10.9697) -- (19.0000,32.9090);
\draw[sd/reverse] (19.0000,10.9697) -- (0.0000,0.0000);
\draw[sd/reverse] (38.0000,0.0000) -- (0.0000,0.0000);
\draw[sd/arrow] (19.0000,32.9090) -- (0.0000,0.0000);
\draw[sd/arrow] (38.0000,0.0000) -- (19.0000,32.9090);
\node[sd/label] at (19.0000,7.6697) {$b_s'$};
\node[sd/label] at (-2.4249,-1.4000) {$\alpha_3$};
\node[sd/label] at (19.0000,35.2090) {$\alpha_2$};
\node[sd/label] at (40.4249,-1.4000) {$\alpha_1$};
\node[sd/label] at (19.7600,-2.2559) {$X_6$};
\node[sd/label] at (6.9563,18.8007) {$X_5$};
\node[sd/label] at (31.8037,17.4843) {$X_4$};
\node[sd/label] at (15.9393,18.8678) {$X_2$};
\node[sd/label] at (11.4207,3.0356) {$X_3$};
\node[sd/label] at (26.5793,3.0356) {$X_1$};
\end{tikzpicture}
}

\expandafter\def\csname NGD@TPE-left-second\endcsname{%
\begin{tikzpicture}[sd/base,x=1mm,y=1mm]
\draw[sd/arrow] (19.0000,10.9697) -- (38.0000,0.0000);
\draw[sd/arrow] (19.0000,10.9697) -- (19.0000,32.9090);
\draw[sd/arrow] (19.0000,10.9697) -- (0.0000,0.0000);
\draw[sd/arrow] (38.0000,0.0000) -- (0.0000,0.0000);
\draw[sd/reverse] (19.0000,32.9090) -- (0.0000,0.0000);
\draw[sd/reverse] (38.0000,0.0000) -- (19.0000,32.9090);
\node[sd/label] at (19.0000,7.9197) {$b_s$};
\node[sd/label] at (-2.4249,-1.4000) {$\alpha_4$};
\node[sd/label] at (19.0000,35.2090) {$\alpha_5$};
\node[sd/label] at (40.4249,-1.4000) {$\alpha_6$};
\node[sd/label] at (19.7600,-2.2559) {$X_9$};
\node[sd/label] at (6.9563,18.8007) {$X_7$};
\node[sd/label] at (31.8037,17.4843) {$X_8$};
\node[sd/label] at (15.9393,18.8678) {$X_2$};
\node[sd/label] at (11.4207,3.0356) {$X_1$};
\node[sd/label] at (26.5793,3.0356) {$X_3$};
\end{tikzpicture}
}

\expandafter\def\csname NGD@TPE-right-1\endcsname{%
\begin{tikzpicture}[sd/base,x=1mm,y=1mm]
\draw[sd/arrow] (19.0000,10.9697) -- (38.0000,0.0000);
\draw[sd/reverse] (19.0000,10.9697) -- (19.0000,32.9090);
\draw[sd/arrow] (19.0000,10.9697) -- (0.0000,0.0000);
\draw[sd/reverse] (38.0000,0.0000) -- (0.0000,0.0000);
\draw[sd/reverse] (19.0000,32.9090) -- (0.0000,0.0000);
\draw[sd/arrow] (38.0000,0.0000) -- (19.0000,32.9090);
\node[sd/label] at (19.0000,7.4197) {$b_{3}'$};
\node[sd/label] at (7.4600,7.4600) {$\alpha_1$};
\node[sd/label] at (19.0000,35.2090) {$b_1$};
\node[sd/label] at (30.5400,7.4600) {$\alpha_4$};
\node[sd/label] at (19.7600,-2.2559) {$X_1$};
\node[sd/label] at (6.9563,18.8007) {$X_4$};
\node[sd/label] at (31.8037,17.4843) {$X_7$};
\node[sd/label] at (16.5618,21.5005) {$X$};
\node[sd/label] at (11.4207,3.0356) {$X_6$};
\node[sd/label] at (26.5793,3.0356) {$X_9$};
\end{tikzpicture}
}

\expandafter\def\csname NGD@TPE-right-2\endcsname{%
\begin{tikzpicture}[sd/base,x=1mm,y=1mm]
\draw[sd/arrow] (19.0000,10.9697) -- (38.0000,0.0000);
\draw[sd/reverse] (19.0000,10.9697) -- (19.0000,32.9090);
\draw[sd/arrow] (19.0000,10.9697) -- (0.0000,0.0000);
\draw[sd/reverse] (38.0000,0.0000) -- (0.0000,0.0000);
\draw[sd/reverse] (19.0000,32.9090) -- (0.0000,0.0000);
\draw[sd/arrow] (38.0000,0.0000) -- (19.0000,32.9090);
\node[sd/label] at (19.0000,7.4197) {$b_{1}'$};
\node[sd/label] at (7.4600,7.4600) {$\alpha_2$};
\node[sd/label] at (19.0000,35.2090) {$b_2$};
\node[sd/label] at (30.5400,7.4600) {$\alpha_5$};
\node[sd/label] at (19.7600,-2.2559) {$X_2$};
\node[sd/label] at (6.9563,18.8007) {$X_5$};
\node[sd/label] at (31.8037,17.4843) {$X_8$};
\node[sd/label] at (16.5618,21.5005) {$X$};
\node[sd/label] at (11.4207,3.0356) {$X_4$};
\node[sd/label] at (26.5793,3.0356) {$X_7$};
\end{tikzpicture}
}

\expandafter\def\csname NGD@TPE-right-3\endcsname{%
\begin{tikzpicture}[sd/base,x=1mm,y=1mm]
\draw[sd/arrow] (19.0000,10.9697) -- (38.0000,0.0000);
\draw[sd/reverse] (19.0000,10.9697) -- (19.0000,32.9090);
\draw[sd/arrow] (19.0000,10.9697) -- (0.0000,0.0000);
\draw[sd/reverse] (38.0000,0.0000) -- (0.0000,0.0000);
\draw[sd/reverse] (19.0000,32.9090) -- (0.0000,0.0000);
\draw[sd/arrow] (38.0000,0.0000) -- (19.0000,32.9090);
\node[sd/label] at (19.0000,7.4197) {$b_{2}'$};
\node[sd/label] at (7.4600,7.4600) {$\alpha_3$};
\node[sd/label] at (19.0000,35.2090) {$b_3$};
\node[sd/label] at (30.5400,7.4600) {$\alpha_6$};
\node[sd/label] at (19.7600,-2.2559) {$X_3$};
\node[sd/label] at (6.9563,18.8007) {$X_6$};
\node[sd/label] at (31.8037,17.4843) {$X_9$};
\node[sd/label] at (16.5618,21.5005) {$X$};
\node[sd/label] at (11.4207,3.0356) {$X_5$};
\node[sd/label] at (26.5793,3.0356) {$X_8$};
\end{tikzpicture}
}

\expandafter\def\csname NGD@annular-D\endcsname{%
\begin{tikzpicture}[sd/base,x=1mm,y=1mm]
\node[sd/box,minimum width=6.7800mm,minimum height=4.0000mm] (b1) at (-5.3155,-11.8400) {$R_{g}$};
\node[sd/box,minimum width=10.6332mm,minimum height=4.0101mm] (b2) at (3.3889,-20.2451) {$L_{-g}$};
\draw[sd/string] (-7.6055,0.0000) -- ([xshift=-2.2900mm]b1.north);
\node[sd/label,anchor=south] at (-7.6055,1.0000) {$\varrho$};
\draw[sd/string] (7.6055,0.0000) -- ([xshift=4.2166mm]b2.north);
\node[sd/label,anchor=south] at (7.6055,1.0000) {$\varrho$};
\draw[sd/string] (-3.6055,-5.4400) -- ([xshift=1.7100mm]b1.north);
\node[sd/label,anchor=west] at (-2.6055,-7.6400) {$g$};
\draw[sd/string] (0.3945,-5.4400) -- ([xshift=-2.9944mm]b2.north);
\node[sd/label,anchor=west] at (1.3945,-7.6400) {$-g$};
\draw[sd/string] ([xshift=-0.0000mm]b1.south) -- (-5.3155,-27.2501);
\node[sd/label,anchor=north] at (-5.3155,-28.2501) {$\varrho$};
\draw[sd/string] ([xshift=-0.0000mm]b2.south) -- (3.3889,-27.2501);
\node[sd/label,anchor=north] at (3.3889,-28.2501) {$\varrho$};
\draw[sd/string] (-3.6055,-5.4400) arc[start angle=180,end angle=0,x radius=2.0000mm,y radius=1.2400mm];
\end{tikzpicture}
}

\expandafter\def\csname NGD@annular-cup\endcsname{%
\begin{tikzpicture}[sd/base,x=1mm,y=1mm]
\node[sd/box,minimum width=6.7800mm,minimum height=4.0000mm] (b2) at (-4.7830,-16.4253) {$R_{g}$};
\node[sd/box,minimum width=9.3553mm,minimum height=4.0101mm] (b3) at (3.4954,-16.4304) {$L_{-g}$};
\draw[sd/string] (-7.0730,-4.3853) -- ([xshift=-2.2900mm]b2.north);
\draw[sd/string] (7.0730,-4.3853) -- ([xshift=3.5777mm]b3.north);
\draw[sd/string] (-3.0730,-10.0253) -- ([xshift=1.7100mm]b2.north);
\node[sd/label,anchor=west] at (-2.0730,-12.2253) {$g$};
\draw[sd/string] (0.9270,-10.0253) -- ([xshift=-2.5684mm]b3.north);
\node[sd/label,anchor=west] at (1.9270,-12.2253) {$-g$};
\draw[sd/string] ([xshift=-0.0000mm]b2.south) -- (-4.7830,-23.4354);
\node[sd/label,anchor=north] at (-4.7830,-24.4354) {$\varrho$};
\draw[sd/string] ([xshift=-0.0000mm]b3.south) -- (3.4954,-23.4354);
\node[sd/label,anchor=north] at (3.4954,-24.4354) {$\varrho$};
\draw[sd/string] (-7.0730,-4.3853) arc[start angle=180,end angle=0,x radius=7.0730mm,y radius=4.3853mm];
\draw[sd/string] (-3.0730,-10.0253) arc[start angle=180,end angle=0,x radius=2.0000mm,y radius=1.2400mm];
\end{tikzpicture}
}

\expandafter\def\csname NGD@basis-h\endcsname{%
\begin{tikzpicture}[sd/base,x=1mm,y=1mm]
\node[sd/box,minimum width=10.4000mm,minimum height=4.0000mm] (b0) at (0.0000,-2.0000) {$u_h$};
\draw[sd/string] ([xshift=-4.1000mm]b0.south) -- (-4.1000,-9.0000);
\node[sd/label,anchor=north] at (-4.1000,-10.0000) {$\varrho$};
\draw[sd/string] ([xshift=-0.1000mm]b0.south) -- (-0.1000,-9.0000);
\node[sd/label,anchor=north] at (-0.1000,-10.0000) {$\varrho$};
\draw[sd/string] ([xshift=3.9000mm]b0.south) -- (3.9000,-9.0000);
\node[sd/label,anchor=north] at (3.9000,-10.0000) {$\varrho$};
\end{tikzpicture}
}

\expandafter\def\csname NGD@bicharacter-end\endcsname{%
\begin{tikzpicture}[sd/base,x=1mm,y=1mm]
\node[sd/box,minimum width=7.6235mm,minimum height=4.0000mm] (b0) at (-0.0270,-6.2000) {$L_{h}$};
\node[sd/box,minimum width=9.9316mm,minimum height=4.0000mm] (b1) at (-2.1342,-14.6000) {$L_{g}$};
\node[sd/box,minimum width=10.3342mm,minimum height=4.0000mm] (b2) at (1.9329,-23.0000) {$R_{k}$};
\draw[sd/string] (-6.0000,0.0000) -- ([xshift=-3.8658mm]b1.north);
\node[sd/label,anchor=south] at (-6.0000,1.0000) {$g$};
\draw[sd/string] (-2.0000,0.0000) -- ([xshift=-1.9730mm]b0.north);
\node[sd/label,anchor=south] at (-2.0000,1.0000) {$h$};
\draw[sd/string] (2.0000,0.0000) -- ([xshift=2.0270mm]b0.north);
\node[sd/label,anchor=south] at (2.0000,1.0000) {$\varrho$};
\draw[sd/string] (6.0000,0.0000) -- ([xshift=4.0671mm]b2.north);
\node[sd/label,anchor=south] at (6.0000,1.0000) {$k$};
\draw[sd/string] ([xshift=-0.2260mm]b0.south) -- ([xshift=1.8813mm]b1.north);
\draw[sd/string] ([xshift=1.1724mm]b1.south) -- ([xshift=-2.8947mm]b2.north);
\draw[sd/string] ([xshift=0.0000mm]b2.south) -- (1.9329,-30.0000);
\node[sd/label,anchor=north] at (1.9329,-31.0000) {$\varrho$};
\end{tikzpicture}
}

\expandafter\def\csname NGD@bicharacter-mid\endcsname{%
\begin{tikzpicture}[sd/base,x=1mm,y=1mm]
\node[sd/box,minimum width=7.6235mm,minimum height=4.0000mm] (b0) at (-0.0270,-6.2000) {$L_{h}$};
\node[sd/box,minimum width=9.9795mm,minimum height=4.0000mm] (b1) at (2.1102,-14.6000) {$R_{k}$};
\node[sd/box,minimum width=10.3102mm,minimum height=4.0000mm] (b2) at (-1.9449,-23.0000) {$L_{g}$};
\draw[sd/string] (-6.0000,0.0000) -- ([xshift=-4.0551mm]b2.north);
\node[sd/label,anchor=south] at (-6.0000,1.0000) {$g$};
\draw[sd/string] (-2.0000,0.0000) -- ([xshift=-1.9730mm]b0.north);
\node[sd/label,anchor=south] at (-2.0000,1.0000) {$h$};
\draw[sd/string] (2.0000,0.0000) -- ([xshift=2.0270mm]b0.north);
\node[sd/label,anchor=south] at (2.0000,1.0000) {$\varrho$};
\draw[sd/string] (6.0000,0.0000) -- ([xshift=3.8898mm]b1.north);
\node[sd/label,anchor=south] at (6.0000,1.0000) {$k$};
\draw[sd/string] ([xshift=0.2370mm]b0.south) -- ([xshift=-1.9002mm]b1.north);
\draw[sd/string] ([xshift=-1.1684mm]b1.south) -- ([xshift=2.8867mm]b2.north);
\draw[sd/string] ([xshift=0.0000mm]b2.south) -- (-1.9449,-30.0000);
\node[sd/label,anchor=north] at (-1.9449,-31.0000) {$\varrho$};
\end{tikzpicture}
}

\expandafter\def\csname NGD@bicharacter-start\endcsname{%
\begin{tikzpicture}[sd/base,x=1mm,y=1mm]
\node[sd/box,minimum width=6.8392mm,minimum height=4.0000mm] (b0) at (3.6804,-6.2000) {$R_{k}$};
\node[sd/box,minimum width=8.0823mm,minimum height=4.0000mm] (b1) at (0.2024,-14.6000) {$L_{h}$};
\node[sd/box,minimum width=8.4024mm,minimum height=4.0000mm] (b2) at (-2.8988,-23.0000) {$L_{g}$};
\draw[sd/string] (-6.0000,0.0000) -- ([xshift=-3.1012mm]b2.north);
\node[sd/label,anchor=south] at (-6.0000,1.0000) {$g$};
\draw[sd/string] (-2.0000,0.0000) -- ([xshift=-2.2024mm]b1.north);
\node[sd/label,anchor=south] at (-2.0000,1.0000) {$h$};
\draw[sd/string] (2.0000,0.0000) -- ([xshift=-1.6804mm]b0.north);
\node[sd/label,anchor=south] at (2.0000,1.0000) {$\varrho$};
\draw[sd/string] (6.0000,0.0000) -- ([xshift=2.3196mm]b0.north);
\node[sd/label,anchor=south] at (6.0000,1.0000) {$k$};
\draw[sd/string] ([xshift=-1.0655mm]b0.south) -- ([xshift=2.4125mm]b1.north);
\draw[sd/string] ([xshift=-0.8504mm]b1.south) -- ([xshift=2.2508mm]b2.north);
\draw[sd/string] ([xshift=0.0000mm]b2.south) -- (-2.8988,-30.0000);
\node[sd/label,anchor=north] at (-2.8988,-31.0000) {$\varrho$};
\end{tikzpicture}
}

\expandafter\def\csname NGD@boundary-two-groups\endcsname{%
\begin{tikzpicture}[sd/base,x=1mm,y=1mm]
\draw[sd/arrow] (0.0000,25.2000) -- (42.0000,25.2000);
\draw[sd/arrow] (12.1800,12.6000) -- (29.8200,12.6000);
\draw[sd/string] (0.0000,0.0000) -- (42.0000,0.0000);
\draw[sd/string] (0.0000,25.2000) -- (12.1800,12.6000);
\draw[sd/string] (12.1800,12.6000) -- (0.0000,0.0000);
\draw[sd/string] (0.0000,0.0000) -- (0.0000,25.2000);
\draw[sd/string] (42.0000,25.2000) -- (29.8200,12.6000);
\draw[sd/string] (29.8200,12.6000) -- (42.0000,0.0000);
\draw[sd/string] (42.0000,0.0000) -- (42.0000,25.2000);
\fill (1.0870,22.5114) circle[radius=.30mm];
\fill (10.2800,12.6000) circle[radius=.30mm];
\node[sd/label] at (2.0804,5.1453) {$p$};
\fill (40.9130,22.5114) circle[radius=.30mm];
\fill (31.7200,12.6000) circle[radius=.30mm];
\node[sd/label] at (39.9196,5.1453) {$q$};
\node[sd/label] at (20.1600,27.1786) {$g$};
\node[sd/label] at (20.6472,10.5237) {$h$};
\end{tikzpicture}
}

\expandafter\def\csname NGD@cancel-L-domain\endcsname{%
\begin{tikzpicture}[sd/base,x=1mm,y=1mm]
\node[sd/box,minimum width=6.9906mm,minimum height=4.0000mm] (b0) at (-0.4680,-6.2000) {$L_{g}$};
\node[sd/box,minimum width=7.9267mm,minimum height=4.8321mm] (b1) at (0.0000,-15.0161) {$L_{g}^{-1}$};
\draw[sd/string] (-2.8634,0.0000) -- ([xshift=-2.3953mm]b0.north);
\node[sd/label,anchor=south] at (-2.8634,1.0000) {$g$};
\draw[sd/string] (1.1366,0.0000) -- ([xshift=1.6047mm]b0.north);
\node[sd/label,anchor=south] at (1.1366,1.0000) {$\varrho$};
\draw[sd/string] ([xshift=0.2340mm]b0.south) -- ([xshift=-0.2340mm]b1.north);
\draw[sd/string] ([xshift=-2.6789mm]b1.south) -- (-2.6789,-22.4321);
\node[sd/label,anchor=north] at (-2.6789,-23.4321) {$g$};
\draw[sd/string] ([xshift=1.3211mm]b1.south) -- (1.3211,-22.4321);
\node[sd/label,anchor=north] at (1.3211,-23.4321) {$\varrho$};
\end{tikzpicture}
}

\expandafter\def\csname NGD@cancel-L-range\endcsname{%
\begin{tikzpicture}[sd/base,x=1mm,y=1mm]
\node[sd/box,minimum width=7.9267mm,minimum height=4.8321mm] (b0) at (0.0000,-6.6161) {$L_{g}^{-1}$};
\node[sd/box,minimum width=6.7800mm,minimum height=4.0000mm] (b1) at (-0.5734,-15.4321) {$L_{g}$};
\draw[sd/string] (0.0000,0.0000) -- ([xshift=0.0000mm]b0.north);
\node[sd/label,anchor=south] at (0.0000,1.0000) {$\varrho$};
\draw[sd/string] ([xshift=-2.8634mm]b0.south) -- ([xshift=-2.2900mm]b1.north);
\node[sd/label,anchor=west] at (-1.8634,-11.2321) {$g$};
\draw[sd/string] ([xshift=1.1366mm]b0.south) -- ([xshift=1.7100mm]b1.north);
\draw[sd/string] ([xshift=0.0000mm]b1.south) -- (-0.5734,-22.4321);
\node[sd/label,anchor=north] at (-0.5734,-23.4321) {$\varrho$};
\end{tikzpicture}
}

\expandafter\def\csname NGD@closed-coefficient\endcsname{%
\begin{tikzpicture}[sd/base,x=1mm,y=1mm]
\node[sd/box,minimum width=6.7800mm,minimum height=4.0000mm] (b2) at (-3.1600,-15.7600) {$H_g$};
\node[sd/box,minimum width=6.7800mm,minimum height=4.0000mm] (b3) at (-3.1600,-24.1600) {$p_k$};
\node[sd/box,minimum width=6.7800mm,minimum height=4.0000mm] (b4) at (-3.1600,-32.5600) {$q_h$};
\draw[sd/string] (-5.4500,-3.7200) -- ([xshift=-2.2900mm]b2.north);
\draw[sd/string] (6.5500,-3.7200) -- (6.5500,-38.9600);
\draw[sd/string] (-1.4500,-9.3600) -- ([xshift=1.7100mm]b2.north);
\draw[sd/string] (2.5500,-9.3600) -- (2.5500,-38.9600);
\draw[sd/string] ([xshift=-2.2900mm]b2.south) -- ([xshift=-2.2900mm]b3.north);
\draw[sd/string] ([xshift=1.7100mm]b2.south) -- ([xshift=1.7100mm]b3.north);
\draw[sd/string] ([xshift=0.0000mm]b3.south) -- ([xshift=0.0000mm]b4.north);
\draw[sd/string] ([xshift=-2.2900mm]b4.south) -- (-5.4500,-38.9600);
\draw[sd/string] ([xshift=1.7100mm]b4.south) -- (-1.4500,-38.9600);
\draw[sd/string] (-5.4500,-3.7200) arc[start angle=180,end angle=0,x radius=6.0000mm,y radius=3.7200mm];
\draw[sd/string] (-1.4500,-9.3600) arc[start angle=180,end angle=0,x radius=2.0000mm,y radius=1.2400mm];
\draw[sd/string] (-1.4500,-38.9600) arc[start angle=180,end angle=360,x radius=2.0000mm,y radius=1.2400mm];
\draw[sd/string] (-5.4500,-38.9600) arc[start angle=180,end angle=360,x radius=6.0000mm,y radius=3.7200mm];
\end{tikzpicture}
}

\expandafter\def\csname NGD@closed-invertible-coefficient\endcsname{%
\begin{tikzpicture}[sd/base,x=1mm,y=1mm]
\node[sd/box,minimum width=6.7800mm,minimum height=4.0000mm] (b2) at (-3.1600,-15.7600) {$H_g$};
\node[sd/box,minimum width=6.7800mm,minimum height=4.0000mm] (b3) at (-3.1600,-24.1600) {$w_l$};
\node[sd/box,minimum width=6.7800mm,minimum height=4.0000mm] (b4) at (-3.1600,-32.5600) {$\check w_l$};
\draw[sd/string] (-5.4500,-3.7200) -- ([xshift=-2.2900mm]b2.north);
\draw[sd/string] (6.5500,-3.7200) -- (6.5500,-38.9600);
\draw[sd/string] (-1.4500,-9.3600) -- ([xshift=1.7100mm]b2.north);
\draw[sd/string] (2.5500,-9.3600) -- (2.5500,-38.9600);
\draw[sd/string] ([xshift=-2.2900mm]b2.south) -- ([xshift=-2.2900mm]b3.north);
\draw[sd/string] ([xshift=1.7100mm]b2.south) -- ([xshift=1.7100mm]b3.north);
\draw[sd/string] ([xshift=0.0000mm]b3.south) -- ([xshift=0.0000mm]b4.north);
\node[sd/label,anchor=west] at (-2.1600,-28.3600) {$l$};
\draw[sd/string] ([xshift=-2.2900mm]b4.south) -- (-5.4500,-38.9600);
\draw[sd/string] ([xshift=1.7100mm]b4.south) -- (-1.4500,-38.9600);
\draw[sd/string] (-5.4500,-3.7200) arc[start angle=180,end angle=0,x radius=6.0000mm,y radius=3.7200mm];
\draw[sd/string] (-1.4500,-9.3600) arc[start angle=180,end angle=0,x radius=2.0000mm,y radius=1.2400mm];
\draw[sd/string] (-1.4500,-38.9600) arc[start angle=180,end angle=360,x radius=2.0000mm,y radius=1.2400mm];
\draw[sd/string] (-5.4500,-38.9600) arc[start angle=180,end angle=360,x radius=6.0000mm,y radius=3.7200mm];
\end{tikzpicture}
}

\expandafter\def\csname NGD@coev-minus\endcsname{%
\begin{tikzpicture}[sd/base,x=1mm,y=1mm]
\draw[sd/string] (-2.7190,-1.6858) -- (-2.7190,-6.6858);
\node[sd/label,anchor=north] at (-2.7190,-7.6858) {$X^*$};
\draw[sd/string] (2.7190,-1.6858) -- (2.7190,-6.6858);
\node[sd/label,anchor=north] at (2.7190,-7.6858) {$X$};
\draw[sd/string] (-2.7190,-1.6858) arc[start angle=180,end angle=0,x radius=2.7190mm,y radius=1.6858mm];
\end{tikzpicture}
}

\expandafter\def\csname NGD@coev-plus\endcsname{%
\begin{tikzpicture}[sd/base,x=1mm,y=1mm]
\draw[sd/string] (-2.7190,-1.6858) -- (-2.7190,-6.6858);
\node[sd/label,anchor=north] at (-2.7190,-7.6858) {$X$};
\draw[sd/string] (2.7190,-1.6858) -- (2.7190,-6.6858);
\node[sd/label,anchor=north] at (2.7190,-7.6858) {$X^*$};
\draw[sd/string] (-2.7190,-1.6858) arc[start angle=180,end angle=0,x radius=2.7190mm,y radius=1.6858mm];
\end{tikzpicture}
}

\expandafter\def\csname NGD@cyclic-result\endcsname{%
\begin{tikzpicture}[sd/base,x=1mm,y=1mm]
\node[sd/box,minimum width=12.3940mm,minimum height=4.1117mm] (b0) at (0.0000,-2.0558) {$\varrho(u)$};
\draw[sd/string] ([xshift=-5.0970mm]b0.south) -- (-5.0970,-9.1117);
\node[sd/label,anchor=north] at (-5.0970,-10.1117) {$Y$};
\draw[sd/string] ([xshift=-0.7879mm]b0.south) -- (-0.7879,-9.1117);
\node[sd/label,anchor=north] at (-0.7879,-10.1117) {$Z$};
\draw[sd/string] ([xshift=5.0970mm]b0.south) -- (5.0970,-9.1117);
\node[sd/label,anchor=north] at (5.0970,-10.1117) {$X^{**}$};
\end{tikzpicture}
}

\expandafter\def\csname NGD@cyclic-rotation\endcsname{%
\begin{tikzpicture}[sd/base,x=1mm,y=1mm]
\node[sd/box,minimum width=13.5948mm,minimum height=4.0000mm] (b1) at (1.0131,-13.8300) {$u$};
\draw[sd/string] (-11.9838,-7.4300) -- (-11.9838,-20.2300);
\node[sd/label,anchor=west] at (-10.9838,-9.6300) {$X^*$};
\draw[sd/string] (11.9838,-7.4300) -- (11.9838,-27.4928);
\node[sd/label,anchor=north] at (11.9838,-28.4928) {$X^{**}$};
\draw[sd/string] ([xshift=-5.6974mm]b1.south) -- (-4.6844,-20.2300);
\node[sd/label,anchor=west] at (-3.6844,-18.0300) {$X$};
\draw[sd/string] ([xshift=0.7767mm]b1.south) -- (1.7898,-27.4928);
\node[sd/label,anchor=north] at (1.7898,-28.4928) {$Y$};
\draw[sd/string] ([xshift=5.0859mm]b1.south) -- (6.0990,-27.4928);
\node[sd/label,anchor=north] at (6.0990,-28.4928) {$Z$};
\draw[sd/string] (-11.9838,-7.4300) arc[start angle=180,end angle=0,x radius=11.9838mm,y radius=7.4300mm];
\draw[sd/string] (-11.9838,-20.2300) arc[start angle=180,end angle=360,x radius=3.6497mm,y radius=2.2628mm];
\end{tikzpicture}
}

\expandafter\def\csname NGD@ev-minus\endcsname{%
\begin{tikzpicture}[sd/base,x=1mm,y=1mm]
\draw[sd/string] (-2.7190,0.0000) -- (-2.7190,-4.2000);
\node[sd/label,anchor=south] at (-2.7190,1.0000) {$X$};
\draw[sd/string] (2.7190,0.0000) -- (2.7190,-4.2000);
\node[sd/label,anchor=south] at (2.7190,1.0000) {$X^*$};
\draw[sd/string] (-2.7190,-4.2000) arc[start angle=180,end angle=360,x radius=2.7190mm,y radius=1.6858mm];
\end{tikzpicture}
}

\expandafter\def\csname NGD@ev-plus\endcsname{%
\begin{tikzpicture}[sd/base,x=1mm,y=1mm]
\draw[sd/string] (-2.7190,0.0000) -- (-2.7190,-4.2000);
\node[sd/label,anchor=south] at (-2.7190,1.0000) {$X^*$};
\draw[sd/string] (2.7190,0.0000) -- (2.7190,-4.2000);
\node[sd/label,anchor=south] at (2.7190,1.0000) {$X$};
\draw[sd/string] (-2.7190,-4.2000) arc[start angle=180,end angle=360,x radius=2.7190mm,y radius=1.6858mm];
\end{tikzpicture}
}

\expandafter\def\csname NGD@group-loop\endcsname{%
\begin{tikzpicture}[sd/base,x=1mm,y=1mm]
\draw[sd/string] (-2.0000,-1.2400) -- (-2.0000,-5.6400);
\node[sd/label,anchor=west] at (-1.0000,-3.4400) {$g$};
\draw[sd/string] (2.0000,-1.2400) -- (2.0000,-5.6400);
\node[sd/label,anchor=west] at (3.0000,-3.4400) {$-g$};
\draw[sd/string] (-2.0000,-1.2400) arc[start angle=180,end angle=0,x radius=2.0000mm,y radius=1.2400mm];
\draw[sd/string] (-2.0000,-5.6400) arc[start angle=180,end angle=360,x radius=2.0000mm,y radius=1.2400mm];
\end{tikzpicture}
}

\expandafter\def\csname NGD@group-loop-l\endcsname{%
\begin{tikzpicture}[sd/base,x=1mm,y=1mm]
\draw[sd/string] (-2.0000,-1.2400) -- (-2.0000,-5.6400);
\node[sd/label,anchor=west] at (-1.0000,-3.4400) {$l$};
\draw[sd/string] (2.0000,-1.2400) -- (2.0000,-5.6400);
\node[sd/label,anchor=west] at (3.0000,-3.4400) {$-l$};
\draw[sd/string] (-2.0000,-1.2400) arc[start angle=180,end angle=0,x radius=2.0000mm,y radius=1.2400mm];
\draw[sd/string] (-2.0000,-5.6400) arc[start angle=180,end angle=360,x radius=2.0000mm,y radius=1.2400mm];
\end{tikzpicture}
}

\expandafter\def\csname NGD@identity-X\endcsname{%
\begin{tikzpicture}[sd/base,x=1mm,y=1mm]
\draw[sd/string] (0.0000,0.0000) -- (0.0000,-8.0000);
\node[sd/label,anchor=south] at (0.0000,1.0000) {$X$};
\node[sd/label,anchor=north] at (0.0000,-9.0000) {$X$};
\end{tikzpicture}
}

\expandafter\def\csname NGD@identity-g-rho\endcsname{%
\begin{tikzpicture}[sd/base,x=1mm,y=1mm]
\draw[sd/string] (-2.0000,0.0000) -- (-2.0000,-8.0000);
\node[sd/label,anchor=south] at (-2.0000,1.0000) {$g$};
\node[sd/label,anchor=north] at (-2.0000,-9.0000) {$g$};
\draw[sd/string] (2.0000,0.0000) -- (2.0000,-8.0000);
\node[sd/label,anchor=south] at (2.0000,1.0000) {$\varrho$};
\node[sd/label,anchor=north] at (2.0000,-9.0000) {$\varrho$};
\end{tikzpicture}
}

\expandafter\def\csname NGD@identity-rho\endcsname{%
\begin{tikzpicture}[sd/base,x=1mm,y=1mm]
\draw[sd/string] (0.0000,0.0000) -- (0.0000,-8.0000);
\node[sd/label,anchor=south] at (0.0000,1.0000) {$\varrho$};
\node[sd/label,anchor=north] at (0.0000,-9.0000) {$\varrho$};
\end{tikzpicture}
}

\expandafter\def\csname NGD@identity-rho2\endcsname{%
\begin{tikzpicture}[sd/base,x=1mm,y=1mm]
\draw[sd/string] (-2.0000,0.0000) -- (-2.0000,-8.0000);
\node[sd/label,anchor=south] at (-2.0000,1.0000) {$\varrho$};
\node[sd/label,anchor=north] at (-2.0000,-9.0000) {$\varrho$};
\draw[sd/string] (2.0000,0.0000) -- (2.0000,-8.0000);
\node[sd/label,anchor=south] at (2.0000,1.0000) {$\varrho$};
\node[sd/label,anchor=north] at (2.0000,-9.0000) {$\varrho$};
\end{tikzpicture}
}

\expandafter\def\csname NGD@identity-simple\endcsname{%
\begin{tikzpicture}[sd/base,x=1mm,y=1mm]
\draw[sd/string] (0.0000,0.0000) -- (0.0000,-8.0000);
\node[sd/label,anchor=south] at (0.0000,1.0000) {$X$};
\node[sd/label,anchor=north] at (0.0000,-9.0000) {$X$};
\end{tikzpicture}
}

\expandafter\def\csname NGD@insert-x\endcsname{%
\begin{tikzpicture}[sd/base,x=1mm,y=1mm]
\node[sd/box,minimum width=7.9267mm,minimum height=4.8321mm] (b0) at (1.6325,-6.6161) {$L_{g}^{-1}$};
\node[sd/box,minimum width=7.2849mm,minimum height=4.0000mm] (b1) at (-1.9534,-15.4321) {$R_{g}$};
\node[sd/box,minimum width=9.6408mm,minimum height=4.0000mm] (b2) at (0.7754,-23.8321) {$u$};
\draw[sd/string] (-4.4958,0.0000) -- ([xshift=-2.5424mm]b1.north);
\node[sd/label,anchor=south] at (-4.4958,1.0000) {$\varrho$};
\draw[sd/string] (1.6325,0.0000) -- ([xshift=0.0000mm]b0.north);
\node[sd/label,anchor=south] at (1.6325,1.0000) {$\varrho$};
\draw[sd/string] ([xshift=-2.1283mm]b0.south) -- ([xshift=1.4576mm]b1.north);
\node[sd/label,anchor=west] at (0.5042,-11.2321) {$g$};
\draw[sd/string] ([xshift=1.8717mm]b0.south) -- ([xshift=2.7288mm]b2.north);
\draw[sd/string] ([xshift=0.5610mm]b1.south) -- ([xshift=-2.1678mm]b2.north);
\draw[sd/string] ([xshift=-0.0000mm]b2.south) -- (0.7754,-30.8321);
\node[sd/label,anchor=north] at (0.7754,-31.8321) {$\varrho$};
\end{tikzpicture}
}

\expandafter\def\csname NGD@insert-y\endcsname{%
\begin{tikzpicture}[sd/base,x=1mm,y=1mm]
\node[sd/box,minimum width=8.2136mm,minimum height=5.2715mm] (b0) at (1.6325,-6.8357) {$R_{-g}^{-1}$};
\node[sd/box,minimum width=7.2449mm,minimum height=4.0000mm] (b1) at (-2.1168,-15.8715) {$u$};
\node[sd/box,minimum width=10.2346mm,minimum height=4.0101mm] (b2) at (0.6220,-24.2765) {$R_{-g}$};
\draw[sd/string] (-4.6393,0.0000) -- ([xshift=-2.5224mm]b1.north);
\node[sd/label,anchor=south] at (-4.6393,1.0000) {$\varrho$};
\draw[sd/string] (1.6325,0.0000) -- ([xshift=0.0000mm]b0.north);
\node[sd/label,anchor=south] at (1.6325,1.0000) {$\varrho$};
\draw[sd/string] ([xshift=-2.2717mm]b0.south) -- ([xshift=1.4776mm]b1.north);
\draw[sd/string] ([xshift=1.7283mm]b0.south) -- ([xshift=2.7388mm]b2.north);
\node[sd/label,anchor=west] at (4.3607,-11.6715) {$-g$};
\draw[sd/string] ([xshift=0.5165mm]b1.south) -- ([xshift=-2.2223mm]b2.north);
\draw[sd/string] ([xshift=0.0000mm]b2.south) -- (0.6220,-31.2816);
\node[sd/label,anchor=north] at (0.6220,-32.2816) {$\varrho$};
\end{tikzpicture}
}

\expandafter\def\csname NGD@insert-z\endcsname{%
\begin{tikzpicture}[sd/base,x=1mm,y=1mm]
\node[sd/box,minimum width=7.9267mm,minimum height=4.8321mm] (b0) at (-1.6325,-6.6161) {$L_{g}^{-1}$};
\node[sd/box,minimum width=6.5691mm,minimum height=4.0000mm] (b1) at (2.3113,-15.4321) {$u$};
\node[sd/box,minimum width=9.0071mm,minimum height=4.0000mm] (b2) at (-1.0923,-23.8321) {$L_{g}$};
\draw[sd/string] (-1.6325,0.0000) -- ([xshift=0.0000mm]b0.north);
\node[sd/label,anchor=south] at (-1.6325,1.0000) {$\varrho$};
\draw[sd/string] (4.4958,0.0000) -- ([xshift=2.1846mm]b1.north);
\node[sd/label,anchor=south] at (4.4958,1.0000) {$\varrho$};
\draw[sd/string] ([xshift=-2.8634mm]b0.south) -- ([xshift=-3.4035mm]b2.north);
\node[sd/label,anchor=west] at (-3.4958,-11.2321) {$g$};
\draw[sd/string] ([xshift=2.0850mm]b0.south) -- ([xshift=-1.8587mm]b1.north);
\draw[sd/string] ([xshift=-0.9512mm]b1.south) -- ([xshift=2.4524mm]b2.north);
\draw[sd/string] ([xshift=0.0000mm]b2.south) -- (-1.0923,-30.8321);
\node[sd/label,anchor=north] at (-1.0923,-31.8321) {$\varrho$};
\end{tikzpicture}
}

\expandafter\def\csname NGD@interchange-left\endcsname{%
\begin{tikzpicture}[sd/base,x=1mm,y=1mm]
\node[sd/box,minimum width=6.9387mm,minimum height=4.0000mm] (b0) at (1.6306,-6.2000) {$R_{h}$};
\node[sd/box,minimum width=7.8306mm,minimum height=4.0000mm] (b1) at (-1.1847,-14.6000) {$L_{g}$};
\draw[sd/string] (-4.0000,0.0000) -- ([xshift=-2.8153mm]b1.north);
\node[sd/label,anchor=south] at (-4.0000,1.0000) {$g$};
\draw[sd/string] (0.0000,0.0000) -- ([xshift=-1.6306mm]b0.north);
\node[sd/label,anchor=south] at (0.0000,1.0000) {$\varrho$};
\draw[sd/string] (4.0000,0.0000) -- ([xshift=2.3694mm]b0.north);
\node[sd/label,anchor=south] at (4.0000,1.0000) {$h$};
\draw[sd/string] ([xshift=-0.7551mm]b0.south) -- ([xshift=2.0602mm]b1.north);
\draw[sd/string] ([xshift=0.0000mm]b1.south) -- (-1.1847,-21.6000);
\node[sd/label,anchor=north] at (-1.1847,-22.6000) {$\varrho$};
\end{tikzpicture}
}

\expandafter\def\csname NGD@interchange-right\endcsname{%
\begin{tikzpicture}[sd/base,x=1mm,y=1mm]
\node[sd/box,minimum width=6.8392mm,minimum height=4.0000mm] (b0) at (-1.6804,-6.2000) {$L_{g}$};
\node[sd/box,minimum width=7.8804mm,minimum height=4.0000mm] (b1) at (1.1598,-14.6000) {$R_{h}$};
\draw[sd/string] (-4.0000,0.0000) -- ([xshift=-2.3196mm]b0.north);
\node[sd/label,anchor=south] at (-4.0000,1.0000) {$g$};
\draw[sd/string] (0.0000,0.0000) -- ([xshift=1.6804mm]b0.north);
\node[sd/label,anchor=south] at (0.0000,1.0000) {$\varrho$};
\draw[sd/string] (4.0000,0.0000) -- ([xshift=2.8402mm]b1.north);
\node[sd/label,anchor=south] at (4.0000,1.0000) {$h$};
\draw[sd/string] ([xshift=0.7634mm]b0.south) -- ([xshift=-2.0768mm]b1.north);
\draw[sd/string] ([xshift=-0.0000mm]b1.south) -- (1.1598,-21.6000);
\node[sd/label,anchor=north] at (1.1598,-22.6000) {$\varrho$};
\end{tikzpicture}
}

\expandafter\def\csname NGD@invariant-u\endcsname{%
\begin{tikzpicture}[sd/base,x=1mm,y=1mm]
\node[sd/box,minimum width=11.0451mm,minimum height=4.0000mm] (b0) at (0.0000,-2.0000) {$u$};
\draw[sd/string] ([xshift=-4.4225mm]b0.south) -- (-4.4225,-9.0000);
\node[sd/label,anchor=north] at (-4.4225,-10.0000) {$X$};
\draw[sd/string] ([xshift=0.1134mm]b0.south) -- (0.1134,-9.0000);
\node[sd/label,anchor=north] at (0.1134,-10.0000) {$Y$};
\draw[sd/string] ([xshift=4.4225mm]b0.south) -- (4.4225,-9.0000);
\node[sd/label,anchor=north] at (4.4225,-10.0000) {$Z$};
\end{tikzpicture}
}

\expandafter\def\csname NGD@left-assoc\endcsname{%
\begin{tikzpicture}[sd/base,x=1mm,y=1mm]
\node[sd/box,minimum width=6.9387mm,minimum height=4.0000mm] (b0) at (1.6306,-6.2000) {$L_{h}$};
\node[sd/box,minimum width=7.8306mm,minimum height=4.0000mm] (b1) at (-1.1847,-14.6000) {$L_{g}$};
\draw[sd/string] (-4.0000,0.0000) -- ([xshift=-2.8153mm]b1.north);
\node[sd/label,anchor=south] at (-4.0000,1.0000) {$g$};
\draw[sd/string] (0.0000,0.0000) -- ([xshift=-1.6306mm]b0.north);
\node[sd/label,anchor=south] at (0.0000,1.0000) {$h$};
\draw[sd/string] (4.0000,0.0000) -- ([xshift=2.3694mm]b0.north);
\node[sd/label,anchor=south] at (4.0000,1.0000) {$\varrho$};
\draw[sd/string] ([xshift=-0.7551mm]b0.south) -- ([xshift=2.0602mm]b1.north);
\draw[sd/string] ([xshift=0.0000mm]b1.south) -- (-1.1847,-21.6000);
\node[sd/label,anchor=north] at (-1.1847,-22.6000) {$\varrho$};
\end{tikzpicture}
}

\expandafter\def\csname NGD@left-fused\endcsname{%
\begin{tikzpicture}[sd/base,x=1mm,y=1mm]
\node[sd/box,minimum width=8.9063mm,minimum height=4.0000mm] (b0) at (-1.7770,-6.2000) {$m_{g,h}$};
\node[sd/box,minimum width=12.4602mm,minimum height=4.1273mm] (b1) at (0.0000,-14.6637) {$L_{g+h}$};
\draw[sd/string] (-4.2926,0.0000) -- ([xshift=-2.5156mm]b0.north);
\node[sd/label,anchor=south] at (-4.2926,1.0000) {$g$};
\draw[sd/string] (-0.2926,0.0000) -- ([xshift=1.4844mm]b0.north);
\node[sd/label,anchor=south] at (-0.2926,1.0000) {$h$};
\draw[sd/string] (5.1301,0.0000) -- ([xshift=5.1301mm]b1.north);
\node[sd/label,anchor=south] at (5.1301,1.0000) {$\varrho$};
\draw[sd/string] ([xshift=-3.3532mm]b0.south) -- ([xshift=-5.1301mm]b1.north);
\node[sd/label,anchor=west] at (-4.1301,-10.4000) {$g+h$};
\draw[sd/string] ([xshift=0.0000mm]b1.south) -- (0.0000,-21.7273);
\node[sd/label,anchor=north] at (0.0000,-22.7273) {$\varrho$};
\end{tikzpicture}
}

\expandafter\def\csname NGD@mate-cup\endcsname{%
\begin{tikzpicture}[sd/base,x=1mm,y=1mm]
\node[sd/box,minimum width=8.1625mm,minimum height=4.8321mm] (b1) at (-1.2500,-8.5152) {$R_{g}^{-1}$};
\node[sd/box,minimum width=6.7800mm,minimum height=4.0000mm] (b2) at (1.9412,-17.3313) {$L_{g}$};
\draw[sd/string] (-1.2500,-1.6992) -- ([xshift=-0.0000mm]b1.north);
\draw[sd/string] (4.2312,-1.6992) -- ([xshift=2.2900mm]b2.north);
\draw[sd/string] ([xshift=-2.5188mm]b1.south) -- (-3.7688,-24.3313);
\node[sd/label,anchor=north] at (-3.7688,-25.3313) {$\varrho$};
\draw[sd/string] ([xshift=1.4812mm]b1.south) -- ([xshift=-1.7100mm]b2.north);
\node[sd/label,anchor=west] at (1.2312,-13.1313) {$g$};
\draw[sd/string] ([xshift=-0.0000mm]b2.south) -- (1.9412,-24.3313);
\node[sd/label,anchor=north] at (1.9412,-25.3313) {$\varrho$};
\draw[sd/string] (-1.2500,-1.6992) arc[start angle=180,end angle=0,x radius=2.7406mm,y radius=1.6992mm];
\end{tikzpicture}
}

\expandafter\def\csname NGD@mixed-theta\endcsname{%
\begin{tikzpicture}[sd/base,x=1mm,y=1mm]
\node[sd/box,minimum width=7.9267mm,minimum height=4.8321mm] (b1) at (-0.7000,-8.4787) {$L_{g}^{-1}$};
\node[sd/box,minimum width=6.7800mm,minimum height=4.0000mm] (b2) at (-1.2734,-17.2948) {$L_{g}$};
\draw[sd/string] (-0.7000,-1.6626) -- ([xshift=-0.0000mm]b1.north);
\draw[sd/string] (4.6634,-1.6626) -- (4.6634,-23.6948);
\draw[sd/string] ([xshift=-2.8634mm]b1.south) -- ([xshift=-2.2900mm]b2.north);
\node[sd/label,anchor=west] at (-2.5634,-13.0948) {$g$};
\draw[sd/string] ([xshift=1.1366mm]b1.south) -- ([xshift=1.7100mm]b2.north);
\draw[sd/string] ([xshift=-0.0000mm]b2.south) -- (-1.2734,-23.6948);
\draw[sd/string] (-0.7000,-1.6626) arc[start angle=180,end angle=0,x radius=2.6817mm,y radius=1.6626mm];
\draw[sd/string] (-1.2734,-23.6948) arc[start angle=180,end angle=360,x radius=2.9684mm,y radius=1.8404mm];
\end{tikzpicture}
}

\expandafter\def\csname NGD@morphism-f\endcsname{%
\begin{tikzpicture}[sd/base,x=1mm,y=1mm]
\node[sd/box,minimum width=6.7359mm,minimum height=4.0000mm] (b0) at (0.0000,-6.2000) {$f$};
\draw[sd/string] (-2.2680,0.0000) -- ([xshift=-2.2680mm]b0.north);
\node[sd/label,anchor=south] at (-2.2680,1.0000) {$X$};
\draw[sd/string] (2.2680,0.0000) -- ([xshift=2.2680mm]b0.north);
\node[sd/label,anchor=south] at (2.2680,1.0000) {$Y$};
\draw[sd/string] ([xshift=0.0000mm]b0.south) -- (0.0000,-13.2000);
\node[sd/label,anchor=north] at (0.0000,-14.2000) {$Z$};
\end{tikzpicture}
}

\expandafter\def\csname NGD@partial-trace-R\endcsname{%
\begin{tikzpicture}[sd/base,x=1mm,y=1mm]
\node[sd/box,minimum width=7.7000mm,minimum height=4.0000mm] (b1) at (0.7000,-11.8400) {$R_{g}$};
\draw[sd/string] (3.4500,0.0000) -- ([xshift=2.7500mm]b1.north);
\node[sd/label,anchor=south] at (3.4500,1.0000) {$g$};
\draw[sd/string] (-4.5500,-5.4400) -- (-4.5500,-18.2400);
\draw[sd/string] (-0.5500,-5.4400) -- ([xshift=-1.2500mm]b1.north);
\draw[sd/string] ([xshift=-1.2500mm]b1.south) -- (-0.5500,-18.2400);
\draw[sd/string] (-4.5500,-5.4400) arc[start angle=180,end angle=0,x radius=2.0000mm,y radius=1.2400mm];
\draw[sd/string] (-4.5500,-18.2400) arc[start angle=180,end angle=360,x radius=2.0000mm,y radius=1.2400mm];
\end{tikzpicture}
}

\expandafter\def\csname NGD@pointed-left\endcsname{%
\begin{tikzpicture}[sd/base,x=1mm,y=1mm]
\node[sd/box,minimum width=8.9063mm,minimum height=4.0000mm] (b0) at (-1.8021,-6.2000) {$m_{g,h}$};
\node[sd/box,minimum width=12.5105mm,minimum height=4.0000mm] (b1) at (0.0000,-14.6000) {$m_{g+h,k}$};
\draw[sd/string] (-4.3177,0.0000) -- ([xshift=-2.5156mm]b0.north);
\node[sd/label,anchor=south] at (-4.3177,1.0000) {$g$};
\draw[sd/string] (-0.3177,0.0000) -- ([xshift=1.4844mm]b0.north);
\node[sd/label,anchor=south] at (-0.3177,1.0000) {$h$};
\draw[sd/string] (5.1553,0.0000) -- ([xshift=5.1553mm]b1.north);
\node[sd/label,anchor=south] at (5.1553,1.0000) {$k$};
\draw[sd/string] ([xshift=-3.3532mm]b0.south) -- ([xshift=-5.1553mm]b1.north);
\node[sd/label,anchor=west] at (-4.1553,-10.4000) {$g+h$};
\draw[sd/string] ([xshift=0.0000mm]b1.south) -- (0.0000,-21.6000);
\node[sd/label,anchor=north] at (0.0000,-22.6000) {$g+h+k$};
\end{tikzpicture}
}

\expandafter\def\csname NGD@pointed-right\endcsname{%
\begin{tikzpicture}[sd/base,x=1mm,y=1mm]
\node[sd/box,minimum width=8.9866mm,minimum height=4.0000mm] (b0) at (1.5966,-6.2000) {$m_{h,k}$};
\node[sd/box,minimum width=12.3159mm,minimum height=4.0000mm] (b1) at (0.0000,-14.6000) {$m_{g,h+k}$};
\draw[sd/string] (-5.0580,0.0000) -- ([xshift=-5.0580mm]b1.north);
\node[sd/label,anchor=south] at (-5.0580,1.0000) {$g$};
\draw[sd/string] (-0.9056,0.0000) -- ([xshift=-2.5022mm]b0.north);
\node[sd/label,anchor=south] at (-0.9056,1.0000) {$h$};
\draw[sd/string] (3.0944,0.0000) -- ([xshift=1.4978mm]b0.north);
\node[sd/label,anchor=south] at (3.0944,1.0000) {$k$};
\draw[sd/string] ([xshift=0.2280mm]b0.south) -- ([xshift=1.8246mm]b1.north);
\node[sd/label,anchor=west] at (2.8246,-10.4000) {$h+k$};
\draw[sd/string] ([xshift=0.0000mm]b1.south) -- (0.0000,-21.6000);
\node[sd/label,anchor=north] at (0.0000,-22.6000) {$g+h+k$};
\end{tikzpicture}
}

\expandafter\def\csname NGD@prism-coupon-dictionary\endcsname{%
\begin{tikzpicture}[sd/base,x=1mm,y=1mm]
\node[sd/box,anchor=south,minimum width=17mm,minimum height=4.3mm] (1) at (-14.0000,-0.0000) {$\alpha_1$};
\node[sd/box,anchor=south,minimum width=17mm,minimum height=4.3mm] (4) at (14.0000,-0.0000) {$\alpha_4$};
\node[sd/box,anchor=south,minimum width=17mm,minimum height=4.3mm] (2) at (-14.0000,-18.5000) {$\alpha_2$};
\node[sd/box,anchor=south,minimum width=17mm,minimum height=4.3mm] (5) at (14.0000,-18.5000) {$\alpha_5$};
\node[sd/box,anchor=south,minimum width=17mm,minimum height=4.3mm] (3) at (-14.0000,-37.0000) {$\alpha_3$};
\node[sd/box,anchor=south,minimum width=17mm,minimum height=4.3mm] (6) at (14.0000,-37.0000) {$\alpha_6$};
\draw[sd/string] ([xshift=0.0000mm]1.south) -- (-14.0000,-3.7000) arc[start angle=180,end angle=360,x radius=14.0000mm,y radius=3.5000mm] -- ([xshift=0.0000mm]4.south);
\node[sd/label,anchor=center] at (0.0000,-4.8000) {$X_1$};
\draw[sd/string] ([xshift=0.0000mm]2.south) -- (-14.0000,-22.2000) arc[start angle=180,end angle=360,x radius=14.0000mm,y radius=3.5000mm] -- ([xshift=0.0000mm]5.south);
\node[sd/label,anchor=center] at (0.0000,-23.3000) {$X_2$};
\draw[sd/string] ([xshift=0.0000mm]3.south) -- (-14.0000,-40.7000) arc[start angle=180,end angle=360,x radius=14.0000mm,y radius=3.5000mm] -- ([xshift=0.0000mm]6.south);
\node[sd/label,anchor=center] at (0.0000,-41.8000) {$X_3$};
\draw[sd/string] ([xshift=-5.2000mm]1.south) -- (-19.2000,-10.0000) arc[start angle=180,end angle=360,x radius=6.3500mm,y radius=2.7000mm] -- (-6.5000,-9.0000) arc[start angle=180,end angle=0,x radius=1.5000mm,y radius=1.3000mm] -- (-3.5000,-18.5000) arc[start angle=0,end angle=-180,x radius=2.6500mm,y radius=1.5000mm] -- ([xshift=5.2000mm]2.south);
\draw[sd/string] ([xshift=5.2000mm]4.south) -- (19.2000,-10.0000) arc[start angle=0,end angle=-180,x radius=6.3500mm,y radius=2.7000mm] -- (6.5000,-9.0000) arc[start angle=0,end angle=180,x radius=1.5000mm,y radius=1.3000mm] -- (3.5000,-18.5000) arc[start angle=180,end angle=360,x radius=2.6500mm,y radius=1.5000mm] -- ([xshift=-5.2000mm]5.south);
\node[sd/label,anchor=west] at (-18.2500,-4.6000) {$X_4$};
\node[sd/label,anchor=east] at (18.2500,-4.6000) {$X_7$};
\draw[sd/string] ([xshift=-5.2000mm]2.south) -- (-19.2000,-28.5000) arc[start angle=180,end angle=360,x radius=6.3500mm,y radius=2.7000mm] -- (-6.5000,-27.5000) arc[start angle=180,end angle=0,x radius=1.5000mm,y radius=1.3000mm] -- (-3.5000,-37.0000) arc[start angle=0,end angle=-180,x radius=2.6500mm,y radius=1.5000mm] -- ([xshift=5.2000mm]3.south);
\draw[sd/string] ([xshift=5.2000mm]5.south) -- (19.2000,-28.5000) arc[start angle=0,end angle=-180,x radius=6.3500mm,y radius=2.7000mm] -- (6.5000,-27.5000) arc[start angle=0,end angle=180,x radius=1.5000mm,y radius=1.3000mm] -- (3.5000,-37.0000) arc[start angle=180,end angle=360,x radius=2.6500mm,y radius=1.5000mm] -- ([xshift=-5.2000mm]6.south);
\node[sd/label,anchor=west] at (-18.2500,-23.1000) {$X_5$};
\node[sd/label,anchor=east] at (18.2500,-23.1000) {$X_8$};
\draw[sd/string] ([xshift=5.2000mm]1.south) arc[start angle=180,end angle=360,x radius=2.0000mm,y radius=1.3000mm] -- (-4.8000,5.0000) arc[start angle=0,end angle=180,x radius=9.8500mm,y radius=2.4000mm] -- (-24.5000,-37.0000) arc[start angle=180,end angle=360,x radius=2.6500mm,y radius=1.5000mm] -- ([xshift=-5.2000mm]3.south);
\draw[sd/string] ([xshift=-5.2000mm]4.south) arc[start angle=0,end angle=-180,x radius=2.0000mm,y radius=1.3000mm] -- (4.8000,5.0000) arc[start angle=180,end angle=0,x radius=9.8500mm,y radius=2.4000mm] -- (24.5000,-37.0000) arc[start angle=0,end angle=-180,x radius=2.6500mm,y radius=1.5000mm] -- ([xshift=5.2000mm]6.south);
\node[sd/label,anchor=east] at (-25.5000,-19.0000) {$X_6$};
\node[sd/label,anchor=west] at (25.5000,-19.0000) {$X_9$};
\end{tikzpicture}
}

\expandafter\def\csname NGD@resolution-term\endcsname{%
\begin{tikzpicture}[sd/base,x=1mm,y=1mm]
\node[sd/box,minimum width=7.7319mm,minimum height=4.1351mm] (b0) at (0.0000,-6.2676) {$b_{X,i}$};
\node[sd/box,minimum width=7.7319mm,minimum height=5.5745mm] (b1) at (0.0000,-15.5224) {$b_{X,i}^{'}$};
\draw[sd/string] (-2.7113,0.0000) -- ([xshift=-2.7113mm]b0.north);
\node[sd/label,anchor=south] at (-2.7113,1.0000) {$\varrho$};
\draw[sd/string] (1.2887,0.0000) -- ([xshift=1.2887mm]b0.north);
\node[sd/label,anchor=south] at (1.2887,1.0000) {$\varrho$};
\draw[sd/string] ([xshift=0.0000mm]b0.south) -- ([xshift=0.0000mm]b1.north);
\node[sd/label,anchor=west] at (1.0000,-10.5351) {$X$};
\draw[sd/string] ([xshift=-2.7113mm]b1.south) -- (-2.7113,-23.3096);
\node[sd/label,anchor=north] at (-2.7113,-24.3096) {$\varrho$};
\draw[sd/string] ([xshift=1.2887mm]b1.south) -- (1.2887,-23.3096);
\node[sd/label,anchor=north] at (1.2887,-24.3096) {$\varrho$};
\end{tikzpicture}
}

\expandafter\def\csname NGD@rho-cup\endcsname{%
\begin{tikzpicture}[sd/base,x=1mm,y=1mm]
\draw[sd/string] (-2.0000,-1.2400) -- (-2.0000,-6.2400);
\node[sd/label,anchor=north] at (-2.0000,-7.2400) {$\varrho$};
\draw[sd/string] (2.0000,-1.2400) -- (2.0000,-6.2400);
\node[sd/label,anchor=north] at (2.0000,-7.2400) {$\varrho$};
\draw[sd/string] (-2.0000,-1.2400) arc[start angle=180,end angle=0,x radius=2.0000mm,y radius=1.2400mm];
\end{tikzpicture}
}

\expandafter\def\csname NGD@rho-loop\endcsname{%
\begin{tikzpicture}[sd/base,x=1mm,y=1mm]
\draw[sd/string] (-2.0000,-1.2400) -- (-2.0000,-5.6400);
\draw[sd/string] (2.0000,-1.2400) -- (2.0000,-5.6400);
\draw[sd/string] (-2.0000,-1.2400) arc[start angle=180,end angle=0,x radius=2.0000mm,y radius=1.2400mm];
\draw[sd/string] (-2.0000,-5.6400) arc[start angle=180,end angle=360,x radius=2.0000mm,y radius=1.2400mm];
\end{tikzpicture}
}

\expandafter\def\csname NGD@rho-root-u\endcsname{%
\begin{tikzpicture}[sd/base,x=1mm,y=1mm]
\node[sd/box,minimum width=10.9351mm,minimum height=4.0000mm] (b0) at (-3.0825,-6.2000) {$u$};
\draw[sd/string] (4.5500,0.0000) -- (4.5500,-12.6000);
\node[sd/label,anchor=south] at (4.5500,1.0000) {$\varrho$};
\draw[sd/string] (8.5500,0.0000) -- (8.5500,-12.6000);
\node[sd/label,anchor=south] at (8.5500,1.0000) {$\varrho$};
\draw[sd/string] ([xshift=-4.3675mm]b0.south) -- (-7.4500,-21.3200);
\node[sd/label,anchor=north] at (-7.4500,-22.3200) {$\varrho$};
\draw[sd/string] ([xshift=-0.3675mm]b0.south) -- (-3.4500,-12.6000);
\draw[sd/string] ([xshift=3.6325mm]b0.south) -- (0.5500,-12.6000);
\draw[sd/string] (0.5500,-12.6000) arc[start angle=180,end angle=360,x radius=2.0000mm,y radius=1.2400mm];
\draw[sd/string] (-3.4500,-12.6000) arc[start angle=180,end angle=360,x radius=6.0000mm,y radius=3.7200mm];
\end{tikzpicture}
}

\expandafter\def\csname NGD@rho-split-u\endcsname{%
\begin{tikzpicture}[sd/base,x=1mm,y=1mm]
\node[sd/box,minimum width=10.9351mm,minimum height=4.0000mm] (b0) at (-1.0825,-6.2000) {$u$};
\draw[sd/string] (6.5500,0.0000) -- (6.5500,-12.6000);
\node[sd/label,anchor=south] at (6.5500,1.0000) {$\varrho$};
\draw[sd/string] ([xshift=-4.3675mm]b0.south) -- (-5.4500,-18.8400);
\node[sd/label,anchor=north] at (-5.4500,-19.8400) {$\varrho$};
\draw[sd/string] ([xshift=-0.3675mm]b0.south) -- (-1.4500,-18.8400);
\node[sd/label,anchor=north] at (-1.4500,-19.8400) {$\varrho$};
\draw[sd/string] ([xshift=3.6325mm]b0.south) -- (2.5500,-12.6000);
\draw[sd/string] (2.5500,-12.6000) arc[start angle=180,end angle=360,x radius=2.0000mm,y radius=1.2400mm];
\end{tikzpicture}
}

\expandafter\def\csname NGD@rho-v\endcsname{%
\begin{tikzpicture}[sd/base,x=1mm,y=1mm]
\node[sd/box,minimum width=6.5000mm,minimum height=4.0000mm] (b0) at (0.0000,-6.2000) {$v_h$};
\draw[sd/string] (-2.1500,0.0000) -- ([xshift=-2.1500mm]b0.north);
\node[sd/label,anchor=south] at (-2.1500,1.0000) {$\varrho$};
\draw[sd/string] (1.8500,0.0000) -- ([xshift=1.8500mm]b0.north);
\node[sd/label,anchor=south] at (1.8500,1.0000) {$\varrho$};
\draw[sd/string] ([xshift=0.0000mm]b0.south) -- (0.0000,-13.2000);
\node[sd/label,anchor=north] at (0.0000,-14.2000) {$\varrho$};
\end{tikzpicture}
}

\expandafter\def\csname NGD@right-assoc\endcsname{%
\begin{tikzpicture}[sd/base,x=1mm,y=1mm]
\node[sd/box,minimum width=6.8392mm,minimum height=4.0000mm] (b0) at (-1.6804,-6.2000) {$R_{g}$};
\node[sd/box,minimum width=7.8804mm,minimum height=4.0000mm] (b1) at (1.1598,-14.6000) {$R_{h}$};
\draw[sd/string] (-4.0000,0.0000) -- ([xshift=-2.3196mm]b0.north);
\node[sd/label,anchor=south] at (-4.0000,1.0000) {$\varrho$};
\draw[sd/string] (0.0000,0.0000) -- ([xshift=1.6804mm]b0.north);
\node[sd/label,anchor=south] at (0.0000,1.0000) {$g$};
\draw[sd/string] (4.0000,0.0000) -- ([xshift=2.8402mm]b1.north);
\node[sd/label,anchor=south] at (4.0000,1.0000) {$h$};
\draw[sd/string] ([xshift=0.7634mm]b0.south) -- ([xshift=-2.0768mm]b1.north);
\draw[sd/string] ([xshift=-0.0000mm]b1.south) -- (1.1598,-21.6000);
\node[sd/label,anchor=north] at (1.1598,-22.6000) {$\varrho$};
\end{tikzpicture}
}

\expandafter\def\csname NGD@right-channel\endcsname{%
\begin{tikzpicture}[sd/base,x=1mm,y=1mm]
\draw[sd/string] (19.0000,10.9697) -- (38.0000,0.0000);
\draw[sd/reverse] (19.0000,10.9697) -- (19.0000,32.9090);
\draw[sd/string] (19.0000,10.9697) -- (0.0000,0.0000);
\draw[sd/reverse] (38.0000,0.0000) -- (0.0000,0.0000);
\draw[sd/string] (19.0000,32.9090) -- (0.0000,0.0000);
\draw[sd/string] (38.0000,0.0000) -- (19.0000,32.9090);
\node[sd/label] at (19.0000,8.6697) {$v$};
\node[sd/label] at (5.5154,5.5154) {$p$};
\node[sd/label] at (19.0000,35.2090) {$u$};
\node[sd/label] at (32.4846,5.5154) {$q$};
\node[sd/label] at (19.7600,-1.9647) {$t$};
\node[sd/label] at (16.5618,21.5005) {$X$};
\end{tikzpicture}
}

\expandafter\def\csname NGD@right-fused\endcsname{%
\begin{tikzpicture}[sd/base,x=1mm,y=1mm]
\node[sd/box,minimum width=8.9063mm,minimum height=4.0000mm] (b0) at (1.6325,-6.2000) {$m_{g,h}$};
\node[sd/box,minimum width=9.6853mm,minimum height=4.1273mm] (b1) at (-1.2430,-14.6637) {$R_{g+h}$};
\draw[sd/string] (-4.9856,0.0000) -- ([xshift=-3.7427mm]b1.north);
\node[sd/label,anchor=south] at (-4.9856,1.0000) {$\varrho$};
\draw[sd/string] (-0.8832,0.0000) -- ([xshift=-2.5156mm]b0.north);
\node[sd/label,anchor=south] at (-0.8832,1.0000) {$g$};
\draw[sd/string] (3.1168,0.0000) -- ([xshift=1.4844mm]b0.north);
\node[sd/label,anchor=south] at (3.1168,1.0000) {$h$};
\draw[sd/string] ([xshift=-0.6306mm]b0.south) -- ([xshift=2.2448mm]b1.north);
\node[sd/label,anchor=west] at (2.0019,-10.4000) {$g+h$};
\draw[sd/string] ([xshift=0.0000mm]b1.south) -- (-1.2430,-21.7273);
\node[sd/label,anchor=north] at (-1.2430,-22.7273) {$\varrho$};
\end{tikzpicture}
}

\expandafter\def\csname NGD@root-f\endcsname{%
\begin{tikzpicture}[sd/base,x=1mm,y=1mm]
\node[sd/box,minimum width=12.9300mm,minimum height=4.0000mm] (b0) at (-5.5310,-6.2000) {$u$};
\draw[sd/string] (7.4601,0.0000) -- (7.4601,-12.6000);
\node[sd/label,anchor=south] at (7.4601,1.0000) {$X$};
\draw[sd/string] (11.9960,0.0000) -- (11.9960,-12.6000);
\node[sd/label,anchor=south] at (11.9960,1.0000) {$Y$};
\draw[sd/string] ([xshift=-5.3650mm]b0.south) -- (-10.8960,-23.4565);
\node[sd/label,anchor=north] at (-10.8960,-24.4565) {$Z$};
\draw[sd/string] ([xshift=-1.3650mm]b0.south) -- (-6.8960,-12.6000);
\node[sd/label,anchor=west] at (-5.8960,-10.4000) {$Y^*$};
\draw[sd/string] ([xshift=4.8533mm]b0.south) -- (-0.6777,-12.6000);
\node[sd/label,anchor=west] at (0.3223,-10.4000) {$X^*$};
\draw[sd/string] (-0.6777,-12.6000) arc[start angle=180,end angle=360,x radius=4.0689mm,y radius=2.5227mm];
\draw[sd/string] (-6.8960,-12.6000) arc[start angle=180,end angle=360,x radius=9.4460mm,y radius=5.8565mm];
\end{tikzpicture}
}

\expandafter\def\csname NGD@rotated-basis\endcsname{%
\begin{tikzpicture}[sd/base,x=1mm,y=1mm]
\node[sd/box,minimum width=10.4000mm,minimum height=4.0000mm] (b0) at (0.0000,-2.0000) {$\varrho u_g$};
\draw[sd/string] ([xshift=-4.1000mm]b0.south) -- (-4.1000,-9.0000);
\node[sd/label,anchor=north] at (-4.1000,-10.0000) {$\varrho$};
\draw[sd/string] ([xshift=-0.1000mm]b0.south) -- (-0.1000,-9.0000);
\node[sd/label,anchor=north] at (-0.1000,-10.0000) {$\varrho$};
\draw[sd/string] ([xshift=3.9000mm]b0.south) -- (3.9000,-9.0000);
\node[sd/label,anchor=north] at (3.9000,-10.0000) {$\varrho$};
\end{tikzpicture}
}

\expandafter\def\csname NGD@slide-Linv-R\endcsname{%
\begin{tikzpicture}[sd/base,x=1mm,y=1mm]
\node[sd/box,minimum width=6.9906mm,minimum height=4.0000mm] (b0) at (-0.4680,-6.2000) {$R_{h}$};
\node[sd/box,minimum width=7.9267mm,minimum height=4.8321mm] (b1) at (0.0000,-15.0161) {$L_{g}^{-1}$};
\draw[sd/string] (-2.8634,0.0000) -- ([xshift=-2.3953mm]b0.north);
\node[sd/label,anchor=south] at (-2.8634,1.0000) {$\varrho$};
\draw[sd/string] (1.1366,0.0000) -- ([xshift=1.6047mm]b0.north);
\node[sd/label,anchor=south] at (1.1366,1.0000) {$h$};
\draw[sd/string] ([xshift=0.2340mm]b0.south) -- ([xshift=-0.2340mm]b1.north);
\draw[sd/string] ([xshift=-2.6789mm]b1.south) -- (-2.6789,-22.4321);
\node[sd/label,anchor=north] at (-2.6789,-23.4321) {$g$};
\draw[sd/string] ([xshift=1.3211mm]b1.south) -- (1.3211,-22.4321);
\node[sd/label,anchor=north] at (1.3211,-23.4321) {$\varrho$};
\end{tikzpicture}
}

\expandafter\def\csname NGD@slide-R-Linv\endcsname{%
\begin{tikzpicture}[sd/base,x=1mm,y=1mm]
\node[sd/box,minimum width=7.9267mm,minimum height=4.8321mm] (b0) at (-1.6804,-6.6161) {$L_{g}^{-1}$};
\node[sd/box,minimum width=7.1030mm,minimum height=4.0000mm] (b1) at (2.0922,-15.4321) {$R_{h}$};
\draw[sd/string] (-1.6804,0.0000) -- ([xshift=0.0000mm]b0.north);
\node[sd/label,anchor=south] at (-1.6804,1.0000) {$\varrho$};
\draw[sd/string] (4.5437,0.0000) -- ([xshift=2.4515mm]b1.north);
\node[sd/label,anchor=south] at (4.5437,1.0000) {$h$};
\draw[sd/string] ([xshift=-2.0450mm]b0.south) -- (-3.7254,-22.4321);
\node[sd/label,anchor=north] at (-3.7254,-23.4321) {$g$};
\draw[sd/string] ([xshift=1.9550mm]b0.south) -- ([xshift=-1.8177mm]b1.north);
\draw[sd/string] ([xshift=0.0000mm]b1.south) -- (2.0922,-22.4321);
\node[sd/label,anchor=north] at (2.0922,-23.4321) {$\varrho$};
\end{tikzpicture}
}

\expandafter\def\csname NGD@slide-two-inverses-a\endcsname{%
\begin{tikzpicture}[sd/base,x=1mm,y=1mm]
\node[sd/box,minimum width=8.5546mm,minimum height=4.8321mm] (b0) at (-1.5467,-6.6161) {$L_{g}^{-1}$};
\node[sd/box,minimum width=8.3867mm,minimum height=5.2715mm] (b1) at (1.6306,-16.0679) {$R_{-h}^{-1}$};
\draw[sd/string] (-1.5467,0.0000) -- ([xshift=0.0000mm]b0.north);
\node[sd/label,anchor=south] at (-1.5467,1.0000) {$\varrho$};
\draw[sd/string] ([xshift=-3.1773mm]b0.south) -- (-4.7240,-23.7036);
\node[sd/label,anchor=north] at (-4.7240,-24.7036) {$g$};
\draw[sd/string] ([xshift=2.3015mm]b0.south) -- ([xshift=-0.8758mm]b1.north);
\draw[sd/string] ([xshift=-2.3546mm]b1.south) -- (-0.7240,-23.7036);
\node[sd/label,anchor=north] at (-0.7240,-24.7036) {$\varrho$};
\draw[sd/string] ([xshift=2.4329mm]b1.south) -- (4.0635,-23.7036);
\node[sd/label,anchor=north] at (4.0635,-24.7036) {$-h$};
\end{tikzpicture}
}

\expandafter\def\csname NGD@slide-two-inverses-b\endcsname{%
\begin{tikzpicture}[sd/base,x=1mm,y=1mm]
\node[sd/box,minimum width=9.5859mm,minimum height=5.2715mm] (b0) at (1.4317,-6.8357) {$R_{-h}^{-1}$};
\node[sd/box,minimum width=7.9267mm,minimum height=4.8321mm] (b1) at (-2.2613,-16.2875) {$L_{g}^{-1}$};
\draw[sd/string] (1.4317,0.0000) -- ([xshift=-0.0000mm]b0.north);
\node[sd/label,anchor=south] at (1.4317,1.0000) {$\varrho$};
\draw[sd/string] ([xshift=-2.6453mm]b0.south) -- ([xshift=1.0477mm]b1.north);
\draw[sd/string] ([xshift=3.6930mm]b0.south) -- (5.1246,-23.7036);
\node[sd/label,anchor=north] at (5.1246,-24.7036) {$-h$};
\draw[sd/string] ([xshift=-2.6789mm]b1.south) -- (-4.9402,-23.7036);
\node[sd/label,anchor=north] at (-4.9402,-24.7036) {$g$};
\draw[sd/string] ([xshift=1.3211mm]b1.south) -- (-0.9402,-23.7036);
\node[sd/label,anchor=north] at (-0.9402,-24.7036) {$\varrho$};
\end{tikzpicture}
}

\expandafter\def\csname NGD@spherical-prism-reference\endcsname{%
\begin{tikzpicture}[sd/base,x=1mm,y=1mm]
\draw[sd/arrow] (0.0000,25.2000) -- (42.0000,25.2000);
\draw[sd/arrow] (12.1800,12.6000) -- (29.8200,12.6000);
\draw[sd/arrow] (0.0000,0.0000) -- (42.0000,0.0000);
\draw[sd/arrow] (0.0000,25.2000) -- (12.1800,12.6000);
\draw[sd/arrow] (12.1800,12.6000) -- (0.0000,0.0000);
\draw[sd/arrow] (0.0000,0.0000) -- (0.0000,25.2000);
\draw[sd/arrow] (42.0000,25.2000) -- (29.8200,12.6000);
\draw[sd/arrow] (29.8200,12.6000) -- (42.0000,0.0000);
\draw[sd/arrow] (42.0000,0.0000) -- (42.0000,25.2000);
\node[sd/label] at (2.7363,18.4322) {$\alpha_1$};
\node[sd/label] at (8.3800,12.6000) {$\alpha_2$};
\node[sd/label] at (2.7363,6.7678) {$\alpha_3$};
\node[sd/label] at (39.2637,18.4322) {$\alpha_4$};
\node[sd/label] at (33.6200,12.6000) {$\alpha_5$};
\node[sd/label] at (39.2637,6.7678) {$\alpha_6$};
\node[sd/label] at (-3.0607,12.0960) {$X_6$};
\node[sd/label] at (45.0607,12.0960) {$X_9$};
\node[sd/label] at (20.1600,27.4559) {$X_1$};
\node[sd/label] at (20.1600,-2.2559) {$X_3$};
\node[sd/label] at (8.2285,21.4547) {$X_4$};
\node[sd/label] at (8.7157,4.2493) {$X_5$};
\node[sd/label] at (33.7715,21.4547) {$X_7$};
\node[sd/label] at (33.2843,4.2493) {$X_8$};
\node[sd/label] at (20.6472,10.3441) {$X_2$};
\end{tikzpicture}
}

\expandafter\def\csname NGD@spherical-tetrahedron-reference\endcsname{%
\begin{tikzpicture}[sd/base,x=1mm,y=1mm]
\draw[sd/arrow] (19.0000,10.9697) -- (38.0000,0.0000);
\draw[sd/arrow] (19.0000,10.9697) -- (19.0000,32.9090);
\draw[sd/arrow] (19.0000,10.9697) -- (0.0000,0.0000);
\draw[sd/arrow] (38.0000,0.0000) -- (0.0000,0.0000);
\draw[sd/arrow] (19.0000,32.9090) -- (0.0000,0.0000);
\draw[sd/arrow] (38.0000,0.0000) -- (19.0000,32.9090);
\node[sd/label] at (19.0000,8.6697) {$\alpha$};
\node[sd/label] at (-1.9919,-1.1500) {$\beta$};
\node[sd/label] at (19.0000,35.2090) {$\gamma$};
\node[sd/label] at (39.9919,-1.1500) {$\delta$};
\node[sd/label] at (19.7600,-2.2559) {$X_4$};
\node[sd/label] at (6.9563,18.8007) {$X_5$};
\node[sd/label] at (31.8037,17.4843) {$X_6$};
\node[sd/label] at (15.9393,18.8678) {$X_2$};
\node[sd/label] at (11.4207,3.0356) {$X_3$};
\node[sd/label] at (26.5793,3.0356) {$X_1$};
\end{tikzpicture}
}

\expandafter\def\csname NGD@tetra-translation\endcsname{%
\begin{tikzpicture}[sd/base,x=1mm,y=1mm]
\draw[sd/string] (16.5000,9.5263) -- (33.0000,0.0000);
\draw[sd/string] (16.5000,9.5263) -- (16.5000,28.5788);
\draw[sd/string] (16.5000,9.5263) -- (0.0000,0.0000);
\draw[sd/string] (33.0000,0.0000) -- (0.0000,0.0000);
\draw[sd/string] (16.5000,28.5788) -- (0.0000,0.0000);
\draw[sd/string] (33.0000,0.0000) -- (16.5000,28.5788);
\node[sd/label] at (13.2091,11.4263) {$\widetilde u_p$};
\node[sd/label] at (8.5206,8.5206) {$\widetilde u_g$};
\node[sd/label] at (13.8989,18.8713) {$h'$};
\node[sd/label] at (24.8329,8.1671) {$\widetilde u_r$};
\end{tikzpicture}
}

\expandafter\def\csname NGD@tetrahedron-coupon-dictionary\endcsname{%
\begin{tikzpicture}[sd/base,x=1mm,y=1mm]
\node[sd/box,anchor=south,minimum width=17mm,minimum height=4.3mm] (b) at (-22.0000,0.0000) {$\beta$};
\node[sd/box,anchor=south,minimum width=17mm,minimum height=4.3mm] (c) at (0.0000,0.0000) {$\gamma$};
\node[sd/box,anchor=south,minimum width=17mm,minimum height=4.3mm] (d) at (22.0000,0.0000) {$\delta$};
\node[sd/box,anchor=south,minimum width=17mm,minimum height=4.3mm] (a) at (0.0000,-20.0000) {$\alpha$};
\draw[sd/string] ([xshift=5.2000mm]b.south) arc[start angle=180,end angle=360,x radius=5.8000mm,y radius=2.4000mm] -- ([xshift=-5.2000mm]c.south);
\draw[sd/string] ([xshift=5.2000mm]c.south) arc[start angle=180,end angle=360,x radius=5.8000mm,y radius=2.4000mm] -- ([xshift=-5.2000mm]d.south);
\draw[sd/string] ([xshift=-5.2000mm]b.south) -- (-27.2000,-25.8000) arc[start angle=180,end angle=360,x radius=27.2000mm,y radius=6.0000mm] -- ([xshift=5.2000mm]d.south);
\draw[sd/string] ([xshift=0.0000mm]b.south) -- (-22.0000,-24.0000) arc[start angle=180,end angle=360,x radius=13.6000mm,y radius=4.2000mm] -- ([xshift=5.2000mm]a.south);
\draw[sd/string] ([xshift=0.0000mm]c.south) -- (0.0000,-5.8000) arc[start angle=0,end angle=-180,x radius=5.0000mm,y radius=2.2000mm] -- (-10.0000,-5.1000) arc[start angle=0,end angle=180,x radius=3.0000mm,y radius=1.7000mm] -- (-16.0000,-22.5000) arc[start angle=180,end angle=360,x radius=8.0000mm,y radius=3.2000mm] -- ([xshift=0.0000mm]a.south);
\draw[sd/string] ([xshift=0.0000mm]d.south) -- (22.0000,-21.7000) arc[start angle=0,end angle=-180,x radius=5.5000mm,y radius=2.6000mm] -- (11.0000,-14.7000) arc[start angle=0,end angle=180,x radius=11.0000mm,y radius=4.2000mm] -- (-11.0000,-20.0000) arc[start angle=180,end angle=360,x radius=2.9000mm,y radius=1.3000mm] -- ([xshift=-5.2000mm]a.south);
\node[sd/label,anchor=center] at (-11.0000,-0.5500) {$X_5$};
\node[sd/label,anchor=center] at (11.0000,-0.5500) {$X_6$};
\node[sd/label,anchor=west] at (28.3000,-9.0000) {$X_4$};
\node[sd/label,anchor=west] at (-4.3500,-21.1000) {$X_1$};
\node[sd/label,anchor=west] at (0.8500,-21.1000) {$X_2$};
\node[sd/label,anchor=west] at (6.0500,-21.1000) {$X_3$};
\end{tikzpicture}
}

\expandafter\def\csname NGD@theta-dual\endcsname{%
\begin{tikzpicture}[sd/base,x=1mm,y=1mm]
\node[sd/box,minimum width=12.0000mm,minimum height=4.4983mm] (l) at (-8.5000,0.0000) {$u_h$};
\node[sd/box,minimum width=12.0000mm,minimum height=4.4983mm] (r) at (8.5000,0.0000) {$u_k'$};
\draw[sd/string] ([xshift=4.0000mm]l.south) arc[start angle=180,end angle=360,x radius=4.5000mm,y radius=2.7900mm];
\draw[sd/string] ([xshift=0.0000mm]l.south) arc[start angle=180,end angle=360,x radius=8.5000mm,y radius=5.2700mm];
\draw[sd/string] ([xshift=-4.0000mm]l.south) arc[start angle=180,end angle=360,x radius=12.5000mm,y radius=7.7500mm];
\end{tikzpicture}
}

\expandafter\def\csname NGD@theta-h-minus\endcsname{%
\begin{tikzpicture}[sd/base,x=1mm,y=1mm]
\node[sd/box,minimum width=12.0000mm,minimum height=4.0000mm] (l) at (-8.5000,0.0000) {$u_h$};
\node[sd/box,minimum width=12.0000mm,minimum height=4.0000mm] (r) at (8.5000,0.0000) {$u_{-h}$};
\draw[sd/string] ([xshift=4.0000mm]l.south) arc[start angle=180,end angle=360,x radius=4.5000mm,y radius=2.7900mm];
\draw[sd/string] ([xshift=0.0000mm]l.south) arc[start angle=180,end angle=360,x radius=8.5000mm,y radius=5.2700mm];
\draw[sd/string] ([xshift=-4.0000mm]l.south) arc[start angle=180,end angle=360,x radius=12.5000mm,y radius=7.7500mm];
\end{tikzpicture}
}

\expandafter\def\csname NGD@theta-hk\endcsname{%
\begin{tikzpicture}[sd/base,x=1mm,y=1mm]
\node[sd/box,minimum width=12.0000mm,minimum height=4.0000mm] (l) at (-8.5000,0.0000) {$u_h$};
\node[sd/box,minimum width=12.0000mm,minimum height=4.0000mm] (r) at (8.5000,0.0000) {$u_k$};
\draw[sd/string] ([xshift=4.0000mm]l.south) arc[start angle=180,end angle=360,x radius=4.5000mm,y radius=2.7900mm];
\draw[sd/string] ([xshift=0.0000mm]l.south) arc[start angle=180,end angle=360,x radius=8.5000mm,y radius=5.2700mm];
\draw[sd/string] ([xshift=-4.0000mm]l.south) arc[start angle=180,end angle=360,x radius=12.5000mm,y radius=7.7500mm];
\end{tikzpicture}
}

\expandafter\def\csname NGD@theta-hk-loop\endcsname{%
\begin{tikzpicture}[sd/base,x=1mm,y=1mm]
\node[sd/box,minimum width=12.0000mm,minimum height=4.0000mm] (l) at (-8.5000,0.0000) {$u_h$};
\node[sd/box,minimum width=12.0000mm,minimum height=4.0000mm] (r) at (8.5000,0.0000) {$u_k$};
\draw[sd/string] ([xshift=4.0000mm]l.south) arc[start angle=180,end angle=360,x radius=4.5000mm,y radius=2.7900mm];
\draw[sd/string] ([xshift=0.0000mm]l.south) arc[start angle=180,end angle=360,x radius=8.5000mm,y radius=5.2700mm];
\draw[sd/string] ([xshift=-4.0000mm]l.south) arc[start angle=180,end angle=360,x radius=12.5000mm,y radius=7.7500mm];
\draw[sd/string] (-2.0000,-13.9500) arc[start angle=180,end angle=0,x radius=2mm,y radius=1.24mm] -- (2.0000,-15.9500) arc[start angle=0,end angle=-180,x radius=2mm,y radius=1.24mm] -- cycle;
\node[sd/label,anchor=west] at (3.2000,-14.9500) {$l$};
\end{tikzpicture}
}

\expandafter\def\csname NGD@theta-translated\endcsname{%
\begin{tikzpicture}[sd/base,x=1mm,y=1mm]
\node[sd/box,minimum width=12.0000mm,minimum height=4.0000mm] (l) at (-8.5000,0.0000) {$u_{h-g}$};
\node[sd/box,minimum width=12.0000mm,minimum height=4.0000mm] (r) at (8.5000,0.0000) {$u_{-h+g}$};
\draw[sd/string] ([xshift=4.0000mm]l.south) arc[start angle=180,end angle=360,x radius=4.5000mm,y radius=2.7900mm];
\draw[sd/string] ([xshift=0.0000mm]l.south) arc[start angle=180,end angle=360,x radius=8.5000mm,y radius=5.2700mm];
\draw[sd/string] ([xshift=-4.0000mm]l.south) arc[start angle=180,end angle=360,x radius=12.5000mm,y radius=7.7500mm];
\end{tikzpicture}
}

\expandafter\def\csname NGD@theta-uv\endcsname{%
\begin{tikzpicture}[sd/base,x=1mm,y=1mm]
\node[sd/box,minimum width=12.0000mm,minimum height=4.0000mm] (l) at (-8.5000,0.0000) {$u$};
\node[sd/box,minimum width=12.0000mm,minimum height=4.0000mm] (r) at (8.5000,0.0000) {$v$};
\draw[sd/string] ([xshift=4.0000mm]l.south) arc[start angle=180,end angle=360,x radius=4.5000mm,y radius=2.7900mm];
\draw[sd/string] ([xshift=0.0000mm]l.south) arc[start angle=180,end angle=360,x radius=8.5000mm,y radius=5.2700mm];
\draw[sd/string] ([xshift=-4.0000mm]l.south) arc[start angle=180,end angle=360,x radius=12.5000mm,y radius=7.7500mm];
\end{tikzpicture}
}

\expandafter\def\csname NGD@theta-xy-inserted\endcsname{%
\begin{tikzpicture}[sd/base,x=1mm,y=1mm]
\node[sd/box,minimum width=14.8754mm,minimum height=4.1976mm] (l) at (-9.9377,0.0000) {$\pi_y(g)u_h$};
\node[sd/box,minimum width=14.8754mm,minimum height=4.1976mm] (r) at (9.9377,0.0000) {$\pi_x(g)u_{-h}$};
\draw[sd/string] ([xshift=4.0000mm]l.south) arc[start angle=180,end angle=360,x radius=5.9377mm,y radius=3.6814mm];
\draw[sd/string] ([xshift=0.0000mm]l.south) arc[start angle=180,end angle=360,x radius=9.9377mm,y radius=6.1614mm];
\draw[sd/string] ([xshift=-4.0000mm]l.south) arc[start angle=180,end angle=360,x radius=13.9377mm,y radius=8.6414mm];
\end{tikzpicture}
}

\expandafter\def\csname NGD@theta-z-inserted\endcsname{%
\begin{tikzpicture}[sd/base,x=1mm,y=1mm]
\node[sd/box,minimum width=12.1155mm,minimum height=4.1117mm] (l) at (-8.5578,0.0000) {$\pi_z(l)u_h$};
\node[sd/box,minimum width=12.1155mm,minimum height=4.1117mm] (r) at (8.5578,0.0000) {$\pi_z(l)u_k$};
\draw[sd/string] ([xshift=4.0000mm]l.south) arc[start angle=180,end angle=360,x radius=4.5578mm,y radius=2.8258mm];
\draw[sd/string] ([xshift=0.0000mm]l.south) arc[start angle=180,end angle=360,x radius=8.5578mm,y radius=5.3058mm];
\draw[sd/string] ([xshift=-4.0000mm]l.south) arc[start angle=180,end angle=360,x radius=12.5578mm,y radius=7.7858mm];
\end{tikzpicture}
}

\expandafter\def\csname NGD@theta-zero\endcsname{%
\begin{tikzpicture}[sd/base,x=1mm,y=1mm]
\node[sd/box,minimum width=12.0000mm,minimum height=4.0000mm] (l) at (-8.5000,0.0000) {$u_0$};
\node[sd/box,minimum width=12.0000mm,minimum height=4.0000mm] (r) at (8.5000,0.0000) {$u_0$};
\draw[sd/string] ([xshift=4.0000mm]l.south) arc[start angle=180,end angle=360,x radius=4.5000mm,y radius=2.7900mm];
\draw[sd/string] ([xshift=0.0000mm]l.south) arc[start angle=180,end angle=360,x radius=8.5000mm,y radius=5.2700mm];
\draw[sd/string] ([xshift=-4.0000mm]l.south) arc[start angle=180,end angle=360,x radius=12.5000mm,y radius=7.7500mm];
\end{tikzpicture}
}

\expandafter\def\csname NGD@trace-dual-bubble\endcsname{%
\begin{tikzpicture}[sd/base,x=1mm,y=1mm]
\node[sd/box,minimum width=7.9933mm,minimum height=5.5745mm] (b0) at (0.0000,-6.9873) {$b_{X,j}^{'}$};
\node[sd/box,minimum width=7.7319mm,minimum height=4.1351mm] (b1) at (-0.1307,-16.2421) {$b_{X,i}$};
\draw[sd/string] (0.0000,0.0000) -- ([xshift=0.0000mm]b0.north);
\node[sd/label,anchor=south] at (0.0000,1.0000) {$X$};
\draw[sd/string] ([xshift=-2.7549mm]b0.south) -- ([xshift=-2.6242mm]b1.north);
\draw[sd/string] ([xshift=1.2451mm]b0.south) -- ([xshift=1.3758mm]b1.north);
\draw[sd/string] ([xshift=0.0000mm]b1.south) -- (-0.1307,-23.3096);
\node[sd/label,anchor=north] at (-0.1307,-24.3096) {$X$};
\end{tikzpicture}
}

\expandafter\def\csname NGD@triangle-charge\endcsname{%
\begin{tikzpicture}[sd/base,x=1mm,y=1mm]
\node[sd/box,minimum width=14.9638mm,minimum height=4.4128mm] (b0) at (1.0112,-6.4064) {$\mathsf S(\varrho^2(u_g))$};
\node[sd/box,minimum width=13.5797mm,minimum height=4.1117mm] (b1) at (-1.7032,-15.0686) {$\mathsf S(\varrho(u_k))$};
\node[sd/box,minimum width=13.7212mm,minimum height=4.1117mm] (b2) at (1.6325,-23.5803) {$\mathsf F(\varrho(u_h))$};
\draw[sd/string] (1.0112,0.0000) -- ([xshift=0.0000mm]b0.north);
\node[sd/label,anchor=south] at (1.0112,1.0000) {$\varrho$};
\draw[sd/string] ([xshift=-2.6042mm]b0.south) -- ([xshift=0.1102mm]b1.north);
\draw[sd/string] ([xshift=5.4755mm]b0.south) -- ([xshift=4.8542mm]b2.north);
\draw[sd/string] ([xshift=-5.6899mm]b1.south) -- (-7.3931,-30.6361);
\node[sd/label,anchor=north] at (-7.3931,-31.6361) {$\varrho$};
\draw[sd/string] ([xshift=1.6560mm]b1.south) -- ([xshift=-1.6796mm]b2.north);
\draw[sd/string] ([xshift=0.0000mm]b2.south) -- (1.6325,-30.6361);
\node[sd/label,anchor=north] at (1.6325,-31.6361) {$\varrho$};
\end{tikzpicture}
}

\expandafter\def\csname NGD@triangle-inserted\endcsname{%
\begin{tikzpicture}[sd/base,x=1mm,y=1mm]
\node[sd/box,minimum width=21.4436mm,minimum height=4.4128mm] (b0) at (1.0112,-6.4064) {$\mathsf S(\varrho^2(\pi_z(l)u_g))$};
\node[sd/box,minimum width=20.0595mm,minimum height=4.1117mm] (b1) at (-1.7032,-15.0686) {$\mathsf S(\varrho(\pi_z(l)u_k))$};
\node[sd/box,minimum width=20.2010mm,minimum height=4.1117mm] (b2) at (1.6325,-23.5803) {$\mathsf F(\varrho(\pi_z(l)u_h))$};
\draw[sd/string] (1.0112,0.0000) -- ([xshift=0.0000mm]b0.north);
\node[sd/label,anchor=south] at (1.0112,1.0000) {$\varrho$};
\draw[sd/string] ([xshift=-3.1441mm]b0.south) -- ([xshift=-0.4298mm]b1.north);
\draw[sd/string] ([xshift=8.7154mm]b0.south) -- ([xshift=8.0941mm]b2.north);
\draw[sd/string] ([xshift=-8.9298mm]b1.south) -- (-10.6330,-30.6361);
\node[sd/label,anchor=north] at (-10.6330,-31.6361) {$\varrho$};
\draw[sd/string] ([xshift=1.6560mm]b1.south) -- ([xshift=-1.6796mm]b2.north);
\draw[sd/string] ([xshift=0.0000mm]b2.south) -- (1.6325,-30.6361);
\node[sd/label,anchor=north] at (1.6325,-31.6361) {$\varrho$};
\end{tikzpicture}
}

\expandafter\def\csname NGD@unroot-f\endcsname{%
\begin{tikzpicture}[sd/base,x=1mm,y=1mm]
\node[sd/box,minimum width=7.9518mm,minimum height=4.0000mm] (b2) at (-5.5690,-18.5328) {$f$};
\draw[sd/string] (-8.4449,-5.5769) -- ([xshift=-2.8759mm]b2.north);
\node[sd/label,anchor=west] at (-7.4449,-7.7769) {$X$};
\draw[sd/string] (9.5449,-5.5769) -- (9.5449,-25.5328);
\node[sd/label,anchor=north] at (9.5449,-26.5328) {$X^*$};
\draw[sd/string] (-3.4684,-12.1328) -- ([xshift=2.1006mm]b2.north);
\node[sd/label,anchor=west] at (-2.4684,-14.3328) {$Y$};
\draw[sd/string] (3.4860,-12.1328) -- (3.4860,-25.5328);
\node[sd/label,anchor=north] at (3.4860,-26.5328) {$Y^*$};
\draw[sd/string] ([xshift=0.0000mm]b2.south) -- (-5.5690,-25.5328);
\node[sd/label,anchor=north] at (-5.5690,-26.5328) {$Z$};
\draw[sd/string] (-8.4449,-5.5769) arc[start angle=180,end angle=0,x radius=8.9949mm,y radius=5.5769mm];
\draw[sd/string] (-3.4684,-12.1328) arc[start angle=180,end angle=0,x radius=3.4772mm,y radius=2.1559mm];
\end{tikzpicture}
}

\expandafter\def\csname NGD@xx-fused\endcsname{%
\begin{tikzpicture}[sd/base,x=1mm,y=1mm]
\node[sd/box,minimum width=12.1304mm,minimum height=5.3886mm] (b0) at (1.6325,-6.8943) {$L_{g+h}^{-1}$};
\node[sd/box,minimum width=9.6853mm,minimum height=4.1273mm] (b1) at (-2.8550,-16.0523) {$R_{g+h}$};
\node[sd/box,minimum width=11.6526mm,minimum height=4.0000mm] (b2) at (1.8713,-24.5159) {$u$};
\draw[sd/string] (-6.5977,0.0000) -- ([xshift=-3.7427mm]b1.north);
\node[sd/label,anchor=south] at (-6.5977,1.0000) {$\varrho$};
\draw[sd/string] (1.6325,0.0000) -- ([xshift=0.0000mm]b0.north);
\node[sd/label,anchor=south] at (1.6325,1.0000) {$\varrho$};
\draw[sd/string] ([xshift=-4.2301mm]b0.south) -- ([xshift=0.2573mm]b1.north);
\node[sd/label,anchor=west] at (-1.5977,-11.7886) {$g+h$};
\draw[sd/string] ([xshift=4.9652mm]b0.south) -- ([xshift=4.7263mm]b2.north);
\draw[sd/string] ([xshift=1.3921mm]b1.south) -- ([xshift=-3.3342mm]b2.north);
\draw[sd/string] ([xshift=0.0000mm]b2.south) -- (1.8713,-31.5159);
\node[sd/label,anchor=north] at (1.8713,-32.5159) {$\varrho$};
\end{tikzpicture}
}

\expandafter\def\csname NGD@xx-stacked\endcsname{%
\begin{tikzpicture}[sd/base,x=1mm,y=1mm]
\node[sd/box,minimum width=7.9267mm,minimum height=4.8321mm] (b0) at (1.6325,-6.6161) {$L_{g}^{-1}$};
\node[sd/box,minimum width=6.5000mm,minimum height=4.0000mm] (b1) at (-2.3458,-15.4321) {$R_{g}$};
\node[sd/box,minimum width=7.9267mm,minimum height=4.9786mm] (b2) at (1.6325,-24.3214) {$L_{h}^{-1}$};
\node[sd/box,minimum width=7.2909mm,minimum height=4.0000mm] (b3) at (-1.9504,-33.2107) {$R_{h}$};
\node[sd/box,minimum width=9.6163mm,minimum height=4.0000mm] (b4) at (0.7877,-41.6107) {$u$};
\draw[sd/string] (-4.4958,0.0000) -- ([xshift=-2.1500mm]b1.north);
\node[sd/label,anchor=south] at (-4.4958,1.0000) {$\varrho$};
\draw[sd/string] (1.6325,0.0000) -- ([xshift=0.0000mm]b0.north);
\node[sd/label,anchor=south] at (1.6325,1.0000) {$\varrho$};
\draw[sd/string] ([xshift=-2.1283mm]b0.south) -- ([xshift=1.8500mm]b1.north);
\node[sd/label,anchor=west] at (0.5042,-11.2321) {$g$};
\draw[sd/string] ([xshift=1.8717mm]b0.south) -- ([xshift=1.8717mm]b2.north);
\draw[sd/string] ([xshift=-2.1500mm]b1.south) -- ([xshift=-2.5454mm]b3.north);
\draw[sd/string] ([xshift=-2.1283mm]b2.south) -- ([xshift=1.4546mm]b3.north);
\node[sd/label,anchor=west] at (0.5042,-29.0107) {$h$};
\draw[sd/string] ([xshift=1.8932mm]b2.south) -- ([xshift=2.7380mm]b4.north);
\draw[sd/string] ([xshift=0.5677mm]b3.south) -- ([xshift=-2.1704mm]b4.north);
\draw[sd/string] ([xshift=0.0000mm]b4.south) -- (0.7877,-48.6107);
\node[sd/label,anchor=north] at (0.7877,-49.6107) {$\varrho$};
\end{tikzpicture}
}

\expandafter\def\csname NGD@xy-left\endcsname{%
\begin{tikzpicture}[sd/base,x=1mm,y=1mm]
\node[sd/box,minimum width=7.9267mm,minimum height=4.8321mm] (b0) at (1.6325,-6.6161) {$L_{g}^{-1}$};
\node[sd/box,minimum width=6.5000mm,minimum height=4.0000mm] (b1) at (-2.3458,-15.4321) {$R_{g}$};
\node[sd/box,minimum width=8.3867mm,minimum height=5.2715mm] (b2) at (1.4025,-24.4679) {$R_{-h}^{-1}$};
\node[sd/box,minimum width=7.2849mm,minimum height=4.0000mm] (b3) at (-1.9534,-33.5036) {$u$};
\node[sd/box,minimum width=9.6408mm,minimum height=4.0101mm] (b4) at (0.7754,-41.9086) {$R_{-h}$};
\draw[sd/string] (-4.4958,0.0000) -- ([xshift=-2.1500mm]b1.north);
\node[sd/label,anchor=south] at (-4.4958,1.0000) {$\varrho$};
\draw[sd/string] (1.6325,0.0000) -- ([xshift=0.0000mm]b0.north);
\node[sd/label,anchor=south] at (1.6325,1.0000) {$\varrho$};
\draw[sd/string] ([xshift=-2.1283mm]b0.south) -- ([xshift=1.8500mm]b1.north);
\node[sd/label,anchor=west] at (0.5042,-11.2321) {$g$};
\draw[sd/string] ([xshift=1.8717mm]b0.south) -- ([xshift=2.1017mm]b2.north);
\draw[sd/string] ([xshift=-2.1500mm]b1.south) -- ([xshift=-2.5424mm]b3.north);
\draw[sd/string] ([xshift=-1.8983mm]b2.south) -- ([xshift=1.4576mm]b3.north);
\draw[sd/string] ([xshift=2.1017mm]b2.south) -- ([xshift=2.7288mm]b4.north);
\node[sd/label,anchor=west] at (4.5042,-29.3036) {$-h$};
\draw[sd/string] ([xshift=0.5610mm]b3.south) -- ([xshift=-2.1678mm]b4.north);
\draw[sd/string] ([xshift=-0.0000mm]b4.south) -- (0.7754,-48.9137);
\node[sd/label,anchor=north] at (0.7754,-49.9137) {$\varrho$};
\end{tikzpicture}
}

\expandafter\def\csname NGD@xy-right\endcsname{%
\begin{tikzpicture}[sd/base,x=1mm,y=1mm]
\node[sd/box,minimum width=8.3867mm,minimum height=5.2715mm] (b0) at (2.9067,-6.8357) {$R_{-h}^{-1}$};
\node[sd/box,minimum width=8.4351mm,minimum height=4.8321mm] (b1) at (0.3825,-16.2875) {$L_{g}^{-1}$};
\node[sd/box,minimum width=7.5447mm,minimum height=4.0000mm] (b2) at (-3.3277,-25.1036) {$R_{g}$};
\node[sd/box,minimum width=9.3425mm,minimum height=4.0000mm] (b3) at (-0.0713,-33.5036) {$u$};
\node[sd/box,minimum width=8.3867mm,minimum height=4.0101mm] (b4) at (2.9067,-41.9086) {$R_{-h}$};
\draw[sd/string] (-6.0000,0.0000) -- ([xshift=-2.6723mm]b2.north);
\node[sd/label,anchor=south] at (-6.0000,1.0000) {$\varrho$};
\draw[sd/string] (2.9067,0.0000) -- ([xshift=0.0000mm]b0.north);
\node[sd/label,anchor=south] at (2.9067,1.0000) {$\varrho$};
\draw[sd/string] ([xshift=-1.9610mm]b0.south) -- ([xshift=0.5632mm]b1.north);
\draw[sd/string] ([xshift=3.0933mm]b0.south) -- ([xshift=3.0933mm]b4.north);
\node[sd/label,anchor=west] at (7.0000,-11.6715) {$-h$};
\draw[sd/string] ([xshift=-2.3825mm]b1.south) -- ([xshift=1.3277mm]b2.north);
\node[sd/label,anchor=west] at (-1.0000,-20.9036) {$g$};
\draw[sd/string] ([xshift=1.6175mm]b1.south) -- ([xshift=2.0713mm]b3.north);
\draw[sd/string] ([xshift=0.8497mm]b2.south) -- ([xshift=-2.4067mm]b3.north);
\draw[sd/string] ([xshift=0.7901mm]b3.south) -- ([xshift=-2.1879mm]b4.north);
\draw[sd/string] ([xshift=0.0000mm]b4.south) -- (2.9067,-48.9137);
\node[sd/label,anchor=north] at (2.9067,-49.9137) {$\varrho$};
\end{tikzpicture}
}

\expandafter\def\csname NGD@yy-fused\endcsname{%
\begin{tikzpicture}[sd/base,x=1mm,y=1mm]
\node[sd/box,minimum width=13.7467mm,minimum height=5.6229mm] (b0) at (1.6325,-7.0115) {$R_{-(g+h)}^{-1}$};
\node[sd/box,minimum width=9.6000mm,minimum height=4.0000mm] (b1) at (-3.7058,-16.2229) {$u$};
\node[sd/box,minimum width=17.0116mm,minimum height=4.3616mm] (b2) at (0.0000,-24.8037) {$R_{-(g+h)}$};
\draw[sd/string] (-7.4058,0.0000) -- ([xshift=-3.7000mm]b1.north);
\node[sd/label,anchor=south] at (-7.4058,1.0000) {$\varrho$};
\draw[sd/string] (1.6325,0.0000) -- ([xshift=0.0000mm]b0.north);
\node[sd/label,anchor=south] at (1.6325,1.0000) {$\varrho$};
\draw[sd/string] ([xshift=-3.0147mm]b0.south) -- ([xshift=2.3236mm]b1.north);
\draw[sd/string] ([xshift=1.7470mm]b0.south) -- ([xshift=3.3794mm]b2.north);
\node[sd/label,anchor=west] at (4.3794,-12.0229) {$-(g+h)$};
\draw[sd/string] ([xshift=0.4353mm]b1.south) -- ([xshift=-3.2705mm]b2.north);
\draw[sd/string] ([xshift=0.0000mm]b2.south) -- (0.0000,-31.9845);
\node[sd/label,anchor=north] at (0.0000,-32.9845) {$\varrho$};
\end{tikzpicture}
}

\expandafter\def\csname NGD@yy-stacked\endcsname{%
\begin{tikzpicture}[sd/base,x=1mm,y=1mm]
\node[sd/box,minimum width=9.8382mm,minimum height=5.2715mm] (b0) at (3.3534,-6.8357) {$R_{-g}^{-1}$};
\node[sd/box,minimum width=9.0839mm,minimum height=5.2715mm] (b1) at (-0.4657,-16.5072) {$R_{-h}^{-1}$};
\node[sd/box,minimum width=7.3798mm,minimum height=4.0000mm] (b2) at (-4.5827,-25.5429) {$u$};
\node[sd/box,minimum width=13.5336mm,minimum height=4.0101mm] (b3) at (-0.9943,-33.9480) {$R_{-h}$};
\node[sd/box,minimum width=10.3668mm,minimum height=4.0101mm] (b4) at (3.0891,-42.3581) {$R_{-g}$};
\draw[sd/string] (-7.1725,0.0000) -- ([xshift=-2.5899mm]b2.north);
\node[sd/label,anchor=south] at (-7.1725,1.0000) {$\varrho$};
\draw[sd/string] (3.3534,0.0000) -- ([xshift=-0.0000mm]b0.north);
\node[sd/label,anchor=south] at (3.3534,1.0000) {$\varrho$};
\draw[sd/string] ([xshift=-2.7294mm]b0.south) -- ([xshift=1.0897mm]b1.north);
\draw[sd/string] ([xshift=3.8191mm]b0.south) -- ([xshift=4.0834mm]b4.north);
\node[sd/label,anchor=west] at (8.1725,-11.6715) {$-g$};
\draw[sd/string] ([xshift=-2.7069mm]b1.south) -- ([xshift=1.4101mm]b2.north);
\draw[sd/string] ([xshift=1.2931mm]b1.south) -- ([xshift=1.8217mm]b3.north);
\node[sd/label,anchor=west] at (1.8275,-21.3429) {$-h$};
\draw[sd/string] ([xshift=0.6664mm]b2.south) -- ([xshift=-2.9220mm]b3.north);
\draw[sd/string] ([xshift=1.1778mm]b3.south) -- ([xshift=-2.9056mm]b4.north);
\draw[sd/string] ([xshift=-0.0000mm]b4.south) -- (3.0891,-49.3632);
\node[sd/label,anchor=north] at (3.0891,-50.3632) {$\varrho$};
\end{tikzpicture}
}

\expandafter\def\csname NGD@yz-left\endcsname{%
\begin{tikzpicture}[sd/base,x=1mm,y=1mm]
\node[sd/box,minimum width=8.2136mm,minimum height=5.2715mm] (b0) at (3.4040,-6.8357) {$R_{-g}^{-1}$};
\node[sd/box,minimum width=7.9267mm,minimum height=4.9786mm] (b1) at (-3.5474,-16.3607) {$L_{h}^{-1}$};
\node[sd/box,minimum width=6.5000mm,minimum height=4.0000mm] (b2) at (0.5608,-25.2500) {$u$};
\node[sd/box,minimum width=11.0017mm,minimum height=4.0000mm] (b3) at (-2.0099,-33.6500) {$L_{h}$};
\node[sd/box,minimum width=10.6207mm,minimum height=4.0101mm] (b4) at (2.2004,-42.0551) {$R_{-g}$};
\draw[sd/string] (-3.5474,0.0000) -- ([xshift=0.0000mm]b1.north);
\node[sd/label,anchor=south] at (-3.5474,1.0000) {$\varrho$};
\draw[sd/string] (3.4040,0.0000) -- ([xshift=0.0000mm]b0.north);
\node[sd/label,anchor=south] at (3.4040,1.0000) {$\varrho$};
\draw[sd/string] ([xshift=-0.9932mm]b0.south) -- ([xshift=1.8500mm]b2.north);
\draw[sd/string] ([xshift=3.0068mm]b0.south) -- ([xshift=4.2103mm]b4.north);
\node[sd/label,anchor=west] at (7.4108,-11.6715) {$-g$};
\draw[sd/string] ([xshift=-2.8634mm]b1.south) -- ([xshift=-4.4008mm]b3.north);
\node[sd/label,anchor=west] at (-5.4108,-21.0500) {$h$};
\draw[sd/string] ([xshift=1.9582mm]b1.south) -- ([xshift=-2.1500mm]b2.north);
\draw[sd/string] ([xshift=-0.3685mm]b2.south) -- ([xshift=2.2021mm]b3.north);
\draw[sd/string] ([xshift=1.2201mm]b3.south) -- ([xshift=-2.9902mm]b4.north);
\draw[sd/string] ([xshift=0.0000mm]b4.south) -- (2.2004,-49.0602);
\node[sd/label,anchor=north] at (2.2004,-50.0602) {$\varrho$};
\end{tikzpicture}
}

\expandafter\def\csname NGD@yz-right\endcsname{%
\begin{tikzpicture}[sd/base,x=1mm,y=1mm]
\node[sd/box,minimum width=7.9267mm,minimum height=4.9786mm] (b0) at (-3.8042,-6.6893) {$L_{h}^{-1}$};
\node[sd/box,minimum width=8.2136mm,minimum height=5.2715mm] (b1) at (3.6608,-16.2143) {$R_{-g}^{-1}$};
\node[sd/box,minimum width=6.8466mm,minimum height=4.0000mm] (b2) at (0.4772,-25.2500) {$u$};
\node[sd/box,minimum width=10.7136mm,minimum height=4.0101mm] (b3) at (2.4108,-33.6551) {$R_{-g}$};
\node[sd/box,minimum width=11.2783mm,minimum height=4.0000mm] (b4) at (-2.1284,-42.0602) {$L_{h}$};
\draw[sd/string] (-3.8042,0.0000) -- ([xshift=0.0000mm]b0.north);
\node[sd/label,anchor=south] at (-3.8042,1.0000) {$\varrho$};
\draw[sd/string] (3.6608,0.0000) -- ([xshift=0.0000mm]b1.north);
\node[sd/label,anchor=south] at (3.6608,1.0000) {$\varrho$};
\draw[sd/string] ([xshift=-2.8634mm]b0.south) -- ([xshift=-4.5392mm]b4.north);
\node[sd/label,anchor=west] at (-5.6676,-11.3786) {$h$};
\draw[sd/string] ([xshift=1.9582mm]b0.south) -- ([xshift=-2.3233mm]b2.north);
\draw[sd/string] ([xshift=-1.5068mm]b1.south) -- ([xshift=1.6767mm]b2.north);
\draw[sd/string] ([xshift=2.4932mm]b1.south) -- ([xshift=3.7432mm]b3.north);
\node[sd/label,anchor=west] at (7.1540,-21.0500) {$-g$};
\draw[sd/string] ([xshift=0.0740mm]b2.south) -- ([xshift=-1.8596mm]b3.north);
\draw[sd/string] ([xshift=-1.3297mm]b3.south) -- ([xshift=3.2094mm]b4.north);
\draw[sd/string] ([xshift=0.0000mm]b4.south) -- (-2.1284,-49.0602);
\node[sd/label,anchor=north] at (-2.1284,-50.0602) {$\varrho$};
\end{tikzpicture}
}

\expandafter\def\csname NGD@zigzag-left\endcsname{%
\begin{tikzpicture}[sd/base,x=1mm,y=1mm]
\draw[sd/string] (6.0689,0.0000) -- (6.0689,-9.8400);
\node[sd/label,anchor=south] at (6.0689,1.0000) {$X$};
\draw[sd/string] (-6.0689,-5.4400) -- (-6.0689,-17.3627);
\node[sd/label,anchor=north] at (-6.0689,-18.3627) {$X$};
\draw[sd/string] (-2.0689,-5.4400) -- (-2.0689,-9.8400);
\node[sd/label,anchor=west] at (-1.0689,-7.6400) {$X^*$};
\draw[sd/string] (-6.0689,-5.4400) arc[start angle=180,end angle=0,x radius=2.0000mm,y radius=1.2400mm];
\draw[sd/string] (-2.0689,-9.8400) arc[start angle=180,end angle=360,x radius=4.0689mm,y radius=2.5227mm];
\end{tikzpicture}
}

\expandafter\def\csname NGD@zigzag-right\endcsname{%
\begin{tikzpicture}[sd/base,x=1mm,y=1mm]
\draw[sd/string] (-6.0689,0.0000) -- (-6.0689,-11.1227);
\node[sd/label,anchor=south] at (-6.0689,1.0000) {$X$};
\draw[sd/string] (-2.0689,-6.7227) -- (-2.0689,-11.1227);
\node[sd/label,anchor=west] at (-1.0689,-8.9227) {$X^*$};
\draw[sd/string] (6.0689,-6.7227) -- (6.0689,-17.3627);
\node[sd/label,anchor=north] at (6.0689,-18.3627) {$X$};
\draw[sd/string] (-2.0689,-6.7227) arc[start angle=180,end angle=0,x radius=4.0689mm,y radius=2.5227mm];
\draw[sd/string] (-6.0689,-11.1227) arc[start angle=180,end angle=360,x radius=2.0000mm,y radius=1.2400mm];
\end{tikzpicture}
}

\expandafter\def\csname NGD@zx-left\endcsname{%
\begin{tikzpicture}[sd/base,x=1mm,y=1mm]
\node[sd/box,minimum width=7.9267mm,minimum height=4.8321mm] (b0) at (-3.5613,-6.6161) {$L_{g}^{-1}$};
\node[sd/box,minimum width=7.9267mm,minimum height=4.9786mm] (b1) at (3.5613,-15.9214) {$L_{h}^{-1}$};
\node[sd/box,minimum width=6.8582mm,minimum height=4.0000mm] (b2) at (0.5270,-24.8107) {$R_{h}$};
\node[sd/box,minimum width=10.2356mm,minimum height=4.0000mm] (b3) at (2.4068,-33.2107) {$u$};
\node[sd/box,minimum width=11.0314mm,minimum height=4.0000mm] (b4) at (-2.0089,-41.6107) {$L_{g}$};
\draw[sd/string] (-3.5613,0.0000) -- ([xshift=0.0000mm]b0.north);
\node[sd/label,anchor=south] at (-3.5613,1.0000) {$\varrho$};
\draw[sd/string] (3.5613,0.0000) -- ([xshift=0.0000mm]b1.north);
\node[sd/label,anchor=south] at (3.5613,1.0000) {$\varrho$};
\draw[sd/string] ([xshift=-2.8634mm]b0.south) -- ([xshift=-4.4157mm]b4.north);
\node[sd/label,anchor=west] at (-5.4246,-11.2321) {$g$};
\draw[sd/string] ([xshift=1.7592mm]b0.south) -- ([xshift=-2.3291mm]b2.north);
\draw[sd/string] ([xshift=-1.3634mm]b1.south) -- ([xshift=1.6709mm]b2.north);
\node[sd/label,anchor=west] at (3.1979,-20.6107) {$h$};
\draw[sd/string] ([xshift=2.6582mm]b1.south) -- ([xshift=3.8126mm]b3.north);
\draw[sd/string] ([xshift=0.0869mm]b2.south) -- ([xshift=-1.7929mm]b3.north);
\draw[sd/string] ([xshift=-1.2886mm]b3.south) -- ([xshift=3.1271mm]b4.north);
\draw[sd/string] ([xshift=0.0000mm]b4.south) -- (-2.0089,-48.6107);
\node[sd/label,anchor=north] at (-2.0089,-49.6107) {$\varrho$};
\end{tikzpicture}
}

\expandafter\def\csname NGD@zx-right\endcsname{%
\begin{tikzpicture}[sd/base,x=1mm,y=1mm]
\node[sd/box,minimum width=7.9267mm,minimum height=4.9786mm] (b0) at (1.6325,-6.6893) {$L_{h}^{-1}$};
\node[sd/box,minimum width=7.2579mm,minimum height=4.0000mm] (b1) at (-1.9668,-15.5786) {$R_{h}$};
\node[sd/box,minimum width=7.9267mm,minimum height=4.8321mm] (b2) at (-1.6325,-24.3946) {$L_{g}^{-1}$};
\node[sd/box,minimum width=6.5000mm,minimum height=4.0000mm] (b3) at (2.2767,-33.2107) {$u$};
\node[sd/box,minimum width=8.9725mm,minimum height=4.0000mm] (b4) at (-1.1096,-41.6107) {$L_{g}$};
\draw[sd/string] (-4.4958,0.0000) -- ([xshift=-2.5290mm]b1.north);
\node[sd/label,anchor=south] at (-4.4958,1.0000) {$\varrho$};
\draw[sd/string] (1.6325,0.0000) -- ([xshift=0.0000mm]b0.north);
\node[sd/label,anchor=south] at (1.6325,1.0000) {$\varrho$};
\draw[sd/string] ([xshift=-1.8554mm]b0.south) -- ([xshift=1.7439mm]b1.north);
\node[sd/label,anchor=west] at (0.7771,-11.3786) {$h$};
\draw[sd/string] ([xshift=2.4942mm]b0.south) -- ([xshift=1.8500mm]b3.north);
\draw[sd/string] ([xshift=0.1672mm]b1.south) -- ([xshift=-0.1672mm]b2.north);
\draw[sd/string] ([xshift=-2.8634mm]b2.south) -- ([xshift=-3.3863mm]b4.north);
\node[sd/label,anchor=west] at (-3.4958,-29.0107) {$g$};
\draw[sd/string] ([xshift=1.7592mm]b2.south) -- ([xshift=-2.1500mm]b3.north);
\draw[sd/string] ([xshift=-0.9454mm]b3.south) -- ([xshift=2.4408mm]b4.north);
\draw[sd/string] ([xshift=0.0000mm]b4.south) -- (-1.1096,-48.6107);
\node[sd/label,anchor=north] at (-1.1096,-49.6107) {$\varrho$};
\end{tikzpicture}
}

\expandafter\def\csname NGD@zz-fused\endcsname{%
\begin{tikzpicture}[sd/base,x=1mm,y=1mm]
\node[sd/box,minimum width=12.9304mm,minimum height=5.3886mm] (b0) at (-1.7825,-6.8943) {$L_{g+h}^{-1}$};
\node[sd/box,minimum width=6.5000mm,minimum height=4.0000mm] (b1) at (4.9977,-15.9886) {$u$};
\node[sd/box,minimum width=14.3453mm,minimum height=4.1273mm] (b2) at (-1.0750,-24.4523) {$L_{g+h}$};
\draw[sd/string] (-1.7825,0.0000) -- ([xshift=0.0000mm]b0.north);
\node[sd/label,anchor=south] at (-1.7825,1.0000) {$\varrho$};
\draw[sd/string] (6.8477,0.0000) -- ([xshift=1.8500mm]b1.north);
\node[sd/label,anchor=south] at (6.8477,1.0000) {$\varrho$};
\draw[sd/string] ([xshift=-5.3652mm]b0.south) -- ([xshift=-6.0727mm]b2.north);
\node[sd/label,anchor=west] at (-6.1477,-11.7886) {$g+h$};
\draw[sd/string] ([xshift=4.6301mm]b0.south) -- ([xshift=-2.1500mm]b1.north);
\draw[sd/string] ([xshift=-1.8409mm]b1.south) -- ([xshift=4.2318mm]b2.north);
\draw[sd/string] ([xshift=0.0000mm]b2.south) -- (-1.0750,-31.5159);
\node[sd/label,anchor=north] at (-1.0750,-32.5159) {$\varrho$};
\end{tikzpicture}
}

\expandafter\def\csname NGD@zz-stacked\endcsname{%
\begin{tikzpicture}[sd/base,x=1mm,y=1mm]
\node[sd/box,minimum width=9.6859mm,minimum height=4.8321mm] (b0) at (-3.1291,-6.6161) {$L_{g}^{-1}$};
\node[sd/box,minimum width=7.9267mm,minimum height=4.9786mm] (b1) at (0.6139,-15.9214) {$L_{h}^{-1}$};
\node[sd/box,minimum width=6.5000mm,minimum height=4.0000mm] (b2) at (4.7220,-24.8107) {$u$};
\node[sd/box,minimum width=7.0215mm,minimum height=4.0000mm] (b3) at (0.1613,-33.2107) {$L_{h}$};
\node[sd/box,minimum width=9.2333mm,minimum height=4.0000mm] (b4) at (-3.3554,-41.6107) {$L_{g}$};
\draw[sd/string] (-3.1291,0.0000) -- ([xshift=0.0000mm]b0.north);
\node[sd/label,anchor=south] at (-3.1291,1.0000) {$\varrho$};
\draw[sd/string] (6.7421,0.0000) -- ([xshift=2.0201mm]b2.north);
\node[sd/label,anchor=south] at (6.7421,1.0000) {$\varrho$};
\draw[sd/string] ([xshift=-3.7429mm]b0.south) -- ([xshift=-3.5166mm]b4.north);
\node[sd/label,anchor=west] at (-5.8720,-11.2321) {$g$};
\draw[sd/string] ([xshift=2.6786mm]b0.south) -- ([xshift=-1.0643mm]b1.north);
\draw[sd/string] ([xshift=-2.8634mm]b1.south) -- ([xshift=-2.4108mm]b3.north);
\node[sd/label,anchor=west] at (-1.2495,-20.6107) {$h$};
\draw[sd/string] ([xshift=2.1283mm]b1.south) -- ([xshift=-1.9799mm]b2.north);
\draw[sd/string] ([xshift=-2.1500mm]b2.south) -- ([xshift=2.4108mm]b3.north);
\draw[sd/string] ([xshift=-0.9889mm]b3.south) -- ([xshift=2.5278mm]b4.north);
\draw[sd/string] ([xshift=0.0000mm]b4.south) -- (-3.3554,-48.6107);
\node[sd/label,anchor=north] at (-3.3554,-49.6107) {$\varrho$};
\end{tikzpicture}
}

\hypersetup{pdftitle={Triangular prism equations for near-group categories: algebraic solutions and a unitary conjecture},
 pdfauthor={Huixuan He, Zhengwei Liu, Fan Lu, Sebastien Palcoux, Yunxiang Ren}}
\begin{document}
\title[Triangular prism equations for near-group categories]
{\texorpdfstring{{Triangular prism equations for near-group categories}: {algebraic solutions and a unitary conjecture}}
{Triangular prism equations for near-group categories: algebraic solutions and a unitary conjecture}}
\author{Huixuan He}
\address{H.~He, Department of Mathematics, Tsinghua University, Beijing, China}
\email{hhx22@mails.tsinghua.edu.cn}

\author{Zhengwei Liu}
\address{Z.~Liu, Yau Mathematical Sciences Center and Department of Mathematics, Tsinghua University, and Beijing Institute of Mathematical Sciences and Applications, Beijing, China}
\email{liuzhengwei@mail.tsinghua.edu.cn}

\author{Fan Lu}
\address{F.~Lu, Yau Mathematical Sciences Center, Tsinghua University, and Beijing Institute of Mathematical Sciences and Applications, Beijing, China}
\email{lufan@bimsa.cn}

\author{Sebastien Palcoux}
\address{S.~Palcoux, Beijing Institute of Mathematical Sciences and Applications, Huairou District, Beijing, China}
\email{sebastienpalcoux@gmail.com}

\author{Yunxiang Ren}
\address{Y.~Ren, LinkedIn Corporation, Sunnyvale, CA, USA}
\email{renyunxiang@gmail.com}

\date{}
\subjclass[2020]{18M20, 18M15, 46L37, 11R37, 11R42}
\keywords{{Near-group fusion category, triangular prism equation, hyperbolic gamma function, Stark conjecture, Hilbert's twelfth problem}}

\begin{abstract}
  We derive the triangular prism equations (TPE) for near-group fusion rules of type $G+|G|$, where $G$ is a finite abelian group. We prove that spherical categorification with trivial Frobenius--Schur indicator over any field in which $|G|$ is invertible is equivalent to the solvability of the TPE system. For every cyclic group, we construct an explicit solution to the TPE system using hyperbolic gamma functions, yielding both a spherical category with negative simple-object dimensions and, after Galois conjugation, a pseudo-unitary category with positive dimensions. For scalar solutions with positive dimensions, we establish an exact singular-value formula and a uniform gap, which provide explicit sufficient criteria for unitarity. Finally, we show that the unitarity of the cyclic Galois-conjugated coefficients for arbitrary orders corresponds to a specialization of a real-multiplication conjecture implied by the order-one abelian Stark conjecture, revealing a deep connection between categorical unitarity and Hilbert's twelfth problem for real quadratic fields.
\end{abstract}

\maketitle
\begin{center}\textit{In memory of Vaughan Jones}\end{center}
\tableofcontents

\section{Introduction}\label{sec: intro}

Near-group fusion rules consist of a finite group $G$ of invertible
simple objects and one non-invertible simple object $\varrho$, with
\[
 g\varrho=\varrho g=\varrho,\qquad
 \varrho^2=\sum_{g\in G}g+m\varrho.
\]
The case $m=0$ corresponds to classical Tambara--Yamagami categories \cite{TY}, whereas Siehler examined the general associative structures for arbitrary $m$ \cite{Siehler}. In this paper, we study the near-group case $m=|G|=n$ for finite abelian groups $G$. Our scalar system of equations and categorical reconstruction apply to every finite abelian $G$, while our explicit complex-analytic constructions specialize to cyclic groups $G=C_n$.

Historically, the classification and construction of near-group fusion rules were deeply rooted in subfactor theory. In his foundational work \cite{Izumi2001LongoRehren2}, Izumi associated two distinct nonlinear systems of equations to finite abelian groups $G$, leading to two celebrated infinite families of conjectural subfactors:
\begin{enumerate}
    \item \textbf{The Haagerup--Izumi series}: Defined for odd-order abelian groups $G$, these subfactors have index $\frac{|G|+2+\sqrt{|G|^2+4}}{2}$. The case $G=\mathbb{Z}_3$ yields the famous exotic Haagerup subfactor, with higher $G$ explored extensively by Evans and Gannon \cite{EvaGan2017}.
    \item \textbf{The Near-Group / Izumi--Xu series} ($m=|G|$): Defined for arbitrary finite abelian groups $G$, these subfactors have index $\frac{|G|+2+\sqrt{|G|^2+4|G|}}{2}$ and possess a star-like principal graph $2^{G}1$ (consisting of a central vertex $\varrho$, one arm of length $1$, and $|G|$ arms of length $2$ corresponding to the elements of $G$, as shown in Figure~\ref{fig:2^G1principalgraph}). For small groups, $G=\mathbb{Z}_1$ and $G=\mathbb{Z}_2$ recover the index $<4$ subfactors $A_4$ and $E_6$ respectively, while $G=\mathbb{Z}_3$ realizes the Izumi--Xu subfactor.
\end{enumerate}
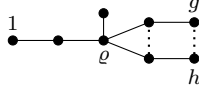
\begin{figure}[h]
\centering
\begin{tikzpicture}[baseline=-.1cm,scale=.6]
	\filldraw (1,0) circle (.1cm) node [above] {{\scriptsize{$1$}}};
	\filldraw (2,0) circle (.1cm);
	\filldraw (3,0) circle (.1cm) node [below] {{\scriptsize{$\varrho$}}};
	\filldraw (3,.6) circle (.1cm);
	\filldraw (4,.4) circle (.1cm);
	\filldraw (4,-.4) circle (.1cm);
	\filldraw (5,.4) circle (.1cm) node [above] {{\scriptsize{$g$}}};
	\filldraw (5,-.4) circle (.1cm) node [below] {{\scriptsize{$h$}}};
	\draw[] (1,0)--(3,0);
	\draw (3,.6)--(3,0)--(4,.4)--(5,.4);
	\draw (5,-.4)--(4,-.4)--(3,0);
	\draw[thick, dotted]  (4,.4)--(4,-.4);
	\draw[thick, dotted]  (5,.4)--(5,-.4);
\end{tikzpicture}
\caption{The $2^G 1$ principal graph.}\label{fig:2^G1principalgraph}
\end{figure}

To properly contextually locate these subfactors, one distinguishes three progressively stronger categorical and operator-algebraic structures:
\[
\text{spherical fusion category}
\ \subset\
\text{unitary fusion category}
\ \subset\
\text{subfactor realization}.
\]
\begin{enumerate}
    \item \textbf{Spherical structure}: A pivotal tensor category equipped with a spherical trace, establishing purely algebraic notions of dimension and duality.
    \item \textbf{Unitary structure}: A $C^*$ tensor category equipped with a positive-definite Hermitian inner product on morphism spaces.
    \item \textbf{Subfactor structure}: A unitary fusion category equipped with a (connected) Frobenius algebra (or Q-system), encoding the subfactor realization $N \subset M$.
\end{enumerate}

Izumi's Cuntz-algebra approach \cite{Izumi2017} directly targets subfactor and $C^*$ near-group structures by realizing fusion categories through endomorphisms of Cuntz algebras. This operator-algebraic program spurred significant progress, including investigations into Drinfeld centers by Evans and Gannon \cite{EvaGan}, categorifiability constraints by Schopieray \cite{Schopieray}, super-modular near-group centers by Rowell, Sun, and Zheng \cite{RSZ}, and modular data evaluations by Yu and Zheng \cite{YZ}.

In contrast to Cuntz-algebra methods, we tackle the categorification problem directly at the foundational polynomial level via the triangular prism equations (TPE) of Liu, Palcoux, Ren, and Vercleyen \cite{LPR}, originally conceived for spherical near-group categories. Sphericality provides a natural, weaker structural framework when omitting unitary structure; moreover, in characteristic zero, Thornton's theorem \cite[Theorem~1.3]{Thornton} guarantees its existence. We establish that the existence of a spherical near-group fusion category with trivial Frobenius--Schur indicator is equivalent to the solvability of the TPE system. We focus on the trivial Frobenius--Schur indicator because our primary motivation is to construct an infinite family of $2^G 1$ subfactors for cyclic groups $G = C_p$ of odd prime order. The key structural insight is that three group actions on morphism spaces determine a coherent $ G $-indexed basis in which every tetrahedral coefficient is explicitly expressed through a nondegenerate symmetric bicharacter $\chi$, an even quadratic refinement $a$, a rotation scalar $c$, and a scalar function $\beta \colon G \to \mathbb{C}$ (these group actions were also independently obtained by Noah Snyder).

For fixed dimension and quadratic data, Proposition~\ref{prop:cubic-reduction} reduces the TPE system to the single boundary condition $\beta(0)=-1/d$ together with the  cubic identity \eqref{eq:two_beta_identity}. This cubic identity from TPE essentially simplifies the pentagon equations of near-group categories by employing a distinct positioning of labeled morphisms in the graphical calculus.  An invertible Fourier transform identifies this TPE system with Izumi's polynomial subsystem for every finite abelian group (Lemma~\ref{lem:izumi-cubic-dictionary}); consequently, under unitary conditions, the two formulations coincide.

For every cyclic group $C_n$, we construct an explicit solution to the TPE system using hyperbolic gamma functions, yielding an irrational simple object dimension
\[
    d_-=\frac{n-\sqrt{n^2+4n}}{2}.
\]
While algebraic reconstruction over $\mathbb{C}$ yields a spherical category with $ d_{\varrho}=d_-<0$, applying a root-reversing field embedding produces a pseudo-unitary category with $ d_{\varrho}=d_+=(n+\sqrt{n^2+4n})/2$. Crucially, this construction relies only on fundamental functional equations satisfied by hyperbolic gamma functions—independent of their explicit analytic formulas—and an integral identity established by Belousov, Sarkissian, and Spiridonov \cite{Belousov_2026}. Then cubic identity~\eqref{eq:two_beta_identity} is proved by the residue theorem in complex analysis.

Unitarity is the primary consideration for subfactor theory. We compute channel singular values and establish a uniform gap from Hermitian reflection (Theorem~\ref{thm:unitarity-gap}); in particular, the bound $\sqrt{n}\,|\beta(g)|\le\sqrt{2}$ for all $g\ne0$ guarantees reflection positivity and enables unitary reconstruction. For the cyclic analytic family, unitarity under a root-reversing embedding $ \iota $ reduces to the condition $\sqrt{n}|\iota(\beta(g))|=1$ for all $1\le g<n$, which is easy to verify for $n\le6$. For general $n$, this condition translates to a cyclic specialization of the Minimalist Real Multiplication Values Conjecture \cite[Conjecture~1.36(2)]{AFK}, and conditionally follows from the order-one abelian Stark conjecture \cite[Theorem~2.20(2)]{AFK}. Consequently, a single common embedding satisfying these analytic conditions suffices for a unitary categorification at each order, establishing a  connection between categorical unitarity and Hilbert's twelfth problem. For the real quadratic fields $\mathbb{Q}(\sqrt{n(n+4)})$, our parameters are linked to Stark class invariants via the Shintani--Faddeev normalization \cite[Theorem~1.1]{Kopp}; see Remark~\ref{rem:hilbert-twelfth}. Whether this constructs an infinite family of $2^{G}1$ subfactors remains a compelling open question.

Notably, hyperbolic gamma functions were also employed in the construction of cyclic Haagerup--Izumi fusion categories at every odd order \cite{huang2026cyclichaagerupizumifusioncategories,gannon2026haagerupizumifusioncategories}. Furthermore, Radchenko and Wheeler \cite{RadchenkoWheeler} solved Izumi's polynomial subsystem for $G = C_n \times C_m$ using finite quantum dilogarithms built from hyperbolic gamma functions.  Crucially, what distinguishes our approach from these prior works is that the integral identity satisfied by the hyperbolic gamma function directly reproduces the TPE system, without requiring any algebraic transformations.

We remark that our investigation of near-group categories over fields of characteristic two was initially stimulated by discussions between the second author and Terry Gannon during a conference last September. In characteristic two, we obtain the Frobenius relation $\beta(2g) = \beta(g)^4$. Parallel to the central role played by hyperbolic gamma functions over $\mathbb{C}$, the unique supersingular elliptic curve $Y^2 + Y = X^3 + d_{\varrho}$ furnishes a universal arithmetic framework for near-group categories in characteristic two. A systematic treatment of this construction will appear in a subsequent paper.

The paper is organized as follows. Section~\ref{sec:abelian-system} formulates the TPE system and presents an explicit construction of solutions. Sections~\ref{sec: prelim} establish our graphical calculus and diagrammatic conventions. In Section~\ref{sec:fsymbols}, we exploit three natural group actions to derive explicit tetrahedral coefficients, while Section~\ref{sec:solve-TPE} verifies the triangular prism system and proves algebraic and unitary reconstruction. Finally, Section~\ref{sec:unitarity} develops criteria for unitarity and formulates a unitary conjecture for the Galois-conjugated solutions.

\section{The TPE system for finite abelian groups}\label{sec:abelian-system}
We first state the TPE system over a general field.
Sections~\ref{sec:fsymbols}--\ref{sec:solve-TPE} subsequently derive these
parameters and equations from spherical/unitary near-group categories and reconstruct categories
from abstract scalar solutions.

\Needspace{60mm}
\subsection{The TPE system and its reduction}
\begin{definition}[TPE system]\label{thm: equation}
{Let $G$ be a finite abelian group of order $n$} and let $K$ be a field in which $n$ is invertible.
The TPE system consist of $d_\varrho,c\in K^\times$, a
nondegenerate symmetric bicharacter $\chi:G\times G\to K^\times$,
$a:G\to K^\times$, and $\beta:G\to K$, satisfying
\begin{align}
 d_\varrho^2&=n+nd_\varrho,\label{eq:dim_rho}\\
 a(-g)&=a(g),\qquad a(0)=1,\label{eq:apm}\\
 a(g+h)\chi(g,h)&=a(g)a(h),\label{eq:achi}\\
 c^3\sum_{g\in G}a(g)&=\frac1n,\label{eq:c_nonzero}\\
 \sum_{g\in G}a(g)^{-1}\chi(h,g)\beta(g)&=cn\beta(h),\label{eq:Fourier}\\
 \beta(g)\beta(-g)+\frac{\delta_{g,0}}{d_\varrho}&=\frac{a(g)}n,\label{eq:betapm}\\
 \frac{a(h-g)}n\sum_{k\in G}\chi(h,k)^{-1}
 \frac{\beta(g-k)\beta(k)}{a(k)}
 &=\beta(-h)\beta(g)\beta(h-g)+\frac{\delta_{h,0}\delta_{g,0}}{d_\varrho}.
 \label{eq:two_beta_identity}
\end{align}
A unitary TPE system over $ K = \mathbb{C} $ requires further
\begin{align}
  d &> 0, \\
  \overline{\beta(g)} &= \frac{\beta(-g)}{a(g)}.
\end{align}
\end{definition}
\begin{remark}
    The primary structural advantage of the TPE system lies in the form of its most complicated equation~\eqref{eq:two_beta_identity}: it equates a sum of quadratic terms to a single cubic term, whereas the corresponding equation in Izumi's system equates a sum of cubic terms to a single quadratic term. While such a shift might appear minor in a general polynomial system, it is crucial in the near-group setting, where this is the \emph{sole} non-trivial equation governing the system.

    This structural simplification provides substantial theoretical and computational leverage. For instance, in characteristic two, setting $g=2h$ in Equation~\eqref{eq:two_beta_identity} directly yields the Frobenius relation $\beta(2h)=\beta(h)^4$. Crucially, the fundamental role of the TPE framework is further validated by the work of Radchenko and Wheeler \cite{RadchenkoWheeler}, who constructed solutions to Izumi's subsystem by first establishing that they satisfy an equation analogous to TPE-equation~\eqref{eq:two_beta_identity}, rather than tackling the pentagon equations directly.
\end{remark}

\begin{theorem}
    Let $G$ be a finite abelian group of order $n$. The existence of a spherical (resp., unitary) near-group fusion category with fusion rules $G + n$ and trivial Frobenius--Schur indicator over a field $K$ (resp., $\mathbb{C}$) is equivalent to the solvability of the TPE system (resp., the unitary TPE system).
\end{theorem}
\begin{proof}
    See Sections~\ref{sec:fsymbols}--\ref{sec:solve-TPE}.
\end{proof}
\begin{remark}
    Two main reasons motivate our focus on trivial Frobenius--Schur indicators:
    \begin{enumerate}
        \item Rather than providing a complete classification of near-group fusion categories, our primary goal is to construct an infinite family of $2^G 1$ subfactors for $G = C_p$, a cyclic group of odd prime order. In this setting, the Frobenius--Schur indicator of any self-dual simple object $X \cong X^*$ is always $1$.
        \item Over fields of characteristic two, the Frobenius--Schur indicator is automatically trivial.
    \end{enumerate}
\end{remark}

The two solutions to Eq.~\eqref{eq:dim_rho} are
\begin{align}
  d_+
  &=
  \frac{n+\sqrt{n^2+4n}}{2},
  &
  d_-
  &=
  \frac{n-\sqrt{n^2+4n}}{2}.
\end{align}
Set $\omega_1=d_+$, $\omega_2=-d_-$ and
$\Omega=\omega_1+\omega_2$. Then
\begin{equation}\label{eq:periods}
  \omega_1-\omega_2
  =
  \omega_1\omega_2
  =
  n,
  \qquad
  \frac{1}{\omega_2}-\frac{1}{\omega_1}=1,
  \qquad
  0<\omega_2<1,
  \qquad
  \Omega>n.
\end{equation}
Let $ \gamma(z):= \gamma^{(2)}(z;\omega_1,\omega_2) $ be the hyperbolic gamma function as defined in \cite[Section~1]{Belousov_2026}.
\begin{lemma}[BSS identity]\label{lem:BSS_identity}
  Let $ \alpha, g\in \mathbb{C} $ such that $ 0 < \mathrm{Re}(\alpha) < \mathrm{Re}(\Omega + g) $
  \begin{align}
&\int_{-i\infty}^{i\infty}e^{\pi i z(2\alpha-g-\Omega)/n}
 \gamma(z-g)\gamma(-z)\frac{dz}{i\sqrt n}\nonumber
&\qquad=e^{\pi i g(2\alpha-g-\Omega)/(2n)}
 \gamma(\alpha)\gamma(\Omega+g-\alpha)\gamma(-g),
 \tag{BSS}\label{eq:BSS}
\end{align}
\end{lemma}
\begin{proof}
    Specializing \cite[arXiv version, equation~(2.17)]{Belousov_2026}
at $f=-g$, $h=0$, and $\omega_1\omega_2=n$ gives the identity.
\end{proof}
\begin{remark}
  The primary constraint \eqref{eq:two_beta_identity} in the TPE system equates a sum of products of two terms of $ \beta $ to a product of three terms of $ \beta $. Remarkably, the Belousov--Sarkissian--Spiridonov (BSS) identity exhibits the exact same algebraic structure, expressing an integral of a product of two terms of $ \gamma $ as a product of three terms of $ \gamma $. This striking structural parallel provides the key motivation for constructing $\beta$ via hyperbolic gamma functions $ \gamma(z) $.
\end{remark}

\begin{proposition}[A single cubic identity]\label{prop:cubic-reduction}\label{prop:beta-normalization}
Let $G$ be a finite abelian group of order $n$, over a field $K$ in
which $n$ is invertible. Fix a nondegenerate symmetric bicharacter
$\chi$, an even quadratic refinement $a$ satisfying
\eqref{eq:apm}--\eqref{eq:achi}, and $d\ne0$ with $d^2=n+nd$; put $d_\varrho=d$.
Then the TPE system is equivalent to
\[
 \beta(0)=-d^{-1}
\]
and the cubic equation~\eqref{eq:two_beta_identity}. More precisely,
these conditions force $B:=\sum_{g\in G}\beta(g)\ne0$, and the
remaining equations hold with the unique choice
\[
 c=-\frac1{dB}.
\]
Every $\beta(g)$ is then nonzero, and
\begin{equation}\label{eq:beta-zero-convolution}
 \beta(0)=-\frac1d,\qquad
 \sum_{k\in G}\frac{\beta(g-k)\beta(k)}{a(k)}
 =\delta_{g,0}-\frac1d.
\end{equation}
\end{proposition}
\begin{proof}
Write $b=\beta(0)=-1/d$ and
$q(g)=\beta(g)\beta(-g)/a(g)$. Thus $q(-g)=q(g)$. For the
unnormalized Fourier transform
\[
 (\mathcal Ff)(h)=\sum_{k\in G}\chi(h,k)^{-1}f(k),
\]
character orthogonality gives $\mathcal F^2q=nq$,
$\mathcal F\delta_0=\mathbf1$, and $\mathcal F\mathbf1=n\delta_0$.
The case $g=0$ of the cubic identity, divided by $a(h)/n$, is
\[
 \mathcal Fq=nbq+\frac nd\delta_0.
\]
Applying $\mathcal F$ again gives
\[
 (1-nb^2)q=\frac{nb}{d}\delta_0+\frac1d\mathbf1.
\]
Since $1-nb^2=n/d$, this is precisely
\[
 q=\frac1n\mathbf1-\frac1d\delta_0,
\]
which proves~\eqref{eq:betapm} and nonvanishing of every $\beta(g)$.

Put
\[
 C_h=\sum_{g\in G}\frac{\beta(g)\beta(h-g)}{a(h-g)}.
\]
The case $h=0$ of the cubic equation, with its other argument renamed,
now gives
\[
 C_h=nbq(h)+\frac nd\delta_{h,0}=b+\delta_{h,0},
\]
using $n/d+n/d^2=1$. Divide the cubic equation by $a(h-g)$ and sum
over $g$ to obtain
\[
 \frac Bn\sum_{k\in G}\chi(h,k)^{-1}\frac{\beta(k)}{a(k)}
 =\beta(-h)C_h+\frac{\delta_{h,0}}d=b\beta(-h).
\]
At $h=0$ the right side is $b^2\ne0$, so $B\ne0$.
Replacing $h$ by $-h$ proves~\eqref{eq:Fourier} for $c=b/B$.

Finally, set $S_a=\sum_g a(g)$, $S_{a^{-1}}=\sum_g a(g)^{-1}$,
and
\[
 (Tf)(h)=\sum_g a(g)^{-1}\chi(h,g)f(g).
\]
The refinement identity and character orthogonality give
\begin{equation}\label{eq:scalar-fourier-toolkit}
 S_aS_{a^{-1}}=n,\qquad
 (T^2)_{h,j}=S_{a^{-1}}a(h)\chi(h,j)^{-1},\qquad
 T^3=nS_{a^{-1}}I.
\end{equation}
Since $T\beta=cn\beta$ and $\beta\ne0$, we get
$c^3S_a=1/n$, proving~\eqref{eq:c_nonzero}.

Conversely, suppose the full scalar system holds. Its quadratic
identity gives $\beta(0)^2=1/n-1/d=1/d^2$. At $(g,h)=(0,0)$
the cubic identity then gives $\beta(0)^3=-1/d^3$, whence
$\beta(0)=-1/d$. The displayed summed cubic identity and the
Fourier equation force $c=b/B$, proving uniqueness.
\end{proof}

\subsection{An explicit construction for cyclic groups}
We construct the soutions to the TPE system for cyclic groups $ C_n $. We only need the fundamental properties of hyperbolic gamma functions as follows:
\begin{align}
  \gamma(z+\omega_1)
  &=
  2\sin\left(\frac{\pi z}{\omega_2}\right)\gamma(z),
  \label{eq:gamma-shift1}\\
  \gamma(z+\omega_2)
  &=
  2\sin\left(\frac{\pi z}{\omega_1}\right)\gamma(z),
  \label{eq:gamma-shift2}\\
  \gamma(z)\gamma(\Omega-z)
  &=
  1,
  \label{eq:gamma-reflection}\\
  \operatorname{Res}_{z=0}\gamma(z)
  &=
  \frac{\sqrt{\omega_1\omega_2}}{2\pi}
  =
  \frac{\sqrt n}{2\pi}.
  \label{eq:gamma-residue}
\end{align}
\begin{equation}\label{eq:gamma-positivity}
  \gamma(x)>0,
  \qquad 0<x<\Omega.
\end{equation}
Fix
\begin{equation}\label{eq:explicit-quadratic-data}
 \chi(g,h)=e^{2\pi i gh/n},\qquad
 a(g)=e^{\pi i g(n-g)/n}.
\end{equation}
These are well defined on $C_n$ and satisfy
\eqref{eq:apm}--\eqref{eq:achi}. The similarity between Equation~\eqref{eq:betapm} and Equation~\eqref{eq:gamma-reflection} implies that we can express $ \beta $ as a product of $ a(g)^{1/2} $ and the shift of gamma function. Hence, for $z\in\mathbb C$ and $x\in\mathbb Z$, put
\begin{equation}\label{eq:s-definition}
 F(z)=\frac1{\gamma(\omega_2+z)},\qquad
 s(x)=\exp\left(\frac{\pi i x(n-x)}{2n}\right).
\end{equation}
Write $\bar g\in\{0,\ldots,n-1\}$ for the standard representative
of $g\in C_n$.

\begin{proposition}\label{prop:F-identities}\label{prop:sF-well-defined}
The function $F$ satisfies
\[
 F(z)=\gamma(\omega_1-z),\qquad F(z)F(n-z)=1,\qquad
 F(z+n)=\frac{\sin(\pi z/\omega_1)}{\sin(\pi z/\omega_2)}F(z).
\]
For every nonzero integer $m$,
\[
 F(m+n)=(-1)^mF(m),\qquad F(-m)=\frac{(-1)^m}{F(m)}.
\]
Moreover,
\[
 F(0)=\frac{\sqrt n}{\omega_2},\qquad
 F(0)^2-1=\sqrt n\,F(0),\qquad F(k)>0\quad(0\leq k<n).
\]
The product $s(x)F(x)$ is $n$-periodic on
$\mathbb Z\setminus n\mathbb Z$. The zero residue class is exceptional:
$F(n)=F(0)^{-1}\ne F(0)$.
\end{proposition}
\begin{proof}
Reflection gives the first two identities, and the two gamma shifts give
\[
 \frac{F(z+n)}{F(z)}
 =\frac{\gamma(\omega_2+z)}{\gamma(\omega_1+z)}
 =\frac{\sin(\pi z/\omega_1)}{\sin(\pi z/\omega_2)}.
\]
At a nonzero integer $m$, the sines do not vanish and their ratio is
$(-1)^m$, since $1/\omega_2-1/\omega_1=1$.
Reflection then gives the formula for $F(-m)$, while
$s(m+n)=(-1)^m s(m)$ proves periodicity on the nonzero classes.
Taking $z\to0$ in~\eqref{eq:gamma-shift1} and using the residue at zero
gives $F(0)=\gamma(\omega_1)=\sqrt n/\omega_2$.
The equation $\omega_2^2+n\omega_2=n$ proves its quadratic identity,
and~\eqref{eq:gamma-positivity} proves positivity.
\end{proof}

Replacing $ \beta(g) $ with $ \frac{s(g)F(g)}{\sqrt{n}} $ in Equation~\eqref{eq:two_beta_identity} forces us to prove the following lemma.
\begin{lemma}\label{lem:S-residue-continuation}
For \(1\le g<n\) and \(0\le h<n\), define
\[
S_{g,h}
:=
\sum_{k=0}^{n-1}
(-1)^k
\exp\left(\frac{\pi i(g-2h)k}{n}\right)
F(k)F(g-k).
\]
Then
\begin{equation}\label{eq:finite-sum-F-final}
S_{g,h}
=
\sqrt{n}\,F(g)
\frac{s(\overline{-h})}{s(\overline{g-h})}
F(\overline{-h})F(\overline{h-g}).
\end{equation}
\end{lemma}

\begin{proof}

The proof is divided into three steps.
\begin{figure}[ht]
\centering
\begin{tikzpicture}[
    x=1.15cm,
    y=0.95cm,
    >=stealth,
    decoration={
        markings,
        mark=at position 0.5 with {\arrow{>}}
    }
]

\def\xsigma{-0.28}
\def\xright{8.50}
\def\ybottom{-1.35}
\def\ytop{1.35}

\draw[->,gray!60]
  (-0.75,0) -- (8.85,0)
  node[right] {$\mathrm{Re} (z)$};

\draw[->,gray!60]
  (0,\ybottom-0.15) -- (0,\ytop+0.18)
  node[above] {$\mathrm{Im} (z)$};

\draw[thick,postaction={decorate}]
  (\xsigma,\ybottom) -- (\xright,\ybottom);

\draw[thick,postaction={decorate}]
  (\xright,\ybottom) -- (\xright,\ytop);

\draw[thick,postaction={decorate}]
  (\xright,\ytop) -- (\xsigma,\ytop);

\draw[thick,postaction={decorate}]
  (\xsigma,\ytop) -- (\xsigma,\ybottom);

\node[above] at (\xsigma,\ytop) {$\mathrm{Re} (z)=\sigma$};
\node[above] at (\xright,\ytop) {$\mathrm{Re} (z)=\sigma+n$};

\foreach \x/\lab in {
  0/{$0$},
  1/{$1$},
  2/{$2$},
  3/{$\cdots$},
  4/{$g-1$},
  5/{$g$},
  6.4/{$\cdots$},
  7.8/{$n-1$}
}{
  \fill (\x,0) circle (1.5pt);
  \node[below=3pt] at (\x,0) {\lab};
}

\fill (0.38,0) circle (1.5pt);
\node[above=5pt] at (0.38,0) {$p_L$};

\fill (4.62,0) circle (1.5pt);
\node[above=5pt] at (4.62,0) {$p_1$};

\node[above=5pt] at (2.55,0) {$p_{L-1},\ldots,p_2$};

\node[below left] at (\xsigma,0) {$\sigma$};
\node[below right] at (\xright,0) {$\sigma+n$};

\end{tikzpicture}
\caption{The positively oriented residue rectangle
$\mathcal R_R$ for $ K_g(z) $ with vertical sides
$\mathrm{Re} (z)=\sigma$ and $\mathrm{Re} (z)=\sigma+n$.}
\label{fig:residue-strip}
\end{figure}
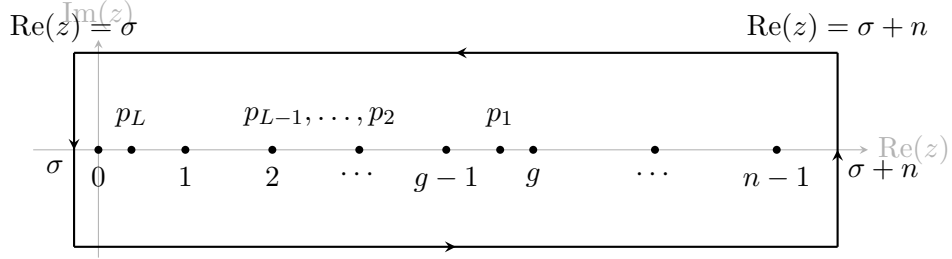

\begin{figure}[ht]
  \centering
  \begin{tikzpicture}[scale=1.2, >=Stealth]

    \def\sigmaVal{-0.3}    
    \def\r{0.25}           
    \def\yBot{-2.2}        
    \def\yTop{2.2}         
    \def\yReturn{0.5}      

    \def\pThree{0.8}       
    \def\pTwo{2.0}         
    \def\pOne{3.2}         
    \def\pZero{4.4}        

    \draw[->, thick, gray!60] (-2.2, 0) -- (5.6, 0) node[right, black] {$ \mathrm{Re} (z) $};
    \draw[->, thick, gray!60] (0, -2.5) -- (0, 2.5) node[above, black] {$\mathrm{Im} (z)$};
    \node[below left, gray!80!black] at (0,0) {$0$};

    \draw[dashed, gray!40] (\sigmaVal, \yBot) -- (\sigmaVal, \yTop);
    \node[below, black] at (\sigmaVal, 0) {$\sigma$};
    \fill (\sigmaVal, 0) circle (1.5pt);

    \foreach \p in {\pThree, \pTwo, \pOne, \pZero} {
        \fill[red!80!black] (\p, 0) circle (2pt);
    }
    
    \node[above=4pt] at (\pThree, 0) {$p_L$};
    \node[above=4pt] at (\pTwo, 0) {$\dots$};
    \node[above=4pt] at (\pOne, 0) {$p_1$};
    \node[above=4pt] at (\pZero, 0) {$p_0 = g$};

    \draw[
        line width=1.2pt, 
        blue!80!black,
        decoration={
            markings,
            mark=at position 0.08 with {\arrow{>}},
            mark=at position 0.28 with {\arrow{>}},
            mark=at position 0.58 with {\arrow{>}},
            mark=at position 0.78 with {\arrow{>}},
            mark=at position 0.95 with {\arrow{>}}
        },
        postaction={decorate}
    ]
        (\sigmaVal, \yBot) -- (\sigmaVal, 0.1)
        
        -- (\pThree - \r, 0.1) arc (160:380:\r)
        -- (\pTwo - \r, 0.1)   arc (160:380:\r)
        -- (\pOne - \r, 0.1)   arc (160:380:\r)
        -- (\pZero - \r, 0.1)  arc (160:380:\r)
        
        -- (\pZero + \r + 0.3, 0.1)
        -- (\pZero + \r + 0.3, \yReturn)
        -- (\sigmaVal, \yReturn)
        
        -- (\sigmaVal, \yTop);

    \node[blue!80!black, left] at (\sigmaVal, -1.2) {$C_{\text{sep}}(g)$};

\end{tikzpicture}
  \caption{The separating contour $C_{\rm sep}(g)$ for the two pole families
$g-j\omega_1-k\omega_2$ and $j\omega_1+k\omega_2$.}
\label{fig:separating-contour}
\end{figure}
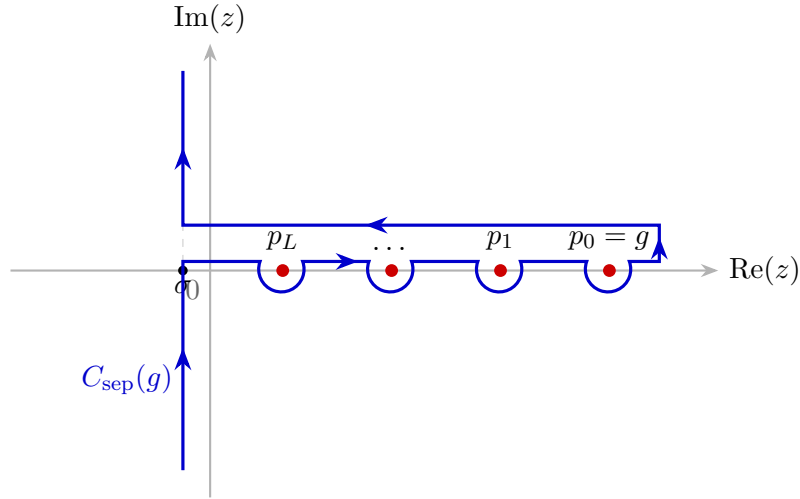

\medskip
\noindent\textbf{Step 1: the finite sum as a residue sum.}

Put
\[
A:=\frac{\pi(g-2h)}{n},
\]
and choose \(\phi\in(-\pi,\pi]\) such that
\begin{equation}\label{eq:phi-direct}
\phi\equiv\pi-A\pmod{2\pi}.
\end{equation}
Then
\[
(-1)^k e^{-i\phi k}=e^{iAk},
\qquad k\in\mathbb Z.
\]

Define
\[
R_g(z):=\frac{F(z)}{F(z-g)},
\qquad
K_g(z):=
\frac{\pi e^{-i\phi z}}{\sin\pi z}R_g(z),
\qquad
J_g(z):=
e^{-i\phi z}\gamma(-z)\gamma(z-g).
\]

Let
\[
L:=\left\lfloor\frac{g}{\omega_2}\right\rfloor
\]
and choose
\[
0<\varepsilon<
\min\{1,(L+1)\omega_2-g\},
\qquad
\sigma:=-\varepsilon.
\]
In the strip
\[
\sigma<\mathrm{Re} (z)<\sigma+n,
\]
the poles of \(K_g\) are
\[
0,1,\ldots,n-1
\]
and
\[
p_\ell:=g-\ell\omega_2,
\qquad
1\leq\ell\leq L.
\]
The points \(p_\ell\) are nonintegral by the irrationality of \(\omega_2\).

At an integer \(k\),
\[
\operatorname{Res}_{z=k}K_g(z)
=
e^{iAk}R_g(k).
\]
For \(k\neq g\),
\[
R_g(k)
=
(-1)^{g-k}F(k)F(g-k),
\]
whereas
\[
R_g(g)=\frac{F(g)}{F(0)}.
\]
Therefore
\begin{align}
\sum_{k=0}^{n-1}\operatorname{Res}_{z=k}K_g(z)
&=
(-1)^gS_{g,h}
+
(1-F(0)^2)e^{iAg}\frac{F(g)}{F(0)}.
\label{eq:integer-residues-direct}
\end{align}

The vertical asymptotics of the hyperbolic gamma function give, uniformly on the strip,
\[
R_g(x+iy)
=
O\left(e^{-\pi g|y|/n}\right),
\]
and hence
\[
K_g(x+iy)
=
O\left(
e^{-[\pi(1+g/n)-|\phi|]|y|}
\right).
\]
Thus the horizontal sides of the residue rectangle vanish as its height tends to infinity.

Using
\[
R_g(z)
=
-4
\sin\frac{\pi z}{\omega_2}
\sin\frac{\pi(z-g)}{\omega_1}
\gamma(-z)\gamma(z-g),
\]
together with
\[
\frac1{\omega_2}-\frac1{\omega_1}=1,
\]
one obtains
\[
K_g(z+n)-K_g(z)
=
-4\pi
\sin\frac{\pi g}{\omega_1}
J_g(z).
\]
Consequently, the residue theorem on the rectangle $\mathcal R_R$ (shown in Figure~\ref{fig:residue-strip}) gives
\begin{align}
&\sum_{k=0}^{n-1}\operatorname{Res}_{z=k}K_g(z)
+
\sum_{\ell=1}^{L}
\operatorname{Res}_{z=p_\ell}K_g(z)
=
-2\sin\frac{\pi g}{\omega_1}
\int_{\sigma-i\infty}^{\sigma+i\infty}
J_g(z)\frac{dz}{i}.
\label{eq:rectangle-direct}
\end{align}

Write
\[
r_\ell
:=
\operatorname{Res}_{z=p_\ell}J_g(z),
\qquad
0\leq\ell\leq L,
\]
where \(p_0=g\). A direct substitution of
\(p_\ell=g-\ell\omega_2\) gives
\[
\operatorname{Res}_{z=p_\ell}K_g(z)
=
4\pi
\sin\frac{\pi g}{\omega_1}\,r_\ell,
\qquad
1\leq\ell\leq L.
\]
Hence
\begin{equation}\label{eq:sigma-direct}
\sum_{k=0}^{n-1}\operatorname{Res}_{z=k}K_g(z)
=
-2\sin\frac{\pi g}{\omega_1}I_\sigma
-
4\pi\sin\frac{\pi g}{\omega_1}
\sum_{\ell=1}^{L}r_\ell,
\end{equation}
where
\[
I_\sigma
:=
\int_{\sigma-i\infty}^{\sigma+i\infty}
J_g(z)\frac{dz}{i}.
\]

Now deform the contour $ \mathrm{Re} (z) = \sigma $ to a separating contour \(C_{\rm sep}(g)\) (shown in Figure~\ref{fig:separating-contour}) for the two pole families of $ J_g(z) $
\[
g-j\omega_1-k\omega_2
\qquad\text{and}\qquad
j\omega_1+k\omega_2.
\]
The deformation crosses precisely
\[
p_0,p_1,\ldots,p_L,
\]
and therefore
\[
I_{\rm sep}
=
I_\sigma
+
2\pi\sum_{\ell=0}^{L}r_\ell.
\]
Substituting this into \eqref{eq:sigma-direct}, all nonintegral residues cancel, and we obtain
\begin{equation}\label{eq:direct-Isep}
\sum_{k=0}^{n-1}\operatorname{Res}_{z=k}K_g(z)
=
-2\sin\frac{\pi g}{\omega_1}\,I_{\rm sep}
+
4\pi\sin\frac{\pi g}{\omega_1}\,r_0.
\end{equation}

\medskip
\noindent\textbf{Step 2: meromorphic continuation of the separating-contour integral.}

For independent complex parameters \(g,\alpha\), consider
\begin{equation}\label{eq:BSS-family}
\mathcal I(g,\alpha)
:=
\int_{C_{\rm sep}(g)}
e^{\pi i z(2\alpha-g-\Omega)/n}
\gamma(z-g)\gamma(-z)\frac{dz}{i\sqrt n}.
\end{equation}
The contour $C_{\rm sep}(g)$ is precisely the separating contour used
in $I_{\rm sep}$.
For \(\mathrm{Re} (g)<0\), the two pole families
\[
g-j\omega_1-k\omega_2
\]
and
\[
j\omega_1+k\omega_2
\]
are separated by a vertical contour $ \mathrm{Re} (z) = \sigma $. The integral is therefore well defined in a nonempty open region of the \((g,\alpha)\)-plane where the convergence conditions of the BSS formula are satisfied.

On compact subsets avoiding pole collisions, the separating contour may be chosen continuously in the parameters, and the vertical asymptotics of the hyperbolic gamma function imply locally uniform convergence. Hence \(\mathcal I(g,\alpha)\) is holomorphic there. Continuing the contour across a pole changes the integral by the corresponding residue; consequently \(\mathcal I(g,\alpha)\) admits a meromorphic continuation in the parameters.

A collision between the two pole families can occur only when
\[
g\in\omega_1\mathbb Z+\omega_2\mathbb Z.
\]
By
\[
(\omega_1\mathbb Z+\omega_2\mathbb Z)\cap\mathbb Z=n\mathbb Z,
\]
none of \(1,\ldots,n-1\) is a collision point. Thus the separating-contour integral may be continued from \(\mathrm{Re} (g)<0\) to each
\[
g=1,\ldots,n-1.
\]

The BSS identity (Lemma~\ref{lem:BSS_identity}) gives
\begin{align}
\mathcal I(g,\alpha)
&=
e^{\pi i g(2\alpha-g-\Omega)/(2n)}
\gamma(\alpha)
\gamma(\Omega+g-\alpha)
\gamma(-g).
\label{eq:BSS-meromorphic}
\end{align}
Both sides are meromorphic in \(g\) and \(\alpha\), so the identity remains valid under the above continuation.

For our fixed integers \(g,h\), choose
\[
\alpha
=
\frac{\Omega+g}{2}
-\frac{n\phi}{2\pi}.
\]
Then
\[
\frac{\pi i}{n}(2\alpha-g-\Omega)z
=
-i\phi z,
\]
and hence
\begin{equation}\label{eq:Isep-direct}
I_{\rm sep}
=
\sqrt n\,
e^{-ig\phi/2}
\gamma(-g)
\gamma(\alpha)
\gamma(\Omega+g-\alpha).
\end{equation}

Moreover,
\[
r_0
=
\frac{\sqrt n}{2\pi}
e^{-i\phi g}\gamma(-g).
\]
Using
\[
2\sin\frac{\pi g}{\omega_1}\gamma(-g)
=
(-1)^{g+1}F(g),
\]
equation \eqref{eq:direct-Isep} becomes
\begin{align}
\sum_{k=0}^{n-1}\operatorname{Res}_{z=k}K_g(z)
&=
(-1)^g\sqrt n\,
e^{-ig\phi/2}
F(g)
\gamma(\alpha)
\gamma(\Omega+g-\alpha)
+
(-1)^{g+1}\sqrt n\,
e^{-i\phi g}F(g).
\label{eq:residues-BSS-direct}
\end{align}
Since
\[
e^{-i\phi g}
=
(-1)^ge^{iAg}
\]
and
\[
\frac{1-F(0)^2}{F(0)}
=
-\sqrt n,
\]
the second term in \eqref{eq:residues-BSS-direct} is exactly the correction term in \eqref{eq:integer-residues-direct}. Cancelling these correction terms gives
\begin{equation}\label{eq:S-BSS-direct}
S_{g,h}
=
\sqrt n\,
e^{-ig\phi/2}
F(g)
\gamma(\alpha)
\gamma(\Omega+g-\alpha).
\end{equation}

\medskip
\noindent\textbf{Step 3: the special choice of \(\alpha\).}

Set
\[
\eta
=
\begin{cases}
0,&2h\leq g,\\
1,&2h>g.
\end{cases}
\]
Since
\[
\phi
\equiv
\pi-\frac{\pi(g-2h)}n
\pmod{2\pi},
\]
we have
\[
\phi
=
\pi-\frac{\pi(g-2h)}n-2\pi\eta.
\]
Using
\[
\Omega-n=2\omega_2,
\qquad
\Omega+n=2\omega_1,
\]
we obtain
\begin{align}
\alpha
&=
\omega_2+g-h+n\eta,
\label{eq:alpha-special}\\
\Omega+g-\alpha
&=
\omega_1+h-n\eta.
\label{eq:beta-special}
\end{align}
Therefore
\begin{equation}\label{eq:gamma-product-direct}
\gamma(\alpha)\gamma(\Omega+g-\alpha)
=
\frac{F(-h+n\eta)}
     {F(g-h+n\eta)}.
\end{equation}

The integer identities
\[
F(x)F(n-x)=1,
\qquad
F(x+n)=(-1)^xF(x)\quad(x\neq0),
\]
together with
\[
s(x+n)=(-1)^x s(x),
\]
give
\begin{equation}\label{eq:phase-conversion-direct}
e^{-ig\phi/2}
\frac{F(-h+n\eta)}
     {F(g-h+n\eta)}
=
\frac{s(\overline{-h})}
     {s(\overline{g-h})}
F(\overline{-h})F(\overline{h-g}).
\end{equation}
The boundary case \(g-h=0\) is interpreted using
\[
F(n)=F(0)^{-1}.
\]
Thus \eqref{eq:phase-conversion-direct} holds for all
\(0\leq h<n\).

Substituting \eqref{eq:phase-conversion-direct} into
\eqref{eq:S-BSS-direct} proves
\[
S_{g,h}
=
\sqrt n\,F(g)
\frac{s(\overline{-h})}
     {s(\overline{g-h})}
F(\overline{-h})F(\overline{h-g}),
\]
as required.
\end{proof}

\begin{theorem}
    Let $\chi, a$ and $ s $ be defined as in \eqref{eq:explicit-quadratic-data} and \eqref{eq:s-definition}, and Let 
    \begin{align*}
      \beta(k) = \frac{s(k)F(k)}{\sqrt{n}}, \qquad 0\le k < n,
    \end{align*}
    Then $ d_{-}, \chi, a, \beta $ satisfy the TPE system.
\end{theorem}
\begin{proof}
  It suffices to verify Equation~\eqref{eq:two_beta_identity} for $ g\ne 0 $.
By the definition of \(\beta\),
\[
\begin{aligned}
&\sum_{k\in G}
\overline{\chi(h,k)}
\frac{\beta(g-k)\beta(k)}{a(k)}
=
\frac1n
\sum_{k=0}^{n-1}
e^{-2\pi i hk/n}
\frac{s(\overline{g-k})}
     {s(k)}
F(\overline{g-k})F(k).
\end{aligned}
\]
Using the \(n\)-periodicity of \(s(x)F(x)\) on integer arguments,
\[
\frac{s(\overline{g-k})}{s(k)}
F(\overline{g-k})
=
s(g)(-1)^k
e^{\pi i gk/n}
F(g-k),
\]
and hence
\begin{equation}\label{eq:cubic-S-direct}
\sum_{k\in G}
\overline{\chi(h,k)}
\frac{\beta(g-k)\beta(k)}{a(k)}
=
\frac{s(g)}nS_{g,h}.
\end{equation}

Applying Lemma~\ref{lem:S-residue-continuation},
\[
\begin{aligned}
\frac{a(h-g)}n
\sum_{k\in G}
\overline{\chi(h,k)}
\frac{\beta(g-k)\beta(k)}{a(k)}&=
\frac{a(h-g)s(g)}{n^2}
S_{g,h}
\\
&=
\frac{
s(g)
s(\overline{h-g})
s(\overline{-h})
}{
n^{3/2}
}
F(g)
F(\overline{h-g})
F(\overline{-h}).
\end{aligned}
\]
Since
\[
a(h-g)=s(\overline{h-g})^2,
\]
the last expression is exactly
\[
\beta(-h)\beta(g)\beta(h-g).
\]
Thus
\[
\frac{a(h-g)}n
\sum_{k\in G}
\overline{\chi(h,k)}
\frac{\beta(g-k)\beta(k)}{a(k)}
=
\beta(-h)\beta(g)\beta(h-g),
\]
which is the desired identity for \(g\neq0\).
\end{proof}

\begin{remark}
  This construction yields a spherical near-group fusion category, while its Galois conjugate produces a pseudo-unitary category. Despite Belousov being a colleague of the authors at Beijing Institute of Mathematical Sciences and Applications, the BSS identity was only brought to our attention through AI-assisted exploration.
\end{remark}

\begin{proposition}[All even quadratic refinements]\label{prop:abelian-refinements}
Let $G=\bigoplus_{i=1}^r\mathbb Z/m_i\mathbb Z\,e_i$ and let
$\chi:G\times G\to\mathbb T$ be a symmetric bicharacter. Choose
$\xi_i$ with
\[
 \xi_i^2=\chi(e_i,e_i),\qquad \xi_i^{m_i^2}=1.
\]
Then
\begin{equation}\label{eq:abelian-quadratic-formula}
 a\!\left(\sum_i x_ie_i\right)
 =\prod_i\xi_i^{-x_i^2}
  \prod_{i<j}\chi(e_i,e_j)^{-x_ix_j}
\end{equation}
is a well-defined even quadratic refinement of $\chi$. Such refinements
form a torsor for $\operatorname{Hom}(G,\{\pm1\})$; in particular
there are exactly $|G[2]|$ of them. If $|G|$ is odd, the unique one is
$a(g)=\chi(g/2,g)^{-1}$.
\end{proposition}
\begin{proof}
If $m_i$ is odd, there is a unique permissible square root,
$\xi_i=\chi(e_i,e_i)^{(m_i+1)/2}$. If $m_i$ is even, both square
roots satisfy the required power condition. Increasing $x_i$ by $m_i$
does not change the diagonal factor because
$\xi_i^{2m_ix_i+m_i^2}=1$, or a mixed factor because
$\chi(e_i,e_j)^{m_i}=1$. Thus the formula is well defined. Negation
preserves it, and expansion of $(x_i+y_i)^2$ gives
$a(x+y)\chi(x,y)=a(x)a(y)$, including the mixed terms.
The quotient of two refinements is a character. Evenness says that
this character has square $1$, and conversely multiplying by any
such character preserves all requirements. Finally,
$|\operatorname{Hom}(G,\{\pm1\})|=|G/2G|=|G[2]|$.
For odd-order $G$, division by $2$ is unique; the stated formula has
the required polarization, and uniqueness follows from $G[2]=0$.
\end{proof}

\begin{lemma}[Algebraic reflection on scalar solutions]
\label{lem:algebraic-reflection}
The rational transformation
\begin{equation}\label{eq:parameter-reflection}
 \tau(d,\chi,a,c,\beta)
 =\left(d,\chi^{-1},a^{-1},\frac1{nc},
        \left(g\longmapsto\frac{\beta(-g)}{a(g)}\right)\right)
\end{equation}
is an involution of the solution set of Definition~\ref{thm: equation},
over every field allowed there.
\end{lemma}
\begin{proof}
The dimension and quadratic-refinement data are preserved by substitution,
as is $\beta(0)=-1/d$. If $E(g,h)$ denotes the left side minus the
right side of~\eqref{eq:two_beta_identity}, substitution and the change
$k\mapsto-k$ give
\[
 E_\tau(g,h)=\frac{E(-g,h-g)}{a(g)a(h)a(h-g)}.
\]
The denominator equals $1$ on the support of the exceptional term.
Proposition~\ref{prop:cubic-reduction} therefore supplies all the remaining
equations. To identify the new value of $c$, use the original Fourier
equation at zero and evenness of $a$:
\[
 \sum_g\beta_\tau(g)=\sum_g\frac{\beta(g)}{a(g)}=cn\beta(0),
 \qquad
 c_\tau=\frac{\beta_\tau(0)}{\sum_g\beta_\tau(g)}=\frac1{nc}.
\]
Applying $\tau$ twice fixes every parameter.
\end{proof}

\begin{lemma}[Fourier form of the cubic identity]\label{lem:izumi-cubic-dictionary}
Assume the quadratic and Fourier equations and $\beta(0)=-1/d$.
Then~\eqref{eq:two_beta_identity} is equivalent to
\begin{equation}\label{eq:izumi-cubic-form}
 \sum_{t\in G}\frac{\beta(-t)\beta(t+u)\beta(t+v)}{a(t)}
 =\chi(u,v)\beta(u)\beta(v)-\frac1{cdn}
 \qquad(u,v\in G).
\end{equation}
Thus Izumi's parameters in \cite[Remark~9.2]{Izumi2017} are
$\langle g,h\rangle_I=\chi(g,h)^{-1}$,
$a_I(g)=a(g)^{-1}$, $c'_I=c$, and $b_I(g)=\beta(g)$.
\end{lemma}
\begin{proof}
Let $E(g,h)$ be the residual of~\eqref{eq:two_beta_identity} and let
$D(u,v)$ be the residual of~\eqref{eq:izumi-cubic-form}, in each case
left side minus right side. Expanding $\beta(t+v)$ by the Fourier
equation and applying the product relation gives
\[
 D(u,v)=\frac1c\sum_{s\in G}\chi(v,s)
       \frac{\beta(s)}{a(s)a(s-u)}E(u,s).
\]
Every $\beta(s)$ is nonzero by the product relation and
$\beta(0)=-1/d$. Character orthogonality gives the inverse
\[
 E(u,s)=\frac{c\,a(s)a(s-u)}{n\beta(s)}
         \sum_{v\in G}\chi(v,s)^{-1}D(u,v).
\]
Thus the two residual systems vanish simultaneously over $K$.
\end{proof}

Over $\mathbb C$, the additional unitary conditions are
$d=(n+\sqrt{n^2+4n})/2$ and
$\overline{\beta(g)}=\beta(-g)/a(g)$.
Under the displayed dictionary, Izumi's antilinear map is
$J_I f(g)=\overline{a_I(g)f(-g)}$, and these are precisely his
additional conditions in \cite[Remark~9.2]{Izumi2017}.
Thus the unitary parameter systems agree for every finite abelian group.
Equivalence classes still require the group-automorphism identifications
of \cite[Theorem~9.1]{Izumi2017}.

\section{Planar algebraic notation for spherical fusion categories}\label{sec: prelim}
We restrict our attention to spherical fusion categories, as Thornton's theorem \cite[Theorem~1.3]{Thornton} guarantees that every near-group category is spherical. Following \cite[Remark~4.20]{LPR}, for the categorification problem, we may assume without loss of generality that both the monoidal and pivotal structures are strict. As noted in \cite[Remark~2.17]{LPR}, even in the strict setting, a simple object satisfying $X \cong X^*$ with second Frobenius--Schur indicator $\nu(X) = -1$ cannot be strictly self-dual (i.e., $X^* \neq X$). Therefore, we further assume that the Frobenius--Schur indicator is $1$ for every self-dual simple object. Together, these strictness assumptions and the triviality of the Frobenius--Schur indicator allow us to adopt a planar algebraic framework to analyze the near-group categorification problem.

Let $G$ be a finite abelian group of order $n \ge 1$, written additively, and let $\mathcal{C}$ be a spherical near-group fusion category over $\mathbb{C}$. The simple objects of $\mathcal{C}$ are indexed by $I = G \sqcup \{\varrho\}$, with fusion rules
\begin{equation}\label{eq:NGfusion}
    g \otimes h = g + h, \qquad
    g \otimes \varrho = \varrho = \varrho \otimes g, \qquad
    \varrho \otimes \varrho = \bigoplus_{g \in G} g \oplus n\varrho,
\end{equation}
where $0 \in G$ denotes the monoidal unit $\mathbf{1}$, $g^* = -g$, and $\varrho^* = \varrho$. Denoting the spherical dimension of simple object $X \in I$ by $d_X$, we have
\begin{equation}\label{eq:dimension}
    d_g = 1, \qquad d = d_\varrho = \frac{n \pm \sqrt{n^2 + 4n}}{2}, \qquad d^2 = nd + n.
\end{equation}

\subsection{String diagrams and invariant trivalent vertices}\label{subsec:string-conventions}
Let $ X $ be an object in $ \cC $. We use the following notations:
$$
\raisebox{-.5cm}{
\begin{tikzpicture}
\draw (0,-.5)--(0,0) node [right] {\tiny $X$};
\draw[<-] (0,0)--(0,.5);
\end{tikzpicture}}:=
\raisebox{-.5cm}{
\begin{tikzpicture}
\draw (0,-.5)--(0,0) node [right] {\tiny $X$};
\draw  (0,0)--(0,.5);
\end{tikzpicture}}
=\id_{X}
\ \ \text{ and } \ \
\raisebox{-.5cm}{
\begin{tikzpicture}
\draw[->] (0,-.5)--(0,0) node [right] {\tiny $X$};
\draw (0,0)--(0,.5);
\end{tikzpicture}}:=
\raisebox{-.5cm}{
\begin{tikzpicture}
\draw (0,-.5)--(0,0) node [right] {\tiny $X^*$};
\draw (0,0)--(0,.5);
\end{tikzpicture}} = \id_{X^*}.
$$
A group-labelled simple object $g\in G$ is represented by a strand labelled by $g$ and oriented from top to bottom. Reversing the orientation replaces the label by the dual object:
$
g^*=-g.
$
The unique non-invertible simple object $\varrho$ is represented by an unlabeled and unoriented strand.

We have the oriented evaluation and coevaluation maps, see \cite[Section~2.2]{LPR}. Let us depict these morphisms as follows
$$
\text{ev}_{X,+}=
\raisebox{-.4cm}{
\begin{tikzpicture}
\draw (0,0)  arc (-180:-90:.5);
\draw [<-] (.5,-.5)  node [right = .25cm] {\tiny $X$} arc (-90:0:.5);
\end{tikzpicture}}
, \ \
\text{ev}_{X,-}=
\raisebox{-.4cm}{
\begin{tikzpicture}
\draw[->] (0,0)  arc (-180:-90:.5);
\draw (.5,-.5)  node [left = .25cm] {\tiny $X$} arc (-90:0:.5);
\end{tikzpicture}}
\ , \ \
\text{coev}_{X,+}= \hspace{-.2cm}
\raisebox{-.4cm}{
\begin{tikzpicture}
\draw[<-] (0,0) arc (90:0:.5);
\draw (-.5,-.5) node [left = -.05cm] {\tiny $X$} arc (180:90:.5);
\end{tikzpicture}}
\ , \ \
\text{coev}_{X,-}=
\raisebox{-.4cm}{
\begin{tikzpicture}
\draw (0,0) arc (90:0:.5) node [right = -.05cm] {\tiny $X$};
\draw[->] (-.5,-.5)  arc (180:90:.5);
\end{tikzpicture}}.
$$

The following dictionary illustrates the correspondence between categorical string diagrams and planar algebraic diagrams for a morphism $f \in \operatorname{Hom}(X \otimes Y, Z)$ and an invariant vector $u \in \operatorname{Hom}(\mathbf{1}, X \otimes Y \otimes Z)$:
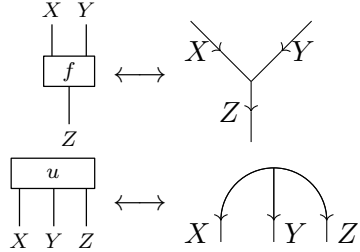
\begin{align*}
  \SD{morphism-f} &\longleftrightarrow \begin{tikzpicture}[scale=0.8,baseline={(current bounding box.center)}]
      \draw[->] (-1,1) -- (-0.5,0.5); \draw (-0.5,0.5) -- (0,0);
  \draw[->] (1,1)  -- (0.5,0.5);  \draw (0.5,0.5)  -- (0,0);
  \draw[->] (0,0)  -- (0,-0.5);   \draw (0,-0.5)   -- (0,-1);
  \node at (-0.5,0.5)[left] {$ X $};
  \node at (0.5,0.5)[right] {$ Y $};
  \node at (0,-0.5)[left] {$ Z $};
  \end{tikzpicture}\\
  \SD{invariant-u} &\longleftrightarrow \begin{tikzpicture}[scale=1.4,baseline={(current bounding box.center)}]
\draw (-.5,0) arc[radius=.5, start angle=180, end angle=0];
\draw (0,.5)--(0,0);
\draw [->](0,.5) arc[radius=.5, start angle=90, end angle=180];
\draw [->](0,.5) arc[radius=.5, start angle=90, end angle=0];
\draw [->](0,.5)--(0,0);
\draw (0,0)--(0,-.2);
\draw (.5,0)--(.5,-.2);
\draw (-.5,0)--(-.5,-.2);
\node at (-.5,-.12)[left] {$X$};
\node at (0,-.12)[right] {$Y$};
\node at (.5,-.12)[right] {$Z$};
\end{tikzpicture}
\end{align*}
Unlike categorical diagrams, which strictly enforce a top-to-bottom reading direction, planar algebraic diagrams can be rotated and interpreted from arbitrary directions via Frobenius reciprocity. Consequently, we can uniformly encode trivalent vertices using the space $\operatorname{Hom}(\mathbf{1}, X \otimes Y \otimes Z)$.

Set $V(a,b,c) = \operatorname{Hom}(\mathbf{1}, a \otimes b \otimes c)$ and $V(a,b,c)^\vee = \operatorname{Hom}(a \otimes b \otimes c, \mathbf{1})$. The spaces $V(a,b,c)$ and $V(a,b,c)^\vee$ are naturally paired via vertical diagrammatic composition. Specifically, for any $u \in V(a,b,c)$ and $v \in V(a,b,c)^\vee$, their composition forms a closed planar diagram identified with a scalar:
\[
\langle v, u \rangle = v \circ u \in \operatorname{End}(\mathbf{1}) \cong K.
\]
We therefore view the canonical bilinear map
\begin{equation}\label{eq:trivalent-pairing}
V(a,b,c)^\vee \times V(a,b,c) \longrightarrow K
\end{equation}
as the fundamental planar pairing. Whenever a basis for $V(a,b,c)$ is chosen, the basis for the opposite trivalent space $V(a,b,c)^\vee$ is taken to be its dual basis with respect to this pairing.

\subsection{The all-\(\varrho\) vertex}
\label{subsec:rho-vertex}

The situation is fundamentally different for trivalent vertices involving three $\varrho$-strands. The fusion rules yield
\[
\dim \operatorname{Hom}(\varrho \otimes \varrho, \varrho) = n.
\]
In Section~\ref{sec:fsymbols}, we establish that this hom-space admits a canonical $G$-indexed basis. Consequently, we decorate a corner of the all-$\varrho$ trivalent vertex with a group element $g \in G$ to specify a basis vector. Recall that individual $\varrho$-strands carry neither object labels nor arrows. The all-$\varrho$ vertex is depicted as
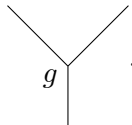
\[
\begin{tikzpicture}[scale=0.8, baseline={(current bounding box.center)}]
    \draw (-1,1) -- (0,0);
    \draw (1,1)  -- (0,0);
    \draw (0,0)  -- (0,-1);
    \node at (0,-0.2)[left] {$g$};
\end{tikzpicture}.
\]
Here, the marked corner is an essential part of the diagrammatic notation. Shifting the mark to an adjacent corner corresponds to applying a $120^\circ$ cyclic rotation to the trivalent vertex.

Throughout this paper, we carefully distinguish between two distinct uses of $G$-labels:
\begin{itemize}
    \item A label $g \in G$ placed \emph{on a strand} specifies an invertible simple object.
    \item A label $g \in G$ placed \emph{at a corner} of an all-$\varrho$ vertex indexes a basis element of $\operatorname{Hom}(\varrho \otimes \varrho, \varrho)$.
\end{itemize}
These two roles carry fundamentally different mathematical meanings and should not be confused. 

\subsection{Tetrahedron and triangular prism graphs}\label{subsec:tetra-coupons}
An admissible invariant vertex has zero, two or three $\varrho$-edges,
and its space has dimension $1$, $1$ or $n$. At an all-invertible
vertex the oriented group labels sum to zero.

We use the labelled spherical graphs of \cite[\S4.3]{LPR}.
The categorical meaning of a tetrahedron is illustrated as follows
\begin{equation}\label{eq:tetra-coupon-definition}
 T(\alpha,\beta,\gamma,\delta)
   =\SD{spherical-tetrahedron-reference}
   =\SD{tetrahedron-coupon-dictionary}.
\end{equation}

We adopt the simplified spherical convention of
\cite[Theorem~4.21 and Figure~1]{LPR} for a triangular prism graph. A prism has edge labels
$X_1,\ldots,X_9\in I$ and vertex vectors
\begin{equation}\label{eq:TP-vertex-spaces}
\begin{aligned}
 \alpha_1&\in V(X_4,X_1,X_6^*),&
 \alpha_2&\in V(X_5,X_2,X_4^*),&
 \alpha_3&\in V(X_6,X_3,X_5^*),\\
 \alpha_4&\in V(X_9^*,X_1^*,X_7),&
 \alpha_5&\in V(X_7^*,X_2^*,X_8),&
 \alpha_6&\in V(X_8^*,X_3^*,X_9).
\end{aligned}
\end{equation}
We denote these data by $\operatorname{TP}(\alpha_1,\ldots,\alpha_6)$,
with the edges understood from the displayed picture.
\begin{figure}[H]
\centering
\begin{equation}\label{eq:prism-coupon-definition}
 \SD{spherical-prism-reference}
   =\SD{prism-coupon-dictionary}.
\end{equation}
\caption{The simplified spherical prism and its categorical meaning}
\label{fig:reference-graphs}
\end{figure}

\subsection{Triangular prism equations}\label{sec: pentagon Equation}
\begin{notation}
For a labelled prism $P$, write $\TPL(P)$ for the contraction into two
tetrahedra, summed over a basis of $V(X_3,X_2,X_1)$ and its bilinear
dual. Write $\TPR(P)$ for its contraction into three tetrahedra, with
internal simple label in
\begin{equation}\label{eq:TP-spectrum}
 \mathcal O(P)=\bigcap_{r=4}^6
 \{X\in I:N_{X_r,X_{r+3}}^X>0\}.
\end{equation}
The latter contraction sums over dual bases of
$V(X_r^*,X,X_{r+3}^*)$, $r=4,5,6$, and carries the factor $d_X$.
With the orientations and corner transports of
\cite[Theorem~4.21]{LPR}, the TPE is
\begin{equation}\label{eq:TPE}
 \TPL(P)=\TPR(P).
\end{equation}
In the pictorial calculations below, the symbol $\tpl$ or $\tpr$ indicates
which of these two TPE-style contractions is being evaluated.
\end{notation}

The two cuts take the simplified form below. Here $(b_s)$ is a basis
of $V(X_3,X_2,X_1)$, and $b_i$ runs independently through a basis of
$V(X_{i+3}^*,X,X_{i+6}^*)$, $i=1,2,3$; primes denote paired duals:
\begin{equation}\label{eq:TPL-coupon-cut}
 \TPL(P)=\sum_s
     \SD{TPE-left-first}\;\SD{TPE-left-second}.
\end{equation}
\begin{equation}\label{eq:TPR-coupon-cut}
 \begin{gathered}
 \TPR(P)=\sum_{X\in\mathcal O(P)}d_X\sum_{b_1,b_2,b_3}
          T_1T_2T_3,\\[2pt]
 \underbrace{\SD{TPE-right-1}}_{T_1}\quad
 \underbrace{\SD{TPE-right-2}}_{T_2}\quad
 \underbrace{\SD{TPE-right-3}}_{T_3}.
 \end{gathered}
\end{equation}

\section{Tetrahedral coefficients}\label{sec:fsymbols}
We determine the tetrahedral coefficients by the number of invertible
edge labels. {Until the final subsection, $\mathcal C$ is a
unitary near-group category with fusion rules {$G+n$}, equipped with
its canonical spherical structure, and $d=d_\varrho>0$. The final
subsection extends the resulting tetrahedral table to either dimension
root without a Hermitian structure.} 

\subsection{Three or more invertible edges}
The mixed trivalent spaces are one-dimensional. We first fix their
phases coherently, since normalizing their lengths alone does not fix
associativity coefficients.

\begin{lemma}[Coherent mixed vertices]\label{lem:mixed-gauge}
There are unitary morphisms
\[
 m_{g,h}:g\otimes h\longrightarrow g+h,\qquad
 L_g:g\otimes\varrho\longrightarrow\varrho,\qquad
 R_g:\varrho\otimes g\longrightarrow\varrho
\]
for which the associativity coefficients on $(g,h,k)$, $(g,h,\varrho)$,
and $(\varrho,g,h)$ are $1$.
\end{lemma}
\begin{proof}
Choose the $L_g$ unitarily, with $L_0$ the unit constraint. Define
$m_{g,h}$ by requiring the two composites
$(g\otimes h)\otimes\varrho\to\varrho$ to agree after inserting the
associator. This determines a unitary $m_{g,h}$, since the relevant
morphism spaces are one-dimensional. The pentagon on $(g,h,k,\varrho)$
then makes the $m_{g,h}$ associative. Thus the pointed associator and
the left module associator are simultaneously trivialized.

The right maps must also respect the bends. Let
$t:\mathbf1\to\varrho\otimes\varrho$ be the standard coevaluation, so
$t^\dagger t=d$. Define $U_g:\varrho\to\varrho\otimes g$ by
\begin{equation}\label{eq:mixed-mate-normalization}
 (\operatorname{id}_\varrho\otimes L_g)
 (U_g\otimes\operatorname{id}_\varrho)t=t,
 \qquad R_g=U_g^\dagger.
\end{equation}
Rigidity makes the displayed assignment from $U_g$ to its left-hand
side an isomorphism of one-dimensional morphism spaces, so $U_g$
exists uniquely. If $U_g^\dagger U_g=s\operatorname{id}_\varrho$,
taking the squared norm in this equation gives $sd=d$; hence $U_g$
is unitary. Apply the same defining equation successively to $U_h$
and $U_g$, and use the left module relation. Uniqueness gives
\[
 (\operatorname{id}_\varrho\otimes m_{g,h})
 (U_g\otimes\operatorname{id}_h)U_h=U_{g+h}.
\]
Taking adjoints proves the right module relation. This is the
categorical normalization of \cite[Remark~3.2]{Izumi2017}; it fixes
both module associativity and the mixed bends, rather than choosing
their phases independently. Associators have been suppressed in
these formulas by coherence.
\end{proof}

\begin{proposition}[Normalization]\label{pro: normalization}
Let $a,b,c$ be simple objects with
$\operatorname{Hom}(\mathbf1,a\otimes b\otimes c)\ne0$, and suppose
that one is invertible. Choose opposite vertices
$\beta:\mathbf1\to a\otimes b\otimes c$ and
$\alpha:a\otimes b\otimes c\to\mathbf1$ with
\begin{equation}\label{equ: normalisation}
 \NGDiagram{D002}=\theta(a,b,c),\qquad
 \theta(a,b,c)=
 \begin{cases}
 1,&a,b,c\text{ are invertible},\\
 d,&\text{two of }a,b,c\text{ equal }\varrho.
 \end{cases}
\end{equation}
We suppress the vertex label on the resulting mixed vertex
$\NGDiagram{D003}$ and its prescribed rotations. For every simple
$a$ and $g\in G$, the associated fusion and splitting maps satisfy
\begin{equation}\label{eq:mixed-splitting}
 \NGDiagram{D004}=\NGDiagram{D005}.
\end{equation}
\end{proposition}
\begin{proof}
The standard trace of an identity on $a\otimes g$ is $d_a$.
Choose its fusion map to the simple object $a\otimes g$ unitarily and
its splitting map as its adjoint. Their composition is the identity,
which gives \eqref{eq:mixed-splitting} (the label $ag$ denotes
$a\otimes g$); bending one leg gives the theta
value $d_a$. This is $1$ or $d$, as in \eqref{equ: normalisation}.
\end{proof}

\begin{lemma}\label{lem: coproduct}
For $g,h,k,g',h',k'\in G$ satisfying
$g+h=g'+h'$ and $k-h=k'-h'$, the coherent mixed vertices give
\begin{align}
 \NGDiagram{D006}&=\NGDiagram{D007},\label{equ: co 6g}\\
 \NGDiagram{D008}&=\NGDiagram{D009}.\label{equ: co 3g}
\end{align}
\end{lemma}
\begin{proof}
Expand the vertices into the multiplication and module maps fixed in
Lemma~\ref{lem:mixed-gauge}. Cancel each adjacent fusion--splitting
pair by \eqref{eq:mixed-splitting}. In \eqref{equ: co 6g}, the two
remaining products have total label
$g+k=g'+k'$; their equality is associativity of $m$. In
\eqref{equ: co 3g}, the remaining identity is the module relation
for $g,h$, with product $g+h$. These cancellations give the displayed
identities with scalar $1$; one-dimensionality by itself would only
give proportionality.
\end{proof}

\begin{lemma}\label{lem: F-symbol 6 inv}
For $g,h,k\in G$, the coefficients with six or three invertible edges are
\begin{equation}
 \NGDiagram{D010}=1,\qquad \NGDiagram{D011}=d.
\end{equation}
\end{lemma}
\begin{proof}
The first tetrahedron contracts coherent group multiplications and their
inverses, so its value is $1$. For the second, the mixed module
relations collapse the tetrahedron to a $\varrho$-loop, of value $d$.
The closed version of the cancellation is
\begin{equation}
 \NGDiagram{D012}=\NGDiagram{D013}=d^2.
\end{equation}
Here the mixed vertex $\NGDiagram{D014}$ is the one fixed by
Lemma~\ref{lem:mixed-gauge}. In particular, {iterating the module relation $\operatorname{ord}(g)$ times gives the unit constraint for every $g\in G$, since $\operatorname{ord}(g)g=0$}; no further independent phase adjustment is needed.
\end{proof}

\subsection{Two invertible edges}
Interchanging the left and right group actions produces the bicharacter
that controls the remaining coefficients.
\begin{lemma}\label{lem: bicha}
The mixed interchange coefficient is a {nondegenerate symmetric} bicharacter
$\chi:G\times G\to\mathbb T$, defined by
\begin{equation}\label{equ: bicha}
 \NGDiagram{D015}=\chi(g,h)\NGDiagram{D016}.
\end{equation}
{In particular $\chi(g,h)=\chi(h,g)=\chi(-g,-h)$.}
\end{lemma}
\begin{proof}
The two mixed composites are proportional because their common
morphism space is one-dimensional. They are unitary composites in the
fixed mixed gauge, so their ratio has modulus $1$. Closing the
interchange relation gives
\begin{equation}\label{equ: def of chi}
 \NGDiagram{D017}=\chi(g,h)d.
\end{equation}
The prism
\begin{equation}
 \NGDiagram{D018}
\end{equation}
expands in its two directions as
\begin{equation}\label{eq:bicharacter-prism}
 d\,\NGDiagram{D019}\NGDiagram{D020}
 =\NGDiagram{D021}\NGDiagram{D022}\NGDiagram{D023}.
\end{equation}
Substitution yields
\[
 d^3\chi(g+h,-k)=d^3\chi(g,-k)\chi(h,-k).
\]
Since $d\ne0$, this proves multiplicativity in the first variable.
Interchanging the roles of the coherent left and right module actions
in this calculation proves multiplicativity in the second variable.
Equivalently, compare the two ways of moving $g$ past the product
$h+k$: right module coherence gives
$\chi(g,h+k)=\chi(g,h)\chi(g,k)$.

The nondegeneracy follows from $\chi_g$ form the regular
representation in the irrational-dimensional case; this applies
since $d=(n+\sqrt{n^2+4n})/2$ is irrational for every $n\geq1$.
The spherical property gives $ \chi(g,h) = \chi(h,g) = \chi(-h,-g) $, since
\begin{equation}
 \NGDiagram{D017}=\NGDiagram{D024}=\NGDiagram{D025}.
\end{equation}
\end{proof}

\subsection{One invertible edge}
Set $V=\operatorname{Hom}(\varrho\otimes\varrho,\varrho)$; thus $\dim V=n$.
We freely identify it with its trivalent invariant-space realization
using the fixed standard duality maps, with the corner of each vertex
specifying the identification. Its positive inner product is the trace inner product
$\langle u,v\rangle=\operatorname{tr}(u v^\dagger)$,
linear in the first variable. Thus an orthonormal trivalent basis has
$\operatorname{tr}(v_h v_k^\dagger)=\delta_{h,k}$.
Separately, closing two identically oriented
trivalent vertices gives a nondegenerate \emph{bilinear} theta pairing
$B$; this pairing is not a Hermitian inner product.

\begin{lemma}\label{pro: 3 reps}
The three mixed insertions define unitary representations
$\pi_x,\pi_y,\pi_z$ of $G$ on $V$, satisfying
\begin{align}
 \pi_z(g)\pi_x(h)&=\chi(g,h)\pi_x(h)\pi_z(g),\label{equ: zx}\\
 \pi_x(g)\pi_y(h)&=\chi(g,h)\pi_y(h)\pi_x(g),\label{equ: xy}\\
 \pi_y(g)\pi_z(h)&=\chi(g,h)\pi_z(h)\pi_y(g).
 \label{equ:yz}
\end{align}
\end{lemma}
\begin{proof}
For $\alpha\in V$, define
\begin{equation}
 \pi_x(g):\NGDiagram{D026}\longmapsto\NGDiagram{D027}.
\end{equation}
The coherent mixed relation gives
\begin{equation}
 \NGDiagram{D028}=\NGDiagram{D029},
\end{equation}
and therefore
\begin{align}
 \pi_x(g)\pi_x(h)(\alpha)
 &=\NGDiagram{D030}=\NGDiagram{D031}\notag\\
 &=\NGDiagram{D032}=\pi_x(g+h)(\alpha).
\end{align}
The other two insertions are
\begin{align}
 \pi_y(g):\NGDiagram{D026}&\longmapsto\NGDiagram{D033},\notag\\
 \pi_z(g):\NGDiagram{D026}&\longmapsto\NGDiagram{D034}.
\end{align}
The same mixed cancellation proves their representation laws.
Reflecting an insertion takes its adjoint and reverses the group strand;
hence $\pi_j(g)^\dagger=\pi_j(-g)$ for $j=x,y,z$, proving unitarity.
Finally, the mixed interchange relation gives the three pictorial
commutations
\begin{align}
 \pi_z(g)\pi_x(h)(\alpha)
 &=\NGDiagram{D035}=\chi(g,h)\NGDiagram{D036}\notag\\
 &=\chi(g,h)\pi_x(h)\pi_z(g)(\alpha),\notag\\[2pt]
 \pi_x(g)\pi_y(h)(\alpha)
 &=\NGDiagram{D037}=\chi(g,h)\NGDiagram{D038}\notag\\
 &=\chi(g,h)\pi_y(h)\pi_x(g)(\alpha),\notag\\[2pt]
 \pi_y(g)\pi_z(h)(\alpha)
 &=\NGDiagram{D039}=\chi(-h,-g)\NGDiagram{D040}\notag\\
 &=\chi(g,h)\pi_z(h)\pi_y(g)(\alpha).
\end{align}
\end{proof}

\begin{corollary}[Normal form]\label{cor: rhorhorho}
Let $G$ be a finite abelian group of order $n$. There are an
orthonormal basis $(v_h)_{h\in G}$ and an even quadratic refinement
$a:G\to\mathbb T$ such that
\begin{align}\label{eq:three-actions-normal-form}
 \pi_z(g)v_h&=\chi(g,h)v_h,&
 \pi_x(g)v_h&=v_{h+g},&
 \pi_y(g)v_h&=\frac{a(h)}{a(h-g)}v_{h-g},
\end{align}
where
\begin{equation}\label{equ: def of a}
 a(0)=1,\qquad a(-g)=a(g),\qquad
 a(g+h)\chi(g,h)=a(g)a(h).
\end{equation}
We represent $v_h$ by $\NGDiagram{D041}$. The associated bilinear theta pairing is
$B(v_h,v_k)=B(v_0,v_0)\delta_{h+k,0}a(h)$.
\end{corollary}
\begin{proof}
The commuting unitary operators $\pi_z(g)$ have simultaneous weight
spaces indexed by $\widehat G$. The relation
$\pi_z(g)\pi_x(h)=\chi(g,h)\pi_x(h)\pi_z(g)$ permutes these
spaces by multiplication by $\chi(-,h)$. Nondegeneracy of $\chi$
therefore implies that all $n$ weights occur, with the same
multiplicity. Since $\dim V=n$, each weight space is one-dimensional.
Choose a unit vector $v_0$ of trivial weight and put
$v_h=\pi_x(h)v_0$. The first two displayed formulas follow, and the
basis is orthonormal.

Applying $ Eq.~\eqref{equ:yz} $ to $ v_k $. We obtain
\begin{align*}
  \pi_z(h)\pi_y(g)v_k = \chi(h,k-g)\pi_y(g)v_k.
\end{align*}
Thus \(\pi_y(g)v_k\) is a simultaneous eigenvector of $ \pi_z(h) $ with eigenvalue $ \chi(h,k-g) $. Hence,
\begin{align*}
  \pi_y(g)v_k = c(g,k)v_{k-g},
\end{align*}
for some scalar \(c(g,k)\).
Next, applying $ Eq.~\eqref{equ: xy} $ to $ v_k $. We get
\begin{align*}
  c(h,k) = \chi(g,h)c(h,g+k).
\end{align*} 
Multiplying by $ \chi(h,k) $ on both sides, we obtain
\begin{align*}
  \chi(h,k)c(h,k) = \chi(h,g+k)c(h,g+k).
\end{align*}
Thus the quantity \(\chi(h,k)c(h,k)\) is independent of \(k\). Denote its common value by \(t(h)\).
Then we obtain
\[
 \pi_y(g)v_h=t(g)\chi(g,h)^{-1}v_{h-g}.
\]
The representation law for $\pi_y$ yields
\begin{equation}\label{eq:abelian-t-cocycle}
 t(g+h)=t(g)t(h)\chi(g,h).
\end{equation}
Thus $a=t^{-1}$ satisfies the quadratic relation, but evenness still
requires the theta pairing.

For every $l\in G$, the group-loop insertion gives
\begin{equation}\label{eq:theta-selection-chain}
 \NGDiagram{D042}=\NGDiagram{D043}=\NGDiagram{D044}
 =\chi(l,h+k)\NGDiagram{D042}.
\end{equation}
By nondegeneracy,
$B(v_h,v_k)=0$ unless $h+k=0$, or pictorially
\begin{equation}\label{equ: delta hk}
 \NGDiagram{D045}=0\qquad(h+k\ne0).
\end{equation}
Write
$b_h=B(v_h,v_{-h})\ne0$. The mixed fusion--splitting cancellation is
\begin{equation}\label{eq:theta-slide-right}
 \NGDiagram{D046}=\NGDiagram{D047}
 =t(g)\chi(g,h)^{-1}\NGDiagram{D048}.
\end{equation}
Consequently
\[
 b_h=t(g)\chi(g,h)^{-1}b_{h-g}.
\]
The theta pairing is symmetric by the pivotal-pairing formula of \cite[\S3.3]{LPR} and $\nu_2(\varrho)=1$, so $b_g=b_{-g}$. Taking $h=0$ and $h=g$ gives
\[
 b_0=t(g)b_{-g},\qquad b_g=t(g)\chi(g,g)^{-1}b_0.
\]
Hence $t(g)^2=\chi(g,g)$ and $b_g/b_0=t(g)^{-1}=a(g)$.
Equation~\eqref{eq:abelian-t-cocycle} at $(g,-g)$ now gives
$t(-g)=t(g)$, proving evenness. Finally,
$a(h)/a(h-g)=a(g)^{-1}\chi(g,h)^{-1}$, which proves the third
normal-form identity.
\end{proof}

\begin{corollary}\label{cor:bilinear-dual-basis}
Normalize the common phase by $B(v_0,v_0)=1$. The basis bilinearly
dual to $(v_h)$ is
\begin{equation}\label{eq:dual-basis}
 v_h'=a(h)^{-1}v_{-h}.
\end{equation}
Its three group actions are
\begin{align*}
 \pi_z(g)v_h'&=\chi(g,-h)v_h',&
 \pi_x(g)v_h'&=\frac{a(h-g)}{a(h)}v_{h-g}',&
 \pi_y(g)v_h'&=v_{h+g}'.
\end{align*}
\end{corollary}
\begin{proof}
Standard unitary duality identifies the theta pairing with the positive
inner product composed with an antiunitary involution. Its weight-zero
line is preserved, so $|B(v_0,v_0)|=1$ and a common phase achieves
the normalization. The pairing formula of
Corollary~\ref{cor: rhorhorho} also has the graphical form
\begin{equation}\label{eq:theta-pairing-value}
 \NGDiagram{D046}=\NGDiagram{D051}
 =a(h)\NGDiagram{D052}=a(h).
\end{equation}
It gives the dual basis, and substituting it into the three normal forms
proves the stated actions. Prime denotes the bilinear dual, not the
Hermitian adjoint.
\end{proof}

\subsection{Rotation of the multiplicity space}

The corner occupied by a vertex label is part of the notation.  A cyclic
change of corner acts on the multiplicity space; it does not leave each
basis vector fixed.  In the following pictures, a prime denotes the
bilinear trace-dual basis morphism, with the bends prescribed in Section~\ref{sec: pentagon Equation}.
Put $S_a=\sum_{g\in G}a(g)$.

\begin{lemma}\label{lem: rotation}
There is a nonzero $c\in\mathbb C$ such that
\begin{equation}\label{eq:vertex-rotation}
 \NGDiagram{D053}
 =c a(g)\sum_{h\in G}\chi(g,h)^{-1}\NGDiagram{D054}.
\end{equation}
Moreover,
\begin{equation}\label{eq:rotation-gauss}
 c^3S_a=\frac1n,\qquad |S_a|=\sqrt n,\qquad |c|=\frac1{\sqrt n}.
\end{equation}
\end{lemma}
\begin{proof}
The two corner conventions give bases
$\{\NGDiagram{D053}:g\in G\}$ and
$\{\NGDiagram{D055}:g\in G\}$.  Thus
\begin{equation}\label{eq:rotate-zero-expand}
 \NGDiagram{D056}=\sum_{h\in G}c_h\NGDiagram{D054}.
\end{equation}
Let $R$ denote this cyclic change of corner.  Transporting the group
strands to the adjacent corner gives
$R\pi_z(l)=\pi_x(l)R$ and $R\pi_x(l)=\pi_y(l)R$.
Since $\pi_z(l)v_0=v_0$, the vector $Rv_0$ is fixed by every
$\pi_x(l)$.  The latter acts by translation in the target basis,
by Corollary~\ref{cor: rhorhorho}; hence all $c_h$ are equal.
Write their common value as $c$.
Sliding a $g$-labelled invertible strand gives
\begin{align}
 \NGDiagram{D053}
 &=\NGDiagram{D057}
   =c\sum_{h\in G}\NGDiagram{D058},\label{eq:rotation-slide}\\
 &=c\sum_{h\in G}\frac{a(h)}{a(h-g)}\NGDiagram{D059},\notag\\
 &=c\sum_{h\in G}\frac{a(h+g)}{a(h)}\NGDiagram{D054}
   =c a(g)\sum_{h\in G}\chi(g,h)^{-1}\NGDiagram{D054}.
   \notag
\end{align}
The last equality uses $a(h+g)\chi(g,h)=a(h)a(g)$.
The scalar $c$ cannot vanish because the change of basis is invertible.

Three successive corner changes are a full cyclic rotation.  Their
pictorial expansion is
\begin{align}
 \NGDiagram{D056}
 &=c\sum_{g\in G}\NGDiagram{D060},\label{eq:rotation-three}\\
 &=c^2\sum_{g,h\in G}a(g)\chi(g,h)^{-1}\NGDiagram{D061},\notag\\
 &=c^3\sum_{g,h,k\in G}
     a(g)a(h)\chi(g,h)^{-1}\chi(h,k)^{-1}\NGDiagram{D062}.
     \notag
\end{align}
The coefficient of the last picture with label $k$ is
\begin{align*}
 c^3\sum_{g,h}a(g)a(h)\chi(g,h)^{-1}\chi(h,k)^{-1}
 &=c^3\sum_{g,h}a(g+h)\chi(h,k)^{-1}\\
 &=nc^3S_a\,\delta_{k,0}.
\end{align*}
Since a full cyclic rotation is the identity, this proves the first
identity in \eqref{eq:rotation-gauss}.  Finally, character orthogonality
gives
\[
 |S_a|^2=\sum_{t,h}a(h+t)\overline{a(h)}
        =\sum_t a(t)\sum_h\chi(t,h)^{-1}=n,
\]
and the asserted modulus of $c$ follows.
\end{proof}

In column convention, the matrix of this corner change is
\begin{equation}\label{eq:rotation-matrix}
 R_{h,g}=c\,a(g)\chi(g,h)^{-1}.
\end{equation}
Thus $R=cHD_a$, where $H_{h,g}=\chi(g,h)^{-1}$ and
$D_a=\operatorname{diag}(a(g))$.  The calculation above gives $R^3=I$.

\subsection{Tetrahedral values with no invertible edge}
We reduce the all-$\varrho$ coefficients to a single function $\beta$ by
a charge-selection rule and an expansion into fusion channels.

\begin{lemma}[Charge selection]\label{lem: ghk}
If $g+h+k\ne0$, then
\begin{equation}\label{eq:triangle-charge}
 \NGDiagram{D063}=0.
\end{equation}
\end{lemma}
\begin{proof}
Insert a loop labelled by $l\in G$.  Its dimension is $1$; sliding it
through the three vertices gives
\begin{align}
 \NGDiagram{D063}
 &=\NGDiagram{D064}
  =\NGDiagram{D065},\notag\\
 &=\chi(l,g)\chi(l,h)\chi(l,k)\NGDiagram{D063}.
 \label{eq:charge-loop-slide}
\end{align}
Hence $(1-\chi(l,g+h+k))$ times the original diagram is zero for every
$l$.  Since $\chi$ is nondegenerate and $g+h+k\ne0$, there exists an
$l$ for which this factor is nonzero.  This proves the claim.
\end{proof}

\begin{lemma}\label{pro: real parameter}
There is a function $\beta:G\to\mathbb C$ such that
\begin{equation}\label{equ: real parameter}
 \NGDiagram{D066}
 =\beta(g)\sum_{h\in G}\chi(h,g)^{-1}\NGDiagram{D067}
   +\frac{\delta_{g,0}}{d_\varrho^2}\sum_{h\in G}\NGDiagram{D068}.
\end{equation}
\end{lemma}
\begin{proof}
Resolving the central channel gives the following basis of
$\Hom(\varrho\otimes\varrho,\varrho\otimes\varrho)$:
\begin{equation}\label{eq:rho-rho-matrix-units}
 E_{h,k}:=\NGDiagram{D069}\quad(h,k\in G),
 \qquad P_l:=\NGDiagram{D070}\quad(l\in G).
\end{equation}
Here $E_{h,k}$ has lower label $k$ and upper label $h'$.
Write the horizontal diagram as $H_g:=\NGDiagram{D066}$.  Its expansion
is
\begin{equation}\label{equ: coef}
 H_g=\sum_{h,k\in G}C^g_{h,k}E_{h,k}
       +\sum_{l\in G}f_l^gP_l.
\end{equation}
The trace normalization of the vertices gives
\[
 \operatorname{tr}(E_{h,k}E_{p,q})
    =d_\varrho^{-1}\delta_{k,p}\delta_{h,q}.
\]
Thus the coefficient $C^g_{h,k}$ is obtained by pairing with
$E_{k,h}=\NGDiagram{D071}$, and its closure is
\begin{equation}\label{eq:extract-rho-coefficient}
 \NGDiagram{D072}
 =\operatorname{tr}(E_{k,h}H_g)=\frac{C^g_{h,k}}{d_\varrho}.
\end{equation}
The primed label is the trace-dual vertex; using
$v_{-k}=a(k)v_k'$ with its prescribed bends, the triangular face of
this closure has total charge $h-k$.  Lemma~\ref{lem: ghk} therefore
gives $C^g_{h,k}=0$ whenever $h\ne k$.

For the invertible channels, close
$\NGDiagram{D073}$ against $\NGDiagram{D074}$.
The mixed tetrahedral evaluation and the same trace normalization give
\begin{equation}\label{eq:extract-invertible-coefficient}
 d_\varrho^2 f_l^g=\chi(-l,0)\delta_{g,0}=\delta_{g,0}.
\end{equation}
Consequently, with $C_h^g=C^g_{h,h}$,
\begin{equation}\label{equ: coef1}
 \NGDiagram{D066}
 =\sum_{h\in G}C_h^g\NGDiagram{D067}
   +\frac{\delta_{g,0}}{d_\varrho^2}\sum_{l\in G}\NGDiagram{D070}.
\end{equation}

It remains to compare the diagonal coefficients.  We extract the
zero-diagonal term from the strand-decorated diagram:
\[
 \NGDiagram{D075}\quad\text{in}\quad\NGDiagram{D076}.
\]
Sliding the invertible strand first gives, by
Corollary~\ref{cor: rhorhorho},
\begin{equation}\label{eq:beta-character-slide}
 \NGDiagram{D076}=\chi(h,g)^{-1}\NGDiagram{D066},
\end{equation}
so that coefficient is $\chi(h,g)^{-1}C_0^g$.
Alternatively, expand the horizontal diagram first.  Only its
$h$-diagonal term contributes, because the strand translates both
indices by $-h$.  Its reduction is
\begin{equation}\label{eq:beta-diagonal-slide}
 \NGDiagram{D077}=\NGDiagram{D078}=\NGDiagram{D075}.
\end{equation}
Here the two dual vertex factors cancel by
Proposition~\ref{pro: normalization}; the normalization at $0$ is
$v_0'=v_0$, after the specified bending identification.
The coefficient is therefore $C_h^g$.
It follows that $C_h^g=\chi(h,g)^{-1}C_0^g$.  Defining
$\beta(g)=C_0^g$ proves \eqref{equ: real parameter}.
\end{proof}

\begin{theorem}[Tetrahedral evaluations]\label{thm: F-symbols}
In the preceding normalized gauge, the tetrahedral evaluations are
determined by $d_\varrho,\chi,a,c,\beta$, with the corner-transport map
\eqref{eq:vertex-rotation}.  The following are representatives;
the vertex corners and primed dual labels are as drawn.

For three or more invertible edges, the representatives are
\begin{equation}\label{eq:table-three-invertibles}
 \NGDiagram{D010}=1,\qquad \NGDiagram{D011}=d_\varrho.
\end{equation}
For two invertible edges,
\begin{equation}\label{eq:table-two-invertibles}
 \NGDiagram{D017}=\chi(g,h)d_\varrho.
\end{equation}
For one invertible edge,
\begin{equation}\label{eq:table-one-invertible}
 \NGDiagram{D079}=\chi(g,h)\delta_{h,k}.
\end{equation}
For no invertible edge, abbreviate $p=h-g$ and $r=k-g$ in the
vertex labels:
\begin{equation}\label{eq:table-no-invertible}
 \NGDiagram{D080}
 =\frac{a(h-g)a(k-g)}{a(h)a(k)}\,
    \chi(h,k)^{-1}\frac{\beta(k)}{d_\varrho}.
\end{equation}
\end{theorem}
\begin{proof}
The first two rows are Lemmas~\ref{lem: F-symbol 6 inv}
and~\ref{lem: bicha}.  The one-invertible-edge row follows by sliding
its strand and pairing the remaining dual basis vertices.  After this
slide, the two lower corners in the middle picture are bivalent bends;
the picture is a closed theta graph:
\begin{equation}\label{eq:one-edge-table-proof}
 \NGDiagram{D079}
 =\chi(g,h)\NGDiagram{D081}
 =\chi(g,h)\delta_{h,k}.
\end{equation}
For the last row, closing the expansion in
Lemma~\ref{pro: real parameter} gives the value with one zero vertex:
\begin{equation}\label{equ: 6rho fsymbol easy}
 \NGDiagram{D082}=\chi(h,k)^{-1}\frac{\beta(k)}{d_\varrho}.
\end{equation}
Retain the abbreviations $p=h-g$ and $r=k-g$. Insert the three
$g$-labelled strands and slide them to the adjacent vertices.  The normalized group action produces the two ratios of
quadratic factors:
\begin{align}
 \NGDiagram{D080}&=\NGDiagram{D083},\notag\\
 &=\frac{a(h-g)}{a(h)}\frac{a(k-g)}{a(k)}\NGDiagram{D082},\notag\\
 &=\frac{a(h-g)a(k-g)}{a(h)a(k)}\,
    \chi(h,k)^{-1}\frac{\beta(k)}{d_\varrho}.
    \label{eq:no-edge-table-proof}
\end{align}
It remains to check that these representatives cover all admissible
edge types.  At a trivalent vertex, exactly two invertible incident
edges would force the third edge to be invertible.  Thus the
invertible-edge subgraph of a tetrahedron can only consist of all six
edges, a three-edge star, two disjoint edges, one edge, or no edge.
The orientation-preserving tetrahedral symmetry group is transitive
on the configurations of each type.  Its action, with the indicated
corner and duality transports, reduces the first four types to
\eqref{eq:table-three-invertibles}--\eqref{eq:table-one-invertible}.

For the last type, choose the corners used in
\eqref{eq:table-no-invertible}, and write its labels as $u,v,w',z$
at the lower left, lower right, top, and interior vertices,
respectively.  The left triangular face has charge $u+z-w$, since
the primed vertex is $a(w)^{-1}v_{-w}$ after bending.  Its value is
therefore zero unless $w=u+z$.  In the remaining case, put
$g=u$, $h=w$ and $k=u+v$.  The four labels are precisely
$g,r,h',p$, with $r=k-g$ and $p=h-g$, as in \eqref{eq:table-no-invertible}.
This proves completeness in the fixed corner convention; arbitrary
corners are obtained by \eqref{eq:vertex-rotation} and the duality maps.
\end{proof}

\subsection{Compatibility with vertex transport}
Rotation and reflection impose two further identities on $\beta$.
We keep the bilinear and Hermitian operations distinct.

\begin{proposition}\label{prop:beta-vertex-compatibility}
The scalar function $\beta$ satisfies
\begin{align}
 c a(g)\sum_{h\in G}\chi(g,h)^{-1}\beta(h)&=\beta(g),
       \label{eq:beta-rotation} \\
  \overline{\beta(g)}&=\frac{\beta(-g)}{a(g)}.
  \label{eq:beta-reality}
\end{align}
\end{proposition}
\begin{proof}
Specialize \eqref{equ: 6rho fsymbol easy} at $h=0$ and $k=g$.
It defines a coefficient functional $\ell$ on the trivalent space:
\begin{equation}\label{eq:beta-zero-face}
 \ell(v_g):=
 d_\varrho\left.\NGDiagram{D082}\right|_{h=0,\,k=g}
 =\beta(g).
\end{equation}
The three vertices of the face opposite the $g$-vertex have label $0$,
with the prime at the top interpreted using $v_0'=v_0$.  Their marked
corners all point into this same triangular face.

Rotate the tetrahedron by the order-three symmetry which fixes the
$g$-vertex and cyclically permutes the opposite face.  The three
zero-labelled vertices, with their inward-facing corners, are
permuted without a coefficient.  At the fixed vertex this symmetry
induces a cyclic change of its marked corner.  It follows that
$\ell\circ R=\ell$ (the opposite choice of cyclic direction gives the
equivalent identity with $R^{-1}$).  Expanding that corner change by
\eqref{eq:vertex-rotation} gives
\[
 \beta(g)=\ell(Rv_g)
 =c a(g)\sum_{h\in G}\chi(g,h)^{-1}\ell(v_h),
\]
which is \eqref{eq:beta-rotation}.

Finally, take the adjoint of the closed diagram in
\eqref{eq:beta-zero-face}, and use the reflection which interchanges
the two zero-labelled vertices in the marked face at the $g$-vertex.
This reflection fixes that vertex and its marked face, so no further
corner rotation occurs there.  Reflection changes each vertex to its
trace-dual vertex.  The three zero-labelled vertices remain unchanged;
the fourth becomes $v_g'$.  The trace-orthonormal normalization and
$v_g'=a(g)^{-1}v_{-g}$ therefore give
\[
 \overline{\beta(g)}=\ell(v_g')
 =a(g)^{-1}\ell(v_{-g})=a(g)^{-1}\beta(-g).
\]
\end{proof}

With the column convention \eqref{eq:rotation-matrix},
\eqref{eq:beta-rotation} is precisely $R^{\mathsf T}\beta=\beta$.
It is a condition on a coefficient functional, distinct from $R^3=I$.
In particular, its $g=0$ instance is
\begin{equation}\label{eq:beta-rotation-at-zero}
 c\sum_{h\in G}\beta(h)=\beta(0).
\end{equation}
This relation fixes the rotation scalar once $\beta$ is known and is
retained in the final parameter system.

\subsection{Algebraic transport of the tetrahedral table}
We now regard the tetrahedral table as numerical data over the coefficient
field of Definition~\ref{thm: equation}. No Hermitian inner product is
used, and $d=d_\varrho$ is any nonzero solution of $d^2=n(d+1)$.
We give the coordinate maps first, then verify proper tetrahedral
transport, and finally describe reflection.

On $V=K^G$, put $D_a=\operatorname{diag}(a(g))$, $Pv_h=v_{-h}$
and $H_\pm=(\chi(h,g)^{\pm1})_{h,g}$. The coordinate maps used below are
\begin{equation}\label{eq:matrix-toolkit}
\begin{gathered}
 X(g)v_h=v_{h+g},\qquad Z(g)v_h=\chi(g,h)v_h,\qquad
 Y(g)v_h=\frac{a(h)}{a(h-g)}v_{h-g},\\
 Q=D_a^{-1}P,\qquad B=Q^{-1},\qquad
 R=cH_-D_a,\qquad A=Q^{-1}R.
\end{gathered}
\end{equation}
Here $Y(g)$ is the algebraic form of $\pi_y(g)$;
$B(v_h,v_k)=a(h)\delta_{h+k,0}$ is the symmetric theta pairing,
also denoted by its Gram matrix, and $Qv_h=v_h'$ is its dual basis.
No conjugation on $K$ is used. Nondegeneracy gives character orthogonality
\begin{equation}\label{eq:character-orthogonality}
 \sum_{h\in G}\chi(g,h)=n\delta_{g,0},\qquad H_-H_+=nI.
\end{equation}
Indeed, for $g\ne0$ translating by $u$ with $\chi(g,u)\ne1$
multiplies the sum by $\chi(g,u)$, forcing it to vanish.
Together with the quadratic relation this gives
\begin{equation}\label{eq:matrix-toolkit-inverses}
\begin{gathered}
 Q^{\mathsf T}=Q,\qquad R^3=I,\qquad
 R^{-1}=\frac1{nc}D_a^{-1}H_+,\qquad
 R^{\mathsf T}B=BR,\\
 A^{\mathsf T}=A,\qquad A^{-1}=R^2Q,
 \qquad B(Rv,R^{-1}w)=B(v,w).
\end{gathered}
\end{equation}
The entries needed in the channel computations are
\begin{equation}\label{eq:explicit-channel-corner}
 A_{h,g}=c\,a(h)a(g)\chi(h,g),\qquad
 (A^{-1})_{h,g}=\frac{\chi(h,g)^{-1}}{nc\,a(h)a(g)}.
\end{equation}
The group/corner transports are
\begin{equation}\label{eq:matrix-toolkit-transports}
 RZ(t)R^{-1}=X(t),\qquad R^{-1}Z(t)R=Y(t),
 \qquad Z(-g)A=AY(g).
\end{equation}
These identities follow by substituting the entries in
\eqref{eq:matrix-toolkit}; the Fourier calculation underlying
\eqref{eq:rotation-three} is valid over $K$ by
\eqref{eq:scalar-fourier-toolkit}.

Let $\tau$ be the parameter involution~\eqref{eq:parameter-reflection}.
Its compatibility with the corner maps is
\begin{equation}\label{eq:algebraic-corner-reflection}
 Q_\tau=Q^{-1},\qquad Q R_\tau=R^{-1}Q,\qquad
 Q^{\mathsf T}BQ=B_\tau.
\end{equation}
Indeed $R_\tau=(nc)^{-1}H_+D_a^{-1}$, and the displayed formulas
give all three identities directly.
The operation is between parameter tuples; it need not arise from a
field automorphism of one particular tuple. On the universal coefficient
ring, the formal semilinear map $J=Q\tau$ satisfies $J^2=I$ and
$JRJ=R^{-1}$. This formal $J$ is distinct from the bent adjoint in the
unitary calculation. The two agree on coefficients precisely under the
Hermitian parameter relations.

\begin{lemma}[Transport of the marked tetrahedra]\label{lem:algebraic-tetra-transport}
The tetrahedral table of Theorem~\ref{thm: F-symbols}, with the
prescribed inverse mixed vertices, is consistent under proper
tetrahedral transport. Reversing the oriented group strands and
reflecting the tetrahedron transports its coefficients by
\eqref{eq:parameter-reflection}, with the bilinear dual changes at
the multiplicity vertices.
\end{lemma}
\begin{proof}
Write $A,B,C,D$ for the lower-left, lower-right, top and interior
vertices, respectively. The all-$\varrho$ coefficient, with labels
$v_u,v_v,v_w',v_z$ at these vertices, is
\begin{equation}\label{eq:algebraic-tetra-tensor}
 T(u,v;w,z)=\delta_{w,u+z}
 \frac{a(z)a(v)}{a(w)a(u+v)}
 \chi(w,u+v)^{-1}\frac{\beta(u+v)}d.
\end{equation}
To check proper transport, use unprimed labels at all four vertices:
\[
 W(u,v,w,z)=a(w)T(u,v;-w,z)
 =\delta_{u+w+z,0}\frac{a(z)a(v)}{d\,a(u+v)}
       \chi(w,u+v)\beta(u+v).
\]
The marked faces are $ACD$ at $A,C,D$ and $BCD$ at $B$.
The generators $(A\,C\,D)$ and $(A\,B)(C\,D)$ of $A_4$ therefore
require, respectively,
\begin{align}
 W(u,v,w,z)&=\sum_j(R^2)_{jv}W(z,j,u,w),
       \label{eq:proper-tetra-three-cycle}\\
 W(u,v,w,z)&=\sum_{j,k}(R^2)_{jz}R_{kw}W(v,u,j,k).
       \label{eq:proper-tetra-double-transposition}
\end{align}
Indeed, $R^2_{jv}=\chi(j,v)/(nc\,a(j))$ reduces the first sum to
\[
 \delta_{u+w+z,0}\frac{a(w)\chi(-z,v)}{nc d}
 \sum_t\chi(u+v,t)\frac{\beta(t)}{a(t)}
 =\delta_{u+w+z,0}\frac{a(w)\chi(-z,v)}d\beta(u+v).
\]
The Fourier equation gives the equality, and the quadratic relation
identifies this with $W$. In the second sum, set $m=u+v$ and first
use the charge constraint. The remaining sum is
\[
 \frac{a(w)a(u)\beta(m)}{nd\,a(m)}
 \sum_j\frac{a(v+j)}{a(j)}
       \chi(j,z)\chi(w,v+j)\chi(j,m).
\]
Its inner sum is
$n a(v)\chi(w,v)\delta_{u+w+z,0}$ by character orthogonality;
the quadratic relation again gives $W$.

For the mixed rows it suffices to check their proper support
stabilizers. With
$U(g;h,k)=\chi(g,h)\delta_{h,k}$ and $V(g,h)=d\chi(g,h)$,
the only nontrivial stabilizer of the single-edge support is
$(A\,B)(C\,D)$. Its coefficient identity is
\[
 \frac{a(h)}{a(k)}U(-g;-k,-h)=U(g;h,k).
\]
The two-edge support has stabilizer generated by
$(A\,B)(C\,D)$ and $(A\,C)(B\,D)$; their identities are
$V(-g,-h)=V(g,h)$ and $V(h,g)=V(g,h)$. The star and all-group
rows have constant values $d$ and $1$. Group charge constraints
are preserved throughout. This proves proper transport for all rows.

For reflection, reverse the group arrows and use $(C\,D)$.
This preserves the marked faces of~\eqref{eq:algebraic-tetra-tensor}
and gives the exact coefficient comparison
\begin{equation}\label{eq:algebraic-tetra-reflection}
 T_\tau(u,v;w,z)=\frac{T(-u,-v;z,w)}{a(u)a(v)}.
\end{equation}
Both sides have support $w=u+z$; their ratio there reduces to
$a(u)^2\chi(u,u)=1$. At $A,B$ the factors are precisely the
conversions $v_u'=a(u)^{-1}v_{-u}$ and
$v_v'=a(v)^{-1}v_{-v}$, while the top and interior labels become
$v_z'$ and $v_w$. For the one-edge and two-edge rows, the same
arrow reversal and permutation give
\[
 U_\tau(g;h,k)=U(-g;k,h),\qquad
 V_\tau(g,h)=V(-g,h).
\]
The all-group parameters change by $(g,h,k)\mapsto(-h,-g,k)$;
for the star use $(A\,B)$ instead, giving
$(g,h)\mapsto(g+h,-h)$. Their constant values remain $1,d$.
Reflection replaces each mixed vertex by its prescribed opposite.
The inverse theta factors already present in dual-basis contractions
are preserved, since $\tau(d)=d$; reflection introduces no additional factor.
Finally, $QR_\tau=R^{-1}Q$ and $Q^{\mathsf T}BQ=B_\tau$
extend these comparisons to every marked corner and its dual
contraction.
\end{proof}

Consequently a prism identity proved for every algebraic scalar
tuple transfers to its reflected family: apply it first at the
$\tau$-transformed tuple, and then use
Lemma~\ref{lem:algebraic-tetra-transport} in each tetrahedron.
Internal changes cancel in the bilinear dual contractions, and
dimension weights are fixed by $\tau$. The changes of external
labels are invertible. This justifies reflection of the scalar TPE
without~\eqref{eq:beta-reality}.

\section{Evaluating the triangular prism equations}\label{sec:solve-TPE}

To organize the derivation of the Triangular Prism Equations (TPE), we classify prism configurations by recording only whether each edge is invertible. Admissible configurations require every vertex to have $0, 2,$ or $3$ $\varrho$-strands. Group labels and vertex basis vectors then range over their allowable values. The resulting labeled prism supports are indexed as follows:
\begin{equation}\label{equ:TPlist}
\begin{array}{c|rrrrrrrr}
\text{Invertible edges} & 0 & 1 & 2 & 3 & 4 & 5 & 6 & 9 \\\hline
\text{Labelled supports} & 1 & 9 & 18 & 10 & 6 & 3 & 2 & 1
\end{array}
\end{equation}
Spherical rotations reduce these configurations into $15$ proper-rotation orbits (or $14$ if orientation-reversing reflections are allowed). The extra proper orbit consists of a reflected skew pair of invertible edges, which is accounted for by the algebraic reflection of Lemma~\ref{lem:algebraic-tetra-transport}.

We divide the resulting equations into two types based on their internal simple labels $\mathcal{O}(P)$:
\begin{itemize}
    \item \textbf{Boundary prisms:} Configurations where the internal simple labels is the full set of simple objects $\mathcal{O}(P) = I = G \sqcup \{\varrho\}$. 
    \item \textbf{Bulk prisms:} Configurations where $\mathcal{O}(P)$ is either a singleton (a unique internal simple object) or empty (contributing zero).
\end{itemize}

For permutations of tetrahedral vertices we retain the names $A,B,C,D$
for the lower-left, lower-right, upper and central vertices of the
simplified tetrahedron, respectively.

\subsection{Bulk equations}
The support list \eqref{equ:TPlist} has fifteen proper-rotation
orbits: eleven bulk and four boundary. The two skew-pair bulk orbits
have invertible supports
\[
 \{\{4,8\},\{5,9\},\{6,7\}\},\qquad
 \{\{4,9\},\{5,7\},\{6,8\}\},
\]
in the numbering of Figure~\ref{fig:reference-graphs}.
Reflection exchanges these two orbits and fixes the remaining types.
By Lemmas~\ref{lem:algebraic-reflection}
and~\ref{lem:algebraic-tetra-transport}, ten bulk representatives
therefore suffice: reflection evaluates the representative at the
transformed scalar tuple and transports its paired corners. Each mixed
dual contributes the inverse theta value, unchanged by $\tau$.
This argument does not require complex-conjugation symmetry.

The following table records the nine routine representatives. The
pictured arrows and vertex labels fix their conventions; every omitted
mixed dual has the inverse theta factor just specified. The last column
gives the identity after substituting the tetrahedral table.
\begingroup
\small
\renewcommand{\arraystretch}{1.1}
\setlength{\tabcolsep}{3.5pt}
\begin{longtable}{c c l l}
\toprule
Invertible edges & Representative & Common contraction & Identity used\\
\midrule
\endfirsthead
\toprule
Invertible edges & Representative & Common contraction & Identity used\\
\midrule
\endhead
$9$ & $\SD{D086}$ & $1$ & Pointed normalization\\\addlinespace[2.5mm]
$6$ & $\SD{D092}$ & $d$ & Mixed theta cancellation\\\addlinespace[2.5mm]
$5$ & $\SD{D098}$ & $d^3$ & Mixed theta cancellation\\\addlinespace[2.5mm]
$4$ & $\SD{D104}$ & $d^3\chi(g+h,-k)$ & Bicharacter multiplication\\\addlinespace[2.5mm]
$3$ & $\SD{D108}$ & $d\chi(g+h,k)\delta_{k,l}$ & Character multiplication\\\addlinespace[2.5mm]
$3$ & $\SD{D112}$ & $d^3\chi(h,k-g)$ & Character multiplication\\\addlinespace[2.5mm]
$2$ & $\SD{D121}$ & $\chi(g+h,k)\delta_{k,l}$ & Kronecker contraction\\\addlinespace[2.5mm]
$2$ & $\SD{D126}$ & $\chi(h,l)\delta_{g+l,k}$ & Kronecker contraction\\\addlinespace[2.5mm]
$2$ & $\SD{D131}$ & $\chi(h,g)\delta_{g+k,l}$ & Character cancellation\\\addlinespace[2.5mm]
\bottomrule
\end{longtable}
\endgroup
For example, the first three-invertible-edge representative has
\[
 \operatorname{TPL}=d\chi(g+h,k)\delta_{k,l},\qquad
 \operatorname{TPR}
 =d\sum_s\chi(g,s)\delta_{s,l}\chi(h,k)\delta_{s,k},
\]
which gives the stated common value. This illustrates the contraction
convention used in every row.

\Needspace{38mm}
\textbf{$\bullet$ Bulk equation with one invertible edge:}
We evaluate both contractions with all vertex labels arbitrary. Put
$H=h_1+h_2+h_3$ and
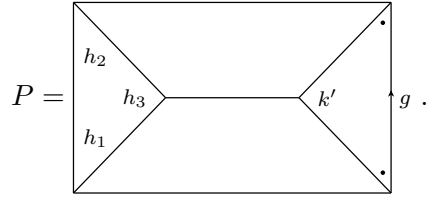
\[P=\SD{D136}.\]
The five tetrahedral factors, with their displayed corners unchanged, are
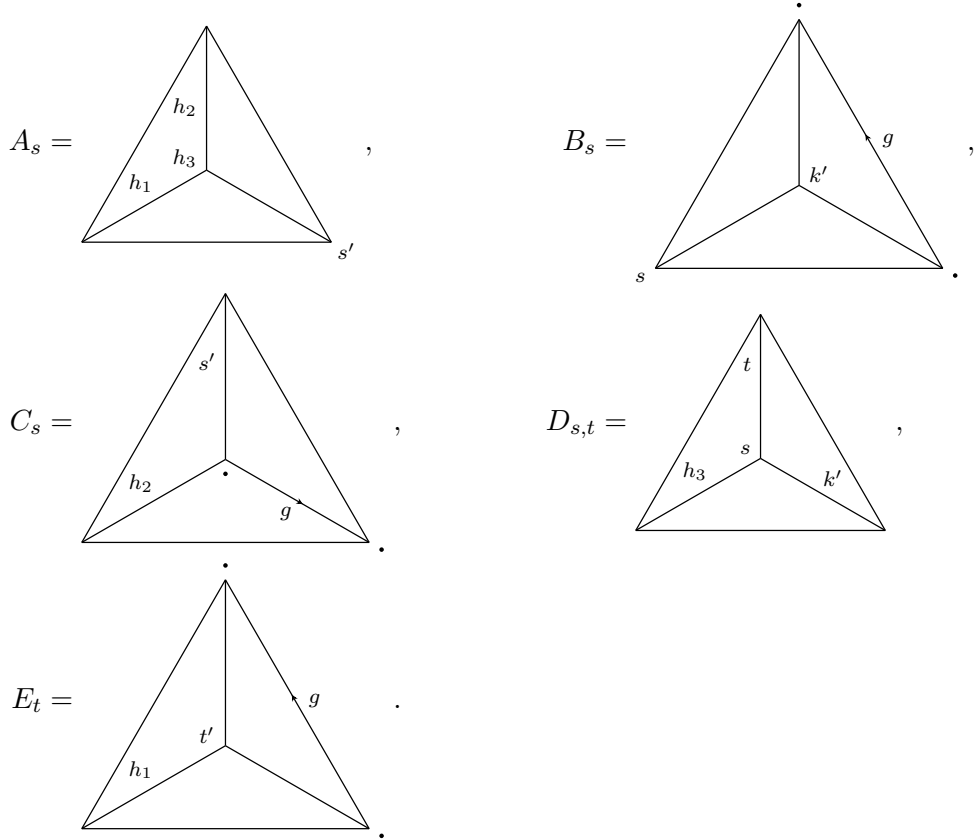
\begin{align*}
 A_s&=\SD{D138},&
 B_s&=\SD{D139},\\
 C_s&=\SD{D144},&
 D_{s,t}&=\SD{D145},\\
 E_t&=\SD{D146}.
\end{align*}
Thus the pictorial TPE is
\[
 \operatorname{TPL}(P)=\sum_s A_sB_s,
 \qquad
 \operatorname{TPR}(P)
 =\sum_{s,t}d_\varrho^{-1}C_s(d_\varrho D_{s,t})E_t.
\]
The tetrahedral table and the corner-rotation formula give
\begin{align}
 A_s&=\delta_{H,0}
 \frac{a(h_3)\chi(-s-h_1,h_1)}
      {d_\varrho a(h_1)a(h_3-s)}\,\beta(h_3-s),
 \label{eq:bulk-one-A}\\
 B_s&=\chi(-g,k)\delta_{s,k},\qquad
 C_s=\delta_{s,h_2-g},\qquad
 E_t=\frac{a(h_1)}{a(h_1+g)}\delta_{t,h_1+g},
 \label{eq:bulk-one-mixed}\\
 D_{s,t}&=\delta_{s+t+h_3,0}
 \frac{a(s)\chi(t,h_3-k)}{d_\varrho a(h_3-k)}\,
 \beta(h_3-k).
 \label{eq:bulk-one-D}
\end{align}
For clarity, the outside right corner of $A_s$ contributes
$(RQ)_{j,s}=c\chi(s,j)$, where $Qv_s=v_s'$.
After writing $v_{h_2}=a(h_2)v_{-h_2}'$, its table evaluation is
\[
 A_s=\delta_{H,0}\frac{c a(h_3)}{d_\varrho}
 \sum_j\chi(s,j)\frac{a(j)}{a(h_1+j)}
              \chi(h_2,h_1+j)\beta(h_1+j).
\]
The substitution $u=h_1+j$, the quadratic identity for $a$, and
\eqref{eq:beta-rotation} at $h_3-s$ yield
\eqref{eq:bulk-one-A}. The factors $C_s$ and $E_t$ are respectively
$B(X(-g)v_{h_2},v_s')$ and $B(\pi_y(-g)v_{h_1},v_t')$, with the
actions of \eqref{eq:three-actions-normal-form}. For
\eqref{eq:bulk-one-D}, substitute
$v_k'=a(k)^{-1}v_{-k}$ and $v_t=a(t)v_{-t}'$ into the all-$\varrho$
row; its selection rule becomes $s+t+h_3=0$.
Consequently
\begin{align*}
 \operatorname{TPL}(P)
 &=\delta_{H,0}
 \frac{a(h_3)\chi(-k-h_1,h_1)\chi(-g,k)}
      {d_\varrho a(h_1)a(h_3-k)}\,\beta(h_3-k),\\
 \operatorname{TPR}(P)
 &=\delta_{H,0}
 \frac{a(h_1)a(h_2-g)\chi(h_1+g,h_3-k)}
      {d_\varrho a(h_1+g)a(h_3-k)}\,\beta(h_3-k).
\end{align*}
Both vanish if $H\ne0$. If $H=0$, the quadratic identity gives
\[
 \frac{a(h_3)}{a(h_1)}
 \chi(-k-h_1,h_1)\chi(-g,k)
 =\frac{a(h_1)a(h_2-g)}{a(h_1+g)}
   \chi(h_1+g,h_3-k),
\]
so the two contractions agree. This calculation uses only the
quadratic and rotation identities, and applies to either root of
$d_\varrho^2=n(d_\varrho+1)$.

\begin{theorem}[Bulk equations]
\label{thm:bulk-equations}
{For algebraic scalar data as in Definition~\ref{thm: equation},
the tetrahedra of Theorem~\ref{thm: F-symbols}, with
the bilinear transports of Lemma~\ref{lem:algebraic-tetra-transport},
satisfies every bulk TPE.}
\end{theorem}
\begin{proof}
The graph support
orbits are those of \eqref{equ:TPlist}; equations at the other supports
{are obtained by proper transport and the algebraic
reflection lemma, including the second skew-pair orbit. Multilinearity
and invertibility of the paired corner maps preserve the equations.}
\end{proof}

\subsection{Boundary equations}
We first fix the channel algebra, so that the same calculation supplies
both the boundary equations and the inverse needed for reconstruction.
Set
\[
 \mathcal A=\bigoplus_{x\in G}K e_x\oplus M_n(K),\qquad
 t(z,M)=\sum_xz_x+d\operatorname{Tr}(M),\qquad
 P_x=de_x,\quad E_{hk}=d^{-1}e_{hk}.
\]
Here the $e_x$ are the scalar idempotents and the $e_{hk}$ are ordinary
matrix units. The channel product is the top-to-bottom stacking product
specified after \eqref{eq:rho-rho-matrix-units}, so $E_{hk}$ has exactly
these, rather than transposed, matrix indices. Thus $t(UV)$ is a nondegenerate symmetric bilinear pairing,
with vertical Gram coefficients
\begin{equation}\label{eq:vertical-channel-gram}
 t(P_xP_y)=d^2\delta_{x,y},\qquad
 t(E_{hk}E_{uv})=d^{-1}\delta_{k,u}\delta_{h,v},\qquad
 t(P_xE_{hk})=0.
\end{equation}
The horizontal channels are
\begin{align*}
 W_g&=\bigl((\chi(g,x))_x,a(g)X(g)Z(-g)\bigr),\\
 H_{pq}&=\left(
  \bigl(d^{-1}a(p)\chi(x,q)\delta_{p+q,0}\bigr)_x,
  \bigl(d^{-1}\chi(-h,q)\beta(p+q)\delta_{k,h+p}\bigr)_{h,k}
 \right).
\end{align*}
We now identify these coordinates with the actual marked diagrams.
Closing against $E_{kh}$ and using \eqref{eq:table-no-invertible} gives
its coefficient
\[
 C^{p,q}_{h,k}
 =\delta_{k,h+p}
 \frac{a(h)a(q)}{a(k)a(p+q)}\chi(k,p+q)^{-1}\beta(p+q)
 =\delta_{k,h+p}\chi(-h,q)\beta(p+q).
\]
For the group coefficient, number the lower-left, lower-right, top and
interior vertices $A,B,C,D$. The even permutation $(A\ C)(B\ D)$
takes the central invertible edge to the bottom edge of
\eqref{eq:table-one-invertible}. It puts $p$ at the outside-top corner
and $q$ at the lower-interior corner. Substituting
$v_p=a(p)v_{-p}'$ therefore gives
$a(p)\chi(x,q)\delta_{p+q,0}$; the coefficient extraction divides
this by $d^2$, as in \eqref{eq:extract-invertible-coefficient}.
Consequently
\begin{align}
 H_{pq}&=\sum_{h,k}C^{p,q}_{h,k}E_{hk}+\sum_x f_x^{p,q}P_x,
 &C^{p,q}_{h,k}&=\delta_{k,h+p}\chi(-h,q)\beta(p+q),
 \label{eq:general-horizontal-rho}\\
 &&f_x^{p,q}&=d^{-2}a(p)\chi(x,q)\delta_{p+q,0}.
 \label{eq:general-horizontal-group}
\end{align}

For the opposite cut, the lower $k$-vertex and upper $h'$-vertex
contribute $A_{i,k}$ and $(A^{-1})_{h,j}$ from
\eqref{eq:explicit-channel-corner}. Thus the change of cut is
\begin{equation}\label{eq:algebraic-recut}
 \Phi(z,M)=(z,AM^{\mathsf T}A^{-1}),\qquad
 \Phi(E_{hk})_{ij}=d^{-1}A_{i,k}(A^{-1})_{h,j}.
\end{equation}
It reverses products and preserves $t$. In the annular convention of
the right TPE, the group channel is
$\Phi(W_g)=((\chi(g,x))_x,Z(-g))$: indeed,
\[
 (A^{-1}Z(-g)A)_{i,j}=\frac{a(j)}{a(i)}\delta_{j,i+g},
\]
whose transpose is $a(g)X(g)Z(-g)$.
In particular, put
\[K_t(X;u,v;p,q)=\SD{right-channel}.\]
The actual right-TPE multiplicity tetrahedron has matrix
\begin{equation}\label{eq:one-horizontal-channel-similarity}
 K_{p,q}:=\bigl(K_\varrho(\varrho;v_i,v_j';v_p,v_q)\bigr)_{i,j}
       =d^{-1}A(C^{p,q})^{\mathsf T}A^{-1}.
\end{equation}
For $x\in G$, its scalar component is
$d^{-1}a(p)\chi(x,q)\delta_{p+q,0}$; the mixed dual has theta factor
$d^{-1}$. These formulas identify all marked-corner coefficients once.
The channel matrix $T$, with rows $(P_x;E_{hk})$ and columns
$(W_g;H_{pq})$, is therefore
\begin{align}
 T_{P_x,W_g}&=d^{-1}\chi(g,x),&
 T_{E_{hk},W_g}&=d\,a(g)\chi(-g,k)\delta_{h,k+g},\notag\\
 T_{P_x,H_{pq}}&=d^{-2}a(p)\chi(x,q)\delta_{p+q,0},&
 T_{E_{hk},H_{pq}}&=\chi(-h,q)\beta(p+q)\delta_{k,h+p}.
 \label{eq:algebraic-channel-blocks}
\end{align}

\smallskip
\noindent\textbf{Three or two horizontal invertible edges.}
The quadratic identity gives $W_gW_h=W_{g+h}$.
Character orthogonality and the ordinary trace of $X(g)Z(-g)$ give
$t(W_g)=n(1+d)\delta_{g,0}=d^2\delta_{g,0}$.
For the representative
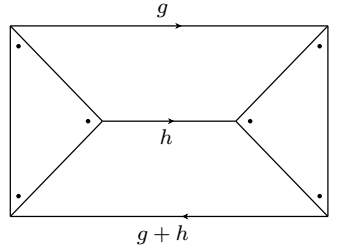
\[\SD{D153},\]
with arbitrary horizontal labels $g,h,l$, the right contraction is consequently
$t(W_gW_hW_l)=d^2\delta_{g+h+l,0}$, equal to the left contraction.
The group and multiplicity contributions are respectively
$n\delta_{g+h+l,0}$ and $nd\delta_{g+h+l,0}$.
In particular,
\begin{equation}\label{eq:horizontal-W-gram}
 t(W_gW_h)=d^2\delta_{g+h,0}.
\end{equation}

For two horizontal invertible edges use
\[P(g,h;v_p,v_q)=\SD{boundary-two-groups}\]
and put $u=g+h$.
The left contraction is zero because $V(\varrho,h,g)=0$.
The right contraction, by the marked-corner identification above, is
\begin{equation}\label{eq:two-horizontal-complete}
 t(W_uH_{pq})
 =n a(u)\left(\beta(0)+d^{-1}\right)
       \delta_{p,u}\delta_{q,-u}.
\end{equation}
For clarity, the group channel contributes
$(n/d)a(u)\delta_{p,u}\delta_{q,-u}$; in the matrix channel the trace
first forces $p=u$ and the character sum forces $q=-u$, giving
$n a(u)\beta(0)\delta_{p,u}\delta_{q,-u}$.
Thus this entire family is equivalent to $\beta(0)=-d^{-1}$, which
is also forced by the product and cubic equations by
Proposition~\ref{prop:cubic-reduction}.

\smallskip
\noindent\textbf{One horizontal invertible edge.}
For the prism
\begin{equation}\label{eq:one-horizontal-prism-picture}\SD{D154},\end{equation}
first take the edge to be the unit and retain all four vertex labels. The right contraction is $t(H_{pq}H_{rs})$.
Multiplying the two channel components gives
\begin{align}
 \operatorname{TPR}
 &=\frac nd\delta_{p+r,0}\delta_{q+s,0}\chi(p,q)
       \beta(p+q)\beta(-p-q)\notag\\
 &\quad+\frac n{d^2}a(p)^2
       \delta_{p+q,0}\delta_{r,-p}\delta_{s,p},
 \label{eq:unit-horizontal-right}\\
 \operatorname{TPL}
 &=d^{-1}B(v_p,v_r)B(v_q,v_s)
 =d^{-1}a(p)a(q)\delta_{p+r,0}\delta_{q+s,0}.
 \label{eq:unit-horizontal-left}
\end{align}
Taking $p=r=0$, $q=g$ and $s=-g$ shows that equality requires
\begin{equation}\label{equ: ess equ 1}
 \beta(g)\beta(-g)+\frac{\delta_{g,0}}d=\frac{a(g)}n.
\end{equation}
Conversely, substitute this product relation at $g=p+q$.
The ordinary terms agree because
$\chi(p,q)a(p+q)=a(p)a(q)$, and the exceptional terms cancel because
$\chi(p,-p)=a(p)^2$. Hence the full family is equivalent to the product
equation and gives the remaining horizontal Gram coefficient
\begin{equation}\label{eq:algebraic-channel-gram}
 t(H_{pq}H_{rs})=d^{-1}a(p)a(q)
                     \delta_{p+r,0}\delta_{q+s,0}.
\end{equation}
Together with \eqref{eq:horizontal-W-gram} and
\eqref{eq:two-horizontal-complete}, this determines the full horizontal
Gram matrix, including the vanishing mixed entries.

Now let the horizontal edge in \eqref{eq:one-horizontal-prism-picture} have any label
$g\in G$, and put
\[
 p=k_2,\quad q=k_3,\quad r=k_1,\quad s=k_4.
\]
Thus $(p,q)$ are the upper vertices and $(r,s)$ the middle vertices.
Use the channel identification \eqref{eq:one-horizontal-channel-similarity}.
The mixed tetrahedron on the $g$-edge has matrix $Z(-g)$.
Since $Z(-g)A=A\pi_y(g)$, the index contraction is
\[
 d_\varrho\operatorname{Tr}\bigl(K_{p,q}Z(-g)K_{r,s}\bigr)
 =d_\varrho^{-1}\operatorname{Tr}
 \bigl((C^{p,q})^{\mathsf T}\pi_y(g)(C^{r,s})^{\mathsf T}\bigr).
\]
Substituting \eqref{eq:general-horizontal-rho} and the action in
\eqref{eq:three-actions-normal-form}, and adding the group channels,
gives
\begin{align}
 \operatorname{TPR}
 &=\frac{n a(g)}{d_\varrho}
    \delta_{p+r,g}\delta_{q+s,-g}\chi(-p,s)
        \beta(p+q)\beta(-p-q)\notag\\
 &\quad+\frac{n a(p)a(r)}{d_\varrho^2}
    \delta_{p+q,0}\delta_{r+s,0}\delta_{q+s+g,0}.
 \label{eq:one-horizontal-general-right}
\end{align}
Equation \eqref{equ: ess equ 1} cancels the exceptional group term.
Using $s=-g-q$ in the remaining term gives
\begin{equation}\label{eq:one-horizontal-general-value}
 \operatorname{TPR}
 =\frac{a(q)a(p-g)}{d_\varrho}
     \delta_{p+r,g}\delta_{q+s,-g}.
\end{equation}

For completeness, the two tetrahedra on the left give this same value
with their actual marked corners. Number their vertices $A,B,C,D$
as above. In the first tetrahedron the invertible edge is $A\to D$;
the even permutation $(B\ C\ D)$ moves it to the bottom edge
$A\to B$. The $p$- and $r$-vertices then occupy the left-face corners
at the top and interior vertices, respectively. In the outside-top
and lower-interior convention of the one-invertible table their
labels are $R^2v_p$ and $Rv_r$. Its value, including the mixed dual
vertex, is
\[
 d_\varrho^{-1}
 \bigl(R^{\mathsf T}Z(g)Q^{-1}R^2\bigr)_{r,p}
 =d_\varrho^{-1}\bigl(Q^{-1}X(-g)\bigr)_{r,p}
 =d_\varrho^{-1}a(r)\delta_{p+r,g}.
\]
In the second tetrahedron the invertible edge is $D\to B$.
The even permutation $(A\ C\ D)$ gives the bottom edge $A\to B$,
while the $q$- and $s$-vertices occupy the right-face top and interior
corners. The corresponding outside-top and lower-interior labels are
$Rv_q$ and $R^2v_s$. The tetrahedron has value
\[
 \bigl((R^2)^{\mathsf T}Z(g)Q^{-1}R\bigr)_{s,q}
 =\bigl(Q^{-1}\pi_y(-g)\bigr)_{s,q}
 =a(q)\delta_{q+s,-g}.
\]
These matrix identities follow from
$R^{\mathsf T}Q^{-1}=Q^{-1}R$,
$RZ(t)R^{-1}=X(t)$ and
$R^{-1}Z(t)R=\pi_y(t)$.
Multiplying the two tetrahedral values gives
\eqref{eq:one-horizontal-general-value}, since $a(r)=a(p-g)$ on
its support. Thus \eqref{equ: ess equ 1} verifies the complete
one-horizontal boundary family, for every edge and vertex label.

\Needspace{38mm}
\textbf{$\bullet$ Boundary equation with no invertible edge}
For the prism
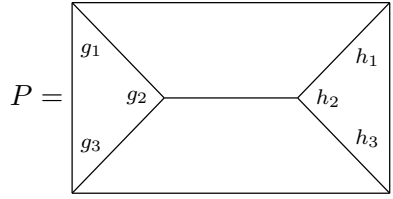
\[P=\SD{D164}\]
put
$G_0=g_1+g_2+g_3$, $H_0=h_1+h_2+h_3$, and $p_i=g_i+h_i$.
Both contractions vanish unless $G_0=H_0=0$. In that charge sector put
\begin{equation}\label{eq:allrho-transport-factor}
 \kappa=\chi(-g_1,h_2)\chi(g_3,h_3).
\end{equation}
We show that both contractions are $\kappa$ times their values at
left labels $(0,0,0)$ and right labels $(p_1,p_2,p_3)$.
Since $\kappa\ne0$, this proves the required reduction for every
choice of the six vertex labels.

For the left contraction, number the vertices of each cut tetrahedron
as $A$ (left), $B$ (right), $C$ (upper) and $D$ (central).
The even permutation $(A\ C)(B\ D)$ puts each tetrahedron into the
all-$\varrho$ convention of Theorem~\ref{thm: F-symbols}. Its outer-face
corners become the three corners of the left face of that table;
its central lower corner becomes the right-hand corner. No further
corner change is used in this step.
For the internal pair $v_s',v_s$, the two resulting values are
\begin{align}
 L_s&=\delta_{G_0,0}\,
 \frac{a(g_1)\chi(g_3,g_2-s)}{d_\varrho a(g_2-s)}\,
 \beta(g_2-s),\notag\\
 M_s&=\delta_{H_0,0}\,
 \frac{a(h_3)a(s)\chi(h_1,h_2+s)}{d_\varrho a(h_2+s)}\,
 \beta(h_2+s).
 \label{eq:allrho-left-factors}
\end{align}
Here the top vertex in the first tetrahedron is converted by
$v_{g_3}=a(g_3)v_{-g_3}'$ and its right vertex by
$v_s'=a(s)^{-1}v_{-s}$; the analogous top conversion in the second
uses $h_1$. These conversions also give the two charge deltas.
When $G_0=H_0=0$, substitute $k=h_2+s$ into $\sum_sL_sM_s$.
The quadratic identity for $a$ gives
\begin{equation}\label{eq:allrho-left-general}
 \operatorname{TPL}(P)=
 \frac{\kappa a(p_3)}{d_\varrho^2}
 \sum_{k\in G}\chi(p_1,k)
       \frac{\beta(p_2-k)\beta(k)}{a(k)}.
\end{equation}
For example, the constant factor simplifies by
\[
 \frac{a(g_1)a(h_3)a(h_2)}{a(p_2)}\chi(g_3,p_2)
 =\kappa a(p_3),
\]
while the remaining $k$-dependent character is $\chi(p_1,k)$.

Here and below, $\operatorname{Tr}$ denotes the ordinary matrix trace.
For the right contraction use the coefficient matrices
$C^{p,q}$ from \eqref{eq:general-horizontal-rho}. By the common
similarity \eqref{eq:one-horizontal-channel-similarity}, all corner
matrices cancel in the cyclic contraction.
The $\varrho$ channel is
\begin{align*}
 &\frac1{d_\varrho^2}
   \operatorname{Tr}\bigl(C^{g_1,h_1}C^{g_2,h_2}C^{g_3,h_3}\bigr)\\
 &\qquad=\frac{n}{d_\varrho^2}\delta_{G_0,0}\delta_{H_0,0}
       \kappa\prod_{i=1}^3\beta(p_i).
\end{align*}
For the group channels, the scalar component of $\Phi(H_{pq})$ above gives
\[
 \frac{n}{d_\varrho^3}\delta_{H_0,0}
   \prod_{i=1}^3a(g_i)\delta_{p_i,0}.
\]
This also vanishes unless $G_0=0$; when all $p_i=0$ and $G_0=0$,
$\prod_i a(g_i)=\kappa$. Thus, in the sector $G_0=H_0=0$,
\begin{equation}\label{eq:allrho-right-general}
 \operatorname{TPR}(P)=\kappa\left(
   \frac{n}{d_\varrho^2}\prod_{i=1}^3\beta(p_i)
   +\frac{n}{d_\varrho^3}\prod_{i=1}^3\delta_{p_i,0}\right).
\end{equation}
Setting $(p_1,p_2,p_3)=(-h,g,h-g)$ in
\eqref{eq:allrho-left-general}--\eqref{eq:allrho-right-general}, and
cancelling the nonzero common factor $\kappa n/d^2$, gives
\begin{equation}\label{equ: ess equ 2}
 \frac{a(h-g)}n\sum_{k\in G}\chi(h,k)^{-1}
       \frac{\beta(g-k)\beta(k)}{a(k)}
 =\beta(-h)\beta(g)\beta(h-g)+\frac{\delta_{h,0}\delta_{g,0}}d.
\end{equation}
Conversely, every charge-admissible triple $(p_1,p_2,p_3)$ has this
form; outside that charge sector both contractions vanish. Thus the
cubic equation verifies the entire all-$\varrho$ boundary family.
Together with the other three families, this proves the full labelled
TPE for every solution of Definition~\ref{thm: equation}.
The normalization and convolution consequences are already established
in Proposition~\ref{prop:cubic-reduction}.

\subsection{Algebraic reconstruction}
For the remainder, $(d,\chi,a,c,\beta)$ satisfies
Definition~\ref{thm: equation}.

\begin{lemma}[Bilinear channel inverse]\label{lem:algebraic-channel-inverse}
Let $G_V$ and $G_H$ be the trace Gram matrices of $(P_x;E_{hk})$ and
$(W_g;H_{pq})$, respectively. The channel matrix is invertible over $K$,
with
\begin{equation}\label{eq:algebraic-channel-inverse}
 T^{\mathsf T}G_VT=G_H,\qquad
 T^{-1}=G_H^{-1}T^{\mathsf T}G_V.
\end{equation}
In particular $T_{P_0,W_0}=(T^{-1})_{W_0,P_0}=d^{-1}$.
\end{lemma}
\begin{proof}
The boundary calculations \eqref{eq:vertical-channel-gram},
\eqref{eq:horizontal-W-gram}, \eqref{eq:two-horizontal-complete} and
\eqref{eq:algebraic-channel-gram} give the stated Gram matrices.
Both are monomial matrices with nonzero entries. Change of coordinates
therefore gives the first identity and hence the inverse. Its vacuum entry is $d^{-1}$ by the displayed
blocks of $T$.
\end{proof}

\begin{proposition}[Exact channel singular values]\label{prop:channel-singular-values}
Work over $\mathbb C$. Let $G$ be finite abelian of order $n$, let
$d=(n+\sqrt{n^2+4n})/2$, and assume only
that $\chi$ is nondegenerate symmetric, $a(-g)=a(g)$,
$a(g+h)\chi(g,h)=a(g)a(h)$, $\beta(0)=-d^{-1}$, and
$\beta(g)\beta(-g)=a(g)/n$ for $g\ne0$.
For the channel matrix $T$ in \eqref{eq:algebraic-channel-blocks}, set
\[
 D=\operatorname{diag}(d^2I_n,d^{-1}I_{n^2}),\qquad
 U=D^{1/2}TD^{-1/2}.
\]
Here $T^\dagger$ denotes the conjugate transpose. Then
\[
 T^\dagger DT=\operatorname{diag}(d^2I_n,d^{-1}\Lambda),\qquad
 \Lambda_{(p,q),(p,q)}=
 \begin{cases}1,&p+q=0,\\ n|\beta(p+q)|^2,&p+q\ne0.\end{cases}
\]
In particular, $U$ has singular values $1$ with multiplicity $2n$,
and $\sqrt n|\beta(g)|$ with multiplicity $n$ for every $g\ne0$
(counting coincident singular values with their multiplicities).
They occur in reciprocal pairs. The matrix $U$ is unitary if and
only if the scalar reflection equation holds.
\end{proposition}
\begin{proof}
All values of $a$ and $\chi$ have modulus one. Direct multiplication
of the four blocks of $T$ gives the entries of $\Gamma=T^\dagger DT$:
\begin{align*}
 \Gamma_{W_g,W_h}&=d^2\delta_{g,h},\\
 \Gamma_{W_g,H_{pq}}
 &=n a(g)(d^{-1}+\beta(0))\delta_{p,-g}\delta_{q,g}=0,\\
 \Gamma_{H_{pq},H_{rs}}
 &=\delta_{p,r}\delta_{q,s}
 \left(\frac{n}{d}|\beta(p+q)|^2
                 +\frac{n}{d^2}\delta_{p+q,0}\right).
\end{align*}

For the last identity, the matrix-unit supports first force $p=r$,
and the character sum then forces $q=s$. The group components have
the same support. When $p+q=0$, the parenthesis is
$n/d^3+n/d^2=1/d$, using $d^2=n(d+1)$. Thus the displayed Gram
matrix follows without any reflection hypothesis. Its eigenvalues,
after the diagonal normalization by $D$, are the squared singular
values asserted. The product equation gives
$(\sqrt n|\beta(g)|)(\sqrt n|\beta(-g)|)=1$.
For $g\ne0$, the reflection equation is equivalent to
$|\beta(g)|^2=1/n$ by the same product equation; at zero it follows
from $\beta(0)\in\mathbb R$.
\end{proof}

\begin{theorem}[Algebraic reconstruction]\label{thm:algebraic-reconstruction}
  Let $ G $ be a finite abelian group of order $ n $.
Let $K$ be a field in which $n=|G|$ is invertible, and let
$(d_\varrho,\chi,a,c,\beta)$ be a $K$-valued solution of the TPE system for $G$. Then the scalar data reconstruct a spherical fusion category
over $K$ with fusion rules $G+n$.
\end{theorem}
\begin{proof}
    Under the strict monoidal and pivotal conventions and the trivial Frobenius--Schur indicator assumption of Section (3), the tetrahedral table constructed from a TPE solution satisfies the pentagon equations by \cite[Theorem~1.1 and Theorem~5.3]{LPR}; hence the standard fusion-category reconstruction gives the required spherical near-group fusion category.
\end{proof}

\Needspace{8\baselineskip}
\begin{corollary}[Unitary reconstruction]\label{thm: reconstruction}
Work over $K=\mathbb C$. {Let $G$ be a finite abelian group of order $n$}, $d_\varrho=(n+\sqrt{n^2+4n})/2$. let
$(d_\varrho,\chi,a,c,\beta)$ be a solution of the unitary TPE system for $G$. Then the scalar data reconstruct a unitary fusion
category with fusion rules $ G+n $.
\end{corollary}
\begin{proof}
The values of $\chi$ and $a$ have modulus $1$, and the Gauss
normalization gives $|c|=n^{-1/2}$. Together with
\eqref{eq:beta-reality}, these identities identify the algebraic
involution $\tau$ with ordinary complex conjugation on the
tetrahedral data.
Proposition~\ref{prop:channel-singular-values} now gives
\begin{equation}\label{eq:reconstruction-unitarity}
 T^\dagger D_{\rm ch}T=D_{\rm ch},\qquad
 D_{\rm ch}=\operatorname{diag}(d_\varrho^2I_n,d_\varrho^{-1}I_{n^2}).
\end{equation}
Since $D_{\rm ch}$ is positive,
rescaling the channel bases by their positive theta norms makes $T$
unitary. A theta norm here is the Hermitian closure $B(v,Jv)$ after
bending, not the bilinear self-pairing $B(v,v)$. These norms are products of trivalent fusion-basis norms,
times the common outer trace factor $d_\varrho$: a rooted vertex with
output $X$ has squared norm equal to its theta value divided by $d_X$.
Thus the normalization is induced by one choice of orthonormal fusion
bases. In these bases the other changes are the unitary corner matrix
$R$, permutation matrices and phases. The resulting fusion
associators are therefore unitary, giving a unitary fusion category.
\end{proof}

\section{Metric criteria and a Galois conjecture for unitarity}\label{sec:unitarity}
\subsection{Metric criteria for unitarity}
We use the channel singular values from
Proposition~\ref{prop:channel-singular-values} to measure failure of
Hermitian reflection. Throughout this subsection the coefficient field is
$\mathbb C$ and $d=d_+=(n+\sqrt{n^2+4n})/2$.

\begin{corollary}[A single nonnegative unitarity defect]\label{cor:unitarity-defect}
Under the assumptions of Proposition~\ref{prop:channel-singular-values}, put
\[
 \mathfrak d(\beta)=\sum_{g\in G}|\beta(g)|^2-(1-d^{-1}).
\]
Then
\[
 \mathfrak d(\beta)=\frac12\sum_{g\in G}
 (|\beta(g)|-|\beta(-g)|)^2\ge0,
 \qquad
 \|U\|_{\mathrm F}^2-(n+n^2)=n^2\mathfrak d(\beta).
\]
The following are equivalent: scalar reflection, $\mathfrak d(\beta)=0$,
and unitarity of $U$. Reflection is automatic at every $g\in G[2]$.
Consequently every positive-root scalar solution on an elementary
abelian $2$-group reconstructs a unitary fusion category.
\end{corollary}
\begin{proof}
Expand the squares using
$|\beta(g)||\beta(-g)|=1/n$ for $g\ne0$ and
$(n-1)/n+d^{-2}=1-d^{-1}$. Summing the squared singular values in
Proposition~\ref{prop:channel-singular-values} gives the norm formula.
At an involution $g\ne0$, the product identity gives
$|\beta(g)|^2=1/n$ directly. Unitary reconstruction then follows from
Corollary~\ref{thm: reconstruction}.
\end{proof}

This automatic reflection statement does not assert existence of scalar
solutions. For $G=(C_2)^r$ with $r\ge3$, such solutions cannot exist:
algebraic reconstruction would contradict
\cite[Theorem~1.1]{Schopieray}. The corollary identifies the Hermitian
condition wherever a positive-root tuple exists.

\begin{theorem}[A uniform gap from the unitary locus]
\label{thm:unitarity-gap}
Let $G$ be a finite abelian group of order $n$, and let
$(d,\chi,a,c,\beta)$ satisfy the TPE system over
$\mathbb C$, with $d=(n+\sqrt{n^2+4n})/2$. Put
\[
 p_g=n|\beta(g)|^2,\qquad t=p_0=\frac1{d+1},\qquad
 q_g=\sum_{k\in G}p_kp_{g-k}.
\]
Then
\begin{equation}\label{eq:positive-convolution}
 p_gp_{-g}=1,\qquad q_g=p_gq_{-g}\qquad(g\ne0).
\end{equation}
Let $\kappa>1$ be the unique root in $(1,\infty)$ of
$X^3-2X^2-X+1=0$, so $\kappa=2.24697960\ldots$.
Either the tuple satisfies Hermitian reflection, or
\begin{equation}\label{eq:unitarity-gap}
 \max_{g\ne0}\sqrt n\,|\beta(g)|>\sqrt\kappa.
\end{equation}
In particular, the coarse inequalities
$\sqrt n\,|\beta(g)|\le\sqrt2$ for all $g\ne0$
already imply Hermitian reflection and unitary reconstruction.
More precisely, every nonunitary tuple must have $n\ge5$ and
\begin{equation}\label{eq:unitarity-gap-finite-n}
 M^{2/(n-2)}<(M-1)(1-M^{-2}),\qquad
 M=\max_{g\ne0}p_g.
\end{equation}
In particular, every positive-root algebraic scalar tuple with
$n\le4$ satisfies Hermitian reflection.
\end{theorem}
\begin{proof}
Since $\chi$ and $a$ take values of modulus $1$, Equation~\eqref{eq:betapm}
gives
$p_gp_{-g}=1$ for $g\ne0$, and $\beta(0)=-1/d$ gives
$t=n/d^2=1/(d+1)\in(0,1)$.
For fixed $g\ne0$, taking absolute squares of the cubic TPE and summing
over $h$, character orthogonality gives
$$
\frac1n\sum_{k\in G}
|\beta(g-k)|^2|\beta(k)|^2
=
|\beta(g)|^2
\sum_{h\in G}
|\beta(-h)|^2|\beta(h-g)|^2.
$$
Multiplying by $n^3$ yields the second identity
in~\eqref{eq:positive-convolution}.

If every nonzero $p_g$ is one, the quadratic equation gives
$\overline{\beta(g)}=\beta(-g)/a(g)$, so Hermitian reflection holds. Otherwise set $M=\max_{g\ne0}p_g>1$ and choose
$g$ with $p_g=M$. Then $g\ne-g$, hence $n\ge3$.
Put $\delta=M-1$, $e_0=0$, and $e_k=p_k-1$ for $k\ne0$.
Every nonzero $p_k$ lies in $[M^{-1},M]$, and
\[
 |e_k|\le\delta,\qquad e_{-k}=-e_k/p_k.
\]
Expanding $p=\mathbf1+(t-1)\delta_0+e$ gives, for $g\ne0$,
\[
 q_g=n-2+2t+2\sum_k e_k+2(t-1)e_g+(e*e)_g.
\]
Since $q_g=Mq_{-g}$ and $e_g-e_{-g}=\delta(1+M^{-1})$,
\begin{align*}
 (e*e)_g-(e*e)_{-g}
 &=\delta q_{-g}+2(1-t)\delta(1+M^{-1})\\
 &>\delta(n-2)M^{2/(n-2)}.
\end{align*}
To obtain the last inequality, retain the $n-2$ terms of $q_{-g}$
with $k\ne0,-g$. Their product is $M^2$: indeed,
$\prod_{k\ne0}p_k=1$ and removing $p_{-g}=M^{-1}$ leaves
product $M$, in each of the two factors. The arithmetic--geometric
mean inequality now gives $q_{-g}>(n-2)M^{2/(n-2)}$.
On the other hand, reindexing the second convolution gives
\[
 (e*e)_g-(e*e)_{-g}
 =\sum_{k\ne0,g}e_ke_{g-k}
       \left(1-\frac1{p_kp_{g-k}}\right).
\]
For $x,y\in[M^{-1},M]$ one has
\[
 |(x-1)(y-1)(1-(xy)^{-1})|
 \le\delta^2(1-M^{-2}).
\]
Indeed, its left side equals
\[
 \left|\frac{x-1}{\sqrt x}\right|
 \left|\frac{y-1}{\sqrt y}\right|
 \left|\frac{xy-1}{\sqrt{xy}}\right|,
\]
whose first two factors are at most $\delta/\sqrt M$ and whose
third factor is at most $M-M^{-1}$. Consequently the convolution
difference is at most $(n-2)\delta^2(1-M^{-2})$. Combining this
with its strict lower bound proves~\eqref{eq:unitarity-gap-finite-n}.
Its left side is greater than one, so
\[
 1<(M-1)(1-M^{-2}),\qquad M^3-2M^2-M+1>0.
\]
The function $(M-1)(1-M^{-2})$ is strictly increasing for $M>1$,
which gives $M>\kappa$. Since $\kappa>2$, the simpler sufficient
bound follows. For $n=3,4$, the left side
of~\eqref{eq:unitarity-gap-finite-n} is at least $M$, whereas the
right side is less than $M$, a contradiction. For $n=1,2$ all
nonzero elements are involutions and already have $p_g=1$.
\end{proof}

\begin{corollary}[A coarse scalar bound for unitarity]
\label{cor:coarse-defect-bound}
Under the assumptions of Theorem~\ref{thm:unitarity-gap}, either
$\mathfrak d(\beta)=0$ or
\[
 \mathfrak d(\beta)>
 \frac{\kappa+\kappa^{-1}-2}{n}>\frac1{2n}.
\]
Consequently the single bound
\begin{equation}\label{eq:coarse-defect-bound}
 \sum_{g\in G}|\beta(g)|^2\le1-\frac1d+\frac1{2n}
\end{equation}
forces Hermitian reflection and unitary reconstruction. Equivalently,
one may check $\|U\|_{\mathrm F}^2\le n+n^2+n/2$.
\end{corollary}
\begin{proof}
If reflection fails, choose $g$ realizing $M>\kappa>2$ in
Theorem~\ref{thm:unitarity-gap}. Since $g\ne-g$, its reciprocal pair
contributes $(M+M^{-1}-2)/n$ to the nonnegative defect of
Corollary~\ref{cor:unitarity-defect}. Monotonicity of $x+x^{-1}-2$
for $x>1$ gives both strict bounds; that corollary's norm identity
gives the last formulation.
\end{proof}

\subsection{A unitary Galois conjecture}

Set
\begin{equation}\label{eq:arithmetic-normalization}
 \Delta_n=n(n+4),\qquad
 \varepsilon_n=\frac{n+2+\sqrt{\Delta_n}}2,\qquad
 E_n(k)=s(k)F(k)\quad(0\leq k<n).
\end{equation}
Thus $E_n(0)=\sqrt{\varepsilon_n}>0$ and
$d_-=\varepsilon_n^{-1}-1$, $d_+=\varepsilon_n-1$.
Write $S_2$ for the Shintani double sine and $\Phi_{S,0,0}$ for the
modular quantum dilogarithm of \cite[\S\S2.1--2.3]{RadchenkoWheeler}.
Their positive-fixed-point values are denoted by $F_A^+(\mathbf u)$;
$\operatorname{shin}_{A}^{\mathbf r}$ denotes the Shintani--Faddeev
real-multiplication value of \cite{Kopp,AFK}.

\begin{proposition}[Real-multiplication normalization]
\label{prop:rm-normalization}
Let
\[
 A_n=\begin{pmatrix}n+2&-1\\1&0\end{pmatrix},\qquad
 \mathbf r_{n,k}=\binom{k/n}{k/n},\qquad
 \mu_n=e^{\pi i(n-1)/12}.
\]
For every $0\leq k<n$,
\begin{equation}\label{eq:rm-value}
 E_n(k)=F_{A_n}^+\bigl((0,-k)\bigr).
\end{equation}
For $1\leq k<n$ this value also equals
$\mu_n\operatorname{shin}_{A_n}^{\mathbf r_{n,k}}
(\varepsilon_n)^{-1}$.
\end{proposition}
\begin{proof}
At zero both sides of \eqref{eq:rm-value} are
$\sqrt{\varepsilon_n}$, by Proposition~\ref{prop:F-identities} and
\cite[equation~(2)]{RadchenkoWheeler}. For $1\leq k<n$, the matrix
$A_n$ fixes $\varepsilon_n$ and satisfies
$(A_n-I)\mathbf r_{n,k}=(k,0)^{\mathsf t}$, so
$A_n\in\Gamma_{\mathbf r_{n,k}}$.
Scaling the gamma periods and using the standard double-sine/Faddeev
conversion \cite[\S2.3]{RadchenkoWheeler} gives, with
$z_{n,k}=k(\varepsilon_n-1)/n$,
\[
 F(k)=S_2(1+z_{n,k};\varepsilon_n,1)
 =\Phi_{S,0,0}(z_{n,k};\varepsilon_n)
   e^{P_{\varepsilon_n}(1+z_{n,k})},
\]
where
\[
 P_\varepsilon(Z)=\frac{\pi i}{2\varepsilon}
 \left(Z^2-(\varepsilon+1)Z+
       \frac{1+3\varepsilon+\varepsilon^2}{6}\right).
\]
Since $(\varepsilon_n-1)^2=n\varepsilon_n$,
\[
 P_{\varepsilon_n}(1+z_{n,k})
 =\pi i\left(\frac{n-1}{12}-\frac{k(n-k)}{2n}\right),
\]
and therefore $s(k)F(k)=\mu_n\Phi_{S,0,0}
(z_{n,k};\varepsilon_n)$.

Write $T=\left(\begin{smallmatrix}1&1\\0&1\end{smallmatrix}\right)$
and $S=\left(\begin{smallmatrix}0&-1\\1&0\end{smallmatrix}\right)$.
Then $A_n=T^{n+2}S$, and the eta transformation laws give
\[
 \mu_{A_n}=e^{\pi i(n+2)/12}e^{-\pi i/4}=\mu_n.
\]
On the upper half-plane the defining products
\cite[equation~(1)]{RadchenkoWheeler} give
$\Phi_{A_n,0,0}(z;\tau)=\Phi_{S,0,0}(z;\tau)$:
both have denominator parameter $e^{-2\pi i/\tau}$ and argument
$z/\tau$. Meromorphic continuation gives the same identity at
$\tau=\varepsilon_n$. For $\mathbf u=(0,-k)$ the fixed-point
argument is
\[
 \frac{u_1\varepsilon_n+u_2}{\varepsilon_n^{-1}-1}
 =z_{n,k},
\]
which proves \eqref{eq:rm-value}. Finally the reciprocal conversion to
$\operatorname{shin}$ has correction factor $1$, because the second
coordinate of $(I-A_n)\mathbf r_{n,k}$ is zero
\cite[Definition~4.18 and Proposition~4.19]{Kopp}.
\end{proof}

\begin{proposition}[A finite coefficient field]
\label{prop:finite-coefficient-field}
The field
\[
 L_n=\mathbb Q\bigl(\zeta_{4n},E_n(k):k\in C_n\bigr)
\]
is a number field, stable under ordinary complex conjugation. It
contains all scalar coefficients of the TPE system
and $K_n=\mathbb Q(\sqrt{\Delta_n})$.
Exactly $[L_n:K_n]$ embeddings $\iota:L_n\hookrightarrow\mathbb C$
send $d_-$ to $d_+$.
\end{proposition}
\begin{proof}
\cite[Theorem~1]{RadchenkoWheeler} give algebraicity of $E_n(k)$;
at zero this also follows from $E_n(0)^2=\varepsilon_n$.
Moreover
\[
 \sqrt{\Delta_n}=2E_n(0)^2-(n+2)\in L_n.
\]
The phases $s(k)$ lie in $\mathbb Q(\zeta_{4n})$, as does
$\sqrt n$ by
\begin{align*}
  \sum_{g\in C_n}^{} a(g) = \sqrt{n}e^{\pi i(n-1)/4}.
\end{align*} 
Its phase is a power of
$i$ for odd $n$, and $\zeta_8\in\mathbb Q(\zeta_{4n})$ for even $n$.
Hence $\beta(k)=E_n(k)/\sqrt n$, $a(k)$, $\chi(g,h)$, $d_-$,
and $c=\beta(0)/\sum_k\beta(k)$ all belong to $L_n$.
Positivity of $F(k)$ and $s(k)^2=a(k)$ give
$\overline{E_n(k)}=a(k)^{-1}E_n(k)$, proving conjugation stability.
Finally each embedding of $K_n$ has $[L_n:K_n]$ extensions to $L_n$,
because number fields are separable.
\end{proof}

\begin{corollary}\label{cor:all-cyclic-neargroup}
For every $n\geq1$, the fusion ring $C_n+n$ has a spherical
categorification with noninvertible simple dimension $d_-$.
Every embedding $\iota:L_n\hookrightarrow\mathbb C$ with
$\iota(d_-)=d_+$ transports it to a pseudo-unitary spherical
categorification.
\end{corollary}
\begin{proof}
Apply Theorem~\ref{thm:algebraic-reconstruction} over $L_n$ and then
extend scalars along the chosen embedding. At the positive root the
simple dimensions are $1,d_+$, so the global dimension equals the
Frobenius--Perron dimension
\cite[Definition~9.4.4]{EGNO15}.
\end{proof}

\begin{proposition}[Unitary Galois criterion and sufficient bounds]
\label{prop:unitary-galois-criterion}
\label{cor:unit-circle-reformulation}
\label{cor:coarse-galois-unitarity}
Let $\iota:L_n\hookrightarrow\mathbb C$ satisfy
$\iota(d_-)=d_+$. The transported tuple satisfies the unitary scalar
condition of Corollary~\ref{thm: reconstruction} precisely when
\begin{equation}\label{eq:unitary-galois-criterion}
 \overline{\iota(\beta(g))}
 =\frac{\iota(\beta(-g))}{\iota(a(g))}\qquad(g\in C_n),
\end{equation}
or, equivalently, when
\begin{equation}\label{eq:unit-circle}
 |\iota(F(k))|=1\qquad(1\leq k<n).
\end{equation}
Either of the following weaker bounds already implies these conditions:
\[
 |\iota(F(k))|\leq\sqrt2\quad(1\leq k<n),
 \qquad\text{or}\qquad
 \sum_{k=1}^{n-1}|\iota(F(k))|^2\leq n-\frac12.
\]
In each case the specifically reconstructed positive-root category
admits a unitary structure.
\end{proposition}
\begin{proof}
The scalar identities survive under $\iota$. The transported values
of $\chi,a$ are roots of unity, and character orthogonality and the
Gauss equation give $|\iota(c)|=n^{-1/2}$. Thus the algebraic
reflection agrees with ordinary complex conjugation on those
parameters. The remaining condition is
\eqref{eq:unitary-galois-criterion}.
For $1\leq k<n$, set $S_k=\iota(s(k))$, $X_k=\iota(F(k))$
and $r=\iota(\sqrt n)\in\{\sqrt n,-\sqrt n\}$.
Since $s(n-k)=s(k)$, $F(n-k)=F(k)^{-1}$ and $a(k)=s(k)^2$,
the two sides of \eqref{eq:unitary-galois-criterion} are
\[
 r^{-1}S_k^{-1}\overline{X_k}
 \quad\text{and}\quad r^{-1}S_k^{-1}X_k^{-1}.
\]
They agree exactly when $|X_k|=1$. At zero the equation is automatic,
because $\iota(\beta(0))=-d_+^{-1}$ is real.
The first sufficient bound follows from
Theorem~\ref{thm:unitarity-gap}, since
$\sqrt n\,|\iota(\beta(k))|=|\iota(F(k))|$ for $k\ne0$.
For the second, $d_+^{-2}=n^{-1}-d_+^{-1}$ turns the bound into
\eqref{eq:coarse-defect-bound}; apply
Corollary~\ref{cor:coarse-defect-bound}.
Finally use Corollary~\ref{thm: reconstruction}.
\end{proof}

\begin{conjecture}[Cyclic unitary Galois conjecture]
\label{conj:unitary-galois}
For every $n\geq2$ and every embedding
$\iota:L_n\hookrightarrow\mathbb C$ with
$\iota(\sqrt{\Delta_n})=-\sqrt{\Delta_n}$,
\[
 |\iota(E_n(k))|=1\qquad(1\leq k<n).
\]
\end{conjecture}
This is a finite set of embedding assertions by
Proposition~\ref{prop:finite-coefficient-field}, equivalently the same
statement for automorphisms of $\overline{\mathbb Q}$, since every
embedding of a number field extends to such an automorphism.
One common embedding must be used for all $k$; separate choices of
conjugates need not extend to an embedding of $L_n$.
The existence of one suitable embedding for each $n$ already gives a
unitary category by Proposition~\ref{prop:unitary-galois-criterion},
as multiplication by the root of unity $\iota(s(k))$ does not
change a modulus. The every-embedding assertion is, by
Proposition~\ref{prop:rm-normalization}, the cyclic specialization of
\cite[Conjecture~1.36(2)]{AFK}. It follows conditionally from the
order-one abelian Stark conjecture for real quadratic fields by
\cite[Theorem~2.20(2)]{AFK}; see also
\cite{Stark1976,Stark1980,TateStark}. This uses the predicted Artin
action, not merely algebraicity. Part~(1) of that theorem places a
positive power of the real-multiplication value in an abelian extension;
it does not assert this placement for the value itself. The finer
placement of the values themselves in ray class fields of the possibly
nonmaximal order $\mathbb Z[\varepsilon_n]$ uses the stronger monoid
form of the Stark conjecture
\cite[Conjecture~2.9 and Theorem~2.21]{AFK}.

\begin{proposition}\label{prop:small-rank-unitarity}
The unit-circle conclusion of Conjecture~\ref{conj:unitary-galois}
holds for $1\leq n\leq6$. Consequently the positive-root categories
admit unitary structures for these orders.
\end{proposition}
\begin{proof}
For $n\leq4$, Theorem~\ref{thm:unitarity-gap} applies to every
positive-root scalar solution.

For $n=5$, Proposition~\ref{prop:rm-normalization} identifies our
solution with the cycle-$(7)$ example of
\cite[\S4.3]{RadchenkoWheeler}. Its nonzero values are products of
roots of unity and roots of
\[
 P(Z)=Z^4-Z^3-\frac{3+3\sqrt5}{2}Z^2-Z+1.
\]
After any embedding reversing $\sqrt5$, divide the conjugate
polynomial by $Z^2$ and put $W=Z+Z^{-1}$. This gives
\[
 W^2-W-\frac{7-3\sqrt5}{2}=0,
 \qquad W=\frac{1\pm\sqrt{15-6\sqrt5}}2.
\]
Both traces lie in $(-2,2)$, so all roots have modulus one.

For $n=6$, put
\[
 r=F(0)=\sqrt{4+\sqrt{15}},\qquad
 X=F(1)+F(5),\qquad Y=F(2)+F(4).
\]
Reciprocity gives $F(5)=F(1)^{-1}$, $F(4)=F(2)^{-1}$ and
$F(3)=1$. Proposition~\ref{prop:cubic-reduction} gives directly
\[
 \sum_{k=0}^5s(k)F(k)=\sqrt6\,r\,e^{5\pi i/12}
 =r+e^{5\pi i/12}X+e^{2\pi i/3}Y+e^{3\pi i/4}.
\]
Equating real and imaginary parts after multiplication by
$e^{-5\pi i/12}$ yields
\begin{equation}\label{eq:n6-exact-traces}
 X=\frac{2+\sqrt{15}+\sqrt3}{2},\qquad
 Y=\frac{(\sqrt3+1)r-\sqrt6}{2},
\end{equation}
using $\sqrt6\,r=3+\sqrt{15}$ in the original embedding.
If $\iota(\sqrt{15})=-\sqrt{15}$, then
\[
 \iota(X)=\frac{2-\sqrt{15}\pm\sqrt3}{2}\in(-2,2).
\]
Also $\iota(r)^2=4-\sqrt{15}>0$, so $\iota(r)$ is real and
every possible choice of the remaining signs satisfies
\[
 |\iota(Y)|\leq
 \frac{(\sqrt3+1)\sqrt{4-\sqrt{15}}+\sqrt6}{2}<2.
\]
Indeed $\sqrt{15}>15/4$, $\sqrt3<7/4$ and $\sqrt6<5/2$
bound the last expression by $31/16<2$.
The conjugate traces of both reciprocal pairs lie in $(-2,2)$,
and hence all four conjugated values lie on the unit circle;
the midpoint remains $\iota(F(3))=1$.
\end{proof}

\begin{remark}[Arithmetic content]\label{rem:arithmetic-content}
\label{rem:galois}
The algebraic scalar equations survive under every field embedding;
ordinary complex conjugation need not. On $\iota(L_n)$ the
transported involution is
$\iota\circ\overline{(\,\cdot\,)}\circ\iota^{-1}$, which can
differ from ordinary conjugation. Likewise the gamma integral
evaluates the original real embedding, not $\iota(F(k))$:
formally conjugating its periods is analytic continuation, during
which the separating contour crosses poles, rather than Galois
transport. The scalar product identity gives only
$\iota(\beta(g))\iota(\beta(-g))=\iota(a(g))/n$ for $g\ne0$;
reflection requires $|\iota(\beta(g))|^2=1/n$.
This is the real/unitary distinction described in
\cite[\S6]{RadchenkoWheeler}.
\end{remark}

\begin{remark}[Hilbert's twelfth problem]\label{rem:hilbert-twelfth}
For $K_n=\mathbb Q(\sqrt{n(n+4)})$, Proposition~\ref{prop:rm-normalization}
identifies $E_n(k)$ for $1\leq k<n$, up to roots of unity and
inversion, with real-multiplication values of the Shintani--Faddeev
modular cocycle. Kopp's limit formula \cite[Theorem~1.1]{Kopp}
relates squares of these values, with explicit root-of-unity factors,
to exponentials of derivatives at $s=0$ of differences of ray class
partial zeta functions. The Stark conjectures predict their
class-field placement and Galois behavior, as explained after
Conjecture~\ref{conj:unitary-galois}; algebraicity alone is already
unconditional by Proposition~\ref{prop:finite-coefficient-field}.
This is the connection with Hilbert's twelfth problem: constructing
abelian extensions through explicit analytic special values
\cite{Stark1976,AFK}. The categorical application uses the predicted
Galois behavior to obtain unitarity; it does not assert that this cyclic
family generates all abelian extensions of $K_n$.
\end{remark}

\begin{ac}
The authors remember Vaughan Jones with gratitude. His questions about
Izumi's equations motivated the categorical approach of this paper.
Zhengwei Liu acknowledges support from the Beijing Natural Science
Foundation Key Program (Grant Z220002) and the Beijing Natural Science
Foundation (Grant Z221100002722017).
Fan Lu acknowledges support from the Beijing Natural Science Foundation
(Grant Z221100002722017).
Sebastien Palcoux acknowledges support from the National Natural
Science Foundation of China (Grant 12471031).
\end{ac}

\textbf{AI-use disclosure.} The AI tool was used as a research assistant.
The authors retain full responsibility for the final version of the manuscript.

\end{document}